\documentclass{amsart}
\usepackage[margin=1.925cm]{geometry}
\usepackage{amsthm}
\usepackage{amsmath}
\usepackage{amssymb}
\usepackage[all]{xy}
\usepackage{graphicx}
\usepackage{indentfirst}
\usepackage{bm}
\usepackage{mathrsfs}
\usepackage{latexsym}
\usepackage{color}
\usepackage{hyperref}
\usepackage{microtype}
\usepackage{tikz}
\usepackage{tikz-cd}
\usepackage{enumerate}
\usepackage{comment}
\usepackage{braket}
\usepackage{manfnt}
\usepackage{hyperref}
\usepackage[utf8]{inputenc}
\usepackage{relsize}
\usepackage{listings}
\usepackage{textcomp}
\usepackage{spverbatim}
\usepackage[T1]{fontenc}
\usepackage{lmodern}
\usepackage{bbm}
\usepackage{dsfont}
\usepackage{mathtools}
\usepackage{enumitem}
\allowdisplaybreaks
\numberwithin{equation}{section}

\definecolor{DarkGreen}{RGB}{0,100,0}      
\definecolor{DarkBlue}{RGB}{0,0,139}       
\definecolor{DarkYellow}{RGB}{204,204,0}   
\definecolor{DarkRed}{RGB}{139,0,0}        
\definecolor{DarkOrange}{RGB}{204,90,0}  

\begin{document}

\theoremstyle{plain}
\newtheorem{theorem}{Theorem}[section]
\newtheorem{main}{Main Theorem}
\newtheorem{proposition}[theorem]{Proposition}
\newtheorem{corollary}[theorem]{Corollary}
\newtheorem{lemma}[theorem]{Lemma}
\newtheorem{conjecture}[theorem]{Conjecture}
\newtheorem{claim}[theorem]{Claim}
\newtheorem{fact}[theorem]{Fact}
\newtheorem{question}[theorem]{Question}
\newtheorem{statement}[theorem]{Statement}
\newtheorem{speculation}[theorem]{Speculation}
\newtheorem*{st}{Statements}	

\theoremstyle{definition}
\newtheorem{definition}[theorem]{Definition}
\newtheorem{notation}[theorem]{Notation}
\newtheorem{convention}[theorem]{Convention}
\newtheorem{example}[theorem]{Example}
\newtheorem{remark}[theorem]{Remark}
\newtheorem*{ac}{Acknowledgements}	

\newcommand{\Rep}{\mathrm{Rep}}
\newcommand{\Corep}{\mathrm{Corep}}
\newcommand{\ch}{\mathrm{ch}}
\newcommand{\VVec}{\mathrm{Vec}}
\newcommand{\PSU}{\mathrm{PSU}}
\newcommand{\SO}{\mathrm{SO}}
\newcommand{\SU}{\mathrm{SU}}
\newcommand{\FPdim}{\mathrm{FPdim}}
\newcommand{\PSL}{\mathrm{PSL}}

\newcommand{\mC}{\mathcal{C}}
\newcommand{\mB}{B}
\newcommand{\hc}{\Hom_{\mC}}
\newcommand{\one}{{\bf 1}} 
\newcommand{\id}{{\mathrm{id}}} 
\newcommand{\spec}{i_0} 
\newcommand{\Sp}{i_0} 
\newcommand{\SpecS}{I_{s}} 
\newcommand{\SpS}{I_{s}'} 
\newcommand{\field}{\mathbbm{k}} 
\newcommand{\white}{\textcolor{red}{\bullet}} 
\newcommand{\black}{\bullet} 
\newcommand{\ev}{{\mathrm{ev}}} 
\newcommand{\coev}{{\mathrm{coev}}} 
\newcommand{\tr}{{\mathrm{tr}}} 
\newcommand{\Tr}{{\mathrm{Tr}}} 
\newcommand{\End}{{\mathrm{End}}}
\newcommand{\Hom}{{\mathrm{Hom}}}
\newcommand{\customcite}[2]{\cite[#2]{#1}}
\newcommand{\Irr}{\rm Irr}
\newcommand{\inner}[2]{\langle #1, #2 \rangle_A}
\newenvironment{restatetheorem}[1]
{
  \par\addvspace{0.25\baselineskip}%
  \noindent\textbf{Theorem~\ref{#1}.}\thinspace\itshape
}
{
  \par\addvspace{0.25\baselineskip}
}
\newenvironment{restatecorollary}[1]
  {
   \par\addvspace{0.25\baselineskip}
   \noindent\textbf{Corollary \ref{#1}.}\thinspace \itshape
  }
  {
   \par\addvspace{0.25\baselineskip}
  }  
\newenvironment{restateproposition}[1]
  {
   \par\addvspace{0.25\baselineskip}
   \noindent\textbf{Proposition \ref{#1}.}\thinspace \itshape
  }
  {
   \par\addvspace{0.25\baselineskip}
  }   

\newcommand{\sebastien}[1]{\textcolor{blue}{#1}}

\def\ZZ{{\mathbb Z}}
\def\NN{{\mathbb N}}
\def\RR{{\mathbb R}}
\def\CC{{\mathbb C}}
\def\QQ{{\mathbb Q}}
\def\HH{{\mathbb H}}
\def\EE{{\mathbb E}}

\def\cL{{\mathcal L}}
\def\cS{{\mathcal S}}
\def\cN{{\mathcal N}}

\def\Aut{\operatorname{Aut}}
\def\csupp{\operatorname{csupp}}

\setcounter{MaxMatrixCols}{20}

\title{Frobenius subalgebra lattices in tensor categories}

\author{Mainak Ghosh}
\address{M. Ghosh, Hetao Institute of Mathematics and Interdisciplinary Sciences, Futian District, Shenzhen, China}
\email{ghosh.main@gmail.com}

\author{Sebastien Palcoux}
\address{S. Palcoux, Beijing Institute of Mathematical Sciences and Applications, Huairou District, Beijing, China}
\email{sebastienpalcoux@gmail.com}
\urladdr{https://sites.google.com/view/sebastienpalcoux}


\maketitle

\begin{abstract} 
This paper generalizes Watatani's finiteness theorem for intermediate subfactors to a wide class of monoidal categories. We characterize the sublattices of Frobenius subalgebra posets in abelian monoidal categories by introducing a notion of ambient selfduality. By extending several key results---such as the planar algebraic exchange relation and Landau's theorems---to linear monoidal categories, we establish a structural rigidity property for a formal angle associated to every coherent pair of Frobenius subalgebras (that is, whose intersection and sum are ambiently selfdual). Furthermore, within a weak positivity framework, we deduce that such coherent sublattices are finite for any connected Frobenius algebra. This significantly generalizes Watatani's theorem, since the unitary Frobenius subalgebra lattices are inherently coherent through a property termed rigidity invariance.

Applications of this work include a unified framework that encompasses several previously unrelated finiteness results. Specifically, we recover the finiteness of the left coideal subalgebra lattice of a finite-dimensional semisimple Hopf algebra (Etingof-Walton theorem) under a well-supported coherence hypothesis, as well as the finiteness of the intermediate $C^*$-algebra lattice for a finite-index unital irreducible inclusion of $C^*$-algebras (relaxing simplicity in Ino-Watatani theorem) under the $E$-compatibility condition shown to be unavoidable. Furthermore, we present a variety of new applications involving abstract spin chains and vertex operator algebras, alongside speculations on quantum arithmetic that include extensions of Ore's theorem, Euler's totient and sigma functions, and RH.
\end{abstract}
%
\section{Introduction}

A \emph{subfactor} is a unital inclusion of factors. The modern theory of subfactors (type ${\rm II}_1$) was initiated by Vaughan Jones \cite{J83}. In \cite{J80}, it was shown that every finite group $G$ admits an outer action on the hyperfinite ${\rm II}_1$ factor $R$, so that the resulting crossed-product $R \rtimes G$ is again a ${\rm II}_1$ factor. The intermediate subfactor lattice of $ (R \subseteq R \rtimes G) $ is isomorphic to the subgroup lattice of $ G $. Since $ G $ is finite, this lattice is also finite. Generalizing this, Yasuo Watatani \cite{Wat96} proved that any irreducible finite index subfactor has a finite intermediate subfactor lattice, henceforth referred to as \emph{Watatani's theorem}. This paper aims to extend this theorem to the setting of rigid abelian monoidal categories and explore its applications. This closely aligns with the approach in \cite{JP19}, where the authors emphasize the importance of distinguishing the functional analysis aspects of subfactor theory from its categorical components.
%
The original proof of Watatani's theorem in \cite{Wat96} relies on functional analysis. An alternative proof using planar algebra and angles between biprojections was later provided in \cite{BDLR19}. Dave Penneys observed that Watatani's theorem can be reformulated within the framework of unitary tensor categories by using the concept of a Frobenius algebra object in a monoidal category \cite{Lon94, Mug03, Str04, FS08, BKLR15, JP20}.

A \emph{Frobenius algebra} in $\VVec$ is a finite-dimensional unital algebra $ A $ that is isomorphic to its dual $ A^* $ as an $ A $-module. Alternatively, it can be characterized by the existence of a nondegenerate associative bilinear form $ \langle \cdot,\cdot \rangle $, as described in \cite[Chapter IV, \S 1]{SY11}. 
Any complexified fusion ring forms a Frobenius algebra, where $ \langle a, b \rangle $ is the coefficient of the unit summand in $ ab $. The standard invariant of subfactors, as discussed in \cite{JS97}, can be axiomatized as a Frobenius algebra object. For further details, refer to \cite{Mug03}, along with Example \ref{ex:subf} and Remark \ref{rk:subf}.


A key first step toward our goal was formulating a monoidal category version of the \emph{exchange relation} from planar algebras \cite{Bi94, BJ00, Lan02, Liu16}, which we establish in Theorem \ref{thm:exchange}. Building on this, we generalize \emph{Landau's theorem} (Theorem \ref{thm:landau}) to connected Frobenius algebras in any $ \mathbbm{k} $-linear monoidal category. The Frobenius algebra connectedness, which generalizes subfactor irreducibility, plays a crucial role in the proof. 

Equivalence classes of Frobenius subalgebras (Definitions \ref{def:sub} and \ref{def:equi})---capturing the concept of substructure categorically---provide the precise framework needed to define the \emph{Frobenius subalgebra poset} (Definition \ref{def:FrobSubPoset}). A Frobenius algebra generalizes the concept of a subfactor, while these equivalence classes correspond to intermediate subfactors. See Proposition~\ref{prop:biprojection} for the connection with biprojections.

A central issue in extending Watatani's theorem in tensor categories lies in equipping the Frobenius subalgebra poset with a \emph{lattice} structure. While such a structure arises naturally for subgroups and intermediate subfactors, it fails for Frobenius subalgebras in general, as demonstrated by Dave Benson and Will Sawin in~\cite{BenMO2, WillMO1}; the first such instance is detailed in Example~\ref{ex:Ben02}.

To clarify this aspect and address this limitation, we introduce the notion of the \emph{ambient dual} in \S\ref{sub:Ambient}, and show that, in an abelian monoidal category, the Frobenius subalgebras are precisely the ambiently selfdual unital subalgebras (Proposition~\ref{prop:AmbientFrobSub}). This allows us to characterize the sublattices of the Frobenius subalgebra poset:

\begin{restatetheorem}{thm:AmbientSublattice}
Let $X$ be a finite-length Frobenius algebra in an abelian rigid monoidal category. A subposet of its Frobenius subalgebra poset extends to a sublattice with intersection if and only if it preserves ambient selfduality.
\end{restatetheorem}

To obtain a large class of such sublattices, we introduce the concept of \emph{rigidity invariance}, developed in the semisimple case in~\S\ref{sec:RigidInvariant}. This concept requires the choice of a basis~$\mathcal{B}$, thereby providing a \emph{perspective} under which the Frobenius subalgebra poset \emph{collapses} to the Frobenius $\mathcal{B}$-subalgebra sublattice (see Notation~\ref{not:B-sub}). Furthermore, every Frobenius subalgebra belongs to one of these sublattices.

\begin{restatetheorem}{thm:Bsublattice}
For every Frobenius algebra that is semisimple of finite length as an object in a $\mathbb{C}$-linear abelian rigid monoidal category, its Frobenius $\mathcal{B}$-subalgebra poset is a lattice.
\end{restatetheorem}


\begin{restatecorollary}{cor:lattice}
For every Frobenius algebra in a semisimple tensor category over $\mathbb{C}$, its Frobenius $\mathcal{B}$-subalgebra poset is a lattice.
\end{restatecorollary}

The dependence on $\mathcal{B}$ can be eliminated by considering the sublattice of Frobenius subalgebras that are rigid invariant with respect to every basis (see Corollary~\ref{cor:lattice0}). While rigidity invariance is \emph{essential} in general---as demonstrated by Example~\ref{ex:Ben02}---it represents a relaxation of the unitary condition (Proposition~\ref{prop:UnitFrobSub}). Hence:

\begin{restatecorollary}{cor:latticeunitary}
For every unitary Frobenius algebra in a unitary tensor category, its unitary Frobenius subalgebra poset is a lattice.
\end{restatecorollary}

The following problem remains open for a general unitary tensor category~$\mathcal{C}$ (see~\cite[Question 4.11]{GP25}).

\begin{question}\label{q:UnitFrob}
Is every Frobenius subalgebra of a unitary Frobenius algebra in~$\mathcal{C}$ necessarily unitary?
\end{question}

In the setting of Hopf algebras, this problem specializes to the following question (see~\cite{PalMO26b} and Theorem \ref{thm:HopfFrobRep}).

\begin{question}\label{q:UnitFrob2}
Is every left coideal subalgebra of a finite-dimensional $C^*$-Hopf algebra necessarily a $*$-subalgebra?
\end{question}


To establish finiteness results, we generally rely on comparison arguments involving a suitable notion of \emph{positivity} (Definition~\ref{def:PositiveTensor}). Every unitary tensor category is positive (Proposition~\ref{prop:C*positive}), although the converse fails in general (see Remark~\ref{rk:PosVsUnitary}).
Since our work is carried out in the broader setting of linear monoidal categories, we introduce the notion of weak positivity (Definition~\ref{def:weaklyPositiveTensor}), which is the most general relaxation of positivity we have found sufficient for our arguments. We prove that every semisimple tensor category over~$\mathbb{C}$ is weakly positive (Theorem~\ref{thm:semisweakly}).



Inspired by ideas from the arXiv version of \cite{BDLR19} (also adapted to the $C^*$-algebra setting in \cite{BG21}), we introduce in \S\ref{sec:Angle} the notion of a \emph{formal angle} (see Remark \ref{rk:formal}) between two Frobenius subalgebras of a connected Frobenius algebra whose intersection and sum are both ambiently selfdual; such a pair is called \emph{coherent} (Definition \ref{def:CoherentPair}). We also define a \emph{coherent sublattice} (Definition \ref{def:CoherentSublattice}) of the Frobenius subalgebra poset as a sublattice in which every pair is coherent. This framework yields the following generalization of Watatani's finiteness theorem, established in \S\ref{sec:Wat}:
\begin{restatetheorem}{thm:allfinite}
Let $\mathcal{C}$ be a $\mathbb{C}$-linear abelian rigid monoidal category with a linear-simple unit object $\one$ and a quasi-pivotal structure $\phi$ (Definition \ref{def:quasipivotal}). Let $X$ be a finite-length connected Frobenius algebra in $\mathcal{C}$ that is $\phi_X$-weakly positive (Definition \ref{def:weakpositive}), where $\phi_X$ is an algebra isomorphism (Definition \ref{def:FrobMor}), and suppose that $\End_{\mathcal{C}}(X)$ is finite-dimensional. Then every coherent sublattice of the Frobenius subalgebra poset of $X$ is finite.
\end{restatetheorem}


The notions of \emph{finite-length}, \emph{connected}, and \emph{Frobenius algebra} within a monoidal category extend the notions of \emph{finite-index}, \emph{irreducible}, and \emph{subfactor}, respectively. The connected assumption cannot be avoided because there are non-irreducible finite index subfactors with infinite intermediate subfactor lattice (see \cite[page 314]{Wat96}). 

The following corollaries make the original Watatani's theorem increasingly transparent.



\begin{restatecorollary}{cor:allfinite1.6}
Let $X$ be a connected Frobenius algebra in a $\phi$-weakly positive tensor category such that $\phi_X$ an algebra isomorphism. Then every coherent sublattice of the Frobenius subalgebra poset of $X$ is finite.
\end{restatecorollary}



Since the Frobenius $\mathcal{B}$-subalgebra sublattices are coherent (Proposition~\ref{prop:Bcoherent}), they are also finite (Corollary \ref{cor:allfinite1.75b}). As discussed in~\S\ref{sub:RigidNonSS}, rigidity invariance is problematic to formulate directly in the non-semisimple setting. We resolve this via semisimplification in the sense of~\cite{EO22} (see~\S\ref{sec:semi}), leading to Corollary~\ref{cor:allfinite1}.
%
Next, we present an application of Theorem~\ref{thm:allfinite} to positive tensor categories (which need not be unitary or semisimple, as explained in Remark~\ref{rk:PosVsUnitary}).
\begin{restatecorollary}{cor:allfinite2b}
Let $X$ be a connected Frobenius algebra in a positive tensor category. Then every coherent sublattice of the Frobenius subalgebra poset of $X$ is finite.
\end{restatecorollary}

Again, since the notion of a Frobenius $\mathcal{B}$-subalgebra extends that of a unitary Frobenius subalgebra:

\begin{restatecorollary}{cor:allfiniteUnitary}
Let \( X \) be a connected unitary Frobenius algebra in a unitary tensor category. Then its unitary Frobenius subalgebra lattice is finite.
\end{restatecorollary}

Furthermore, using Corollary~\ref{cor:allfinite2b} again and Proposition~\ref{prop:pseudo}, we obtain:
\begin{restatecorollary}{cor:allfinite3}
Let $X$ be a connected Frobenius algebra in a pseudo-unitary fusion category over $\mathbb{C}$. Then every coherent sublattice of the Frobenius subalgebra poset of $X$ is finite.
\end{restatecorollary}
%
 

In \S\ref{sec:IncC*}, \S\ref{sec:hopf}, and \S\ref{sec:other}, we present several applications of Theorem~\ref{thm:allfinite}. In \cite{IW14}, Watatani's theorem for subfactors was extended to irreducible unital inclusions of finite-index simple $C^*$-algebras. In \S\ref{sec:IncC*}, we generalize this further by removing the simplicity assumption. Specifically, we show that the $E$-compatible (Definition~\ref{def:Ecomp}) intermediate $C^*$-algebras of a finite-index unital irreducible inclusion form a finite lattice (Corollary \ref{cor:InoWat}), a result previously obtained by functional-analytic methods in \cite{GK24}. The $E$-compatibility is equivalent to the inclusion map to be adjointable (Proposition~\ref{prop:AdjointEcomp}), which is required in the unitary context of Corollary~\ref{cor:allfiniteUnitary}. Example~\ref{ex:coburn} shows that this condition is necessary, exhibiting a unital irreducible finite-index inclusion with uncountably many intermediates.

In \S\ref{sec:hopf}, we provide a detailed proof that every finite-dimensional Hopf algebra $H$ is a connected Frobenius algebra object in $\Corep(H)$. Moreover, when $H$ is semisimple over $\mathbb{C}$, we show that its Frobenius subalgebras coincide with its left coideal subalgebras (Theorem~\ref{thm:HopfFrobRep}). 
Consequently, Corollary~\ref{cor:allfinite3} recovers \cite[Theorem~3.6]{EW14} on the finiteness of the left coideal subalgebra lattice, under the well-supported coherence hypothesis---equivalent to the nondegeneracy of $\lambda_{L+K}$ for all left coideal subalgebras $L,K \subseteq H$, where $\lambda$ denotes the left integral (Question~\ref{Q:TwoCoideal})---which is established when either $L$ or $K$ is a Hopf subalgebra (Theorem~\ref{thm:NonDegSum}) and in particular implies the coherence of the Hopf subalgebra lattice (Corollary~\ref{cor:HopfSubCoherent}).
Note that semisimplicity is essential for finiteness, even when restricting to Hopf subalgebras alone, as illustrated by the classical Example~\ref{ex:nichols}.
%

In \S \ref{sec:other}, we introduce additional examples of connected Frobenius algebra objects. In \S \ref{sub:canon}, we demonstrate that the canonical Frobenius algebra object in a unimodular multitensor category $ \mathcal{C} $ \cite[\S 7.20]{EGNO15} is connected if and only if $ \mathcal{C} $ is tensor (Proposition \ref{prop:canon}). This result allows us to provide a connected Frobenius algebra structure on $ H^* $ as an object in $ \Rep(H \otimes H^{\rm cop}) $, where $ H $ is a finite-dimensional unimodular Hopf algebra. This provides a class of non-semisimple examples where Corollary \ref{cor:allfinite1} is not applicable (Remark \ref{rk:failpositive}).
In \S \ref{sub:ASC}, we discuss inclusions of abstract spin chains. In \cite{JSW24}, an algebraic model of categorical inclusions was employed to examine the extensions of bounded-spread isomorphisms of symmetric local algebras to quantum cellular automata (QCA), defined on either the full or edge-restricted local operator algebras. We demonstrate that the lattice of categorical inclusions is finite. In \S \ref{sub:voa}, we report a brief discussion with Kenichi Shimizu about the challenge of establishing a Frobenius algebra structure on a vertex operator algebra $V$ as an object within $\Rep(V)$. 

In \S \ref{sec:QA}, we present a collection of open problems and speculations in quantum arithmetic that have emerged from our research and that we plan to investigate further in future work. These developments are enabled by the finiteness of the coherent sublattices. In \S \ref{sub:Euler}, we introduce an analogue of Euler's totient function for connected Frobenius algebra objects in tensor categories, and in \S \ref{sub:Ore}, we propose generalizations of Ore's theorem. We also discuss extensions of the notions of subfactor depth (\S\ref{sub:depth}) and index (\S\ref{sub:Dim}). In \S\ref{sub:RH}, we propose a tensor-categorical generalization of the Riemann Hypothesis, inspired by Guy Robin’s reformulation \cite{Rob84} in terms of the sigma function.


The primary aim of this paper is to further develop the interplay between tensor category theory and subfactor theory. It is intended to serve as an accessible gateway between these two subjects. For a general introduction to tensor categories, we refer the reader to \cite{EGNO15}.

\tableofcontents

\section{Frobenius algebra} \label{sec:Frob}

Let us review some definitions and fundamental results. For the basic definition of monoidal and tensor categories, we refer to \cite{EGNO15}. According to Mac Lane's strictness theorem, we can assume that monoidal categories are strict without any loss of generality. We will use graphical calculus, interpreting diagrams from top to bottom.

\begin{definition}
A \emph{unital algebra} in a monoidal category $\mathcal{C}$ consists of a triple $(X, m, e)$, where $X$ is an object in $\mathcal{C}$, $m$ is a multiplication morphism in $\Hom_{\mathcal{C}}(X \otimes X, X)$, and $e$ is a unit morphism in $\Hom_{\mathcal{C}}(\one, X)$, depicted as follows:
\[ m = \raisebox{-6mm}{
	\begin{tikzpicture}
	\draw[blue,in=-90,out=-90,looseness=2] (-0.5,0.5) to (-1.5,0.5);
	\draw[blue] (-1,-.1) to (-1,-.6);
	\node[left,scale=0.7] at (-1,-.4) {$X$};
	\node[left,scale=0.7] at (-1.6,0.5) {$X$};
	\node[right,scale=0.7] at (-.5,.5) {$X$};
	\end{tikzpicture}} \ \ \ \ \ \ \ \ \ \
	e= \raisebox{-6mm}{
		\begin{tikzpicture}
		\draw [blue] (-0.8,-.6) to (-.8,.6);
		\node at (-.8,.6) {${\color{blue}\bullet}$};
		\node[left,scale=.8] at (-.8,.9) {$\one$};
		\node[left,scale=0.7] at (-.8,-.5) {$X$};
		\end{tikzpicture}}\]
These satisfy the following axioms:
\begin{itemize}
\item (Associativity) \qquad $m \circ (m \otimes \id_X) = m \circ (\id_X \otimes m)$,
\item (Unitality) \qquad $m \circ (e \otimes \id_X) = \id_X = m \circ (\id_X \otimes e)$.
\end{itemize}
These relations are typically represented as follows:

	 \[\text{(associativity)} \ \ \ \ \raisebox{-6mm}{
		\begin{tikzpicture}[rotate=180,transform shape]
		\draw[blue,in=90,out=90,looseness=2] (0,0) to (1,0);
		\draw[blue,in=90,out=90,looseness=2] (0.5,.6) to (-.5,.6);
		\draw[blue] (-.5,.6) to (-.5,0);
		 \draw[blue] (0,1.2) to (0,1.6);
		\end{tikzpicture}}
		= \raisebox{-6mm}{\begin{tikzpicture}[rotate=180,transform shape]
		\draw[blue,in=90,out=90,looseness=2] (0,0) to (1,0);
		\draw[blue,in=90,out=90,looseness=2] (.5,.6) to (1.5,.6);
		 \draw[blue] (1.5,.6) to (1.5,0);
		 \draw[blue] (1,1.2) to (1,1.6);
			\end{tikzpicture}} \]
		
		\[ \text{(unitality)} \ \ \ \ 
		\raisebox{-4mm}{
			\begin{tikzpicture}[rotate=180,transform shape]
			\draw[blue,in=90,out=90,looseness=2] (0,0) to (1,0);
			\node at (1,0) {$\textcolor{blue}{\bullet}$};
			\draw[blue] (.5,.6) to (.5,1.2);
			\end{tikzpicture}}
		=
		\raisebox{-4mm}{
			\begin{tikzpicture}
			\draw[blue] (0,0) to (0,1.2);
			\end{tikzpicture}} 
		= 
			\raisebox{-4mm}{
			\begin{tikzpicture}[rotate=180,transform shape]
			\draw[blue,in=90,out=90,looseness=2] (0,0) to (1,0);
			\node at (0,0) {$\textcolor{blue}{\bullet}$};
			\draw[blue] (.5,.6) to (.5,1.2);
			\end{tikzpicture}} \]
\end{definition}

\begin{definition} \label{def:coalgebra}
A \emph{counital coalgebra} in a monoidal category $\mathcal{C}$ is defined as a triple $(X, \delta, \epsilon)$, where $X$ is an object in $\mathcal{C}$, $\delta$ is a comultiplication morphism in $\Hom_{\mathcal{C}}(X, X \otimes X)$, and $\epsilon$ is a counit morphism in $\Hom_{\mathcal{C}}(X, \one)$, depicted as:
\[ \delta = \raisebox{-6mm}{
		\begin{tikzpicture}
		\draw[blue,in=90,out=90,looseness=2] (-0.5,0.5) to (-1.5,0.5);
		\draw[blue] (-1,1.1) to (-1,1.6);
		\node[left,scale=0.7] at (-1,1.6) {$X$};
		\node[left,scale=0.7] at (-1.6,0.5) {$X$};
		\node[right,scale=0.7] at (-.5,.5) {$X$};
		\end{tikzpicture}} \ \ \ \ \ \ 
		\epsilon = \raisebox{-8mm}{
		\begin{tikzpicture}
		\draw [blue] (-0.8,-.6) to (-.8,.6);
		\node at (-.8,-.6) {${\color{blue}\bullet}$};
		\node[left,scale=0.7] at (-.8,.6) {$X$};
		\node[scale=.8] at (-.8,-.9) {$\one$};
		\end{tikzpicture}} \]
These satisfy the following axioms:
\begin{itemize}
\item (Coassociativity) \qquad $(\id_X \otimes \delta) \circ \delta = (\delta \otimes \id_X) \circ \delta$,
\item (Counitality) \qquad $(\epsilon \otimes \id_X) \circ \delta = \id_X = (\id_X \otimes \epsilon) \circ \delta$.
\end{itemize}
These relations are typically represented as follows:
	\[\text{(coassociativity)} \ \ \ \ \ \ \ \raisebox{-6mm}{
		\begin{tikzpicture}
		\draw[blue,in=90,out=90,looseness=2] (0,0) to (1,0);
		\draw[blue,in=90,out=90,looseness=2] (0.5,.6) to (-.5,.6);
		\draw[blue] (-.5,.6) to (-.5,0);
		\draw[blue] (0,1.2) to (0,1.6);
		\end{tikzpicture}}
	= \raisebox{-6mm}{\begin{tikzpicture}
		\draw[blue,in=90,out=90,looseness=2] (0,0) to (1,0);
		\draw[blue,in=90,out=90,looseness=2] (.5,.6) to (1.5,.6);
		\draw[blue] (1.5,.6) to (1.5,0);
		\draw[blue] (1,1.2) to (1,1.6);
		\end{tikzpicture}} \]
	
	\[ \text{(counitality)} \ \ \ \ 
		\raisebox{-6mm}{
		\begin{tikzpicture}
		\draw[blue,in=90,out=90,looseness=2] (0,0) to (1,0);
		\node at (0,0) {$\textcolor{blue}{\bullet}$};
		\draw[blue] (.5,.6) to (.5,1.2);
		\end{tikzpicture}} 
		\ \ = \ \  
		\raisebox{-4mm}{\begin{tikzpicture}
		\draw[blue] (0,0) to (0,1.2);
		\end{tikzpicture}} 
		\ \ = \ \  
		\raisebox{-6mm}{
		\begin{tikzpicture}
		\draw[blue,in=90,out=90,looseness=2] (0,0) to (1,0);
		\node at (1,0) {$\textcolor{blue}{\bullet}$};
		\draw[blue] (.5,.6) to (.5,1.2);
		\end{tikzpicture}} \]
\end{definition}

\begin{lemma} \label{lem:Dualgebra}
Let $(X,m,e)$ be a unital algebra object in a monoidal category $\mathcal{C}$. If $X$ admits a left dual $X^*$, then $(X^*, m^*, e^*)$ naturally forms a counital coalgebra in $\mathcal{C}$.
\end{lemma}

\begin{proof}
This follows from a straightforward verification, using the identities $(f \circ g)^* = g^* \circ f^*$ and $(f \otimes g)^* = g^* \otimes f^*$:
\[
(\id_{X^*} \otimes m^*) \circ m^* 
= (m \circ (m \otimes \id_X))^* 
= (m \circ (\id_X \otimes m))^* 
= (m^* \otimes \id_{X^*}) \circ m^*,
\]
\[
(e^* \otimes \id_{X^*}) \circ m^* 
= (m \circ (\id_X \otimes e))^* 
= \id_X^* 
= \id_{X^*} 
= (m \circ (e \otimes \id_X))^* 
= (\id_{X^*} \otimes e^*) \circ m^*.
\qedhere
\]
\end{proof}

\begin{definition} \label{def:Frob}
A \emph{Frobenius algebra} in a monoidal category $\mathcal{C}$ is a quintuple $(X, m, e, \delta, \epsilon)$, where $(X, m, e)$ is a unital algebra and $(X, \delta, \epsilon)$ is a counital coalgebra. These must satisfy the following axiom:

\begin{itemize}
\item (Frobenius) \qquad $(\id_X \otimes m) \circ (\delta \otimes \id_X) = \delta \circ m = (m \otimes \id_X) \circ (\id_X \otimes \delta)$.
\end{itemize}

These relations are typically illustrated as follows:  
	\[ \text{(Frobenius)} \ \ \ \ \ \ \raisebox{-10mm}{
		\begin{tikzpicture}
		\draw[blue,in=90,out=90,looseness=2] (0,0) to (1,0);
		\draw[blue,in=-90,out=-90,looseness=2] (1,0) to (2,0);
		 \draw[blue] (.5,.6) to (.5,1.2);
		\draw[blue] (1.5,-.6) to (1.5,-1.2);
		\draw[blue] (0,0) to (0,-1.2);
		\draw[blue] (2,0) to (2,1.2);
		\end{tikzpicture}} \ \ = \ \
	\raisebox{-8mm}{
		\begin{tikzpicture}
		\draw[blue,in=90,out=90,looseness=2] (0,0) to (1,0);
		\draw[blue] (.5,.6) to (.5,1.2);
		\draw[blue,in=-90,out=-90,looseness=2] (0,1.8) to (1,1.8);
		\end{tikzpicture}}  \ \ = \ \
	\raisebox{-10mm}{
		\begin{tikzpicture}
		\draw[blue,in=-90,out=-90,looseness=2] (0,0) to (1,0);
		\draw[blue,in=90,out=90,looseness=2] (1,0) to (2,0);
		 \draw[blue] (.5,-.6) to (.5,-1.2);
		\draw[blue] (1.5,.6) to (1.5,1.2);
		\draw[blue] (0,0) to (0,1.2);
		\draw[blue] (2,0) to (2,-1.2);
		\end{tikzpicture}} \]

We define a \emph{weak Frobenius algebra} by relaxing the Frobenius condition in the following way:
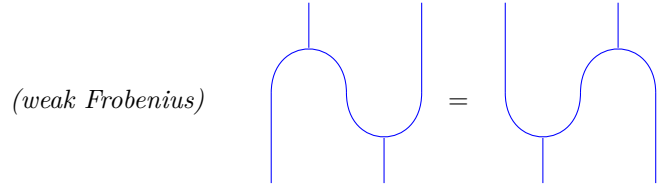
\[\text{(weak Frobenius)} \ \ \ \ \ \ \raisebox{-10mm}{
		\begin{tikzpicture}
		\draw[blue,in=90,out=90,looseness=2] (0,0) to (1,0);
		\draw[blue,in=-90,out=-90,looseness=2] (1,0) to (2,0);
		\draw[blue] (.5,.6) to (.5,1.2);
		\draw[blue] (1.5,-.6) to (1.5,-1.2);
		\draw[blue] (0,0) to (0,-1.2);
		\draw[blue] (2,0) to (2,1.2);
		\end{tikzpicture}}   \ \ = \ \ 
	\raisebox{-10mm}{
		\begin{tikzpicture}
		\draw[blue,in=-90,out=-90,looseness=2] (0,0) to (1,0);
		\draw[blue,in=90,out=90,looseness=2] (1,0) to (2,0);
		\draw[blue] (.5,-.6) to (.5,-1.2);
		\draw[blue] (1.5,.6) to (1.5,1.2);
		\draw[blue] (0,0) to (0,1.2);
		\draw[blue] (2,0) to (2,-1.2);
		\end{tikzpicture}} \] 
  
\end{definition}

We will demonstrate (Lemma \ref{lem:weakfull}) that an algebra is weak Frobenius if and only if it is Frobenius.

\begin{lemma} \label{lem:selfdual}
Let $(X, m, \delta, e, \epsilon)$ be a weak Frobenius algebra. Then $X$ is selfdual (i.e., $X^* = X$), with the evaluation morphism defined as $\ev_X := \epsilon \circ m$ and the coevaluation morphism defined as $\coev_X := \delta \circ e$, as illustrated below:
 \[ \raisebox{-4mm}
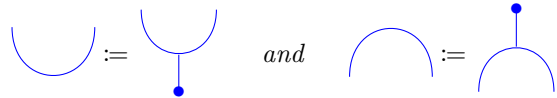
{\begin{tikzpicture}
  	\draw[blue,in=-90,out=-90,looseness=2] (-0.5,0.5) to (-1.6,0.5);
  	\end{tikzpicture}}
  	\coloneqq \,
    \raisebox{-6mm}{\begin{tikzpicture}
  	\draw[blue,in=-90,out=-90,looseness=2] (-0.5,0.5) to (-1.5,0.5);
  	\node at (-1,-.6) {${\color{blue}\bullet}$};
  	\draw[blue] (-1,-.1) to (-1,-.6);
   	\end{tikzpicture}} \hspace*{6mm} \text{and} \hspace*{6mm} \raisebox{-2mm}{\begin{tikzpicture}
   \draw[blue,in=90,out=90,looseness=2] (-0.5,0.5) to (-1.6,0.5);
\end{tikzpicture}}  \coloneqq \, \raisebox{-4mm}
   	{\begin{tikzpicture}
   	\draw[blue,in=90,out=90,looseness=2] (-0.5,0.5) to (-1.5,0.5);
   	\node at (-1,1.6) {${\color{blue}\bullet}$};
   	\draw[blue] (-1,1.1) to (-1,1.6);
   	\end{tikzpicture}}  \]
\end{lemma}

\begin{proof}
The following diagram illustrates that $\epsilon \circ m$ and $\delta \circ e$ satisfy the zigzag relations (see \cite[Definition 2.10.1]{EGNO15}).
  \[ \begin{tikzpicture}
   \draw[blue,in=-90,out=-90,looseness=2] (0,0) to (1,0);
   \draw[blue,in=90,out=90,looseness=2] (1,0) to (2,0);
   \draw[blue] (.5,-.6) to (.5,-1.2);
   \draw[blue] (1.5,.6) to (1.5,1.2);
   \draw[blue] (0,0) to (0,1.2);
   \draw[blue] (2,0) to (2,-1.2);
   \node at (.5,-1.2) {${\color{blue}\bullet}$};
   \node at (1.5,1.2) {${\color{blue}\bullet}$};
   
   \draw[blue,in=90,out=90,looseness=2] (3,0) to (4,0);
   \draw[blue,in=-90,out=-90,looseness=2] (4,0) to (5,0);
   \draw[blue] (3.5,.6) to (3.5,1.2);
   \draw[blue] (4.5,-.6) to (4.5,-1.2);
   \draw[blue] (3,0) to (3,-1.2);
   \draw[blue] (5,0) to (5,1.2);
   \node at (3,-1.2) {${\color{blue}\bullet}$};
   \node at (5,1.2) {${\color{blue}\bullet}$};
   
   \draw[blue] (6,-1.2) to (6,1.2);
   
   \draw[blue,in=-90,out=-90,looseness=2] (7,0) to (8,0);
   \draw[blue,in=90,out=90,looseness=2] (8,0) to (9,0);
   \draw[blue] (7.5,-.6) to (7.5,-1.2);
   \draw[blue] (8.5,.6) to (8.5,1.2);
   \draw[blue] (7,0) to (7,1.2);
   \draw[blue] (9,0) to (9,-1.2);
   \node at (9,-1.2) {${\color{blue}\bullet}$};
   \node at (7,1.2) {${\color{blue}\bullet}$};
   
   \draw[blue,in=90,out=90,looseness=2] (10,0) to (11,0);
   \draw[blue,in=-90,out=-90,looseness=2] (11,0) to (12,0);
   \draw[blue] (10.5,.6) to (10.5,1.2);
   \draw[blue] (11.5,-.6) to (11.5,-1.2);
   \draw[blue] (10,0) to (10,-1.2);
   \draw[blue] (12,0) to (12,1.2);
   \node at (11.5,-1.2) {${\color{blue}\bullet}$};
   \node at (10.5,1.2) {${\color{blue}\bullet}$};
   
   
   \draw[dashed, DarkOrange, thick] (-0.2,-.8) rectangle (2.2,1);
   \draw[dashed, DarkOrange, thick] (6.8,-.8) rectangle (9.2,1);
   \draw[dashed, DarkGreen!70!black, thick] (2.75,-1.35) rectangle (3.95,1.25);
   \draw[dashed, DarkGreen!70!black, thick] (4.1,-1.25) rectangle (5.2,1.35);
   
   \node at (2.8,1.3) [align=center, above, DarkOrange] {weak Frobenius};
   \node at (5.8,1.3) [align=center, above, DarkGreen!70!black] {(co)unitality};
   \node at (9.8,1.3) [align=center, above, DarkOrange] {weak Frobenius};
   
   \node at (2.5,0) {$=$};
   \node at (5.6,0) {$=$};
   \node at (6.5,0) {$=$};
   \node at (9.5,0) {$=$};
   
   \draw[<-, DarkOrange, thick] (2.5,.2) -- (2.5,1.2);
   \draw[<-, DarkGreen!70!black, thick] (5.6,.2) -- (5.6,1.2);
   \draw[<-, DarkOrange, thick] (9.5,.2) -- (9.5,1.2);
   \end{tikzpicture} \qedhere \]  
\end{proof}

\begin{remark} \label{rk:nosym}
The definitions $\ev_X := \epsilon \circ m$ and $\coev_X := \delta \circ e$ may not align perfectly with a given rigid structure in the monoidal category. However, it is possible to adjust the rigid structure to achieve a perfect match, as noted in \cite[Definition 2.10.11]{EGNO15}. Additionally, the rigid structure on an object is unique up to a unique isomorphism, as stated in \cite[Proposition 2.10.5]{EGNO15}. In this paper, we will assume that this alignment is always satisfied. 
\end{remark}

\begin{lemma} \label{lem:intermultdualcomult}
A unital algebra and counital coalgebra $(X, m, \delta, e, \epsilon)$ is weak Frobenius if and only if the following equalities hold:
$$
(m \otimes \id_X) \circ (\id_X \otimes \coev_X) = \delta = (\id_X \otimes m) \circ (\coev_X \otimes \id_X),
$$
if and only if the following equalities hold
$$
(\ev_X \otimes \id_X) \circ (\id_X \otimes \delta) = m = (\id_X \otimes \ev_X) \circ (\delta \otimes \id_X).
$$
They are depicted as follows:
   \[\raisebox{-10mm}{
   	\begin{tikzpicture}
   	\draw[blue,in=-90,out=-90,looseness=2] (0,0) to (1,0);
   	\draw[blue,in=90,out=90,looseness=2] (1,0) to (2,0);
   	\draw[blue] (.5,-.6) to (.5,-1.2);
   	 \draw[blue] (0,0) to (0,1.2);
   	\draw[blue] (2,0) to (2,-1.2);
   	\end{tikzpicture}}  
   	=
   	\raisebox{-8mm}{
   	\begin{tikzpicture}
   	\draw[blue,in=90,out=90,looseness=2] (-0.5,0.5) to (-1.5,0.5);
   	\draw[blue] (-1,1.1) to (-1,2);
   	\end{tikzpicture}}
    =
    \raisebox{-10mm}{
    \begin{tikzpicture}
    \draw[blue,in=90,out=90,looseness=2] (0,0) to (1,0);
    \draw[blue,in=-90,out=-90,looseness=2] (1,0) to (2,0);
     \draw[blue] (1.5,-.6) to (1.5,-1.2);
    \draw[blue] (0,0) to (0,-1.2);
    \draw[blue] (2,0) to (2,1.2);
    \end{tikzpicture}} \hspace*{6mm} \Leftrightarrow \hspace*{6mm} 
    \raisebox{-10mm}{
    \begin{tikzpicture}
    \draw[blue,in=-90,out=-90,looseness=2] (0,0) to (1,0);
    \draw[blue,in=90,out=90,looseness=2] (1,0) to (2,0);
     \draw[blue] (1.5,.6) to (1.5,1.2);
    \draw[blue] (0,0) to (0,1.2);
    \draw[blue] (2,0) to (2,-1.2);
    \end{tikzpicture}} 
	=  
	\raisebox{-6mm}{
		\begin{tikzpicture}
		\draw[blue,in=-90,out=-90,looseness=2] (-0.5,0.5) to (-1.5,0.5);
		 \draw[blue] (-1,-.1) to (-1,-1);
		 \end{tikzpicture}}
	 = 
	 \raisebox{-10mm}{
	 	\begin{tikzpicture}
	 	\draw[blue,in=90,out=90,looseness=2] (0,0) to (1,0);
	 	\draw[blue,in=-90,out=-90,looseness=2] (1,0) to (2,0);
	 	\draw[blue] (.5,.6) to (.5,1.2);
	 	 \draw[blue] (0,0) to (0,-1.2);
	 	\draw[blue] (2,0) to (2,1.2);
	 	\end{tikzpicture}} \]
\end{lemma}

\begin{proof}
Assume the weak Frobenius condition. We prove the first equality by the following diagram (where \emph{selfdual} refers to Lemma \ref{lem:selfdual}):
   \[ \raisebox{-8mm}{
   	\begin{tikzpicture}
   	\draw[blue,in=-90,out=-90,looseness=2] (0,0) to (1,0);
   	\draw[blue,in=90,out=90,looseness=2] (1,0) to (2,0);
   	\draw[blue] (.5,-.6) to (.5,-1.2);
   	\draw[blue] (0,0) to (0,1.2);
   	\draw[blue] (2,0) to (2,-1.2);
   	\draw[dashed, DarkOrange, thick] (.9,0) rectangle (2.1,.8);
   	\draw[->, DarkOrange] (2.4,.8) to (2.4,0);
   	\node at (2.4,1.1) [orange] {selfdual};
   	\node at (2.4,-.2) {$=$};
   	\end{tikzpicture}}
  	\raisebox{-8mm}{
   		\begin{tikzpicture}
   		\draw[blue,in=-90,out=-90,looseness=2] (0,0) to (1,0);
   		\draw[blue,in=90,out=90,looseness=2] (1,0) to (2,0);
   		\draw[blue] (.5,-.6) to (.5,-1.2);
   		\draw[blue] (0,0) to (0,1.2);
   		\draw[blue] (2,0) to (2,-1.2);
   		\draw[blue] (1.5,.6) to (1.5,1.2);
   		\draw[dashed, DarkGreen!70!black, thick] (-.2,-.8) rectangle (2.2,.8);
   		\draw[->, DarkGreen!70!black] (2.7,.8) to (2.7,0);
   		\node at (2.9,1.1) [green!70!black] {weak Frobenius};
   		\node at (1.5,1.2) {$\color{blue}{\bullet}$};
   		\node at (2.7,-.2) {$=$};
   		\end{tikzpicture}}
   	   	\raisebox{-8mm}{
   		\begin{tikzpicture}
   		\draw[blue,in=90,out=90,looseness=2] (0,0) to (1,0);
   		\draw[blue,in=-90,out=-90,looseness=2] (1,0) to (2,0);
   		\draw[blue] (.5,.6) to (.5,1.2);
   		\draw[blue] (1.5,-.6) to (1.5,-1.2);
   		\draw[blue] (0,0) to (0,-1.2);
   		\draw[dashed, red, thick] (.8,-.7) rectangle (2.2,.2);
   		\draw[->, red] (2.6,.8) to (2.6,0);
   		\node at (2.6,1) [red] {unitality};
   		\node at (2,0) {$\color{blue}{\bullet}$};
   		\node at (2.6,-.2) {$=$};
   		\end{tikzpicture}} 
   	  	\raisebox{-2mm}{
   		\begin{tikzpicture}
   		\draw[blue,in=90,out=90,looseness=2] (-0.5,0.5) to (-1.5,0.5);
   		 \draw[blue] (-1,1.1) to (-1,1.8);
   		 \end{tikzpicture}} \]
Similarly, the second equality follows. Next, Lemma~\ref{lem:selfdual} can also be established from these first two equalities, which, by the zigzag relations, are equivalent to the last two equalities. Finally, assuming these equalities hold, we now verify the weak Frobenius condition:
   	 \[\raisebox{-8mm}{
   	 	\begin{tikzpicture}
   	 	\draw[blue,in=90,out=90,looseness=2] (0,0) to (1,0);
   	 	\draw[blue,in=-90,out=-90,looseness=2] (1,0) to (2,0);
   	 	\draw[blue] (.5,.6) to (.5,1.2);
   	 	\draw[blue] (1.5,-.6) to (1.5,-1.2);
   	 	\draw[blue] (0,0) to (0,-1.2);
   	 	\draw[blue] (2,0) to (2,1.2);
   	 	\draw[dashed, violet, thick] (.6,-.4) rectangle (1.3,.4);
   	 	\draw[->, violet] (2.7,.8) to (2.7,0);
   	 	\node at (2.7,1) [violet] {zigzag};
   	 	\node at (2.7,-.2) {$=$};
   	 	\end{tikzpicture}} 
    	\raisebox{-8mm}{
    		\begin{tikzpicture}
    		\draw[blue,in=90,out=90,looseness=2] (0,0) to (1,0);
    		\draw[blue,in=-90,out=-90,looseness=2] (1,0) to (2,0);
    		\draw[blue,in=90,out=90,looseness=2] (2,0) to (3,0);
    		\draw[blue,in=-90,out=-90,looseness=2] (3,0) to (4,0);
    		\draw[blue] (.5,.6) to (.5,1.2);
    		\draw[blue] (0,0) to (0,-1.2);
    		\draw[blue] (3.5,-.58) to (3.5,-1.2);
    		\draw[blue] (4,0) to (4,1.2);
    		\draw[dashed, yellow!70!black, thick] (-.2,-.75) rectangle (1.9,1);
    		\draw[dashed, yellow!70!black, thick] (2.05,-1) rectangle (4.1,.75);
    		\node at (5,1) [yellow!70!black] {assumption};
    		\draw[->, yellow!70!black] (4.7,.8) to (4.7,0);
    		\node at (4.7,-.2) {$=$};
    		\end{tikzpicture}}
    	\raisebox{-8mm}{
    		\begin{tikzpicture}
    		\draw[blue,in=-90,out=-90,looseness=2] (0,0) to (1,0);
    		\draw[blue,in=90,out=90,looseness=2] (1,0) to (2,0);
    		\draw[blue] (.5,-.6) to (.5,-1.2);
    		\draw[blue] (1.5,.6) to (1.5,1.2);
    		\draw[blue] (0,0) to (0,1.2);
    		\draw[blue] (2,0) to (2,-1.2);
    		\end{tikzpicture}} \]
The result follows.
\end{proof}

\begin{lemma} \label{lem:weakfull}
A unital algebra and counital coalgebra is weak Frobenius if and only if it is Frobenius.
\end{lemma}

\begin{proof}
The proof is illustrated by the following diagram (where \emph{weak Frobenius} refers to Lemma \ref{lem:intermultdualcomult}):
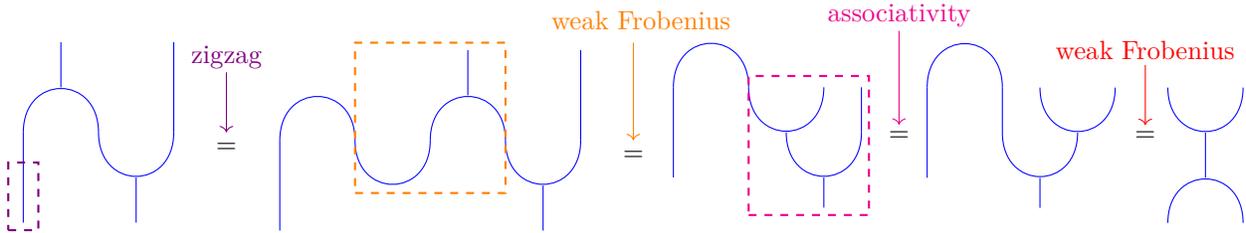
		\[ \raisebox{-8mm}{
			\begin{tikzpicture}
			\draw[blue,in=90,out=90,looseness=2] (0,0) to (1,0);
			\draw[blue,in=-90,out=-90,looseness=2] (1,0) to (2,0);
			\draw[blue] (.5,.6) to (.5,1.2);
			\draw[blue] (1.5,-.6) to (1.5,-1.2);
			\draw[blue] (0,0) to (0,-1.2);
			\draw[blue] (2,0) to (2,1.2);
			\draw[dashed, violet, thick] (-.2,-.4) rectangle (.2,-1.3);
			\draw[->, violet] (2.7,.8) to (2.7,0);
			\node at (2.7,1) [violet] {zigzag};
			\node at (2.7,-.2) {$=$};
			\end{tikzpicture}}
		\raisebox{-8mm}{
			\begin{tikzpicture}
			\draw[blue,in=90,out=90,looseness=2] (-1,0) to (-2,0);
			\draw[blue,in=-90,out=-90,looseness=2] (0,0) to (-1,0);
			\draw[blue,in=90,out=90,looseness=2] (0,0) to (1,0);
			\draw[blue,in=-90,out=-90,looseness=2] (1,0) to (2,0);
			\draw[blue] (.5,.6) to (.5,1.2);
			\draw[blue] (1.5,-.6) to (1.5,-1.2);
			\draw[blue] (-2,0) to (-2,-1.2);
			\draw[blue] (2,0) to (2,1.2);
			\draw[dashed, DarkOrange, thick] (1,1.3) rectangle (-1,-.7);
			\node at (2.8,1.6) [orange] {weak Frobenius};
			\draw[->, DarkOrange] (2.7,1.3) to (2.7,0);
			\node at (2.7,-.2) {$=$};
			\end{tikzpicture}}  
		\hspace*{-1cm}
		\raisebox{-6mm}{
			\begin{tikzpicture}[rotate=180,transform shape]
			\draw[blue,in=90,out=90,looseness=2] (0,0) to (1,0);
			\draw[blue,in=90,out=90,looseness=2] (0.5,.6) to (-.5,.6);
			\draw[blue,in=-90,out=-90,looseness=2] (1,0) to (2,0);
			\draw[blue] (-.5,.6) to (-.5,0);
			\draw[blue] (0,1.2) to (0,1.6);
			\draw[blue] (2,0) to (2,1.2);
			\draw[dashed, DarkRed, thick] (-.6,-.15) rectangle (1,1.7);
			\draw[->, DarkRed] (-1,-.75) to (-1,.5);
			\node[rotate=180] at (-1,-.95) [DarkRed] {associativity};
			\node at (-1,.6) {$=$};
			\end{tikzpicture}}
		\hspace*{-.7cm}
		\raisebox{-7mm}{\begin{tikzpicture}[rotate=180,transform shape]
			\draw[blue,in=90,out=90,looseness=2] (0,0) to (1,0);
			\draw[blue,in=90,out=90,looseness=2] (.5,.6) to (1.5,.6);
			\draw[blue,in=-90,out=-90,looseness=2] (1.5,0) to (2.5,0);
			\draw[blue] (1.5,.6) to (1.5,0);
			\draw[blue] (1,1.2) to (1,1.6);
			\draw[blue] (2.5,0) to (2.5,1.2);
			\draw[->, red] (-.4,-.3) to (-.4,.5);
			\node[rotate=180] at (-.4,-.5) [red] {weak Frobenius};
			\node at (-.4,.6) {$=$};
			\draw[blue,in=90,out=90,looseness=2] (-1.7,0) to (-.7,0);
			\draw[blue] (-1.2,.6) to (-1.2,1.2);
			\draw[blue,in=-90,out=-90,looseness=2] (-1.7,1.8) to (-.7,1.8);
			\end{tikzpicture}} \]
The reverse direction is immediate.
\end{proof}

\begin{lemma} \label{lem:multdualcomult}
Let $(X, m, \delta, e, \epsilon)$ be a Frobenius algebra. Then $m^* = \delta$ and $\delta^* = m$, as shown below:
			\[ \raisebox{-10mm}
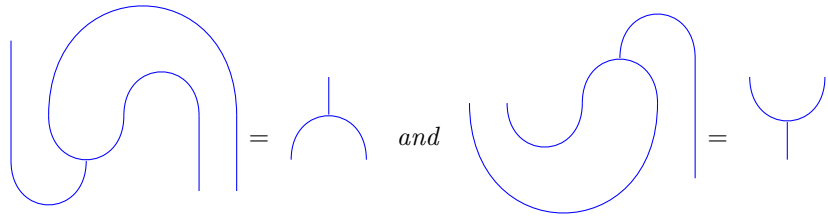
{\begin{tikzpicture}
   	 		\draw[blue,in=-90,out=-90,looseness=2] (0,0) to (1,0);
   	 		\draw[blue,in=-90,out=-90,looseness=2] (-.5,-.6) to (.5,-.6);
   	 		\draw[blue,in=90,out=90,looseness=2] (1,0) to (2,0);
   	 		\draw[blue,in=90,out=90,looseness=2] (0,0) to (2.5,0);
   	 		\draw[blue] (-.5,-.6) to (-.5,1);
   	 		\draw[blue] (2,0) to (2,-1);
   	 		\draw[blue] (2.5,0) to (2.5,-1);
   	 		\end{tikzpicture}} \, = \, 
   	 		\raisebox{-2mm}{
   	 			\begin{tikzpicture}
   	 			\draw[blue,in=90,out=90,looseness=2] (-0.5,0.5) to (-1.5,0.5);
   	 			\draw[blue] (-1,1.1) to (-1,1.6);
   	 			\end{tikzpicture}} \hspace*{4mm} \text{and} \hspace*{4mm} 
    		\raisebox{-14mm}{\begin{tikzpicture}
    		\draw[blue,in=90,out=90,looseness=2] (0,0) to (1,0);
    		\draw[blue,in=90,out=90,looseness=2] (.5,.6) to (1.5,.6);
    		\draw[blue,in=-90,out=-90,looseness=2] (0,0) to (-1,0);
    		\draw[blue,in=-90,out=-90,looseness=2] (1,0) to (-1.5,0);
    		\draw[blue] (1.5,.6) to (1.5,-1);
    		\end{tikzpicture}}
    		 \, = \,
    		\raisebox{-2mm}{
    		\begin{tikzpicture}
    		\draw[blue,in=-90,out=-90,looseness=2] (-0.5,0.5) to (-1.5,0.5);
    		\draw[blue] (-1,-.1) to (-1,-.6);
    		\end{tikzpicture}}   \]
\end{lemma}

\begin{proof}
To prove the first equality, we apply Lemma \ref{lem:intermultdualcomult} three times:
			\[ \raisebox{-8mm}
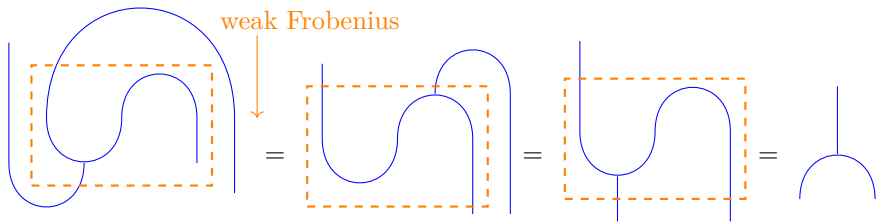
{\begin{tikzpicture}
				\draw[blue,in=-90,out=-90,looseness=2] (0,0) to (1,0);
				\draw[blue,in=-90,out=-90,looseness=2] (-.5,-.6) to (.5,-.6);
				\draw[blue,in=90,out=90,looseness=2] (1,0) to (2,0);
				\draw[blue,in=90,out=90,looseness=2] (0,0) to (2.5,0);
				\draw[blue] (-.5,-.6) to (-.5,1);
				\draw[blue] (2,0) to (2,-.6);
				\draw[blue] (2.5,0) to (2.5,-1);
				\draw[dashed, DarkOrange, thick] (2.2, .7) rectangle (-.2,-.9);
				\node at (3.5,1.3) [orange] {weak Frobenius};
				\draw[->,orange] (2.8,1.1) to (2.8,0);
				\end{tikzpicture}} 
				\hspace*{-2cm} = \,
			\raisebox{-7mm}{
				\begin{tikzpicture}
				\draw[blue,in=-90,out=-90,looseness=2] (0,0) to (1,0);
				\draw[blue,in=90,out=90,looseness=2] (1,0) to (2,0);
				\draw[blue,in=90,out=90,looseness=2] (1.5,.6) to (2.5,.6);
				\draw[blue] (2.5,.6) to (2.5,-1);
				\draw[blue] (0,0) to (0,1);
				\draw[blue] (2,0) to (2,-1);
				\draw[dashed, DarkOrange, thick] (2.2, .7) rectangle (-.2,-.9);
				\end{tikzpicture}} 
				 \, = \,
			 \raisebox{-8mm}{
				\begin{tikzpicture}
				\draw[blue,in=-90,out=-90,looseness=2] (0,0) to (1,0);
				\draw[blue,in=90,out=90,looseness=2] (1,0) to (2,0);
				\draw[blue] (.5,-.6) to (.5,-1.2);
				\draw[blue] (0,0) to (0,1.2);
				\draw[blue] (2,0) to (2,-1.2);
				\draw[dashed, DarkOrange, thick] (2.2, .7) rectangle (-.2,-.9);
				\end{tikzpicture}}  
			    \, = \,
			\raisebox{-5mm}{
				\begin{tikzpicture}
				\draw[blue,in=90,out=90,looseness=2] (-0.5,0.5) to (-1.5,0.5);
				\draw[blue] (-1,1.1) to (-1,2);
				\end{tikzpicture}} \]
The second equality follows similarly.
\end{proof}

\begin{lemma} \label{lem:intermultdualcomultalpha}
Let $X$ and $X'$ be two selfdual objects, and let $\alpha \in \Hom_{\mathcal{C}}(X', X)$.
\begin{itemize}
\item If $(X, m, \delta, e, \epsilon)$ is a Frobenius algebra, then  
			\[ \raisebox{-8mm}{
				\begin{tikzpicture}
				\draw[blue,in=-90,out=-90,looseness=2] (0,0) to (1,0);
				\draw[blue,in=90,out=90,looseness=2] (1,0) to (2,0);
				\draw[blue] (.5,-.6) to (.5,-1.2);
				\draw[blue] (0,0) to (0,1.2);
				\draw[blue] (2,0) to (2,-1.2);
				\node[draw,thick, rounded corners, fill=white,minimum width = 20] at (1,0) {$\alpha$}; 
				\end{tikzpicture}}  
			=
			\raisebox{-7mm}{
				\begin{tikzpicture}
				\draw[blue,in=90,out=90,looseness=2] (-0.5,0.5) to (-1.5,0.5);
				\draw[blue] (-1,1.1) to (-1,2);
				\draw[blue] (-.5,.2) to (-.5,-.3);
				\draw[blue] (-1.5,.5) to (-1.5,-.3);
				\node[draw,thick, rounded corners, fill=white,minimum width = 20] at (-.5,.3) {$\alpha^*$};
				\end{tikzpicture}} \hspace*{4mm} \text{and} \hspace*{4mm}
				 \raisebox{-6mm}{
				 	\begin{tikzpicture}
				 	\draw[blue,in=-90,out=-90,looseness=2] (-0.5,0.5) to (-1.5,0.5);
				 	\draw[blue] (-1,-.1) to (-1,-1);
				 	\draw[blue] (-.5,.6) to (-.5,1);
				 	\draw[blue] (-1.5,.5) to (-1.5,1);
				 	\node[draw,thick, rounded corners, fill=white,minimum width = 20] at (-.5,.5) {$\alpha$};
				 	\end{tikzpicture}}
				 = \raisebox{-6mm}{
				 	\begin{tikzpicture}
				 	\draw[blue,in=90,out=90,looseness=2] (0,0) to (1,0);
				 	\draw[blue,in=-90,out=-90,looseness=2] (1,0) to (2,0);
				 	\draw[blue] (.5,.6) to (.5,1.2);
				 	\draw[blue] (0,0) to (0,-.8);
				 	\draw[blue] (2,0) to (2,1.2);
				 	\node[draw,thick, rounded corners, fill=white,minimum width = 20] at (1,0) {$\alpha^*$};
				 	\end{tikzpicture}}\]
\item If $(X', m', \delta', e', \epsilon')$ is a Frobenius algebra, then  
		\[ \raisebox{-7mm}{
			\begin{tikzpicture}
			\draw[blue,in=-90,out=-90,looseness=2] (0,0) to (1,0);
			\draw[blue,in=90,out=90,looseness=2] (1,0) to (2,0);
			\draw[blue] (1.5,.6) to (1.5,1.2);
			\draw[blue] (0,0) to (0,1.2);
			\draw[blue] (2,0) to (2,-.8);
			\node[draw,thick, rounded corners, fill=white,minimum width = 20] at (1,0) {$\alpha$};
			\end{tikzpicture}} 
		=  
		\raisebox{-6mm}{
			\begin{tikzpicture}
			\draw[blue,in=-90,out=-90,looseness=2] (-0.5,0.5) to (-1.5,0.5);
			\draw[blue] (-1,-.1) to (-1,-1);
			\draw[blue] (-1.5,.5) to (-1.5,1);
			\draw[blue] (-.5,.5) to (-.5,1);
			\node[draw,thick, rounded corners, fill=white,minimum width = 20] at (-1.5,.4) {$\alpha^*$};
			\end{tikzpicture}} \hspace*{4mm} \text{and} \hspace*{4mm} 
		\raisebox{-6mm}{
			\begin{tikzpicture}
			\draw[blue,in=90,out=90,looseness=2] (-0.5,0.5) to (-1.5,0.5);
			\draw[blue] (-1,1.1) to (-1,2);
			\draw[blue] (-1.5,.4) to (-1.5,0);
			\draw[blue] (-.5,.5) to (-.5,0);
			\node[draw,thick, rounded corners, fill=white,minimum width = 20] at (-1.5,.6) {$\alpha$};
			\end{tikzpicture}}
		=
		\raisebox{-6mm}{
			\begin{tikzpicture}
			\draw[blue,in=90,out=90,looseness=2] (0,0) to (1,0);
			\draw[blue,in=-90,out=-90,looseness=2] (1,0) to (2,0);
			\draw[blue] (1.5,-.6) to (1.5,-1.2);
			\draw[blue] (0,0) to (0,-1.2);
			\draw[blue] (2,0) to (2,1);
			\node[draw,thick, rounded corners, fill=white,minimum width = 20] at (1,0) {$\alpha^*$};
			\end{tikzpicture}} \]
\end{itemize} 
\end{lemma}

\begin{proof}
We can prove the first equality using the following diagram:
		\[\raisebox{-4mm}
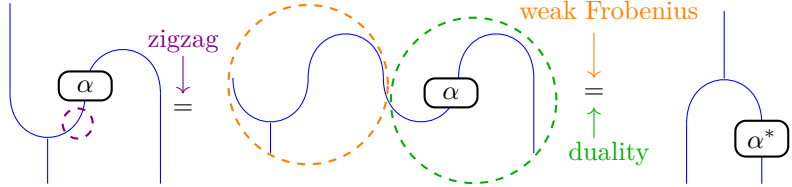
{
			\begin{tikzpicture}
			\draw[blue,in=-90,out=-90,looseness=2] (0,0) to (1,0);
			\draw[blue,in=90,out=90,looseness=2] (1,0) to (2,0);
			\draw[blue] (.5,-.6) to (.5,-1.2);
			\draw[blue] (0,0) to (0,1.2);
			\draw[blue] (2,0) to (2,-1.2);
			\draw[dashed, violet, thick] (.9,-.4) circle (0.2);
			\draw[->,violet] (2.3,.5) to (2.3, 0);
			\node[draw,thick, rounded corners, fill=white,minimum width = 20] at (1,.1) {$\alpha$};
			\node at (2.3,.7) [violet] {zigzag};
			\node at (2.3,-.2) {$=$}; 
			\end{tikzpicture}}  
		 	\raisebox{-4mm}{\begin{tikzpicture}
		 	\draw[blue,in=-90,out=-90,looseness=2] (0,0) to (1,0);
		 	\draw[blue,in=90,out=90,looseness=2] (1,0) to (2,0);
		 	\draw[blue,in=-90,out=-90,looseness=2] (2,0) to (3,0);
		 	\draw[blue,in=90,out=90,looseness=2] (3,0) to (4,0);
		 	\draw[blue] (.5,-.6) to (.5,-1);
		 	\draw[blue] (4,0) to (4,-1);
		 	\draw[dashed, DarkOrange, thick] (1,-.09) circle (1.06);
		 	\draw[dashed, DarkGreen!70!black, thick] (3.2,-.3) circle (1.1);
		 	\draw[->,orange] (4.8,.6) to (4.8,0);
		 	\draw[<-,green!70!black] (4.8,-.4) to (4.8,-.8);
		 	\node[draw,thick, rounded corners, fill=white,minimum width = 20] at (2.9,-.2) {$\alpha$};
		 	\node at (5,.9) [orange] {weak Frobenius};
		 	\node at (5,-1) [green!70!black] {duality};
		 	\node at (4.8,-.2) {$=$};
		 	\end{tikzpicture}}
	 		\hspace*{-.4cm}
	 		\raisebox{-4mm}{
	 			\begin{tikzpicture}
	 			\draw[blue,in=90,out=90,looseness=2] (-0.5,0.5) to (-1.5,0.5);
	 			\draw[blue] (-1,1.1) to (-1,2);
	 			\draw[blue] (-.5,.2) to (-.5,-.3);
	 			\draw[blue] (-1.5,.5) to (-1.5,-.3);
	 			\node[draw,thick, rounded corners, fill=white,minimum width = 20] at (-.5,.3) {$\alpha^*$};
	 			\end{tikzpicture}} \]
The proof of the other equalities follows a similar approach.
\end{proof}

In the rest of the paper, we will refer to the application of Lemma \ref{lem:intermultdualcomultalpha} as \emph{weak Frobenius}.

\begin{lemma} \label{lem:unitdualcounit}
Let $(X, m, \delta, e, \epsilon)$ be a Frobenius algebra. Then $e^* = \epsilon$ and $\epsilon^* = e$, as depicted below:
 			\[ \raisebox{-7mm}
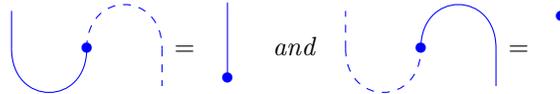
{\begin{tikzpicture}
 				\draw[blue,in=-90,out=-90,looseness=2] (0,0) to (1,0);
 				\draw[blue,dashed,in=90,out=90,looseness=2] (1,0) to (2,0);
 				\draw[blue,dashed] (2,0) to (2,-.5);
 				\draw[blue] (0,0) to (0,.5);
 				\node at (1,0) {$\color{blue}{\bullet}$};
 			\end{tikzpicture}}
 			\, =
 			\raisebox{-5mm}{
 			\begin{tikzpicture}
 			\draw[blue] (0,0) to (0,1);
 			\node at (0,0) {$\color{blue}{\bullet}$};
 			\end{tikzpicture}} \hspace*{4mm} \text{and} \hspace*{4mm}
 			\raisebox{-7mm}{\begin{tikzpicture}
 				\draw[blue,dashed,in=-90,out=-90,looseness=2] (0,0) to (1,0);
 				\draw[blue,in=90,out=90,looseness=2] (1,0) to (2,0);
 				\draw[blue] (2,0) to (2,-.5);
 				\draw[blue,dashed] (0,0) to (0,.5);
 				\node at (1,0) {$\color{blue}{\bullet}$};
 				\end{tikzpicture}}
 			\, =
 			\raisebox{-5mm}{
 				\begin{tikzpicture}
 				\draw[blue] (0,0) to (0,1);
 				\node at (0,1) {$\color{blue}{\bullet}$};
 				\end{tikzpicture}} \]
\end{lemma}

\begin{proof}
We can demonstrate the first equality using the following diagram:
 			\[\raisebox{-2mm}
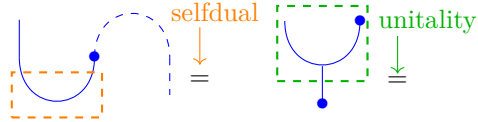
{\begin{tikzpicture}
 				\draw[blue,in=-90,out=-90,looseness=2] (0,0) to (1,0);
 				\draw[blue,dashed,in=90,out=90,looseness=2] (1,0) to (2,0);
 				\draw[blue,dashed] (2,0) to (2,-.5);
 				\draw[dashed, DarkOrange,thick] (-.1,-.2) rectangle (1.1,-.8);
 				\draw[blue] (0,0) to (0,.5);
 				\draw[->,orange] (2.4,.4) to (2.4,-.1);
 				\node at (1,0) {$\color{blue}{\bullet}$};
 				\node at (2.6,.6) [orange] {selfdual};
 				\node at (2.4,-.3) {$=$};
 				\end{tikzpicture}}
 				\raisebox{-2mm}{
 					\begin{tikzpicture}
 					\draw[blue,in=-90,out=-90,looseness=2] (-0.5,0.5) to (-1.5,0.5);
 					\draw[blue] (-1,-.1) to (-1,-.6);
 					\draw[dashed,green!70!black,thick] (-1.6,.7) rectangle (-.4,-.3);
 					\draw[->,green!70!black] (0,.3) to (0,-.2);
 					\node at (.4,.5) [green!70!black] {unitality};
 					\node at (-.5,.5) {$\color{blue}{\bullet}$};
 					\node at (-1,-.6) {$\color{blue}{\bullet}$};
 					\node at (0,-.3) {$=$};
 					\end{tikzpicture}}
 				\raisebox{-1.5mm}{
 				\begin{tikzpicture}
 				\draw[blue] (0,0) to (0,1);
 				\node at (0,0) {$\color{blue}{\bullet}$};
 				\end{tikzpicture}} \]
The second equality follows similarly.
\end{proof}

\begin{lemma} \label{lem:evdualcoev}
Let $X$ be a Frobenius algebra. Then $\ev_X^* = \coev_X$ and $\coev_X^* = \ev_X$.
\end{lemma}
\begin{proof}
Immediate from Lemmas \ref{lem:selfdual}, \ref{lem:multdualcomult} and \ref{lem:unitdualcounit}.
\end{proof}

\begin{lemma} \label{lem:FrobExtra}
The following equalities hold in a Frobenius algebra:
 			\[ \raisebox{-8mm}
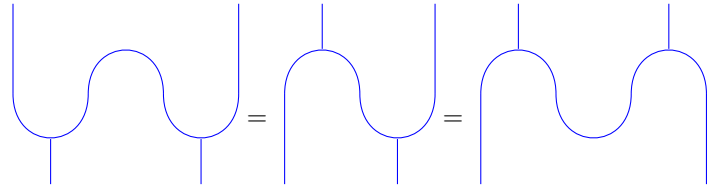
{
 				\begin{tikzpicture}
 				\draw[blue,in=-90,out=-90,looseness=2] (0,0) to (1,0);
 				\draw[blue,in=90,out=90,looseness=2] (1,0) to (2,0);
 				\draw[blue,in=-90,out=-90,looseness=2] (2,0) to (3,0);
 				\draw[blue] (.5,-.6) to (.5,-1.2);
 				\draw[blue] (0,0) to (0,1.2);
 				\draw[blue] (2.5,-.6) to (2.5,-1.2);
 				\draw[blue] (3,0) to (3,1.2);
  				\end{tikzpicture}} 
  				=
  				 \raisebox{-8mm}{
  				 	\begin{tikzpicture}
  				 	\draw[blue,in=90,out=90,looseness=2] (0,0) to (1,0);
  				 	\draw[blue,in=-90,out=-90,looseness=2] (1,0) to (2,0);
  				 	\draw[blue] (.5,.6) to (.5,1.2);
  				 	\draw[blue] (1.5,-.6) to (1.5,-1.2);
  				 	\draw[blue] (0,0) to (0,-1.2);
  				 	\draw[blue] (2,0) to (2,1.2);
  				 	\end{tikzpicture}} 
  			 		= 
  			 		\raisebox{-8mm}{
  			 			\begin{tikzpicture}
  			 			\draw[blue,in=90,out=90,looseness=2] (0,0) to (1,0);
  			 			\draw[blue,in=-90,out=-90,looseness=2] (1,0) to (2,0);
  			 			\draw[blue,in=90,out=90,looseness=2] (2,0) to (3,0);
  			 			\draw[blue] (.5,.6) to (.5,1.2);
  			 			\draw[blue] (0,0) to (0,-1.2);
  			 			\draw[blue] (3,0) to (3,-1.2);
  			 			\draw[blue] (2.5,.6) to (2.5,1.2);
  			 			\end{tikzpicture}}  \]
\end{lemma}

\begin{proof}
This can be proven using the following diagram:
  		 	\[ \raisebox{-8mm}
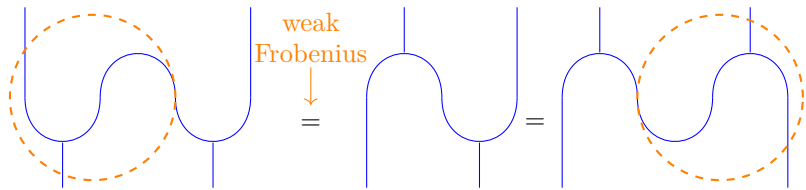
{
  		 		\begin{tikzpicture}
  		 		\draw[blue,in=-90,out=-90,looseness=2] (0,0) to (1,0);
  		 		\draw[blue,in=90,out=90,looseness=2] (1,0) to (2,0);
  		 		\draw[blue,in=-90,out=-90,looseness=2] (2,0) to (3,0);
  		 		\draw[blue] (.5,-.6) to (.5,-1.2);
  		 		\draw[blue] (0,0) to (0,1.2);
  		 		\draw[blue] (2.5,-.6) to (2.5,-1.2);
  		 		\draw[blue] (3,0) to (3,1.2);
  		 		\draw[dashed, DarkOrange, thick] (.9,0) circle (1.1);
  		 		\draw[->,orange] (3.8,.4) to (3.8,-.1);
  		 		\node at (3.8,1) [orange] {weak};
  		 		\node at (3.8,.6) [orange] {Frobenius};
  		 		\node at (3.8,-.35) {$=$};
  		 		\end{tikzpicture}}
  		 		\hspace*{-.25cm}
  	 			\raisebox{-8mm}{
  	 				\begin{tikzpicture}
  	 				\draw[blue,in=90,out=90,looseness=2] (0,0) to (1,0);
  	 				\draw[blue,in=-90,out=-90,looseness=2] (1,0) to (2,0);
  	 				\draw[blue] (.5,.6) to (.5,1.2);
  	 				\draw[blue] (1.5,-.6) to (1.5,-1.2);
  	 				\draw[blue] (0,0) to (0,-1.2);
  	 				\draw[blue] (2,0) to (2,1.2);
  	 				\end{tikzpicture}}
   				= 
   				  \raisebox{-8mm}{
   				  	\begin{tikzpicture}
   				  	\draw[blue,in=90,out=90,looseness=2] (0,0) to (1,0);
   				  	\draw[blue,in=-90,out=-90,looseness=2] (1,0) to (2,0);
   				  	\draw[blue,in=90,out=90,looseness=2] (2,0) to (3,0);
   				  	\draw[blue] (.5,.6) to (.5,1.2);
   				  	\draw[blue] (0,0) to (0,-1.2);
   				  	\draw[blue] (3,0) to (3,-1.2);
   				  	\draw[blue] (2.5,.6) to (2.5,1.2);
   				  	\draw[dashed, DarkOrange, thick] (2.1,0) circle (1.1);
   				  	\end{tikzpicture}} \qedhere \]
\end{proof}

\begin{lemma} \label{lem:grap}
Let $(X, m, \delta, e, \epsilon)$ be a Frobenius algebra, and let $\alpha$ be a morphism in $\End_{\mathcal C}(X)$. Then:
$$
m \circ (\alpha \otimes \id_X) \circ \delta = m \circ (( m \circ (\alpha \otimes \id_X) \circ (\delta \circ e)) \otimes \id_X),
$$
which can be depicted as follows:
 			\[ \raisebox{-12mm}
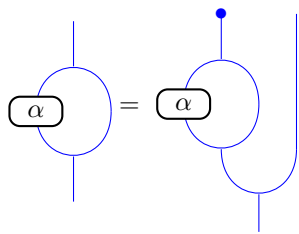
{
 				\begin{tikzpicture}
 				\draw[blue,in=90,out=90,looseness=2] (0,0) to (1,0);
 				\draw[blue,in=-90,out=-90,looseness=2] (0,0) to (1,0);
 				\draw[blue] (.5,-.6) to (.5,-1.2);
 				\draw[blue] (.5,.6) to (.5,1.2);
 				\node[draw,thick,rounded corners, fill=white,minimum width=20] at (0,0){$\alpha$};  
 				\end{tikzpicture}} 
 				= 
 				\raisebox{-16mm}{
 					\begin{tikzpicture}
 					\draw[blue,in=90,out=90,looseness=2] (0,0) to (1,0);
 					\draw[blue,in=-90,out=-90,looseness=2] (0,0) to (1,0);
 					\draw[blue,in=-90,out=-90,looseness=2] (.5,-.6) to (1.5,-.6);
 					\draw[blue] (1.5,-.6) to (1.5,1.2);
 					\draw[blue] (.5,.6) to (.5,1.2);
 					\draw[blue] (1,-1.2) to (1,-1.7);
 					\node[draw,thick,rounded corners, fill=white,minimum width=20] at (0,0){$\alpha$};
 					\node at (.5,1.2) {$\color{blue}{\bullet}$};
 					\end{tikzpicture}} \qedhere \]
\end{lemma}

\begin{proof}
Here is a graphical proof where \emph{weak Frobenius} refers to Lemma \ref{lem:intermultdualcomult}:
 			\[ \raisebox{-6.5mm}{
 				\begin{tikzpicture}
 				\draw[blue,in=90,out=90,looseness=2] (0,0) to (1,0);
 				\draw[blue,in=-90,out=-90,looseness=2] (0,0) to (1,0);
 				\draw[blue,in=-90,out=-90,looseness=2] (.5,-.6) to (1.5,-.6);
 				\draw[blue] (1.5,-.6) to (1.5,1.2);
 				\draw[blue] (.5,.6) to (.5,1.2);
 				\draw[blue] (1,-1.2) to (1,-1.7);
 				\node[draw,thick,rounded corners, fill=white,minimum width=20] at (0,0){$\alpha$};
 				\node at (.5,1.2) {$\color{blue}{\bullet}$};
 				\draw[dashed,red,thick] (-.3,-.35) rectangle (1.7,-1.4);
 				\draw[->,red] (2,.7) to (2,.2);
 				\node at (2.5,1) [red] {associativity};
 				\node at (2,0) {$=$}; 
 				\end{tikzpicture}} 
 				\raisebox{-8mm}{
 					\begin{tikzpicture}
 					\draw[blue,in=90,out=90,looseness=2] (0,0) to (1,0);
 					\draw[blue,in=-90,out=-90,looseness=2] (1,0) to (2,0);
 					\draw[blue,in=-90,out=-90,looseness=2] (0,-.6) to (1.5,-.6);
 					\draw[blue] (.5,.6) to (.5,1.2);
 					\draw[blue] (0,0) to (0,-.6);
 					\draw[blue] (2,0) to (2,1.2);
 					\draw[blue] (.75,-1.5) to (.75,-2);
 					\node[draw,thick,rounded corners, fill=white,minimum width=20] at (0,-.6) {$\alpha$};
 					\node at (.5,1.2) {$\color{blue}{\bullet}$};
 					\node at (3.5,.9) [green!70!black] {selfdual};
		 			\node at (3.5,-1) [orange] {weak Frobenius};
		 			\node at (3.5,-.2) {$=$};
					\draw[dashed, DarkGreen!70!black, thick] (.5,.8) circle (.6);
		 			\draw[dashed, DarkOrange, thick] [rotate around={-30:(3.2,-.3)}] (.9,-.7) ellipse (1.7 and 1);
		 			\draw[->,green!70!black] (3.5,.6) to (3.5,0);
		 			\draw[<-,orange] (3.5,-.4) to (3.5,-.8);
 					\end{tikzpicture}}
 					\raisebox{-3mm}{
 						\begin{tikzpicture}
 						\draw[blue,in=90,out=90,looseness=2] (0,0) to (1,0);
 						\draw[blue,in=-90,out=-90,looseness=2] (0,0) to (1,0);
 						\draw[blue] (.5,-.6) to (.5,-1.2);
 						\draw[blue] (.5,.6) to (.5,1.2);
 						\node[draw,thick,rounded corners, fill=white,minimum width=20] at (0,0){$\alpha$}; 
 						\end{tikzpicture}} \qedhere \]
\end{proof}

\begin{definition} \label{def:connected}
An object $X$ in a $\mathbbm{k}$-linear monoidal category is called \emph{connected} (or haploid) if $\Hom_{\mathcal{C}}(\one, X)$ is one-dimensional, so equal to $\mathbbm{k} e$ for some nonzero $e$.
An algebra and coalgebra $(X, m, \delta)$ is termed \emph{weakly separable} if there exists $\lambda \in \mathbbm{k}$ such that $m \circ \delta = \lambda \id_X$, as illustrated below:
 					\[ \hspace*{-1.5cm} (\text{weakly separable}) \hspace*{1.5cm}  \raisebox{-10mm}{
 						\begin{tikzpicture}
 						\draw[blue,in=90,out=90,looseness=2] (0,0) to (1,0);
 						\draw[blue,in=-90,out=-90,looseness=2] (0,0) to (1,0);
 						\draw[blue] (.5,-.6) to (.5,-1.2);
 						\draw[blue] (.5,.6) to (.5,1.2);
 						\end{tikzpicture}} = 
 					\lambda \raisebox{-10mm}{
 						\begin{tikzpicture}
 						\draw[blue] (0,0) to (0,2.4);
 						\end{tikzpicture}} \]
A \emph{separable} Frobenius algebra is precisely a weakly separable one for which the scaling factor $\lambda$ is nonzero.
\end{definition}

\begin{lemma} \label{lem:cose}
A connected Frobenius algebra is weakly separable.
\end{lemma}

\begin{proof}
Note that $(m \circ \delta) \circ e \in \Hom_{\mathcal{C}}(\one, X)$. By the connectedness condition, there exists $\lambda \in \mathbbm{k}$ such that $(m \circ \delta) \circ e = \lambda e$, as shown below:
 				\[ \raisebox{-10mm}
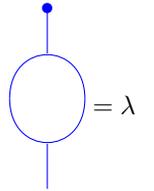
{
 					\begin{tikzpicture}
 					\draw[blue,in=90,out=90,looseness=2] (0,0) to (1,0);
 					\draw[blue,in=-90,out=-90,looseness=2] (0,0) to (1,0);
 					\draw[blue] (.5,-.6) to (.5,-1.2);
 					\draw[blue] (.5,.6) to (.5,1.2);
 					\node at (.5,1.2) {$\color{blue}{\bullet}$};
 					\end{tikzpicture}} = 
 				\lambda \raisebox{-10mm}{
 					\begin{tikzpicture}
 					\draw[blue] (0,0) to (0,2.4);
 					\node at (0,2.4) {$\color{blue}{\bullet}$};
 					\end{tikzpicture}} \]
Applying Lemma \ref{lem:grap} with $\alpha = \id_X$, along with the previous equality and unitality, we obtain:
 				\[ \raisebox{-10mm}
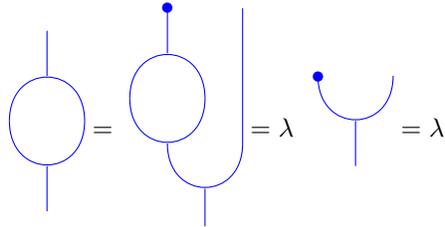
{
 					\begin{tikzpicture}
 					\draw[blue,in=90,out=90,looseness=2] (0,0) to (1,0);
 					\draw[blue,in=-90,out=-90,looseness=2] (0,0) to (1,0);
 					\draw[blue] (.5,-.6) to (.5,-1.2);
 					\draw[blue] (.5,.6) to (.5,1.2);
 					\end{tikzpicture}}
 					=
 					\raisebox{-12mm}{
 						\begin{tikzpicture}
 						\draw[blue,in=90,out=90,looseness=2] (0,0) to (1,0);
 						\draw[blue,in=-90,out=-90,looseness=2] (0,0) to (1,0);
 						\draw[blue,in=-90,out=-90,looseness=2] (.5,-.6) to (1.5,-.6);
 						\draw[blue] (1.5,-.6) to (1.5,1.2);
 						\draw[blue] (.5,.6) to (.5,1.2);
 						\draw[blue] (1,-1.2) to (1,-1.7);
 						\node at (.5,1.2) {$\color{blue}{\bullet}$};
 						\end{tikzpicture}}
 						= \lambda \raisebox{-4mm}{
 							\begin{tikzpicture}[rotate=180,transform shape]
 							\draw[blue,in=90,out=90,looseness=2] (0,0) to (1,0);
 							\node at (1,0) {$\textcolor{blue}{\bullet}$};
 							\draw[blue] (.5,.6) to (.5,1.2);
 							\end{tikzpicture}} 
 							=
 							 \lambda \, \raisebox{-4mm}{\begin{tikzpicture}
 							 \draw[blue] (0,0) to (0,1);
 							 \end{tikzpicture}} \]
Thus, the result follows.
\end{proof}

If there is no possibility for confusion, a Frobenius algebra $(X, m, \delta, e, \epsilon)$ may simply be referred to as $X$.

\begin{definition} \label{def:trace}
Let $X$ be an object in a monoidal category $\mathcal{C}$. Assume that $X$ has a left dual and a double dual, denoted $X^*$ and $X^{**}$, respectively. The \emph{trace} of a morphism $\alpha \in \Hom_{\mathcal{C}}(X, X^{**})$, denoted $\tr(\alpha)$, is defined as 
$$
\tr(\alpha) = \ev_{X^*} \circ (\alpha \otimes \id_{X^*}) \circ \coev_X \in \Hom_{\mathcal{C}}(\one, \one),
$$
as illustrated below:
 			\[ \tr(\alpha) := \raisebox{-6mm}{\begin{tikzpicture}
 					\draw[blue,in=90,out=90,looseness=2] (0,0) to (1,0);
 					\draw[blue,in=-90,out=-90,looseness=2] (0,0) to (1,0);
 					\node[draw,thick,rounded corners, fill=white,minimum width=20] at (0,0){$\alpha$};
 					\node[scale=.8] at (-.2,-.4) {$X^{**}$};
 					\node[scale=.8] at (-.2,.4) {$X$};
 					\node[scale=.8] at (1.3,0) {$X^*$};
 			\end{tikzpicture}}\]
If $X^{**}$ also has a left dual, then the following holds:
 		\[\raisebox{-6mm}
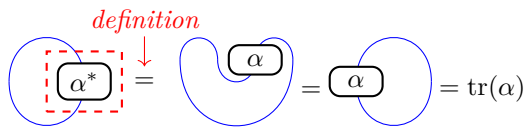
{\begin{tikzpicture}
 			\draw[blue,in=90,out=90,looseness=2] (0,0) to (1,0);
 			\draw[blue,in=-90,out=-90,looseness=2] (0,0) to (1,0);
 			\node[draw,thick,rounded corners, fill=white,minimum width=20] at (1,0){$\alpha^*$};
 			\draw[dashed,red,thick] (.5,.4) rectangle (1.5,-.4);
 			\draw[->,red] (1.8,.6) to (1.8,.2);
 			\node at (1.8,.8) [red] {definition};
 			\node at (1.8,0) {$=$}; 
 			\end{tikzpicture}} \hspace*{-.5cm}
 					\raisebox{-7mm}{
 					\begin{tikzpicture}
 					\draw[blue,in=-90,out=-90,looseness=2] (0,0) to (.5,0);
 					\draw[blue,in=90,out=90,looseness=2] (.5,0) to (1,0);
 					\draw[blue,in=90,out=90,looseness=2] (0,0) to (-.5,0);
 					\draw[blue,in=-90,out=-90,looseness=2] (-.5,0) to (1,0);
 					\node[draw,thick, rounded corners, fill=white,minimum width = 20] at (.5,0) {$\alpha$}; 
 					\end{tikzpicture}}
 				=
 				 \raisebox{-6mm}{\begin{tikzpicture}
 					\draw[blue,in=90,out=90,looseness=2] (0,0) to (1,0);
 					\draw[blue,in=-90,out=-90,looseness=2] (0,0) to (1,0);
 					\node[draw,thick,rounded corners, fill=white,minimum width=20] at (0,0){$\alpha$};
 					\end{tikzpicture}}
 					= \tr(\alpha) \]
\end{definition}

\begin{remark} \label{rk:LinSimple}
If $\mathcal{C}$ is a $\mathbbm{k}$-linear category and the unit object $\one$ is linear-simple (i.e., $\Hom_{\mathcal{C}}(\one, \one) = \mathbbm{k}$, as in a tensor category), then the elements of $\Hom_{\mathcal{C}}(\one, \one)$ can be interpreted as scalars. If $\mathcal{C}$ is abelian, then by Schur's lemma, if $\one$ is simple, $\Hom_{\mathcal{C}}(\one, \one)$ forms a division algebra. Consequently, if $\mathbbm{k}$ is algebraically closed, $\one$ must be linear-simple.
\end{remark}

\begin{lemma} \label{lem:trfg}
Let \(A\) and \(B\) be objects in a monoidal category \(\mathcal{C}\), and suppose that both admit double duals \(A^{**}\) and \(B^{**}\).
Consider two morphisms $f: A \to B$ and $g: B \to A^{**}$. Then, it holds that 
$$
\tr(g \circ f) = \tr(f^{**} \circ g).
$$
\end{lemma}

\begin{proof}
This can be proven pictorially:
 		\[ \tr(f^{**} \circ g) 
 			=
 			\raisebox{-12mm}{\begin{tikzpicture}
 				\draw[blue,in=90,out=90,looseness=2] (0,0) to (1,0);
 				\draw[blue,in=-90,out=-90,looseness=2] (0,-1) to (1,-1);
 				\draw[blue] (0,-1) to (0,0);
 				\draw[blue] (1,-1) to (1,0);
 				\node[draw,thick,rounded corners, fill=white,minimum width=20] at (0,0){$g$};
 				\node[draw,thick,rounded corners, fill=white,minimum width=20] at (0,-.8){$f^{**}$};
 				\end{tikzpicture}}
 			= 
 			\raisebox{-12mm}{
 			\begin{tikzpicture}
 			\draw[blue,in=90,out=90,looseness=2] (0,.2) to (.5,.2);
 			\draw[blue,in=90,out=90,looseness=2] (-.5,.2) to (.7,.2);
 			\draw[blue,in=-90,out=-90,looseness=2] (0,-.2) to (-.5,-.2);
 			\draw[blue,in=-90,out=-90,looseness=2] (.5,-.2) to (-.85,-.2);
 			\draw[blue,in=-90,out=-90,looseness=2] (.7,.2) to (1.2,.2);
 			\draw[blue] (-.5,.2) to (-.5,-.2);
 			\draw[blue] (.5,.2) to (.5,-.2);
 			\draw[blue,in=90,out=90,looseness=2] (-.85,.2) to (1.2,.2);
 			\node[draw,thick,rounded corners, fill=white,minimum width=20] at (0,0){$f$};
 			\node[draw,thick,rounded corners, fill=white,minimum width=15] at (-.85,0){$g$};
 			\end{tikzpicture}} 
 			=
 			\raisebox{-12mm}{
 				\begin{tikzpicture}
 				\draw[blue,in=-90,out=-90,looseness=2] (0,0) to (.75,0);
 				\draw[blue,in=90,out=90,looseness=2] (.75,0) to (1.5,0);
 				\draw[blue,in=90,out=90,looseness=2] (0,0) to (-.75,0);
 				\draw[blue,in=-90,out=-90,looseness=2] (-.75,0) to (1.5,0);
 				\node[draw,thick, rounded corners, fill=white,minimum width = 20] at (.75,0) {$f$};
 				\node[draw,thick, rounded corners, fill=white,minimum width = 20] at (-.75,0) {$g$}; 
 				\end{tikzpicture}}
 			=
 			\raisebox{-12mm}{\begin{tikzpicture}
 				\draw[blue,in=90,out=90,looseness=2] (0,0) to (1,0);
 				\draw[blue,in=-90,out=-90,looseness=2] (0,-1) to (1,-1);
 				\draw[blue] (0,-1) to (0,0);
 				\draw[blue] (1,-1) to (1,0);
 				\node[draw,thick,rounded corners, fill=white,minimum width=20] at (0,0){$f$};
 				\node[draw,thick,rounded corners, fill=white,minimum width=20] at (0,-.8){$g$};
 				\end{tikzpicture}}
 			= \tr(g \circ f) \qedhere \]
\end{proof}

\begin{definition}[{\cite[Definition 4.7.7]{EGNO15}}] \label{def:pivotal}
A \emph{pivotal structure} on a rigid monoidal category $\mathcal{C}$ is a natural isomorphism of monoidal functors $\phi:\id_{\mathcal{C}} \to ()^{**}$, i.e. a collection of isomorphisms $\phi_X: X \to X^{**}$ such that:
\begin{enumerate}
\item  for all $f$ in $\Hom_{\mathcal{C}}(X,Y)$ then $f = \phi_Y^{-1} \circ f^{**} \circ \phi_X$ (naturality) \label{nat}
\item  $\phi_{X \otimes Y} = \phi_{X} \otimes \phi_{Y} $ (monoidal)
\end{enumerate}
\end{definition}

Let us define a weaker version without the monoidal structure:
\begin{definition} \label{def:quasipivotal}
A \emph{quasi-pivotal structure} on a rigid monoidal category $\mathcal{C}$ is a natural isomorphism $\phi:\id_{\mathcal{C}} \to ()^{**}$, i.e. a collection of isomorphisms $\phi_X: X \to X^{**}$ such that (\ref{nat}) from Definition \ref{def:pivotal} holds.
\end{definition}
 
\noindent If $X=Y$, let us define $\tr_\phi(f):=\tr(\phi_X \circ f)$. 

\begin{lemma} \label{lem:trfgpivotalquasi}
Let $X$ and $Y$ be two objects in a rigid monoidal category $\mathcal{C}$ with quasi-pivotal structure $\phi$. Consider two morphisms $f: X \to Y$ and $g: Y \to X$. Then, it holds that $\tr_\phi(g \circ f) = \tr_\phi(f \circ g)$.
\end{lemma}
\begin{proof}
By Lemma \ref{lem:trfg} and the quasi-pivotal structure: $$ \tr_\phi(g \circ f) = \tr(\phi_X \circ g \circ f)  = \tr(f^{**} \circ \phi_X \circ g) = \tr(\phi_Y \circ f \circ g) = \tr_\phi(f \circ g). \qedhere$$ 
\end{proof}


\begin{lemma} \label{lem:loop}
Let $(X, m, \delta, e, \epsilon)$ be a Frobenius algebra. Then, $
\tr(\id_X) = \ev_{X} \circ \coev_X = \epsilon \circ m \circ \delta \circ e$.
In the weakly separable case, where $m \circ \delta = \lambda_X \id_X$, and assuming that the unit object $\one$ is linear-simple, let $\mu_X$ denote the scalar defined by $\epsilon \circ e$. Then, we have $\tr(\id_X) = \lambda_X \mu_X$.
\end{lemma}

\begin{proof}
This follows immediately from Lemma \ref{lem:selfdual}. The relation is illustrated below:
 		\[ \tr(\id_X) 
 			=
 			 \raisebox{-6mm}{
 			 \begin{tikzpicture}
 			 \draw[blue,in=90,out=90,looseness=2] (0,0) to (1,0);
 			 \draw[blue,in=-90,out=-90,looseness=2] (0,0) to (1,0);
 			 \draw[dashed,green!70!black,thick] (-.1,.1) rectangle (1.1,.7);
 			 \draw[dashed,green!70!black,thick] (-.1,-.1) rectangle (1.1,-.7);
 			 \draw[->,green!70!black] (1.5,.7) to (1.5,.2);
 			 \node at (1.5,1) [green!70!black] {selfdual};
 			 \node at (1.5,-.1) {$=$};
 			 \end{tikzpicture}}
 		 	\raisebox{-10mm}{
 		 		\begin{tikzpicture}
 		 		\draw[blue,in=90,out=90,looseness=2] (0,0) to (1,0);
 		 		\draw[blue,in=-90,out=-90,looseness=2] (0,0) to (1,0);
 		 		\draw[blue] (.5,-.6) to (.5,-1.2);
 		 		\draw[blue] (.5,.6) to (.5,1.2);
 		 		\draw[dashed,violet,thick] (-.2,.7) rectangle (1.2,-.7);
 		 		\draw[->,violet] (1.7,.7) to (1.7,0);
 		 		\node at (1.5,1.2) [violet] {\begin{tabular}{c} weak \\ separability \end{tabular}};
 		 		\node at (.5,1.2) {$\color{blue}{\bullet}$};
 		 		\node at (.5,-1.2) {$\color{blue}{\bullet}$};
 		 		\node at (1.7,-.3) {$=$};
 		 		\end{tikzpicture}} 
 		 	\lambda_X \raisebox{-10mm}{
 		 		\begin{tikzpicture}
 		 		\draw[blue] (0,0) to (0,2.4);
 		 		\node at (0,2.4) {$\color{blue}{\bullet}$};
 		 		\node at (0,0) {$\color{blue}{\bullet}$};
 		 		\end{tikzpicture}}
 	 		= \lambda_X \mu_X  \qedhere \]
\end{proof}

\begin{lemma} \label{lem:contr}
Using the notation from Lemma \ref{lem:loop}, let $\alpha \in \End_{\mathcal C}(X)$ and assume the connectedness condition. Then, $
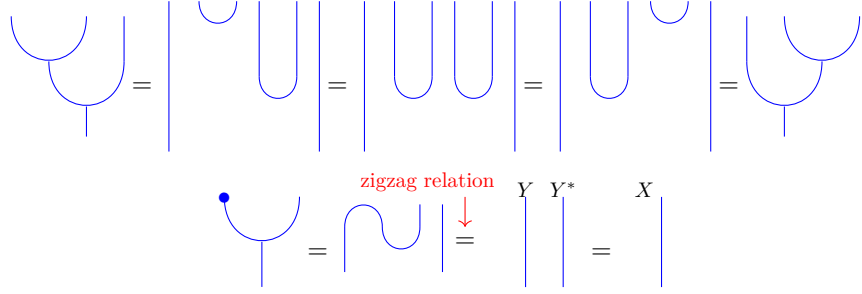
\mu_X m \circ (\alpha \otimes \id_X) \circ \delta = \tr(\alpha) \id_X$, as depicted below:
 	 		\[ \mu_X \raisebox{-11mm}{
 	 			\begin{tikzpicture}
 	 			\draw[blue,in=90,out=90,looseness=2] (0,0) to (1,0);
 	 			\draw[blue,in=-90,out=-90,looseness=2] (0,0) to (1,0);
 	 			\draw[blue] (.5,-.6) to (.5,-1.2);
 	 			\draw[blue] (.5,.6) to (.5,1.2);
 	 			\node[draw,thick,rounded corners, fill=white,minimum width=20] at (0,0){$\alpha$}; 
 	 			\end{tikzpicture}} = \tr(\alpha) \, \raisebox{-4mm}{\begin{tikzpicture}
  			\draw[blue] (0,0) to (0,1.2);
  		\end{tikzpicture}}  \]
\end{lemma}

\begin{proof}
Here is a graphical proof:
  		\[\raisebox{-11mm}{
  			\begin{tikzpicture}
  			\draw[blue,in=90,out=90,looseness=2] (0,0) to (1,0);
  			\draw[blue,in=-90,out=-90,looseness=2] (0,0) to (1,0);
  			\draw[blue] (.5,-.6) to (.5,-1.2);
  			\draw[blue] (.5,.6) to (.5,1.2);
  			\draw[->,orange] (1.5,.5) to (1.5,.1);
  			\node[draw,thick,rounded corners, fill=white,minimum width=20] at (0,0){$\alpha$};
  			\node at (.5,1.2) {$\color{blue}{\bullet}$};
  			\node at (1.5,.7) [orange] {connected}; 
  			\node at (1.5,0) {$=$};
  			\end{tikzpicture}}
  			c_X \raisebox{-4mm}{\begin{tikzpicture}
  			\draw[blue] (0,0) to (0,1);
  			\node at (0,1) {$\color{blue}{\bullet}$};
  			\end{tikzpicture}}
  			\implies 
  			\tr(\alpha) 
  			=
  			\raisebox{-9mm}{
  				\begin{tikzpicture}
  				\draw[blue,in=90,out=90,looseness=2] (0,0) to (1,0);
  				\draw[blue,in=-90,out=-90,looseness=2] (0,0) to (1,0);
  				\draw[dashed,green!70!black,thick] (-.1,.4) rectangle (1.1,.9);
  				\draw[dashed,green!70!black,thick] (-.1,-.4) rectangle (1.1,-.9);
  				\draw[->,green!70!black] (1.5,.7) to (1.5,.2);
  				\node[draw,thick,rounded corners, fill=white,minimum width=20] at (0,0){$\alpha$};
  				\node at (1.5,1.1) [green!70!black] {selfdual};
  				\node at (1.5,0) {$=$};
  				\end{tikzpicture}}
  			\raisebox{-11mm}{
  				\begin{tikzpicture}
  				\draw[blue,in=90,out=90,looseness=2] (0,0) to (1,0);
  				\draw[blue,in=-90,out=-90,looseness=2] (0,0) to (1,0);
  				\draw[blue] (.5,-.6) to (.5,-1.2);
  				\draw[blue] (.5,.6) to (.5,1.2);
  				\node[draw,thick,rounded corners, fill=white,minimum width=20] at (0,0){$\alpha$};
  				\node at (.5,1.2) {$\color{blue}{\bullet}$};
  				\node at (.5,-1.2) {$\color{blue}{\bullet}$};
  				\end{tikzpicture}}
  			=
  			c_X \raisebox{-6mm}{\begin{tikzpicture}
  				\draw[blue] (0,0) to (0,1);
  				\node at (0,1) {$\color{blue}{\bullet}$};
  				\node at (0,0) {$\color{blue}{\bullet}$};
  				\end{tikzpicture}}
  			=
  			c_X \mu_X \]
Thus, we have $c_X \mu_X = \tr(\alpha)$. Next, by applying Lemma \ref{lem:grap} along with the preceding equalities, we obtain:
  			\[ \mu_X \raisebox{-11mm}{
  				\begin{tikzpicture}
  				\draw[blue,in=90,out=90,looseness=2] (0,0) to (1,0);
  				\draw[blue,in=-90,out=-90,looseness=2] (0,0) to (1,0);
  				\draw[blue] (.5,-.6) to (.5,-1.2);
  				\draw[blue] (.5,.6) to (.5,1.2);
  				\node[draw,thick,rounded corners, fill=white,minimum width=20] at (0,0){$\alpha$}; 
  				\end{tikzpicture}} 
  				= 
  				\mu_X \raisebox{-12mm}{
  					\begin{tikzpicture}
  					\draw[blue,in=90,out=90,looseness=2] (0,0) to (1,0);
  					\draw[blue,in=-90,out=-90,looseness=2] (0,0) to (1,0);
  					\draw[blue,in=-90,out=-90,looseness=2] (.5,-.6) to (1.5,-.6);
  					\draw[blue] (1.5,-.6) to (1.5,1.2);
  					\draw[blue] (.5,.6) to (.5,1.2);
  					\draw[blue] (1,-1.2) to (1,-1.7);
  					\node[draw,thick,rounded corners, fill=white,minimum width=20] at (0,0){$\alpha$};
  					\node at (.5,1.2) {$\color{blue}{\bullet}$};
  					\end{tikzpicture}}
  				=  
  				\mu_X c_X \raisebox{-4mm}{
  					\begin{tikzpicture}[rotate=180,transform shape]
  					\draw[blue,in=90,out=90,looseness=2] (0,0) to (1,0);
  					\node at (1,0) {$\textcolor{blue}{\bullet}$};
  					\draw[blue] (.5,.6) to (.5,1.2);
  					\end{tikzpicture}}
  				=
  				\tr(\alpha) \, \raisebox{-4mm}{\begin{tikzpicture}
  					\draw[blue] (0,0) to (0,1.2);
  					\end{tikzpicture}} \qedhere \]
 
\end{proof}

\begin{proposition} \label{prop:YY}
Let $Y$ be an object with left dual $Y^*$ and double dual $Y^{**}$ in a monoidal category. Assume that $Y^{**} = Y$. Then, the object $X := Y \otimes Y^*$ can be equipped with a structure of Frobenius algebra, defined by $m = \id_Y \otimes \ev_Y \otimes \id_{Y^*}$, $\delta = \id_Y \otimes \coev_{Y^*}\otimes \id_{Y^*}$, $e = \coev_Y$, and $\epsilon = \ev_{Y^*}$, as depicted below:
  				\[ m = \raisebox{-6mm}{
  					\begin{tikzpicture}
  					\draw[blue,in=-90,out=-90,looseness=2] (-0.5,0.5) to (-1.5,0.5);
  					\draw[blue] (-1,-.1) to (-1,-.6);
  					\node[left,scale=0.7] at (-1,-.4) {$X$};
  					\node[left,scale=0.7] at (-1.6,0.5) {$X$};
  					\node[right,scale=0.7] at (-.5,.5) {$X$};
  					\end{tikzpicture}}
  					\coloneqq
  					\raisebox{-8mm}{
  					\begin{tikzpicture}
  					\draw[blue] (0,.5) to (0,2);
  					\draw[blue,in=-90,out=-90,looseness=2] (.4,2) to (1.4,2);
  					\draw[blue] (1.8,.5) to (1.8,2);
  					\node[scale=.8] at (0,2.3) {$Y$};
  					\node[scale=.8] at (.5,2.3) {$Y^*$};
  					\node[scale=.8] at (1.3,2.3) {$Y$};
  					\node[scale=.8] at (1.9,2.3) {$Y^*$};
  					\end{tikzpicture}} 
  					=
  					\id_Y \otimes \ev_Y \otimes \id_{Y^*} \]
  					
  					\[\delta 
  						= 
  						\raisebox{-6mm}{
  						\begin{tikzpicture}
  						\draw[blue,in=90,out=90,looseness=2] (-0.5,0.5) to (-1.5,0.5);
  						\draw[blue] (-1,1.1) to (-1,1.6);
  						\node[left,scale=0.7] at (-1,1.6) {$X$};
  						\node[left,scale=0.7] at (-1.6,0.5) {$X$};
  						\node[right,scale=0.7] at (-.5,.5) {$X$};
  						\end{tikzpicture}} 
  						\coloneqq
  						\raisebox{-10mm}{
  						\begin{tikzpicture}
  						\draw[blue] (0,0) to (0,1.5);
  						\draw[blue,in=90,out=90,looseness=2] (.4,0) to (1.4,0);
  						\draw[blue] (1.8,0) to (1.8,1.5);
  						\node[scale=.8] at (0,-.3) {$Y$};
  						\node[scale=.8] at (.5,-.3) {$Y^*$};
  						\node[scale=.8] at (1.3,-.3) {$Y$};
  						\node[scale=.8] at (1.9,-.3) {$Y^*$};
  						\end{tikzpicture}}
  						=
  						\id_Y \otimes \coev_{Y^*} \otimes \id_{Y^*} \]
  						
  						\[ e= \raisebox{-6mm}{
  							\begin{tikzpicture}
  							\draw [blue] (-0.8,-.6) to (-.8,.6);
  							\node at (-.8,.6) {${\color{blue}\bullet}$};
  							\node[left,scale=.8] at (-.8,.9) {$\one$};
  							\node[left,scale=0.7] at (-.8,-.5) {$X$};
  							\end{tikzpicture}}
  							\coloneqq
  							\raisebox{-6mm}{
  							\begin{tikzpicture}
  							\draw[blue,in=90,out=90,looseness=2] (0,0) to (1,0);
  							\node[scale=.8] at (0,-.3) {$Y$};
  							\node[scale=.8] at (1,-.3) {$Y^*$};
  							\end{tikzpicture}}
  							=
  							\coev_Y\]
  							
  								\[ \epsilon = \raisebox{-8mm}{
  								\begin{tikzpicture}
  								\draw [blue] (-0.8,-.6) to (-.8,.6);
  								\node at (-.8,-.6) {${\color{blue}\bullet}$};
  								\node[left,scale=0.7] at (-.8,.6) {$X$};
  								\node[scale=.8] at (-.8,-.9) {$\one$};
  								\end{tikzpicture}}
  								\coloneqq
  								\raisebox{-6mm}{
  								\begin{tikzpicture}
  								\draw[blue,in=-90,out=-90,looseness=2] (0,0) to (1,0);
  								\node[scale=.8] at (0,.3) {$Y$};
  								\node[scale=.8] at (1,.3) {$Y^*$};
  								\end{tikzpicture}}
  								=
  								\ev_{Y^*} \]
Moreover, if $\mathcal{C}$ is a $\mathbbm{k}$-linear category and $Y$ is linear-simple, then $X$ is connected.								
\end{proposition}

\begin{proof}
We can verify all the assumptions graphically:
  				\[\raisebox{-8mm}{
  					\begin{tikzpicture}
  					\draw[blue,in=90,out=90,looseness=2] (0,0) to (1,0);
  					\draw[blue,in=-90,out=-90,looseness=2] (1,0) to (2,0);
  					\draw[blue] (.5,.6) to (.5,1.2);
  					\draw[blue] (1.5,-.6) to (1.5,-1.2);
  					\draw[blue] (0,0) to (0,-1.2);
  					\draw[blue] (2,0) to (2,1.2);
  					\end{tikzpicture}} 
  					=
  					\raisebox{-8mm}{
  					\begin{tikzpicture}
  					\draw[blue] (0,0) to (0,2);
  					\draw[blue,in=90,out=90,looseness=2] (.4,1.2) to (.9,1.2);
  					\draw[blue,in=-90,out=-90,looseness=2] (1.2,1.2) to (1.7,1.2);
  					\draw[blue] (2,0) to (2,2);
  					\draw[blue] (.4,1.2) to (.4,0);
  					\draw[blue] (.9,1.2) to (.9,0);
  					\draw[blue] (1.2,1.2) to (1.2,2);
  					\draw[blue] (1.7,1.2) to (1.7,2);
  					\end{tikzpicture}}
  					=
  					\raisebox{-8mm}{
  						\begin{tikzpicture}
  						\draw[blue] (0,0) to (0,2);
  						\draw[blue,in=90,out=90,looseness=2] (.6,.5) to (1.1,.5);
  						\draw[blue,in=-90,out=-90,looseness=2] (.6,1.5) to (1.1,1.5);
  						\draw[blue] (1.5,0) to (1.5,2);
  						\draw[blue] (.6,.5) to (.6,0);
  						\draw[blue] (1.1,.5) to (1.1,0);
  						\draw[blue] (.6,1.5) to (.6,2);
  						\draw[blue] (1.1,1.5) to (1.1,2);
  						\end{tikzpicture}}
  					=
  					\raisebox{-8mm}{
  						\begin{tikzpicture}
  						\draw[blue] (0,0) to (0,2);
  						\draw[blue,in=-90,out=-90,looseness=2] (.4,1.2) to (.9,1.2);
  						\draw[blue,in=90,out=90,looseness=2] (1.2,1.2) to (1.7,1.2);
  						\draw[blue] (2,0) to (2,2);
  						\draw[blue] (.4,1.2) to (.4,2);
  						\draw[blue] (.9,1.2) to (.9,2);
  						\draw[blue] (1.2,1.2) to (1.2,0);
  						\draw[blue] (1.7,1.2) to (1.7,0);
  						\end{tikzpicture}}
  						=
  						\raisebox{-8mm}{
  							\begin{tikzpicture}
  							\draw[blue,in=-90,out=-90,looseness=2] (0,0) to (1,0);
  							\draw[blue,in=90,out=90,looseness=2] (1,0) to (2,0);
  							\draw[blue] (.5,-.6) to (.5,-1.2);
  							\draw[blue] (1.5,.6) to (1.5,1.2);
  							\draw[blue] (0,0) to (0,1.2);
  							\draw[blue] (2,0) to (2,-1.2);
  							\end{tikzpicture}}\]
  				
  				\[\raisebox{-6mm}{
  					\begin{tikzpicture}[rotate=180,transform shape]
  					\draw[blue,in=90,out=90,looseness=2] (0,0) to (1,0);
  					\draw[blue,in=90,out=90,looseness=2] (0.5,.6) to (-.5,.6);
  					\draw[blue] (-.5,.6) to (-.5,0);
  					\draw[blue] (0,1.2) to (0,1.6);
  					\end{tikzpicture}}
  				=
  				\raisebox{-8mm}{
  				\begin{tikzpicture}
  				\draw[blue] (0,0) to (0,2);
  				\draw[blue,in=-90,out=-90,looseness=2] (.4,2) to (.9,2);
  				\draw[blue,in=-90,out=-90,looseness=2] (1.2,1) to (1.7,1);
  				\draw[blue] (2,0) to (2,2);
  				\draw[blue] (1.2,1) to (1.2,2);
  				\draw[blue] (1.7,1) to (1.7,2);
  				\end{tikzpicture}}
  				=
  				\raisebox{-8mm}{
  					\begin{tikzpicture}
  					\draw[blue] (0,0) to (0,2);
  					\draw[blue,in=-90,out=-90,looseness=2] (.4,1) to (.9,1);
  					\draw[blue,in=-90,out=-90,looseness=2] (1.2,1) to (1.7,1);
  					\draw[blue] (2,0) to (2,2);
  					\draw[blue] (1.2,1) to (1.2,2);
  					\draw[blue] (1.7,1) to (1.7,2);
  					\draw[blue] (.4,1) to (.4,2);
  					\draw[blue] (.9,1) to (.9,2);
  					\end{tikzpicture}}
  				=
  				\raisebox{-8mm}{
  					\begin{tikzpicture}
  					\draw[blue] (0,0) to (0,2);
  					\draw[blue,in=-90,out=-90,looseness=2] (.4,1) to (.9,1);
  					\draw[blue,in=-90,out=-90,looseness=2] (1.2,2) to (1.7,2);
  					\draw[blue] (2,0) to (2,2);
  					\draw[blue] (.4,1) to (.4,2);
  					\draw[blue] (.9,1) to (.9,2);
  					\end{tikzpicture}}
  				=
  				 \raisebox{-6mm}{\begin{tikzpicture}[rotate=180,transform shape]
  					\draw[blue,in=90,out=90,looseness=2] (0,0) to (1,0);
  					\draw[blue,in=90,out=90,looseness=2] (.5,.6) to (1.5,.6);
  					\draw[blue] (1.5,.6) to (1.5,0);
  					\draw[blue] (1,1.2) to (1,1.6);
  					\end{tikzpicture}} \]
  				
  				
  				\[ \raisebox{-4mm}{
  					\begin{tikzpicture}[rotate=180,transform shape]
  					\draw[blue,in=90,out=90,looseness=2] (0,0) to (1,0);
  					\node at (1,0) {$\textcolor{blue}{\bullet}$};
  					\draw[blue] (.5,.6) to (.5,1.2);
  					\end{tikzpicture}} 
  					=
  					\raisebox{-2mm}{
  					\begin{tikzpicture}
  					\draw[blue,in=90,out=90,looseness=2] (0,0) to (.5,0);
  					\draw[blue,in=-90,out=-90,looseness=2] (.5,0) to (1,0);
  					\draw[blue] (0,0) to (0,-.6);
  					\draw[blue] (1,0) to (1,.3);
  					\draw[blue] (1.3,-.6) to (1.3,.3);
  					\draw[->,red] (1.6,.4) to (1.6,0);
  					\node[scale=.8] at (1.1,.6) [red] {zigzag relation};
  					\node at (1.6,-.2) {$=$};
  					\end{tikzpicture}}
  					\raisebox{-4mm}{
  					\begin{tikzpicture}
  					\draw[blue] (0,0) to (0,1.2);
  					\draw[blue] (.5,0) to (.5,1.2);
  					\node[scale=.8] at (0,1.3) {$Y$};
  					\node[scale=.8] at (.5,1.3) {$Y^*$};
  					\end{tikzpicture}}
  					=
  					\raisebox{-4mm}{
  					\begin{tikzpicture}
  					\draw[blue] (0,0) to (0,1.2);
  					\node[scale=.8,left] at (0,1.3) {$X$};
  					\end{tikzpicture}}\]
The same applies for coassociativity and counitality. 

Finally, if $\mathcal{C}$ is a $\mathbbm{k}$-linear category, then $\Hom_{\mathcal{C}}(\one, X) = \Hom_{\mathcal{C}}(\one, Y \otimes Y^*)$ is isomorphic to $\Hom_{\mathcal{C}}(Y, Y)$ through a zigzag relation and is one-dimensional if $Y$ is linear-simple. \end{proof}

Proposition \ref{prop:YY} extends to $2$-categories as follows, using a similar pictorial approach (see \cite[Lemma 3.4]{Mug03}) that also incorporates a zigzag relation.
\begin{proposition} \label{prop:bimodYY}
Consider a $2$-category $\mathfrak{A}$, with two objects, denoted as $C$ and $D$ (i.e., $\mathfrak{A}_0 = \{C, D\}$). Define $\mathcal{C}$ and $\mathcal{D}$ as the monoidal categories $\End_{\mathfrak{A}}(C)$ and $\End_{\mathfrak{A}}(D)$, respectively. Let $\mathcal{M}$ represent the $(\mathcal{C}, \mathcal{D})$-bimodule category, specifically $\Hom_{\mathfrak{A}}(C, D)$. Suppose $Y$ is an object in $\mathcal{M}$, and let $Y^{\vee}$ in $\Hom_{\mathfrak{A}}(D, C)$ be a two-sided dual of $Y$. Then the object $X := Y \circ Y^{\vee}$ in $\mathcal{D}$ can be structured as a Frobenius algebra. Moreover, if $\mathfrak{A}$ is $\mathbbm{k}$-linear and $Y$ is linear-simple, then $X$ is connected.
\end{proposition}

\begin{remark} \label{rk:MugConverse}
It is worth noting that the converse of Proposition \ref{prop:bimodYY} also holds, as shown in \cite[Proposition 3.8]{Mug03}.
\end{remark}


\begin{example}[\cite{Mug03}] \label{ex:subf} 
Let $(N \subset M)$ be a finite-index inclusion of $\mathrm{II}_1$ factors. Then $X = {}_N M_N$ is a Frobenius algebra in the tensor category $\mathcal{C}$ of $N$-$N$-bimodules. This follows the structure outlined in Proposition \ref{prop:bimodYY}, where $\mathcal{D}$ is the tensor category of $M$-$M$-bimodules, $Y = {}_N M_M$, and $Y^{\vee} = {}_M M_N$, such that $X = Y \otimes_M Y^{\vee}$. Moreover, the subfactor $(N \subset M)$ is irreducible (i.e., $N' \cap M = \mathbb{C}$) if and only if the Frobenius algebra $X$ is connected.
\end{example}

\begin{remark} \label{rk:subf}
It is a well-known fact that every unitary Frobenius algebra in a unitary tensor category can be realized in the form presented in Example~\ref{ex:subf}. Here is a sketch of one possible argument: starting with a unitary Frobenius algebra in a unitary tensor category, one invokes the construction in \cite{Mug03} (see Remark~\ref{rk:MugConverse}) to obtain a 2-category as described in Proposition~\ref{prop:bimodYY}, commonly known as a Morita context. Leveraging the results of \cite{Shami} along with unitarity, this structure gives rise to a subfactor planar algebra. Finally, the Popa–Jones reconstruction theorem \cite{J21} recovers the desired realization. In the irreducible case, this result traces back to \cite{Lon94}; see also the more recent developments in \cite{JP17, JP19}.
\end{remark}

The notion of \emph{unitary} is reviewed in Remark \ref{rk:UnitaryTensor} and Definition \ref{def:UnitFrob}.

\section{Frobenius subalgebras} \label{sec:FrobSub}

\subsection{Basic material} \label{sub:FrobSubPre}

Example \ref{ex:subf} shows the sense in which Theorem \ref{thm:allfinite} extends Watatani's theorem from subfactor theory. To generalize the concept of intermediate subfactors, it is essential to clarify what we mean by Frobenius subalgebras. First, let us recall why a morphism of Frobenius algebras is an isomorphism:

\begin{definition} \label{def:FrobMor}
A morphism $f$ between Frobenius algebras $(X, m, \delta, e, \epsilon)$ and $(X', m', \delta', e', \epsilon')$ in a monoidal category $\mathcal{C}$ is called a Frobenius algebra morphism if it satisfies the following conditions:
\begin{itemize}
\item (Unital algebra morphism) $e' = f \circ e$ and $m' \circ (f \otimes f) = f \circ m$,
\item (Counital coalgebra morphism) $\epsilon = \epsilon' \circ f$ and $ (f \otimes f) \circ \delta = \delta' \circ f$.
\end{itemize}
These conditions can be illustrated as follows:
  				\[ 
  				 \raisebox{-6mm}
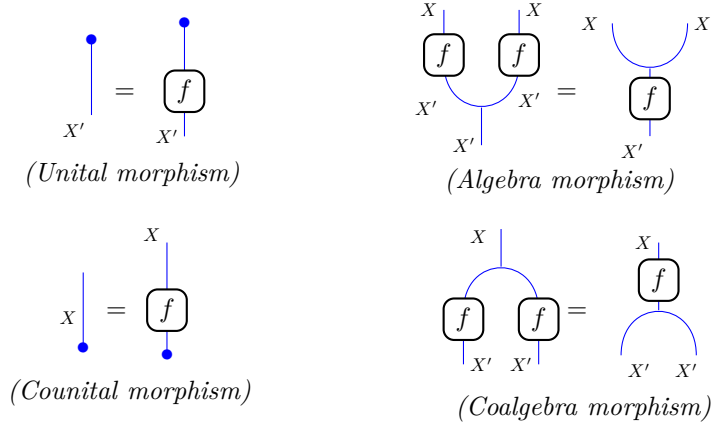
{\begin{tikzpicture}
  				 \draw[blue] (0,0) to (0,1);
  				 \node at (0,1) {$\textcolor{blue}{\bullet}$};
  				 \node[left,scale=0.7] at (0,-.2) {$X'$};
  				 \end{tikzpicture}}
  			 		= \hspace*{-1.7cm}
  			    \raisebox{-13mm}{\begin{tikzpicture}
  			 			\draw[blue] (0,0) to (0,1.5);
  			 			 \node[draw,thick,rounded corners, fill=white,minimum width=15] at (0,.6){$f$};
  			 			 \node at (0,1.5) {$\textcolor{blue}{\bullet}$};
  			 			\node[left,scale=0.7] at (0,0) {$X'$};
  			 	\node at (-.7,-.5) {\text{(Unital morphism)}};	
  			 			\end{tikzpicture}}  				
  				 \hspace*{2cm} 
  				\raisebox{-8mm}{
  					\begin{tikzpicture}
  					\draw[blue,in=-90,out=-90,looseness=2] (-0.5,0.5) to (-1.5,0.5);
  					\draw[blue] (-1,-.1) to (-1,-.6);
  					\draw[blue] (-.5,.6) to (-.5,1.2);
  					\draw[blue] (-1.5,.6) to (-1.5,1.2);
  					\node[draw,thick,rounded corners, fill=white,minimum width=15] at (-.5,.6){$f$};
  					\node[draw,thick,rounded corners, fill=white,minimum width=15] at (-1.5,.6){$f$};
  					\node[left,scale=0.7] at (-1,-.6) {$X'$};
  					\node[left,scale=0.7] at (-1.5,-.1) {$X'$};
  					\node[right,scale=0.7] at (-.6,0) {$X'$};
  					\node[left,scale=0.7] at (-1.5,1.2) {$X$};
  					\node[right,scale=0.7] at (-.5,1.2) {$X$};
  					\end{tikzpicture}} 
  				= \hspace*{-2.1cm}
  				\raisebox{-14mm}{
  					\begin{tikzpicture}
  					\draw[blue,in=-90,out=-90,looseness=2] (-0.5,0.5) to (-1.5,0.5);
  					\draw[blue] (-1,-.1) to (-1,-1);
  					\node[draw,thick,rounded corners, fill=white,minimum width=15] at (-1,-.5){$f$};
  					\node[left,scale=0.7] at (-1,-1.2) {$X'$};
  					\node[left,scale=0.7] at (-1.6,0.5) {$X$};
  					\node[right,scale=0.7] at (-.5,.5) {$X$};
  			      	\node at (-2.2,-1.6) {\text{(Algebra morphism)}};	
  					\end{tikzpicture}}
  				 \hspace*{0mm}
  			 		\]

				\[
				\raisebox{-6mm}{\begin{tikzpicture}
					\draw[blue] (0,0) to (0,1);
					\node at (0,0) {$\textcolor{blue}{\bullet}$};
					\node[left,scale=0.7] at (0,.4) {$X$};
					\end{tikzpicture}}
				= \hspace*{-1.8cm}
				\raisebox{-13mm}{\begin{tikzpicture}
					\draw[blue] (0,0) to (0,1.5);
					\node[draw,thick,rounded corners, fill=white,minimum width=15] at (0,.6){$f$};
					\node at (0,0) {$\textcolor{blue}{\bullet}$};
					\node[left,scale=.7] at (0,1.6) {$X$};
				\node at (-.5,-.5) {\text{(Counital morphism)}};	
					\end{tikzpicture}}				
		  		\hspace*{2.3cm} 
				\raisebox{-8mm}{
					\begin{tikzpicture}[rotate=180,transform shape]
					\draw[blue,in=-90,out=-90,looseness=2] (-0.5,0.5) to (-1.5,0.5);
					\draw[blue] (-1,-.1) to (-1,-.6);
					\draw[blue] (-.5,.6) to (-.5,1.2);
					\draw[blue] (-1.5,.6) to (-1.5,1.2);
					\node[rotate=180,draw,thick,rounded corners, fill=white,minimum width=15] at (-.5,.6){$f$};
					\node[rotate=180,draw,thick,rounded corners, fill=white,minimum width=15] at (-1.5,.6){$f$};
					\node[rotate=180,right,scale=0.7] at (-.5,-.5) {$X$};
					\node[rotate=180,left,scale=0.7] at (-1.5,1.2) {$X'$};
					\node[rotate=180,right,scale=0.7] at (-.5,1.2) {$X'$};
					\end{tikzpicture}} 
				= \hspace*{-2.1cm}
				\raisebox{-15mm}{
					\begin{tikzpicture}[rotate=180,transform shape]
					\draw[blue,in=-90,out=-90,looseness=2] (-0.5,0.5) to (-1.5,0.5);
					\draw[blue] (-1,-.1) to (-1,-1);
					\node[rotate=180,draw,thick,rounded corners, fill=white,minimum width=15] at (-1,-.5){$f$};
					\node[rotate=180,left,scale=0.7] at (-1,-1) {$X$};
					\node[rotate=180,left,scale=0.7] at (-1.6,0.7) {$X'$};
					\node[rotate=180,right,scale=0.7] at (-.5,.7) {$X'$};
					\node[rotate=180] at (0,1.2) {\text{(Coalgebra morphism)}};	
					\end{tikzpicture}}
				\hspace*{0mm}
				\]  			 				 		
\end{definition}

\begin{lemma} \label{lem:DualgMor}
Let $f$ be a unital algebra morphism between $(X, m, e)$ and $(X', m', e')$ in a monoidal category. If $X$ and $X'$ admit left duals, then the dual map $f^*$ is a counital coalgebra morphism.
\end{lemma}

\begin{proof}
By Lemma~\ref{lem:Dualgebra}, the left duals are counital coalgebras. The verification for $f^*$ is then analogous to that lemma:
\[
(f^* \otimes f^*) \circ m_{X'}^* 
= (m_{X'} \circ (f \otimes f))^* 
= (f \circ m_X)^* 
= m_X^* \circ f^*,
\quad\text{and}\quad
e_X^* \circ f^* = (f \circ e_X)^* = e_{X'}^*. \qedhere
\]
\end{proof}

\begin{lemma} \label{lem:TripleAlg}
Let $(X, m, e)$ be a unital algebra in a monoidal category. If $(X', m', e')$ is a triple along with a monomorphism $f: X' \to X$ that satisfies the conditions of a unital algebra morphism from Definition \ref{def:FrobMor}, then $(X', m', e')$ is also a unital algebra.
\end{lemma}
\begin{proof}
To prove unitality, we begin with the following illustration:
				\[
				\raisebox{-7mm}
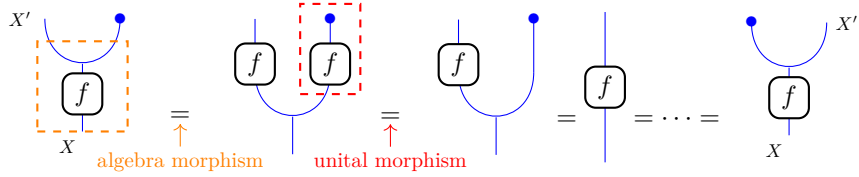
{
					\begin{tikzpicture}
					\draw[blue,in=-90,out=-90,looseness=2] (-0.5,0.5) to (-1.5,0.5);
					\draw[blue] (-1,-.1) to (-1,-1);
					\node[draw,thick,rounded corners, fill=white,minimum width=15] at (-1,-.5){$f$};
					\draw[dashed,orange,thick] (-1.6,.2) rectangle (-.4,-1);
					\node[left,scale=0.7] at (-1,-1.2) {$X$};
					\node[left,scale=0.7] at (-1.6,0.5) {$X'$};
					\node at (-0.5,0.5) {$\textcolor{blue}{\bullet}$};
					\draw[->,orange] (.3,-1.2) to (.3,-.9);
					\node[scale = .8] at (.3,-1.4) [orange] {algebra morphism};
					\node at (.3,-.8) {$=$};
					\end{tikzpicture}}
					\hspace*{-.6cm}
					\raisebox{-7mm}{
					\begin{tikzpicture}
					\draw[blue,in=-90,out=-90,looseness=2] (-0.5,0.5) to (-1.5,0.5);
					\draw[blue] (-1,-.1) to (-1,-.6);
					\draw[blue] (-.5,.6) to (-.5,1.2);
					\draw[blue] (-1.5,.6) to (-1.5,1.2);
					\draw[dashed,red,thick] (-.9,1.4) rectangle (-.1,.2);
					\draw[->,red] (.3,-.5) to (.3,-.2);
					\node[scale = .8] at (.3,-.7) [red] {unital morphism};
					\node at (.3,-.1) {$=$};
					\node at (-0.5,1.2) {$\textcolor{blue}{\bullet}$};
					\node[draw,thick,rounded corners, fill=white,minimum width=15] at (-.5,.6){$f$};
					\node[draw,thick,rounded corners, fill=white,minimum width=15] at (-1.5,.6){$f$};
					\end{tikzpicture}}
				    \hspace*{-.6cm}
					\raisebox{-4mm}{
						\begin{tikzpicture}
						\draw[blue,in=-90,out=-90,looseness=2] (-0.5,0.5) to (-1.5,0.5);
						\draw[blue] (-1,-.1) to (-1,-.6);
						\draw[blue] (-.5,.5) to (-.5,1.2);
						\draw[blue] (-1.5,.6) to (-1.5,1.2);
						\node at (-0.5,1.2) {$\textcolor{blue}{\bullet}$};
						\node[draw,thick,rounded corners, fill=white,minimum width=15] at (-1.5,.6){$f$};
						\end{tikzpicture}}
					 = 
					 \raisebox{-5mm}{\begin{tikzpicture}
					 \draw[blue] (0,-1) to (0,1);
					 \node[draw,thick,rounded corners, fill=white,minimum width=15] at (0,0){$f$};
					 \end{tikzpicture}}
				 	=
				 	\cdots
				 	=
				\raisebox{-5mm}{
					\begin{tikzpicture}
					\draw[blue,in=-90,out=-90,looseness=2] (-0.5,0.5) to (-1.5,0.5);
					\draw[blue] (-1,-.1) to (-1,-1);
					\node[draw,thick,rounded corners, fill=white,minimum width=15] at (-1,-.5){$f$};
					\node[left,scale=0.7] at (-1,-1.2) {$X$};
					\node[left,scale=0.7] at (0,0.5) {$X'$};
					\node at (-1.5,0.5) {$\textcolor{blue}{\bullet}$};
					\end{tikzpicture}}
				 		\]
Since the morphism $f$ is left-cancellative (by the definition of a monomorphism), the result follows. The proof for associativity is analogous.
\end{proof}

\begin{lemma} \label{lem:TripleCoalg}
Let $(X, \delta, \epsilon)$ be a counital coalgebra in a monoidal category. If $(X', \delta', \epsilon')$ is a triple along with an epimorphism $f: X \to X'$ that satisfies the conditions of a counital coalgebra morphism from Definition \ref{def:FrobMor}, then $(X', \delta', \epsilon')$ is also a counital coalgebra.
\end{lemma}

\begin{proof}
The proof is similar to that of Lemma \ref{lem:TripleAlg}, as an epimorphism is defined to be right-cancellative.
\end{proof}

The following is a well-known result:

\begin{lemma} \label{lem:FrobMorIso}
A Frobenius algebra morphism is an isomorphism.
\end{lemma}

\begin{proof}
Let $f$ be a morphism as defined in Definition \ref{def:FrobMor}. The following illustration demonstrates that $f^* \circ f = \id_X$:
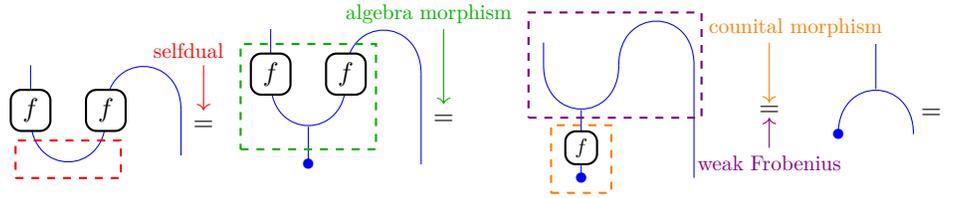
			 		\[\raisebox{-8mm}{
			 			\begin{tikzpicture}
			 			\draw[blue,in=-90,out=-90,looseness=2] (-0.5,0.5) to (-1.5,0.5);
			 			\draw[blue,in=90,out=90,looseness=2] (-0.5,0.6) to (.5,0.6);
			 			\draw[blue] (-1.5,.5) to (-1.5,1.2);
			 			\draw[blue] (.5,0.6) to (.5,0); 
			 			\draw[dashed,red,thick] (-1.7,.2) rectangle (-.3,-.3);
			 			\draw[->,red] (.8,1.2) to (.8,.6);
			 			\node[draw,thick,rounded corners, fill=white,minimum width=15] at (-.5,.6){$f$};
			 			\node[draw,thick,rounded corners, fill=white,minimum width=15] at (-1.5,.6){$f$};
			 			\node[scale=.8] at (.6,1.4) [red] {selfdual};
			 			\node at (.8,.4) {$=$};
			 			\end{tikzpicture}}
		 				\raisebox{-8mm}{
		 				\begin{tikzpicture}
		 				\draw[blue,in=-90,out=-90,looseness=2] (-0.5,0.5) to (-1.5,0.5);
		 				\draw[blue,in=90,out=90,looseness=2] (-0.5,0.6) to (.5,0.6);
		 				\draw[blue] (-1,-.1) to (-1,-.6);
		 				\draw[blue] (-1.5,.5) to (-1.5,1.2);
		 				\draw[blue] (.5,0.6) to (.5,-.6); 
		 				\draw[->,green!70!black] (.8,1.2) to (.8,.2);
		 				\draw[dashed,green!70!black,thick] (-1.9,1) rectangle (-.1,-.4);
		 				\node at (-1,-.6) {$\textcolor{blue}{\bullet}$};
		 				\node[draw,thick,rounded corners, fill=white,minimum width=15] at (-.5,.6){$f$};
		 				\node[draw,thick,rounded corners, fill=white,minimum width=15] at (-1.5,.6){$f$};
		 				\node[scale=.8] at (.6,1.4) [green!70!black] {algebra morphism};
		 				\node at (.8,0) {$=$};		 				\end{tikzpicture}}
	 					\raisebox{-10mm}{
	 						\begin{tikzpicture}
	 						\draw[blue,in=-90,out=-90,looseness=2] (-0.5,0.5) to (-1.5,0.5);
	 						\draw[blue,in=90,out=90,looseness=2] (-0.5,0.5) to (.5,0.5);
	 						\draw[blue] (-1,-.1) to (-1,-1);
	 						\draw[blue] (.5,0.5) to (.5,-1); 
	 						\draw[blue] (-1.5,.5) to (-1.5,.8);
	 						\node[draw,thick,rounded corners, fill=white,minimum width=15,scale=.8] at (-1,-.6){$f$};
	 						\draw[dashed,orange,thick] (-1.4,-.3) rectangle (-.6,-1.2);
	 						\draw[dashed,violet,thick] (-1.7,1.2) rectangle (.6,-.2);
	 						\draw[->,orange] (1.5,.8) to (1.5,-.0);
	 						\draw[->,violet] (1.5,-.6) to (1.5,-.2);
	 						\node at (-1,-1) {$\textcolor{blue}{\bullet}$};
	 						\node[scale = .8] at (1.85,1) [orange] {counital morphism};
	 						\node[scale=.8] at (1.5,-.8) [violet] {weak Frobenius};
	 						\node at (1.5,-.1) {$=$};
	 						\end{tikzpicture}}
	 						\hspace*{-1cm}
 							\raisebox{-4mm}{
 								\begin{tikzpicture}
 								\draw[blue,in=90,out=90,looseness=2] (0,0) to (1,0);
 								\node at (0,0) {$\textcolor{blue}{\bullet}$};
 								\draw[blue] (.5,.6) to (.5,1.2);
 								\end{tikzpicture}}
 							=
 							\raisebox{-2mm}{ 
 							\begin{tikzpicture}
 							\draw[blue] (0,0) to (0,1);
 							\end{tikzpicture}} \]
Similarly, we can show that $f \circ f^* = \id_{X'}$. Thus, the result follows.
\end{proof}

Lemma \ref{lem:FrobMorIso} indicates that the concept of a Frobenius algebra morphism is too restrictive for defining the notion of a Frobenius subalgebra.

%
\begin{definition} \label{def:sub}
A Frobenius algebra \((X', m', \delta', e', \epsilon')\) in a monoidal category \(\mathcal{C}\) is called a \emph{Frobenius subalgebra} of \((X, m, \delta, e, \epsilon)\) if there exists a unital algebra monomorphism 
\(i: X' \to X\) such that \(i^{**} = i\) and \(i^* \circ i = \mathrm{id}_{X'}\) (equivalently, \(i\) is counital).  
We then define the selfdual unital idempotent $b_{X'} := i \circ i^* \in \End_{\mathcal C}(X)$.
\end{definition}

Let us explain why the equivalence in Definition~\ref{def:sub} holds: 
\begin{itemize}
\item If \(i\) is counital, then by definition of \(i^*\), Remark \ref{rk:FrobSubAmbiently} and zigzag relation:
\[i^* \circ i =  \raisebox{-6.25mm}{
			 			\begin{tikzpicture}
			 			\draw[blue,in=-90,out=-90,looseness=2] (-0.5,0.5) to (-1.5,0.5);
			 			\draw[blue,in=90,out=90,looseness=2] (-0.5,0.6) to (.5,0.6);
			 			\draw[blue] (-1.5,.5) to (-1.5,1.2);
			 			\draw[blue] (.5,0.6) to (.5,0); 
			 			\node[draw,thick,rounded corners, fill=white,minimum width=15] at (-.5,.6){$i$};
			 			\node[draw,thick,rounded corners, fill=white,minimum width=15] at (-1.5,.6){$i$};
			 			\node at (.8,.4) {$=$};
			 			\end{tikzpicture}
			 			\begin{tikzpicture}
			 			\draw[blue,in=-90,out=-90,looseness=2] (-0.5,0.5) to (-1.5,0.5);
			 			\draw[blue,in=90,out=90,looseness=2] (-0.5,0.5) to (.5,0.5);
			 			\draw[blue] (-1.5,.5) to (-1.5,1);
			 			\draw[blue] (.5,0.5) to (.5,0); 
			 			\node at (.8,.4) {$=$};
			 			\end{tikzpicture}}
\id_{X'}.			 			
\]

\item Conversely, since \(i\) is unital, we have $i \circ e' = e$, and applying $i^*$ to both sides yields $i^* \circ i \circ e' = i^* \circ e$, so $e' = i^* \circ e$, since $i^* \circ i = \mathrm{id}_{X'}$. Using Lemma \ref{lem:unitdualcounit} and the fact that $i^{**} = i$, we conclude that $\epsilon \circ i = \epsilon'$. 
\end{itemize}

\begin{remark} \label{rk:FrobSubAmbiently}
Following Definition \ref{def:sub}, since $i$ is a counital algebra morphism, Lemma \ref{lem:selfdual} implies that:
\[
\ev_X \circ (i \otimes i) = \epsilon \circ m \circ (i \otimes i) = \epsilon \circ i \circ m' = \epsilon' \circ m' = \ev_{X'}.
\]
A similar argument establishes an analogous identity for the coevaluation morphisms. Consequently, we have:
\[
\ev_X \circ (i \otimes i) = \ev_{X'} \quad \text{and} \quad (i^* \otimes i^*) \circ \coev_X = \coev_{X'}.
\]
\end{remark}

Let us explain why the morphism $b_{X'}$ in Definition \ref{def:sub} is a unital selfdual idempotent. Since $i$ is both unital and counital, so is its dual $i^*$; consequently, their composition $b_{X'} = i \circ i^*$ is likewise unital and counital. The remaining properties are established as follows:
\[
b_{X'}^* = (i \circ i^*)^* = i^{**} \circ i^* = i \circ i^* = b_{X'} \quad \text{and} \quad b_{X'} \circ b_{X'} = (i \circ i^*) \circ (i \circ i^*) = i \circ (i^* \circ i) \circ i^* = i \circ \id_{X'} \circ i^* = b_{X'}.
\]
It generalizes the concept of a biprojection in subfactor planar algebras (see \cite[\S 4]{Liu16}, Proposition \ref{prop:biprojection}, and Statement \ref{sta:liu}). In Definition \ref{def:sub}, since $ i^* \circ i = \id_{X'} $, it follows that $ i: X' \to X $ is a split unital algebra monomorphism, and $ i^*: X \to X' $ is a split counital coalgebra epimorphism. Recall from Lemma \ref{lem:selfdual} that both $ X $ and $ X' $ are selfdual. Note that $ X $ is trivially a Frobenius subalgebra of itself with $ i = i^* = \id_X $, hence $ b_X = \id_X $.

\begin{lemma} \label{lem:b1}
Let $(X, m, \delta, e, \epsilon)$ be a Frobenius algebra in a $\mathbbm{k}$-linear monoidal category $\mathcal{C}$, and suppose the unit object $\one$ is linear-simple. If the scalar $\mu_X := \epsilon \circ e$ is nonzero, then $\one$ forms a Frobenius subalgebra of $X$, and
$
b_{\one} = \mu_X^{-1} \, e \circ \epsilon.
$
\end{lemma}
\begin{proof}
Assume $\mu_X \neq 0$. Then the equation $\mu_X x^2 = 1$ has a solution in $\mathbbm{k}$; let $c_X \in \mathbbm{k}$ be such a solution. Define $i := c_X e$. By Lemma \ref{lem:unitdualcounit}, we have $i^* = c_X \epsilon$, so
$
i^* \circ i = c_X \epsilon \circ c_X e = c_X^2 \mu_X = 1,
$
as required. It follows that
$
b_{\one} = i \circ i^* = \mu_X^{-1} e \circ \epsilon.
$
\end{proof}
\begin{lemma} \label{lem:b1nonzero}
Let $(X, m, \delta, e, \epsilon)$ be a Frobenius algebra that is also a semisimple object in a $\mathbbm{k}$-linear abelian monoidal category $\mathcal{C}$. Suppose that the unit object $\mathbf{1}$ is simple.  
Then the scalar $\mu_X := \epsilon \circ e$ is nonzero. Consequently, $\mathbf{1}$ forms a Frobenius subalgebra of $X$, with  
\(
b_{\mathbf{1}} = \mu_X^{-1} \, e \circ \epsilon.
\)
\end{lemma}
\begin{proof}
Let $A$ and $B$ be selfdual objects in $\mathcal{C}$, and consider the split monomorphism $i = \id_A \oplus 0 \in \Hom_{\mathcal{C}}(A, A \oplus B)$, and the split epimorphism $p = \id_A \oplus 0 \in \Hom_{\mathcal{C}}(A \oplus B, A)$. They satisfy $p \circ i = \id_A$.

Since $\one$ is simple and $X$ is unital and semisimple, it follows that $X = \one \oplus Y$ for some object $Y$, such that the morphism $i \in \Hom_{\mathcal{C}}(\one, X)$ defined above, with $A = \one$ and $B = Y$, is a scalar multiple of the unit morphism $e$; that is, there exists a nonzero scalar $k \in \mathbbm{k}$ such that $e = k i$. 
By the natural adjunction isomorphism \cite[Proposition 2.10.8]{EGNO15}, the space $\Hom_{\mathcal{C}}(X^*, \one)$ is also one-dimensional. Since $X^* = X$, it follows that $\epsilon = k' p$ for some nonzero scalar $k' \in \mathbbm{k}$. Thus,
$$
\mu_X = \epsilon \circ e = (k' p) \circ (k i) = k k' (p \circ i) = k k' \id_{\one} \neq 0.
$$
The last assertion is an immediate consequence of Lemma \ref{lem:b1}.
\end{proof}

\subsection{Frobenius subalgebra poset} \label{sub:poset}

This subsection introduces an equivalence class of Frobenius subalgebras (Definition \ref{def:equi}) to capture the concept of a substructure, leading naturally to the notion of the Frobenius subalgebra poset (Definition \ref{def:FrobSubPoset}).
Let $X$ be an object in a category. A subobject of $X$ is a pair $(Y,i)$, where $i : Y \to X$ is a monomorphism. Note that, at the beginning of \S \ref{sec:lattice}, following \cite[\S 1.5]{Fr64}, subobjects are defined as equivalence classes of such pairs; here, however, we adopt a different convention. A split subobject of $X$ is a triple $(Y,i,p)$, where $i : Y \to X$ and $p : X \to Y$ satisfy $p \circ i = \id_Y$. It follows that $i$ is a monomorphism (i.e., left-cancellative), that $p$ is an epimorphism (i.e., right-cancellative), and that $b = i \circ p$ is an idempotent. 
\begin{remark} \label{rk:Karou}
A category $\mathcal{C}$ is called \emph{Karoubian} if it is idempotent-complete, meaning that for every idempotent $b$ in $\End_{\mathcal{C}}(X)$ there is a split subobject $(Y,i,p)$ such that $b=i \circ p$. Notably, every abelian category is Karoubian. 
\end{remark}
We define an equivalence relation on split subobjects, which completely characterizes the associated idempotent.

\begin{definition} \label{def:equisplit}
Two split subobjects $(Y_1,i_1,p_1)$ and $(Y_2,i_2,p_2)$ of $X$ are said to be \emph{equivalent} if there exists an isomorphism $i_{1,2} : Y_1 \to Y_2$ with inverse $i_{2,1}$ such that $i_1 = i_2 \circ i_{1,2}$ and $p_1 = i_{2,1} \circ p_2$. It follows that $i_1 \circ i_{2,1} = i_2 $ and $i_{1,2} \circ p_1 = p_2$, but also $i_{1,2} = p_2 \circ i_1$ and  $i_{2,1} = p_1 \circ i_2$. The equivalence class of $Y_1$ is denoted by $[Y_1]$.
\end{definition}


\begin{proposition} \label{prop:idempotentsplit}
Using the notation from Definition \ref{def:equisplit}, let $b_k$ be the idempotent $i_k \circ p_k$. Then
$$
[Y_1] = [Y_2] \Leftrightarrow b_1 = b_2.
$$
\end{proposition}

\begin{proof}
If $[Y_1] = [Y_2]$, then
$$
b_1 = i_1 \circ p_1 = (i_2 \circ i_{1,2}) \circ (i_{2,1} \circ p_2)
= i_2 \circ (i_{1,2} \circ i_{2,1}) \circ p_2
= i_2 \circ \id_{X_2} \circ p_2
= i_2 \circ p_2
= b_2.
$$
Conversely, suppose that $b_1=b_2$, and define $i_{1,2} = p_2 \circ i_1$ and $i_{2,1} = p_1 \circ i_2$. Then
$$
i_2 \circ i_{1,2}
= i_2 \circ p_2 \circ i_1
= b_2 \circ i_1
= b_1 \circ i_1
= i_1 \circ p_1 \circ i_1
= i_1,
$$
%
%
$$
i_{1,2} \circ i_{2,1}
= p_2 \circ i_1 \circ p_1 \circ i_2
= p_2 \circ b_1 \circ i_2
= p_2 \circ b_2 \circ i_2
= \id_{Y_2}.
$$
Similarly, we show that $i_{2,1} \circ p_2 = p_1$ and $i_{2,1} \circ i_{1,2} = \id_{Y_1}$.
\end{proof}

Typically, idempotents in $\End_{\mathcal{C}}(X)$ carry a natural partial order, defined by $b_1 \le b_2$ when $b_1 \circ b_2 = b_1 = b_2 \circ b_1$.

\begin{lemma} \label{lem:idemposet}
Using the notations from Proposition \ref{prop:idempotentsplit}, $b_1 \le b_2$ if and only if there exist morphisms $i_{1,2}$ and $p_{2,1}$ such that $i_1 = i_2 \circ i_{1,2}$ and $p_1 = p_{2,1} \circ p_2$. In this case, $i_{1,2} = p_2 \circ i_1$ and $p_{2,1} = p_1 \circ i_2$, so $p_{2,1} \circ i_{1,2} = \id_{Y_1}$, meaning that $(Y_1,i_{1,2}, p_{2,1})$ is a split subobject of $Y_2$. Consequently, in an abelian category, $Y_2 \cong Y_1 \oplus Z$ for some object $Z$.
\end{lemma}

\begin{proof}
By the cancellativity of the monomorphism $i_1$ and epimorphism $p_1$, we obtain the following equivalences:
\begin{align*}
b_1 \le b_2 &\iff b_1 \circ b_2 = b_1 \text{ and } b_2 \circ b_1 = b_1 \\
&\iff i_1 \circ p_1 \circ i_2 \circ p_2 = i_1 \circ p_1 \text{ and } i_2 \circ p_2 \circ i_1 \circ p_1 = i_1 \circ p_1 \\
&\iff p_1 \circ i_2 \circ p_2 = p_1 \text{ and } i_2 \circ p_2 \circ i_1 = i_1 \\
&\iff \exists i_{1,2}, p_{2,1} \text{ such that } i_1 = i_2 \circ i_{1,2} \text{ and } p_1 = p_{2,1} \circ p_2.
\end{align*}
The final equivalence is uniquely witnessed by assigning $i_{1,2} = p_2 \circ i_1$ and $p_{2,1} = p_1 \circ i_2$, since $p_2 \circ i_2 = \id_{Y_2}$. Thus,
$$
p_{2,1} \circ i_{1,2} = p_1 \circ i_2 \circ p_2 \circ i_1 = p_1 \circ b_2 \circ i_1 = p_1 \circ i_1 = \id_{Y_1}.
$$
The final direct sum decomposition follows directly from the splitting lemma.
\end{proof}

Let $X$ be a Frobenius algebra object in a monoidal category. 

\begin{definition} \label{def:equi}
Using the notation of Definition \ref{def:sub}, two Frobenius subalgebras $X_1$ and $X_2$ of $X$ are said to be \emph{equivalent} if they are equivalent as split subobjects $(X_1, i_1, i_1^*)$ and $(X_2, i_2, i_2^*)$ in the sense of Definition \ref{def:equisplit}. It follows that $i_{2,1} = p_1 \circ i_2 = i_1^* \circ i_2 = (i_2^* \circ i_1)^* = (p_2 \circ i_1)^* = i_{1,2}^*$. The equivalence class of $X_1$ is denoted by $[X_1]$.
\end{definition}

As with Proposition \ref{prop:idempotentsplit}, the equivalence classes correspond to idempotents, viewed as generalized biprojections:

\begin{proposition} \label{prop:biprojection}
Using the notation from Definitions \ref{def:sub} and \ref{def:equi}, the following holds,
$$[X_1] = [X_2] \Leftrightarrow b_{X_1} = b_{X_2}.$$
\end{proposition} 
\begin{proof}
Immediate from Proposition \ref{prop:idempotentsplit}.
\end{proof}

\begin{definition} \label{def:FrobSubPoset}
An equivalence class $[X_1]$ is said to be \emph{contained} in $[X_2]$ if $b_{X_1} \le b_{X_2}$ (equivalently, if there exists a morphism $i_{1,2}$ such that $i_1 = i_2 \circ i_{1,2}$, by Lemma \ref{lem:idemposet} with $p_{2,1} = i_{1,2}^*$). It defines a partial order on the equivalence classes. The resulting poset is called the \emph{Frobenius subalgebra poset} of the Frobenius algebra $X$.
\end{definition}



\subsection{Ambient dual} \label{sub:Ambient}
This subsection introduces the notion of an ambient dual for a subobject (Definition~\ref{def:InducedDual}), leading to the notion of ambient selfduality (Definition~\ref{def:AmbientlySelfdual}). This extends the classical notion of a nondegenerate subspace (see \S\ref{sub:RigidVec}), a property that notably fails for the intersection in Example~\ref{ex:Ben02}. 
In a Karoubian (Remark \ref{rk:Karou}) monoidal category, Proposition \ref{prop:AmbientFrobSub} shows that the Frobenius subalgebras are precisely the ambiently selfdual unital subalgebras. For an object in a monoidal category, a left dual (when it exists) is defined only up to isomorphism. For a subobject, however, it is sometimes necessary to consider a specific left dual induced by the ambient object.

\begin{definition} \label{def:InducedDual}
Let $C$ be an object of a Karoubian monoidal category $\mathcal{C}$ admitting a left dual $(C^*, \ev_C, \coev_C)$. Let $(B,i,p)$ be a split subobject of $C$, where $i \in \Hom_{\mathcal{C}}(B,C)$, $p \in \Hom_{\mathcal{C}}(C,B),$ satisfy $p \circ i = \id_B$. Denote by $b := i \circ p \in \End_{\mathcal{C}}(C)$ the associated idempotent.
Since $\mathcal{C}$ is Karoubian, the dual idempotent $b^* \in \End_{\mathcal{C}}(C^*)$ likewise splits. That is, there exists a triple $(B', i', p')$, with $i' \colon B' \to C^*$ and $p' \colon C^* \to B'$, such that $i' \circ p' = b^*$, $p' \circ i' = \id_{B'}$. 
An \emph{ambient dual} of $B$ (relative to $C$) is the triple $(B^*,\ev_B,\coev_B):=\bigl(B',\, \ev_C \circ (i' \otimes i),\, (p \otimes p') \circ \coev_C \bigr).$ The zigzag identities for these morphisms follow automatically. Moreover, we immediately verify that $i^* = p'$ and $p^* = i'$. \end{definition}

\begin{lemma} \label{lem:InducedRefor}
Let $B$ and $C$ be objects in a Karoubian monoidal category $\mathcal{C}$, each admitting a left dual $(B^*, \ev_B, \coev_B)$ and $(C^*, \ev_C, \coev_C)$, respectively. Then the dual of $B$ is ambient relative to $C$ if and only if there exist morphisms $i \in \Hom_{\mathcal{C}}(B, C)$ and $p \in \Hom_{\mathcal{C}}(C, B)$ such that
\(
\ev_B = \ev_C \circ (p^* \otimes i).
\)
\end{lemma}

\begin{proof}
The forward direction is immediate from Definition \ref{def:InducedDual}. Conversely, suppose that $\ev_B = \ev_C \circ (p^* \otimes i)$. Then, by the zigzag relations, we have $i^* \circ p^* = \id_{B^*}$, which implies $p \circ i = \id_B$. Hence, $(B, i, p)$ forms a split subobject of $C$.  
Next, for the coevaluation, we compute:
\[
\coev_B = (\id_B \otimes \id_{B^*}) \circ \coev_B = (\id_B \otimes (i^* \circ p^*)) \circ \coev_B = (p \otimes i^*) \circ \coev_C,
\] 
where the last equality again follows from the zigzag relation. The lemma follows.
\end{proof}

\begin{remark} \label{rk:InducedDual}
An object \(X\) is called \emph{selfdual} if the equality \(X^* = X\) holds strictly, not merely up to isomorphism.
\end{remark}

\begin{definition} \label{def:AmbientlySelfdual}
Following Definition \ref{def:InducedDual}, let $C$ be a selfdual object in a Karoubian monoidal category. A split subobject $(B, i, p)$ of $C$ is said to be \emph{ambiently selfdual} if it coincides with its ambient dual. In this case, $i^* = p$ and $i^{**} = p^* = i$, which implies that $i \circ i^* = b$ and $i^* \circ i = \id_B$.
\end{definition}

\begin{lemma} \label{lem:AmbientRefor}
Let $C$ be a selfdual object in a Karoubian monoidal category $\mathcal{C}$. A selfdual subobject $(B,i)$ is ambiently selfdual (with $p=i^*$) if and only if \(
\ev_B = \ev_C \circ (i \otimes i).
\)
\end{lemma}
\begin{proof}
Immediate from Lemma \ref{lem:InducedRefor} and Definition \ref{def:AmbientlySelfdual}.
\end{proof}

\begin{lemma} \label{lem:AmbientlyInter}
Following Definition \ref{def:AmbientlySelfdual}, let $C$ be a selfdual object in a Karoubian monoidal category, and let $(B, i, i^*)$ be an ambiently selfdual subobject of $C$. Let $(A, j)$ be a subobject of $B$. Then $(A, j, j^*)$ is an ambiently selfdual subobject of $B$ if and only if $(A, k, k^*)$ is an ambiently selfdual subobject of $C$, where $k = i \circ j$.  
\end{lemma}
\begin{proof}
The result follows from the equalities $k^* \circ k = (i \circ j)^* \circ (i \circ j) = j^* \circ (i^* \circ i) \circ j = j^* \circ j$ and
\[
(A, \ev_B \circ (j \otimes j), (j^* \otimes j^*) \circ \coev_B) = (A, \ev_C \circ (k \otimes k), (k^* \otimes k^*) \circ \coev_C). \qedhere
\]
\end{proof}

The following is an immediate consequence of Remark \ref{rk:FrobSubAmbiently}:
\begin{lemma} \label{lem:FrobSubAmbiently}
In a Karoubian monoidal category, a Frobenius subalgebra is ambiently selfdual.
\end{lemma}

\begin{proposition}\label{prop:AmbientFrobSub}
Let $C$ be a Frobenius algebra in a Karoubian monoidal category. Let $(B, m_B, e_B)$ be a unital subalgebra of $C$, with unital algebra monomorphism $i_B$, and with a left dual. Then $(B, i_B, i_B^*)$ is an ambiently selfdual subobject of $C$ if and only if $(B, m_B, e_B, m_B^*, e_B^*)$ is a Frobenius subalgebra of $C$.
\end{proposition}
\begin{proof}
One implication is exactly Lemma~\ref{lem:FrobSubAmbiently}. Conversely, if $(B, i_B, i_B^*)$ is an ambiently selfdual subobject of $C$, we have $i_B^*\circ i_B =\id_{B}$ and $i_B^{**} = i_B$, by Definition~\ref{def:AmbientlySelfdual}. Since $(B, m_B, e_B)$ is a unital algebra and $i_B$ is a unital algebra monomorphism, then $(B, m_B^*, e_B^*)$ is a counital coalgebra and $i_B^*$ is a counital coalgebra epimorphism, by Lemmas \ref{lem:Dualgebra} and \ref{lem:DualgMor}. We deduce the Frobenius structure pictorially: 
\[\raisebox{-5mm}{
\begin{tikzpicture}
\draw[blue,in=90,out=90,looseness=2] (0,0) to (1,0);
\draw[blue,in=-90,out=-90,looseness=2] (1,0) to (2,0);
\draw[blue] (.5,.6) to (.5,1.2);
\draw[blue] (1.5,-.6) to (1.5,-1.2);
\draw[blue] (0,0) to (0,-1.2);
\draw[blue] (2,0) to (2,1.2);
\node at (2.2,0) {$=$};
\end{tikzpicture}} 
\hspace*{-.2cm}
\raisebox{-16mm}{
\begin{tikzpicture}
\draw[blue,in=90,out=90,looseness=2] (0,0) to (1,0);
\draw[blue,in=-90,out=-90,looseness=2] (1,0) to (2,0);
\draw[blue] (.5,.6) to (.5,2);
\draw[blue] (1.5,-.6) to (1.5,-2);
\draw[blue] (0,0) to (0,-1.2);
\draw[blue] (2,0) to (2,1.2);
\draw[dashed,orange,thick] (-.1,1.3) rectangle (1,0);
\draw[dashed,violet,thick] (.9,-.1) rectangle (2,-1.3);
\draw[->,orange] (2.2,1.2) to (2.2,.2);
\draw[->,violet] (2.2,-1.2) to (2.2,-.2);
\node at (2.2,1.8) [orange, align=center] {\small coalgebra\\\small morphism};
\node at (2.4,-1.8) [violet, align=center] {\small algebra\\\small  morphism};
\node[draw,thick,rounded corners, fill=white,minimum width=15,scale=.7] at (.5,1) {$i_B^*$};
\node[draw,thick,rounded corners, fill=white,minimum width=15,scale=.7] at (.5,1.6) {$i_B$};
\node[draw,thick,rounded corners, fill=white,minimum width=15,scale=.7] at (1.5,-1) {$i_B$};
\node[draw,thick,rounded corners, fill=white,minimum width=15,scale=.7] at (1.5,-1.6) {$i_B^*$};
\node at (2.2,0) {$=$};
\end{tikzpicture}}
\hspace*{-1cm}
\raisebox{-11mm}{
\begin{tikzpicture}
\draw[blue,in=90,out=90,looseness=2] (0,0) to (1,0);
\draw[blue,in=-90,out=-90,looseness=2] (1,0) to (2,0);
\draw[blue] (.5,.6) to (.5,1.7);
\draw[blue] (1.5,-.6) to (1.5,-1.7);
\draw[blue] (0,0) to (0,-1.2);
\draw[blue] (2,0) to (2,1.2);
\draw[dashed,red,thick] (-.1,.8) rectangle (1.3,-.05);
\draw[->,red] (2.3,1.2) to (2.3,.2);
\node at (2.3,1.6) [red] {\small Lemma \ref{lem:intermultdualcomultalpha}};
\node[draw,thick,rounded corners, fill=white,minimum width=15,scale=.7] at (1,.2) {$i_B^*$};
\node[draw,thick,rounded corners, fill=white,minimum width=15,scale=.7] at (0,-.6) {$i_B^*$};
\node[draw,thick,rounded corners, fill=white,minimum width=15,scale=.7] at (.5,1.2) {$i_B$};
\node[draw,thick,rounded corners, fill=white,minimum width=15,scale=.7] at (1,-.3) {$i_B$};
\node[draw,thick,rounded corners, fill=white,minimum width=15,scale=.7] at (2,.6) {$i_B$};
\node[draw,thick,rounded corners, fill=white,minimum width=15,scale=.7] at (1.5,-1.2) {$i_B^*$};
\node at (2.3,0) {$=$};
\end{tikzpicture}} 
\hspace*{-.8cm}
\raisebox{-8mm}{
\begin{tikzpicture}
\draw[blue,in=-90,out=-90,looseness=2] (0,0) to (1,0);
\draw[blue,in=90,out=90,looseness=2] (1,0) to (2,0);
\draw[blue,in=-90,out=-90,looseness=2] (2,0) to (3,0);
\draw[blue] (.5,-.6) to (.5,-1.4);
\draw[blue] (0,0) to (0,1);
\draw[blue] (2.5,-.6) to (2.5,-1.4);
\draw[blue] (3,0) to (3,1);
\node[draw,thick,rounded corners, fill=white,minimum width=15,scale=.7] at (0,0) {$i_B$};
\node[draw,thick,rounded corners, fill=white,minimum width=15,scale=.7] at (1,0) {$i_B$};
\node[draw,thick,rounded corners, fill=white,minimum width=15,scale=.7] at (2,0) {$i_B$};
\node[draw,thick,rounded corners, fill=white,minimum width=15,scale=.7] at (3,0) {$i_B$};
\node[draw,thick,rounded corners, fill=white,minimum width=15,scale=.7] at (.5,-1) {$i_B^*$};
\node[draw,thick,rounded corners, fill=white,minimum width=15,scale=.7] at (2.5,-1) {$i_B^*$};
\draw[dashed, DarkGreen!70!black, thick] [rotate around={-55:(1.9,.1)}] (2,.2) ellipse (.9 and .6);
\draw[->,green!70!black] (3.5,1.2) to (3.5,.2);
\node at (3.5,1.7) [green!70!black, align=center] {\small $i_B = i_B^{**}$ and\\\small Lemma \ref{lem:intermultdualcomultalpha}};
\node at (3.5,0) {$=$}; 
\end{tikzpicture}}
\hspace*{-.85cm}
\raisebox{-8mm}{
\begin{tikzpicture}
\draw[blue,in=-90,out=-90,looseness=2] (0,0) to (1,0);
\draw[blue,in=90,out=90,looseness=2] (1,0) to (2,0);
\draw[blue] (.5,-.6) to (.5,-1.4);
\draw[blue] (1.5,.6) to (1.5,1.4);
\draw[blue] (0,0) to (0,1.2);
\draw[blue] (2,0) to (2,-1.2);
\node[draw,thick,rounded corners, fill=white,minimum width=15,scale=.7] at (0,-.1) {$i_B$};
\node[draw,thick,rounded corners, fill=white,minimum width=15,scale=.7] at (1.5,1) {$i_B$};
\node[draw,thick,rounded corners, fill=white,minimum width=15,scale=.7] at (1,-.25) {$i_B$};
\node[draw,thick,rounded corners, fill=white,minimum width=15,scale=.7] at (1,.25) {$i_B^*$};
\node[draw,thick,rounded corners, fill=white,minimum width=15,scale=.7] at (.5,-1) {$i_B^*$};
\node[draw,thick,rounded corners, fill=white,minimum width=15,scale=.7] at (2,.1) {$i_B^*$};
\node at (2.45,0) {$=$};
\end{tikzpicture}}
\hspace*{-.15cm}
\raisebox{-8mm}{
\begin{tikzpicture}
\draw[blue,in=-90,out=-90,looseness=2] (0,0) to (1,0);
\draw[blue,in=90,out=90,looseness=2] (1,0) to (2,0);
\draw[blue] (.5,-.6) to (.5,-1.4);
\draw[blue] (1.5,.6) to (1.5,1.4);
\draw[blue] (0,0) to (0,1.4);
\draw[blue] (2,0) to (2,-1.4);
\end{tikzpicture}}
\] 
Thus, the Frobenius property holds for $B$ (recall Lemma \ref{lem:weakfull}). Finally, since $i_B$ is a unital algebra monomorphism satisfying $i_B^* \circ i_B = \id_{B}$ and $i_B^{**} = i_B$, it follows that $B$ is a Frobenius subalgebra of $C$ (Definition \ref{def:sub}).
\end{proof}

The notion of a \emph{coherent pair} of subobjects, defined below, will be used in the definition of the formal angle in \S\ref{sec:Angle}. It involves the notions of intersection and sum, defined in general in \S\ref{sec:lattice} (see Remarks~\ref{rk:intersection} and~\ref{rk:union}).

\begin{definition} \label{def:CoherentPair}
A pair $(A,B)$ of subobjects of a selfdual object $C$ in a Karoubian monoidal category is called \emph{coherent} if $A$, $B$, $A \cap B$, and $A + B$ are ambiently selfdual.
\end{definition}

A broad class of coherent pairs of Frobenius subalgebras is furnished by Theorem~\ref{thm:CoherentPairB}.

\subsection{Exchange relations} \label{sub:exchange}
This subsection presents a monoidal category version of the \emph{exchange relation} from planar algebras \cite{Bi94, BJ00, Lan02, Liu16}. For further developments, see \cite{GP25}.
\begin{theorem}[Exchange Relations] \label{thm:exchange}
Let $X$ and $X'$ be two Frobenius algebras in a monoidal category $\mathcal{C}$. Let $f \in \Hom_{\mathcal{C}}(X', X)$ be an algebra morphism such that $f^{**} = f$, and define $g := f \circ f^*$. The following relation holds:
 			\[\raisebox{-6mm}{
 				\begin{tikzpicture}
 				\draw[blue,in=90,out=90,looseness=2] (-0.5,0.5) to (-1.5,0.5);
 				\draw[blue] (-1,1.1) to (-1,2);
 				\draw[blue] (-.5,.5) to (-.5,-.2);
 				\draw[blue] (-1.5,.5) to (-1.5,-.2);
 				\node[draw,thick,rounded corners, fill=white,minimum width=15] at (-1.5,.4){$g$};
 				\node[draw,thick,rounded corners, fill=white,minimum width=15] at (-1,1.6){$g$};
 				\end{tikzpicture}}
 				=
 				\raisebox{-6mm}{
 				\begin{tikzpicture}
 				\draw[blue,in=90,out=90,looseness=2] (-0.5,0.5) to (-1.5,0.5);
 				\draw[blue] (-1,1.1) to (-1,2);
 				\draw[blue] (-.5,.5) to (-.5,-.2);
 				\draw[blue] (-1.5,.5) to (-1.5,-.2);
 				\node[draw,thick,rounded corners, fill=white,minimum width=15] at (-1.5,.4){$g$};
 				\node[draw,thick,rounded corners, fill=white,minimum width=15] at (-.5,.4){$g$};
 				\end{tikzpicture}}
 				=
 				\raisebox{-6mm}{
 					\begin{tikzpicture}
 					\draw[blue,in=90,out=90,looseness=2] (-0.5,0.5) to (-1.5,0.5);
 					\draw[blue] (-1,1.1) to (-1,2);
 					\draw[blue] (-.5,.5) to (-.5,-.2);
 					\draw[blue] (-1.5,.5) to (-1.5,-.2);
 					\node[draw,thick,rounded corners, fill=white,minimum width=15] at (-.5,.4){$g$};
 					\node[draw,thick,rounded corners, fill=white,minimum width=15] at (-1,1.6){$g$};
 					\end{tikzpicture}} \]
\end{theorem}

\begin{proof}
First, we show that $f^* \in \Hom_{\mathcal{C}}(X, X')$ is a coalgebra morphism. Since $f$ is an algebra morphism, meaning $f \circ m = m \circ (f \otimes f)$, we have $(f \circ m)^* = (m \circ (f \otimes f))^*$, which can be reformulated as $\delta \circ f^* = (f^* \otimes f^*) \circ \delta$ (noting that $m^* = \delta$ by Lemma \ref{lem:multdualcomult}). Thus, $f^*$ is confirmed as a coalgebra morphism. Next, we note that $g^* = g$ because $(f \circ f^*)^* = f^{**} \circ f^*$ and $f^{**} = f$. The rest of the proof follows from the illustration:
 			\[ \raisebox{-6mm}{
 				\begin{tikzpicture}
 				\draw[blue,in=90,out=90,looseness=2] (-0.5,0.5) to (-1.5,0.5);
 				\draw[blue] (-1,1.1) to (-1,2);
 				\draw[blue] (-.5,.5) to (-.5,-.2);
 				\draw[blue] (-1.5,.5) to (-1.5,-.2);
 				\node[draw,thick,rounded corners, fill=white,minimum width=15] at (-1.5,.4){$g$};
 				\node[draw,thick,rounded corners, fill=white,minimum width=15] at (-1,1.6){$g$};
 				\node at (.2,1.9) {Lemma \ref{lem:intermultdualcomultalpha}};
 				\draw [->] (.2,1.7) to (.2,.55);
 				\node at (.2,.4) {$=$};
 				\end{tikzpicture}}
 				\hspace*{-.6cm}
 				\raisebox{-6mm}{
 					\begin{tikzpicture}
 					\draw[blue,in=90,out=90,looseness=2] (0,0) to (1,0);
 					\draw[blue,in=-90,out=-90,looseness=2] (1,0) to (2,0);
 					\draw[blue] (1.5,-.6) to (1.5,-1.2);
 					\draw[blue] (0,0) to (0,-1.2);
 					\draw[blue] (2,0) to (2,.7);
 					\node[draw,thick,rounded corners, fill=white,minimum width=15] at (1,0) {$g$};
 					\node[draw,thick,rounded corners, fill=white,minimum width=15] at (2,0) {$g$};
 					\end{tikzpicture}}
 				 =
 				  	\raisebox{-6mm}{
 				  	\begin{tikzpicture}
 				  	\draw[blue,in=90,out=90,looseness=2] (0,0) to (1,0);
 				  	\draw[blue,in=-90,out=-90,looseness=2] (1,0) to (2,0);
 				  	\draw[blue] (1.5,-.6) to (1.5,-1.2);
 				  	\draw[blue] (0,0) to (0,-1.2);
 				  	\draw[blue] (2,0) to (2,.7);
 				  	\draw[dashed,green!70!black,thick] (.65,-0.05) rectangle (2.3,-1.1);
 				  	\draw[->,green!70!black] (2.6,.8) to (2.6,-.3);
 				  	\node[draw,thick,rounded corners, fill=white,minimum width=15,scale=.7] at (1,.2) {$f^*$};
 				  	\node[draw,thick,rounded corners, fill=white,minimum width=15,scale=.7] at (1,-.3) {$f$};
 				  	\node[draw,thick,rounded corners, fill=white,minimum width=15,scale=.7] at (2,.2) {$f^*$};
 				  	\node[draw,thick,rounded corners, fill=white,minimum width=15,scale=.7] at (2,-.3) {$f$};
 				  	\node[scale=.8] at (2.8,1) [green!70!black] {algebra morphism};
 				  	\node at (2.6,-.5) {$=$};
 				  	\end{tikzpicture}}
 				  	\hspace*{-.75cm}
 			  		\raisebox{-6mm}{
 			  			\begin{tikzpicture}
 			  			\draw[blue,in=90,out=90,looseness=2] (0,0) to (1,0);
 			  			\draw[blue,in=-90,out=-90,looseness=2] (1,0) to (2,0);
 			  			\draw[blue] (1.5,-.6) to (1.5,-1.4);
 			  			\draw[blue] (0,0) to (0,-1.2);
 			  			\draw[blue] (2,0) to (2,.7);
 			  			\node[draw,thick,rounded corners, fill=white,minimum width=15,scale=.8] at (1,0) {$f^*$};
 			  			\node[draw,thick,rounded corners, fill=white,minimum width=15,scale=.8] at (2,0) {$f^*$};
 			  			\node[draw,thick,rounded corners, fill=white,minimum width=15,scale=.8] at (1.5,-1) {$f$};
 			  			\end{tikzpicture}}
 		  				=
 		  				 \raisebox{-6mm}{
 		  				 	\begin{tikzpicture}
 		  				 	\draw[blue,in=90,out=90,looseness=2] (-0.5,0.5) to (-1.5,0.5);
 		  				 	\draw[blue] (-1,1.1) to (-1,2);
 		  				 	\draw[blue] (-.5,.5) to (-.5,-.2);
 		  				 	\draw[blue] (-1.5,.5) to (-1.5,-.2);
 		  				 	\draw[dashed,orange,thick] (-1.6,2) rectangle (-.5,.8);
 		  				 	\draw[->,orange] (0,2) to (0,.7);
 		  				 	\node[draw,thick,rounded corners, fill=white,minimum width=15] at (-.5,.4){$f$};
 		  				 	\node[draw,thick,rounded corners, fill=white,minimum width=15] at (-1.5,.4){$f$};
 		  				 	\node[draw,thick,rounded corners, fill=white,minimum width=15] at (-1,1.6){$f^*$};
 		  				 	\node[scale=.8] at (0,2.2) [orange] {coalgebra morphism};
 		  				 	\node at (0,.5) {$=$};
 		  				 	\end{tikzpicture}}
 	  				 	\hspace*{-.75cm}
 	  				 	\raisebox{-6mm}{
 	  				 		\begin{tikzpicture}
 	  				 		\draw[blue,in=90,out=90,looseness=2] (-0.5,0.5) to (-1.5,0.5);
 	  				 		\draw[blue] (-1,1.1) to (-1,2);
 	  				 		\draw[blue] (-.5,.5) to (-.5,-.2);
 	  				 		\draw[blue] (-1.5,.5) to (-1.5,-.2);
 	  						\node[draw,thick,rounded corners, fill=white,minimum width=15,scale=.7] at (-1.5,.65) {$f^*$};
 				  			\node[draw,thick,rounded corners, fill=white,minimum width=15,scale=.7] at (-1.5,.15) {$f$};
 				  			\node[draw,thick,rounded corners, fill=white,minimum width=15,scale=.7] at (-.5,.65) {$f^*$};
 				  			\node[draw,thick,rounded corners, fill=white,minimum width=15,scale=.7] at (-.5,.15) {$f$};
 	  				 		\end{tikzpicture}}
 	  				 	=
 	  				 	\raisebox{-6mm}{
 	  				 		\begin{tikzpicture}
 	  				 		\draw[blue,in=90,out=90,looseness=2] (-0.5,0.5) to (-1.5,0.5);
 	  				 		\draw[blue] (-1,1.1) to (-1,2);
 	  				 		\draw[blue] (-.5,.5) to (-.5,-.2);
 	  				 		\draw[blue] (-1.5,.5) to (-1.5,-.2);
 	  				 		\node[draw,thick,rounded corners, fill=white,minimum width=15] at (-1.5,.4){$g$};
 	  				 		\node[draw,thick,rounded corners, fill=white,minimum width=15] at (-.5,.4){$g$};
 	  				 		\end{tikzpicture}}  \]
The second equality follows a similar reasoning. \qedhere
\end{proof}

We need to expand the concept of the convolution product.

\begin{definition}[Convolution Product] \label{def:conv}
Let $(X, m, \delta)$ be an algebra and coalgebra in a monoidal category $\mathcal{C}$. For two morphisms $a, b \in \End_{\mathcal C}(X)$, we define the \emph{convolution product} as $ a * b := \delta \circ (a \otimes b) \circ m$, illustrated as follows:
   				\[ a*b \coloneqq \delta \circ (a \otimes b) \circ m  = \raisebox{-11mm}{
   					\begin{tikzpicture}
   					\draw[blue,in=90,out=90,looseness=2] (0,0) to (1,0);
   					\draw[blue,in=-90,out=-90,looseness=2] (0,0) to (1,0);
   					\draw[blue] (.5,-.6) to (.5,-1.2);
   					\draw[blue] (.5,.6) to (.5,1.2);
   					\node[draw,thick,rounded corners, fill=white,minimum width=20] at (0,0){$a$};
   					\node[draw,thick,rounded corners, fill=white,minimum width=20] at (1,0){$b$}; 
   					\end{tikzpicture}} \]
\end{definition}

\begin{remark} \label{rk:ConvId}
Note that Lemma \ref{lem:contr} can be restated as $\mu_X (\alpha * \id_X) = \tr(\alpha) \id_X$.
\end{remark}

\begin{remark}[Fourier Transform] \label{rk:fourier}
Although it may not be utilized in this paper, it is noteworthy that the convolution product defined in Definition \ref{def:conv} can be derived from the Fourier transform:
$$ 
\begin{array}{cccc}
\mathcal{F}:& \End_{\mathcal C}(X) & \to & \Hom_{\mathcal{C}}(X \otimes X, X \otimes X) \\
& a & \mapsto & (\id_X \otimes m) \circ (\id_X \otimes a \otimes \id_X) \circ (\delta \otimes \id_X)
\end{array}
$$
with a left inverse given by $\mathcal{F}^{-1}(x) = (\epsilon \otimes \id_X) \circ x \circ (\id_X \otimes e)$. These transformations are depicted below:
   				\[ \mathcal F (a) = \raisebox{-8mm}{
   					\begin{tikzpicture}
   					\draw[blue,in=90,out=90,looseness=2] (0,0) to (1,0);
   					\draw[blue,in=-90,out=-90,looseness=2] (1,0) to (2,0);
   					\draw[blue] (.5,.6) to (.5,1.2);
   					\draw[blue] (1.5,-.6) to (1.5,-1.2);
   					\draw[blue] (0,0) to (0,-1.2);
   					\draw[blue] (2,0) to (2,1.2);
   					\node[draw,thick,rounded corners, fill=white,minimum width=20] at (1,0) {$a$};
   					\end{tikzpicture}}
   				\hspace*{.5cm} \text{ and } \hspace*{.5cm} 
   				\mathcal F^{-1} (x) = \raisebox{-4mm}{\begin{tikzpicture}
   				\draw[blue] (0,0) to (0,1);
   				\draw[blue] (.5,0) to (.5,1);
   				\node[draw,thick,rounded corners, fill=white,minimum width=22] at (.25,.5) {$x$};
   				\node at (0,0) {$\textcolor{blue}{\bullet}$};
   				\node at (.5,1) {$\textcolor{blue}{\bullet}$};
   				\end{tikzpicture}}  \]
It can be shown pictorially that $ a * b = \mathcal{F}^{-1} (\mathcal{F}(a) \circ \mathcal{F}(b))$; refer to \cite[Proposition 1]{BD21} for further details.
\end{remark}

\begin{corollary} \label{cor:exchangeconnected}
Let $X$, $X_1$, and $X_2$ be three Frobenius algebras in a $\mathbbm{k}$-linear monoidal category $\mathcal{C}$, where $\one$ is linear-simple. Let $f_i \in \Hom_{\mathcal{C}}(X_i, X)$ be an algebra morphism such that $f_i^{**} = f_i$, and define $g_i := f_i \circ f_i^*$. If $X$ is connected, then the following relation holds:
$$ \mu_X (g_1 * g_2) \circ g_i = \tr(g_1 \circ g_2) g_i = \mu_X g_i \circ (g_1 * g_2). $$
\end{corollary}

\begin{proof}
We first prove the left-hand side for $i = 2$ using the following illustration, applying the exchange relation (Theorem \ref{thm:exchange}) and the connectedness condition (Lemma \ref{lem:contr}):
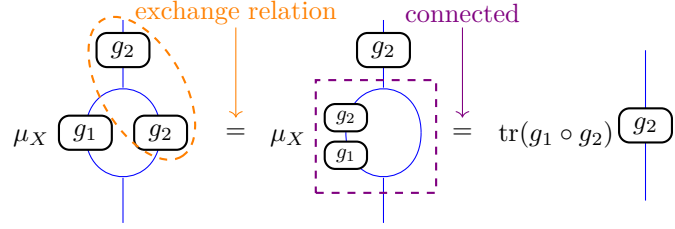
  			\[ \mu_X \raisebox{-11mm}{
   				\begin{tikzpicture}
   				\draw[blue,in=90,out=90,looseness=2] (0,0) to (1,0);
   				\draw[blue,in=-90,out=-90,looseness=2] (0,0) to (1,0);
   				\draw[blue] (.5,-.6) to (.5,-1.2);
   				\draw[blue] (.5,.6) to (.5,1.5);
   				\node[draw,thick,rounded corners, fill=white,minimum width=20] at (0,0){$g_1$};
   				\node[draw,thick,rounded corners, fill=white,minimum width=20] at (1,0){$g_2$};
   				\node[draw,thick,rounded corners, fill=white,minimum width=20] at (.5,1.1){$g_2$}; 
   				\draw[dashed, DarkOrange, thick] [rotate around={-60:(1.9,.1)}] (.9,-.65) ellipse (1.05 and .55);
   				\node at (2,1.6)[orange]{exchange relation};
   				\draw[->, DarkOrange] (2,1.4) to (2,.2);
   				\node at (2,0) {$=$};
   				\end{tikzpicture}}
   				\hspace*{-1cm}
   				\mu_X \raisebox{-11mm}{
   					\begin{tikzpicture}
   					\draw[blue,in=90,out=90,looseness=2] (0,0) to (1,0);
   					\draw[blue,in=-90,out=-90,looseness=2] (0,0) to (1,0);
   					\draw[blue] (.5,-.6) to (.5,-1.2);
   					\draw[blue] (.5,.6) to (.5,1.5);
   					\draw[->,violet] (1.55,1.4) to (1.55,.2);
   					\draw[violet,dashed,thick] (-.4,.7) rectangle (1.2,-.8);
   					\node[draw,thick,rounded corners, fill=white,minimum width=20,scale=.8] at (0,-.3){$g_1$};
   					\node[draw,thick,rounded corners, fill=white,minimum width=20,scale=.8] at (0,.2){$g_2$};
   					\node[draw,thick,rounded corners, fill=white,minimum width=20] at (.5,1.1){$g_2$};
   					\node at (1.55,1.6) [violet] {connected};
   					\node at (1.55,0) {$=$};
   					\end{tikzpicture}}
   					\hspace*{-.4cm}
   					\text{tr}(g_1 \circ g_2) \, \raisebox{-8mm}{\begin{tikzpicture}
   					\draw[blue] (0,-1) to (0,1);
   				\node[draw,thick,rounded corners, fill=white,minimum width=20] at (0,0) {$g_2$};
   					\end{tikzpicture}} \]
For the case when $i = 1$, the reasoning is similar. 
The second equality follows by a similar argument, concluding with $ \tr(g_2 \circ g_1) = \tr(g_1 \circ g_2) $ by Lemma~\ref{lem:trfg}, using the fact that $ g_i^{**} = g_i $.
\end{proof}

\begin{corollary} \label{cor:exchangesub}
Following the notations of Definition \ref{def:sub}, the morphism $b_{X'}$ satisfies the exchange relations. Moreover, if $X_1$ and $X_2$ are Frobenius subalgebras of $X$ in a $\mathbbm{k}$-linear monoidal category $\mathcal{C}$ where $\one$ is linear-simple, then:
$$ \mu_X (b_{X_1} * b_{X_2}) \circ b_{X_i} = \tr(b_{X_1} \circ b_{X_2}) b_{X_i}.$$
\end{corollary}

\begin{proof}
The first part is immediate from Theorem \ref{thm:exchange} with $f = i$. For the second part, we can apply Corollary \ref{cor:exchangeconnected} with $g_i = b_{X_i}$.
\end{proof}

\begin{proposition} \label{prop:coprod}
Let $X$ and $X'$ be two algebras-coalgebras in a $\mathbbm{k}$-linear monoidal category $\mathcal{C}$, with the assumption that $X'$ is weakly separable. Let $f \in \Hom_{\mathcal{C}}(X', X)$ be an algebra morphism and $h \in \Hom_{\mathcal{C}}(X, X')$ a coalgebra morphism. Define $g := f \circ h$. Then, it holds that $g * g = \lambda_{X'} g$, where $\lambda_{X'}$ is the constant associated with the weak separability of $X'$ (see Definition \ref{def:connected}). Moreover, if both $X$ and $X'$ are Frobenius, $X$ is connected, and $\one$ is linear-simple, with $h = f^*$, $h^{*} = f$, and $g = f \circ f^*$ being a nonzero idempotent (e.g., $f^* \circ f = \id_{X'}$), then $\tr(g) = \mu_X \lambda_{X'}$, where $\mu_X$ is the scalar defined by $\epsilon \circ e$. Consequently, we have $\mu_X (g * g) = \tr(g) g$.
\end{proposition}

\begin{proof}
The first part is illustrated by the following diagram:
   			\[ g*g = \raisebox{-11mm}
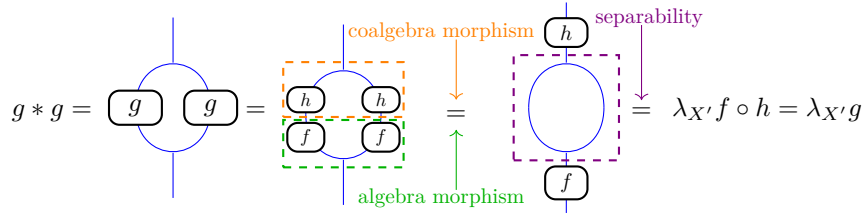
{
   				\begin{tikzpicture}
   				\draw[blue,in=90,out=90,looseness=2] (0,0) to (1,0);
   				\draw[blue,in=-90,out=-90,looseness=2] (0,0) to (1,0);
   				\draw[blue] (.5,-.6) to (.5,-1.2);
   				\draw[blue] (.5,.6) to (.5,1.2);
   				\node[draw,thick,rounded corners, fill=white,minimum width=20] at (0,0){$g$};
   				\node[draw,thick,rounded corners, fill=white,minimum width=20] at (1,0){$g$}; 
   				\end{tikzpicture}}
   				=
   				\raisebox{-13mm}{
   					\begin{tikzpicture}
   					\draw[blue,in=90,out=90,looseness=2] (0,0) to (1,0);
   					\draw[blue,in=-90,out=-90,looseness=2] (0,0) to (1,0);
   					\draw[blue] (.5,-.6) to (.5,-1.2);
   					\draw[blue] (.5,.6) to (.5,1.2);
   					\draw[dashed,orange,thick] (-.3,.7) rectangle (1.3,-.03);
   					\draw[dashed,green!70!black,thick] (-.3,-.07)rectangle (1.3,-.7);
   					\draw[->,green!70!black] (2,-1) to (2,-.2);
   					\draw[->,orange] (2,1) to (2,.2);
   					\node[scale=.8] at (1.8,1.1) [orange] {coalgebra morphism};
   					\node[scale=.8] at (1.8,-1.1) [green!70!black] {algebra morphism}; 
   					 	\node[draw,thick,rounded corners, fill=white,minimum width=20,scale=.7] at (0,.2){$h$};
   					\node[draw,thick,rounded corners, fill=white,minimum width=20,scale=.7] at (1,.2){$h$};
   					\node[draw,thick,rounded corners, fill=white,minimum width=20,scale=.7] at (1,-.3) {$f$};
   					\node[draw,thick,rounded corners, fill=white,minimum width=20,scale=.7] at (0,-.3) {$f$};
   				\node at (2,0) {$=$};
   					\end{tikzpicture}}
   					\hspace{-.5cm}
   				\raisebox{-13mm}{
   					\begin{tikzpicture}
   					\draw[blue,in=90,out=90,looseness=2] (0,0) to (1,0);
   					\draw[blue,in=-90,out=-90,looseness=2] (0,0) to (1,0);
   					\draw[blue] (.5,-.6) to (.5,-1.4);
   					\draw[blue] (.5,.6) to (.5,1.4);
   					\draw[dashed,violet,thick] (-.2,.7) rectangle (1.2,-.7);
   					\draw[->,violet] (1.5,.8) to (1.5,.1);
   					\node[scale=.8] at (1.6,1.2) [violet] {\begin{tabular}{c} weak \\ separability \end{tabular}};
   					\node[draw,thick,rounded corners, fill=white,minimum width=20,scale=.8] at (.5,1) {$h$};
   					\node[draw,thick,rounded corners, fill=white,minimum width=20,scale=.8] at (.5,-1) {$f$};
   					\node at (1.5,-.05) {$=$};
   					\end{tikzpicture}}
   					\hspace*{-.5cm}
   				\lambda_{X'} f \circ h = \lambda_{X'} g \]
   			
For the second part, define $A \in \Hom_{\mathcal{C}}(X, X \otimes X)$ as follows:
   				\[ A \coloneqq \raisebox{-6mm}{
   					\begin{tikzpicture}
   					\draw[blue,in=90,out=90,looseness=2] (-0.5,0.5) to (-1.5,0.5);
   					\draw[blue] (-1,1.1) to (-1,2);
   					\draw[blue] (-.5,.5) to (-.5,-.2);
   					\draw[blue] (-1.5,.5) to (-1.5,-.2);
   					\node[draw,thick,rounded corners, fill=white,minimum width=15] at (-1.5,.4){$g$};
   					\node[draw,thick,rounded corners, fill=white,minimum width=15] at (-1,1.6){$g$};
   					\end{tikzpicture}} - \raisebox{-6mm}{
   					\begin{tikzpicture}
   					\draw[blue,in=90,out=90,looseness=2] (-0.5,0.5) to (-1.5,0.5);
   					\draw[blue] (-1,1.1) to (-1,2);
   					\draw[blue] (-.5,.5) to (-.5,-.2);
   					\draw[blue] (-1.5,.5) to (-1.5,-.2);
   					\node[draw,thick,rounded corners, fill=white,minimum width=15] at (-.5,.4){$g$};
   					\node[draw,thick,rounded corners, fill=white,minimum width=15] at (-1,1.6){$g$};
   					\end{tikzpicture}}  \]
We observe that
   				\[ A^* \circ A = 2 \raisebox{-11mm}{
   					\begin{tikzpicture}
   					\draw[blue,in=90,out=90,looseness=2] (0,0) to (1,0);
   					\draw[blue,in=-90,out=-90,looseness=2] (0,0) to (1,0);
   					\draw[blue] (.5,-.6) to (.5,-1.2);
   					\draw[blue] (.5,.6) to (.5,1.2);
   					\node[draw,thick,rounded corners, fill=white,minimum width=15,scale=.8] at (0,0){$g$};
   					\node[draw,thick,rounded corners, fill=white,minimum width=15,scale=.8] at (1,0){$g$};
   					\node[draw,thick,rounded corners, fill=white,minimum width=15,scale=.8] at (.5,.85){$g$};
   					\node[draw,thick,rounded corners, fill=white,minimum width=15,scale=.8] at (.5,-.85){$g$}; 
   					\end{tikzpicture}}
   				-
   				 \raisebox{-11mm}{
   				 	\begin{tikzpicture}
   				 	\draw[blue,in=90,out=90,looseness=2] (0,0) to (1,0);
   				 	\draw[blue,in=-90,out=-90,looseness=2] (0,0) to (1,0);
   				 	\draw[blue] (.5,-.6) to (.5,-1.2);
   				 	\draw[blue] (.5,.6) to (.5,1.2);
   				 	\node[draw,thick,rounded corners, fill=white,minimum width=15,scale=.8] at (0,0){$g$};
   				 	\node[draw,thick,rounded corners, fill=white,minimum width=15,scale=.8] at (.5,.85){$g$};
   				 	\node[draw,thick,rounded corners, fill=white,minimum width=15,scale=.8] at (.5,-.85){$g$}; 
   				 	\end{tikzpicture}}
   			 	-
   			 	\raisebox{-11mm}{
   			 		\begin{tikzpicture}
   			 		\draw[blue,in=90,out=90,looseness=2] (0,0) to (1,0);
   			 		\draw[blue,in=-90,out=-90,looseness=2] (0,0) to (1,0);
   			 		\draw[blue] (.5,-.6) to (.5,-1.2);
   			 		\draw[blue] (.5,.6) to (.5,1.2);
   			 		\node[draw,thick,rounded corners, fill=white,minimum width=15,scale=.8] at (1,0){$g$};
   			 		\node[draw,thick,rounded corners, fill=white,minimum width=15,scale=.8] at (.5,.85){$g$};
   			 		\node[draw,thick,rounded corners, fill=white,minimum width=15,scale=.8] at (.5,-.85){$g$}; 
   			 		\end{tikzpicture}} \]
By the first part and Lemma \ref{lem:contr}, we obtain 
$$
\mu_X A^* \circ A = 2(\mu_X \lambda_{X'} - \tr(g))g,
$$
but $A = 0$ by Theorem \ref{thm:exchange}. The result follows.
\end{proof}

\begin{corollary} \label{cor:coprod}
Using the notation from Definition \ref{def:sub}, the following holds in the connected case if $\one$ is linear-simple:
$$
\tr(b_{X'}) = \mu_X \lambda_{X'} \quad \text{and} \quad \mu_X (b_{X'} * b_{X'}) = \tr(b_{X'}) b_{X'}.
$$
\end{corollary}

\begin{remark}[\cite{Bi94, BJ00, Lan02, Liu16}]
Theorem \ref{thm:exchange} and Proposition \ref{prop:coprod} align with the theory of irreducible finite index subfactors $(N \subset P \subset M)$, where $X = {}_N M_N$, $X' = {}_N P_N$, $g = b_{X'} = e^M_P$, and $\lambda_X = \mu_X = |M:N|^{1/2}$. Note that the trace $\tr$ here is not normalized, in contrast to its typical normalization within the subfactor framework.
\end{remark}

\subsection{Landau's theorem} \label{sub:landau}

The following result generalizes Landau's theorem (e.g., see \cite[Theorem 3.10]{GJ06}):

\begin{theorem}[Landau's theorem] \label{thm:landau}
Let $X$, $X_1$, and $X_2$ be three Frobenius algebras in a $\mathbbm{k}$-linear monoidal category $\mathcal{C}$ with $\one$ being linear-simple. Assume that $X$ is connected. For $i=1,2$, let $f_i \in \Hom_{\mathcal{C}}(X_i, X)$ be an algebra morphism such that $f_i^{**} = f_i$, and let $g_i := f_i \circ f_i^*$. Then 
$$
\mu_X (g_1 * g_2) \circ (g_1 * g_2) = \tr(g_1 \circ g_2)(g_1 * g_2).
$$
Moreover, for every algebra morphism $\alpha \in \End_{\mathcal C}(X)$—for instance, $\alpha = \id_X$—
$$
\mu_X \tr_{\alpha}(g_1 * g_2) = \tr_{\alpha}(g_1) \tr_{\alpha}(g_2),
$$
where $\tr_{\alpha}(h):=\tr(\alpha \circ h)$. As a result, if $\tr(g_1 \circ g_2)$ is nonzero, then 
$$
g_{12} := \frac{\mu_X}{\tr(g_1 \circ g_2)} g_1 * g_2
$$
is an idempotent, and we have\footnote{The trace in the denominator is \emph{not} indexed by $\alpha$; this is correct and not a typo} 
$$
\tr_{\alpha}(g_{12}) = \frac{\tr_{\alpha}(g_1) \tr_{\alpha}(g_2)}{\tr(g_1 \circ g_2)}.
$$
Finally, $g_{12} \circ g_i = g_i = g_i \circ g_{12}$ for $i \in \{1, 2\}$.
\end{theorem}

\begin{proof}
The proof begins with the diagram below, utilizing Lemmas \ref{lem:intermultdualcomultalpha}, \ref{lem:FrobExtra}, and Theorem \ref{thm:exchange}:
 			\[\raisebox{-6mm}{
 				\begin{tikzpicture}
 				\draw[blue,in=90,out=90,looseness=2] (0,0) to (1,0);
 				\draw[blue] (.5,.6) to (.5,1.2);
 				\draw[blue,in=-90,out=-90,looseness=2] (0,1.8) to (1,1.8);
 				\draw[blue,in=90,out=90,looseness=2] (0,2) to (1,2);
 				\draw[blue,in=-90,out=-90,looseness=2] (0,-.2) to (1,-.2);
 				\draw[blue] (.5,2.6) to (.5,2.9);
 				\draw[blue] (.5,-.8) to (.5,-1.2);
 				\draw[dashed,green!70!black,thick] (-.1,1.5) rectangle (1.1,.3);
 				\draw[->,green!70!black] (1.6,2.5) to (1.6,1);
 				\node at (1.6,2.7) [green!70!black] {Frobenius};
 				\node[draw,thick,rounded corners, fill=white,minimum width=15,scale=.8] at (0,1.8) {$g_1$};
 				\node[draw,thick,rounded corners, fill=white,minimum width=15,scale=.8] at (0,0) {$g_1$};
 				\node[draw,thick,rounded corners, fill=white,minimum width=15,scale=.8] at (1,1.8) {$g_2$};
 				\node[draw,thick,rounded corners, fill=white,minimum width=15,scale=.8] at (1,0) {$g_2$};
 				\node at (1.6,.8) {$=$};
 				\end{tikzpicture}} 
 				\raisebox{-8mm}{\begin{tikzpicture}
 				\draw[blue,in=-90,out=-90,looseness=2] (0,0) to (1,0);
 				\draw[blue,in=90,out=90,looseness=2] (1,0) to (2,0);
 				\draw[blue,in=-90,out=-90,looseness=2] (2,0) to (3,0);
 				\draw[blue,in=90,out=90,looseness=2] (0,0) to (3,0);
 				\draw[blue,in=-90,out=-90,looseness=2] (.5,-.6) to (2.5,-.6);
 				\draw[blue] (1.5,-1.77) to (1.5,-2.2);
 				\draw[blue] (1.5,1.75) to (1.5,2.2);
 				\draw[dashed,red,thick] (-.1,1.3) rectangle (1,-1.3);
 				\draw[dashed,red,thick] (2,1.3) rectangle (3.1,-1.3);
 				\draw[->,red] (3.3,1.5) to (3.3,.2);
 				\node at (3.3,1.8) [red] {exchange relation};
 				\node[draw,thick,rounded corners, fill=white,minimum width=15,scale=.8] at (.2,1) {$g_1$};
 				\node[draw,thick,rounded corners, fill=white,minimum width=15,scale=.8] at (2.7,1) {$g_2$};
 				\node[draw,thick,rounded corners, fill=white,minimum width=15,scale=.8] at (.6,-1) {$g_1$};
 				\node[draw,thick,rounded corners, fill=white,minimum width=15,scale=.8] at (2.4,-1) {$g_2$};
 				\node at (3.3,0) {$=$};
 				\end{tikzpicture}} 
 				\hspace*{-.6cm}
 				\raisebox{-6mm}{\begin{tikzpicture}
 					\draw[blue,in=-90,out=-90,looseness=2] (0,0) to (1,0);
 					\draw[blue,in=90,out=90,looseness=2] (1,0) to (2,0);
 					\draw[blue,in=-90,out=-90,looseness=2] (2,0) to (3,0);
 					\draw[blue,in=90,out=90,looseness=2] (0,0) to (3,0);
 					\draw[blue,in=-90,out=-90,looseness=2] (.5,-.6) to (2.5,-.6);
 					\draw[blue] (1.5,-1.77) to (1.5,-2.2);
 					\draw[blue] (1.5,1.75) to (1.5,2.2);
 					\node[draw,thick,rounded corners, fill=white,minimum width=15,scale=.8] at (1,0) {$g_1$};
 					\node[draw,thick,rounded corners, fill=white,minimum width=15,scale=.8] at (2,0) {$g_2$};
 					\node[draw,thick,rounded corners, fill=white,minimum width=15,scale=.8] at (.6,-1) {$g_1$};
 					\node[draw,thick,rounded corners, fill=white,minimum width=15,scale=.8] at (2.4,-1) {$g_2$};
 					\draw[dashed, pink!70!black, thick] (.65,.4) circle (1.1);
 					\node at (3.3,0) {$=$};
 					\end{tikzpicture}}
 					\raisebox{1mm}{\begin{tikzpicture}
 						\draw[blue,in=90,out=90,looseness=2] (0,0) to (1,0);
 						\draw[blue,in=90,out=90,looseness=2] (.5,.6) to (1.5,.6);
 						\draw[blue] (1.5,.6) to (1.5,0);
 						\draw[blue] (1,1.2) to (1,1.6);
 						\draw[blue,in=-90,out=-90,looseness=2] (1,0) to (1.5,0);
 						\draw[dashed,orange,thick] (0,1.7) rectangle (1.6,.35);
 						\draw[blue,in=-90,out=-90,looseness=2] (0,-.3) to (1.25,-.3);
 						\draw[blue] (0,0) to (0,-.3);
 						\draw[blue] (.6,-1.03) to (.6,-1.5);
 						\draw[->,orange] (1.8,1.5) to (1.8,.2);
 						\node at (1.8,1.9) [orange] {coassociativity};
 					\node[draw,thick,rounded corners, fill=white,minimum width=15,scale=.8] at (.8,0) {$g_2 \circ g_1$};
 						\node[draw,thick,rounded corners, fill=white,minimum width=15,scale=.8] at (.1,-.7) {$g_1$};
 						\node[draw,thick,rounded corners, fill=white,minimum width=15,scale=.8] at (1.15,-.7) {$g_2$};
 						\node at (1.8,0) {$=$};
 						\end{tikzpicture}}
 						\hspace*{-.6cm}
 						\raisebox{-2mm}{
 							\begin{tikzpicture}
 							\draw[blue,in=90,out=90,looseness=2] (.5,0) to (1.5,0);
 							\draw[blue,in=90,out=90,looseness=2] (1,.6) to (-.2,.6);
 							\draw[blue] (-.2,.6) to (-.2,-1);
 							\draw[blue] (.4,1.3) to (.4,1.7);
 							\draw[blue,in=-90,out=-90,looseness=2] (.5,0) to (1.5,0);
 							 \draw[blue,in=-90,out=-90,looseness=2] (-.2,-1) to (1,-1);
 							\draw[blue] (.4,-1.7) to (.4,-2.2);
 							 \draw[blue] (1,-1) to(1,-.6);
 					\node[draw,thick,rounded corners, fill=white,minimum width=15,scale=.8] at (.5,0) {$g_2 \circ g_1$};
 							\node[draw,thick,rounded corners, fill=white,minimum width=15,scale=.8] at (-.2,-1) {$g_1$};
 							\node[draw,thick,rounded corners, fill=white,minimum width=15,scale=.8] at (1,-1) {$g_2$};
 							\end{tikzpicture}}\]
Because $X$ is connected, Lemma~\ref{lem:contr} applies; in view of Lemma~\ref{lem:trfg} and the fact that $g_i^{**} = g_i$, it follows that
$$
\mu_X (g_1 * g_2) \circ (g_1 * g_2) = \tr(g_1 \circ g_2)(g_1 * g_2).
$$
The following figure proves the second equality, applying the connectedness of $X$ via Lemma~\ref{lem:contr}:
 			\[ \mu_X \, \tr_{\alpha}(g_1 * g_2) = \mu_X \, \raisebox{-20mm}{
 				\begin{tikzpicture}
 				\draw[blue,in=90,out=90,looseness=2] (0,0) to (1,0);
 				\draw[blue,in=-90,out=-90,looseness=2] (0,0) to (1,0);
 				\draw[blue,in=90,out=90,looseness=2] (.5,.6) to (1.5,.6);
 				\draw[blue,in=-90,out=-90,looseness=2] (.5,-1) to (1.5,-1);
 				\draw[blue] (.5,-.6) to (.5,-1);
 				\draw[blue] (1.5,.6) to (1.5,-1);
 				\draw[dashed,red,thick] (.4,1.4) rectangle (1.6,.7);
 				\draw[dashed,red,thick] (.4,-1.25) rectangle (1.6,-1.75);
 				\draw[->,red] (1.8,1.5) to (1.8,0);
 				\node at (1.8,1.7) [red] {selfdual};
 				\draw[->, DarkOrange] (1.8,-1.75) to (1.8,-.45);
 				\node at (2,-1.95)[orange] {algebra morphism};
 				\draw[dashed,orange,thick] (.1,-1.15) rectangle (.9,-.475);
 				\node[draw,thick,rounded corners, fill=white,minimum width=15] at (0,0){$g_1$};
 				\node[draw,thick,rounded corners, fill=white,minimum width=15] at (1,0){$g_2$};
 				\node[draw,thick,rounded corners, fill=white,minimum width=15] at (.5,-.875){$\alpha$};
 				\node at (1.8,-.25) {$=$}; 
 				\end{tikzpicture}} 
 				\hspace*{-1.4cm}
 				\mu_X \, \raisebox{-20mm}{
 					\begin{tikzpicture}
 					\draw[blue,in=90,out=90,looseness=2] (0,0) to (1,0);
 					\draw[blue,in=-90,out=-90,looseness=2] (0,-.4) to (1,-.4);
 					\draw[blue,in=90,out=90,looseness=2] (.5,.6) to (1.5,.6);
 					\draw[blue,in=-90,out=-90,looseness=2] (.5,-1) to (1.5,-1);
 					\draw[blue] (0,0) to (0,-.4);
 					\draw[blue] (1,0) to (1,-.4);
 					\draw[blue] (1.5,.6) to (1.5,-1);
 					\draw[dashed,green!70!black,thick] (0,1.5) rectangle (1.6,.3);
 					\draw[dashed,green!70!black,thick] (0,-.8) rectangle (1.6,-1.8);
 					\draw[->,green!70!black] (2,1.8) to (2,0);
 					\node at (1.5,2) [green!70!black] {(co)associativity};
 					\node[draw,thick,rounded corners, fill=white,minimum width=15] at (0,0){$g_1$};
 					\node[draw,thick,rounded corners, fill=white,minimum width=15] at (1,0){$g_2$};
 					 \node[draw,thick,rounded corners, fill=white,minimum width=15] at (0,-.5){$\alpha$};
 					\node[draw,thick,rounded corners, fill=white,minimum width=15] at (1,-.5){$\alpha$};
 					\node at (1,1.7) {$\color{blue}{\bullet}$};
 					\node at (1,-2.1) {$\color{blue}{\bullet}$};
 					\draw[blue] (1,1.2) to (1,1.7);
 					\draw[blue] (1,-1.6) to (1,-2.1);
 					\node at (2,-.2) {$=$}; 
 					\end{tikzpicture}}
 					\hspace*{-.7cm}
 					\mu_X \, \raisebox{-20mm}{
 					\begin{tikzpicture}
 					\draw[blue,in=90,out=90,looseness=2] (1,0) to (2,0);
 					\draw[blue,in=-90,out=-90,looseness=2] (1,-.4) to (2,-.4);
 					\draw[blue,in=90,out=90,looseness=2] (0.3,.6) to (1.5,.6);
 					\draw[blue,in=-90,out=-90,looseness=2] (0.3,-1) to (1.5,-1);
 					\draw[blue] (2,0) to (2,-.4);
 					\draw[blue] (0.3,0) to (0.3,-.4);
 					\draw[blue] (1,0) to (1,-.4);
 					\draw[blue] (0.3,.6) to (0.3,-1);
 					\node[draw,thick,rounded corners, fill=white,minimum width=15] at (0.3,0){$g_1$};
 					\node[draw,thick,rounded corners, fill=white,minimum width=15] at (1,0){$g_2$};
 					\node[draw,thick,rounded corners, fill=white,minimum width=15] at (0.3,-.5){$\alpha$};
 					\node[draw,thick,rounded corners, fill=white,minimum width=15] at (1,-.5){$\alpha$};
 					\node at (.9,1.8) {$\color{blue}{\bullet}$};
 					\node at (.9,-2.2) {$\color{blue}{\bullet}$};
 					\draw[blue] (.9,1.3) to (.9,1.8);
 					\draw[blue] (.9,-1.7) to (.9,-2.2);
 					\draw[dashed,violet,thick] (.65,.75) rectangle (2.1,-1.1);
 					\draw[->,violet] (2.3,1.3) to (2.3,-.1);
 					\draw[->,red] (2.3,-1.3) to (2.3,-.5);
 					\node at (2.7,1.5) [violet] {connectedness};
 					\node at (2.3,-1.5) [red] {selfdual};
 					\draw[dashed,red,thick] (.4,2) rectangle (1.4,1);
 					\draw[dashed,red,thick] (.4,-1.45) rectangle (1.4,-2.35);
 					\node at (2.3,-.35) {$=$}; 
 					\end{tikzpicture}}		
 					\hspace*{-1.3cm}	
 						\raisebox{-9mm}{\begin{tikzpicture}
 					\draw[blue,in=90,out=90,looseness=2] (0,0) to (1,0);
 					\draw[blue,in=-90,out=-90,looseness=2] (0,-.4) to (1,-.4);
 					\draw[blue] (1,0) to (1,-.4);
 					\draw[blue] (0,0) to (0,-.4);		
 					\node[draw,thick,rounded corners, fill=white,minimum width=15] at (0,0){$g_1$};
 					\node[draw,thick,rounded corners, fill=white,minimum width=15] at (0,-.5){$\alpha$};
 							\end{tikzpicture}} \hspace*{.1cm}
 						\tr_{\alpha}(g_2) =
 						\tr_{\alpha}(g_1) \tr_{\alpha}(g_2) \]
The final equalities follows from Corollary \ref{cor:exchangeconnected}.
\end{proof}

\begin{corollary} \label{cor:landau}
Let $ X $ be a connected Frobenius algebra in a $ \mathbbm{k} $-linear monoidal category $ \mathcal{C} $ with $ \one $ being linear-simple. Let $ A $ and $ B $ be two Frobenius subalgebras of $ X $. Using the notation from Definition \ref{def:sub}, we have 
$$
\mu_X (b_{A} * b_{B}) \circ (b_{A} * b_{B}) = \tr(b_{A} \circ b_{B})(b_{A} * b_{B}).
$$
Moreover, for every algebra morphism $\alpha \in \End_{\mathcal C}(X)$—for instance, $\alpha = \id_X$—, 
$$
\mu_X \tr_{\alpha}(b_{A} * b_{B}) = \tr_{\alpha}(b_{A}) \tr_{\alpha}(b_{B}).
$$
Consequently, if $ \tr(b_{A} \circ b_{B}) $ is nonzero, then 
$$
b_{AB} := \frac{\mu_X}{\tr(b_{A} \circ b_{B})} b_{A} * b_{B}
$$
is an idempotent, and therefore 
$$
\tr_{\alpha}(b_{AB}) = \frac{\tr_{\alpha}(b_{A}) \tr_{\alpha}(b_{B})}{\tr(b_{A} \circ b_{B})},
$$
and, $ b_{AB} \circ b_Y = b_Y = b_Y \circ b_{AB} $ for all $ Y \in \{A, B\} $. Lastly, if $\tr_{\alpha}(\id_X)$ is nonzero, we have  $$\tr_{\alpha}(b_A) = \tr(b_A), \ \tr_{\alpha}(b_B) = \tr(b_B), \ \text{ and } \ \tr_{\alpha}(b_{AB}) = \tr(b_{AB}).$$
\end{corollary}
\begin{proof}
The result follows immediately from Theorem~\ref{thm:landau}, except for the final claim, which we prove now.

Define $b'_{AB} := \mu_X (b_A * b_B)$. Then $\tr_{\alpha}(b'_{AB}) = \mu_X \tr_{\alpha}(b_A * b_B) = \tr_{\alpha}(b_A)\tr_{\alpha}(b_B)$. In particular, since $b_X = \id_X$, we have $\tr_{\alpha}(b'_{AX}) = \tr_{\alpha}(b_A)\tr_{\alpha}(\id_X)$. On the other hand, applying Remark~\ref{rk:ConvId} yields $b'_{AX} = \mu_X (b_A * \id_X) = \tr(b_A) \id_X$, which implies $\tr_{\alpha}(b'_{AX}) = \tr(b_A)\tr_{\alpha}(\id_X)$. Given that $\tr_{\alpha}(\id_X) \neq 0$, it follows that $\tr_{\alpha}(b_A) = \tr(b_A)$. A similar argument shows that $\tr_{\alpha}(b_B) = \tr(b_B)$.
Finally, we obtain the desired equality:
\[
\tr_{\alpha}(b_{AB}) = \frac{\tr_{\alpha}(b_{A}) \tr_{\alpha}(b_{B})}{\tr(b_{A} \circ b_{B})} = \frac{\tr(b_{A}) \tr(b_{B})}{\tr(b_{A} \circ b_{B})} = \tr(b_{AB}). \qedhere
\]
\end{proof}

\begin{remark} \label{rk:AB}
The notation $ b_{AB} $ is inspired by \cite[Theorem 3.10]{GJ06}. However, it is used here solely as a notation; there is no requirement to show that $ b_{AB} $ corresponds to the idempotent associated with the image of $ m \circ (i_A \otimes i_B) $.
\end{remark}

\section{Rigidity invariance} \label{sec:RigidInvariant}

The notion of a Frobenius subalgebra (Definition~\ref{def:sub}) is not stable under intersection (see Example~\ref{ex:Ben02}).
To remedy this, we introduce the notion of \emph{rigidity invariance}, first in the category $\VVec$ (Definition \ref{def:RigidInvVec}). Although this notion depends on a choice of basis, Lemma~\ref{lem:FixedBasis} guarantees that every Frobenius subalgebra is rigid invariant for at least one such basis. We then extend this notion in the semisimple case (Definition \ref{def:RigidInvariance}), under which the failure mentioned above is resolved by ensuring that ambient selfduality is preserved under intersection (see Lemma~\ref{lem:RigidInvInterSum}). 
\begin{notation} \label{not:B-sub}
A $\mathcal{B}$-rigid invariant Frobenius subalgebra will be denoted a Frobenius $\mathcal{B}$-subalgebra.
\end{notation}
The choice of a basis~$\mathcal{B}$ can be viewed as a \emph{perspective} under which the Frobenius subalgebra poset \emph{collapses} to the Frobenius $\mathcal{B}$-subalgebra sublattice (see \S \ref{sec:lattice}). Furthermore, every Frobenius subalgebra belongs to one of these sublattices. 
Within a unitary tensor category, every unitary Frobenius subalgebra is rigid invariant with respect to \emph{every} choice of basis (see Proposition~\ref{prop:UnitFrobSub}). In particular, the notion of Frobenius $\mathcal{B}$-subalgebra is a proper generalization of the notion of unitary Frobenius subalgebra. This permits to fully recover Watatani's theorem in the unitary setting and suggests two natural directions for extending it beyond the unitary case: one considers Frobenius subalgebras that are rigid-invariant with respect to \emph{every} basis, and the other focuses on those invariant under a \emph{fixed} basis. The first approach is inherently basis-independent, while the second offers a broader generalization.

Finally, note that rigidity invariance appears to be unrelated to the Nakayama automorphism: in Example~\ref{ex:Ben02}, the bilinear form~$\kappa$ is symmetric.

\begin{remark} \label{rk:anisotropic}
From now on, we require a field equipped with an involution $\sigma$ such that the standard $\sigma$-Hermitian form is \emph{anisotropic}. The real field with the identity involution satisfies this condition, as does the complex field with complex conjugation. Certain real-like fields and some degree-two field extensions also meet this requirement. However, for simplicity, we will henceforth assume that we are working over the complex field.
\end{remark}

\subsection{Vector spaces} \label{sub:RigidVec}

In \(\VVec\), a Frobenius algebra \(U\) is a unital algebra with a nondegenerate, associative bilinear form \(\kappa: U \times U \to \mathbb{C}\). A Frobenius subalgebra \(V \subseteq U\) is unital subalgebra where \(\kappa|_{V \times V}\) remains nondegenerate.

More generally, let \(U\) be a finite-dimensional \(\mathbb{C}\)-vector space, and let \(\kappa : U \times U \to \mathbb{C}\) be a nondegenerate bilinear form, meaning that \(U^{\perp_\kappa} = \{0\}\). For any subspace \(V \subseteq U\), define
\[
V^{\perp_\kappa} := \{u \in U \mid \kappa(u, v) = 0 \text{ for all } v \in V\}.
\]

Then \(\kappa\) is completely determined by a conjugate-linear automorphism \(M : U \to U\), defined with respect to a chosen basis \((e_i)\) of \(U\), by the relation
\[
\langle e_i, M e_j \rangle = \kappa(e_i, e_j),
\]
where \(\langle \cdot, \cdot \rangle\) denotes the Hermitian form, conjugate-linear in the \emph{second} argument, given by \(\langle e_i, e_j \rangle = \delta_{i,j}\). 

By an abuse of notation, the conjugate-linear automorphism $M$ will be represented as a matrix in several subsequent examples, leaving its conjugate-linearity implicit.

\begin{definition} \label{def:RigidInvVec}
A subspace \( V \subseteq U \) is called:
\begin{itemize}
  \item \emph{nondegenerate} if the restriction \(\kappa|_{V \times V}\) is nondegenerate;
  \item \((e_i)\)-\emph{rigid invariant} (or \emph{rigid invariant with respect to \((e_i)\)}) if \( M(V) = V \);
  \item \emph{rigid invariant} if it is \((e_i)\)-rigid invariant for \emph{every} basis~\((e_i)\).
\end{itemize}
\end{definition}
  

\begin{lemma} \label{lem:two}
Let $(e_i)$ and $(f_i)$ be two bases of a complex vector space $V$. Let $\kappa$ be a bilinear form on $V$. Let $\langle \cdot, \cdot \rangle_e$ and $\langle \cdot, \cdot \rangle_f$ be Hermitian forms on $V$ for which the bases $(e_i)$ and $(f_i)$ are orthonormal, respectively. 
Define conjugate-linear transformations $M$ and $N$ of $V$ by
$
\kappa(e_i, e_j) = \langle e_i, M e_j \rangle_e$, and $\kappa(f_i, f_j) = \langle f_i, N f_j \rangle_f.
$
Let $P$ be the change-of-basis matrix defined by $f_i = P e_i$. Then
\(
N = P^* P M,
\)
where $P^*$ denotes the Hermitian adjoint of $P$ with respect to $\langle \cdot, \cdot \rangle_f$.
\end{lemma}
\begin{proof}
Observe that for all $u, v \in V$, we have
\(
\langle u, v \rangle_e = \langle Pu, Pv \rangle_f,
\)
and
\(
\langle u, Mv \rangle_e = \kappa(u, v) = \langle u, Nv \rangle_f.
\)
We obtain
\[
\langle u, Nv \rangle_f = \langle u, Mv \rangle_e = \langle Pu, P Mv \rangle_f = \langle u, P^* P Mv \rangle_f.
\]
Since this holds for all $u, v \in V$ and the Hermitian form $\langle \cdot, \cdot \rangle_f$ is nondegenerate, it follows that
\(
N = P^* P M.
\)
\end{proof}


\begin{remark} \label{rk:UnitaryChange}
Note that if a subspace $V$ is $(e_i)$-rigid invariant, then using the notation of Lemma~\ref{lem:two}, we have $MV = V$. Consequently,
\(
NV = P^* P MV = P^* P V.
\)
In particular, if $P$ is unitary, then $V$ is $(e_i)$-rigid invariant if and only if it is $(f_i)$-rigid invariant. 
\end{remark}

Finally, observe that the bilinear form $\kappa$ is symmetric if and only if the matrix $M$ is symmetric, i.e., $M = M^T$.
The following lemma is classical; see, for instance, \cite[Theorem~3.12]{Con12}.

\begin{lemma} \label{lem:EquivND}
A subspace \(V \subseteq U\) is nondegenerate if and only if
\(
U = V \oplus V^{\perp_\kappa}.
\)
\end{lemma}
Recall that a direct sum \(A = B \oplus C\) just means that \(B \cap C = \{0\}\) and \(B + C = A\).

\begin{lemma} \label{lem:FixedBasis}
A subspace \(V \subseteq U\) is nondegenerate if and only if it is \((e_i)\)-rigid invariant for some basis \((e_i)\) of \(U\).
\end{lemma}

\begin{proof}
Assume first that \(V \subseteq U\) is nondegenerate. Then
\[
V^{\perp_{\kappa|_{V \times V}}} = V^{\perp_\kappa} \cap V = \{u \in V \mid \kappa(u, v) = 0 \text{ for all } v \in V\} = \{0\}.
\]
By Lemma~\ref{lem:EquivND}, we have \(U = V \oplus V^{\perp_\kappa}\), so we can choose a basis \((f_i)\) of \(U\) such that the first vectors span \(V\) and the remaining ones span \(V^{\perp_\kappa}\). The matrix \(M\) defined with respect to this basis satisfies \(M(V) = V\), since
\[
\langle V^{\perp_\kappa}, M V \rangle = \kappa(V^{\perp_\kappa}, V) = \{0\}.
\]
Hence, \(V\) is \((f_i)\)-rigid invariant.

Conversely, suppose that \(V\) is \((e_i)\)-rigid invariant for some basis \((e_i)\) of \(U\), meaning \(M(V) = V\). Thus,
\[
V^{\perp_{\kappa|_{V \times V}}} = \{u \in V \mid \langle u, Mv \rangle = 0 \text{ for all } v \in V\} = \{u \in V \mid \langle u, v \rangle = 0 \text{ for all } v \in V\} = \{0\}.
\]
The final equality holds because the Hermitian form is anisotropic. Therefore, \(V\) is nondegenerate.
\end{proof}

It is important to emphasize that a nondegenerate subspace need not be rigid invariant—that is, \((e_i)\)-\emph{rigid invariant} for \emph{every} basis~\((e_i)\). For instance, consider:
\[
U = \mathbb{C}^2, \quad M = \begin{pmatrix} 1 & 1 \\ 0 & 1 \end{pmatrix}, \quad V = \mathbb{C}v \quad \text{with} \quad v = \begin{pmatrix} 1 \\ 1 \end{pmatrix}.
\]
The matrix \(M\) is invertible, so the bilinear form \(\kappa(u, v) = \langle u, Mv \rangle\) is nondegenerate. Its restriction to \(V\) is also nondegenerate, since
\[
\kappa(v, v) = \langle v, Mv \rangle = \left\langle \begin{pmatrix} 1 \\ 1 \end{pmatrix}, \begin{pmatrix} 2 \\ 1 \end{pmatrix} \right\rangle = 3 \neq 0.
\]
However, \(Mv = \begin{pmatrix} 2 \\ 1 \end{pmatrix} \notin V\), so \(V\) is not rigid invariant with respect to the standard basis \(\left\{ \begin{pmatrix} 1 \\ 0 \end{pmatrix}, \begin{pmatrix} 0 \\ 1 \end{pmatrix} \right\}\). On the other hand, \(V\) is rigid invariant with respect to the basis \(\left\{ \begin{pmatrix} 1 \\ 1 \end{pmatrix}, \begin{pmatrix} 1 \\ -2 \end{pmatrix} \right\}\).

\begin{lemma} \label{lem:InterRigidVec}
If the subspaces $V, W \subseteq U$ are $(e_i)$-rigid invariant, then their intersection $V \cap W$ and their sum $V + W$ are also $(e_i)$-rigid invariant.
\end{lemma}

\begin{proof}
Let $M$ be the matrix representing the bilinear form $\kappa$ with respect to the basis $(e_i)$. By hypothesis, $M(V) = V$ and $M(W) = W$. It follows that 
\[
M(V \cap W) \subseteq M(V) \cap M(W) = V \cap W.
\]
Because $M$ is an automorphism, the containment is in fact an equality, ensuring that $V \cap W$ is $(e_i)$-rigid invariant. A similar argument applies to the sum $V + W$.
\end{proof}

The following example, due to Dave Benson~\cite{BenMO2}, exhibits two nondegenerate subspaces whose intersection is degenerate. In particular, by Lemmas~\ref{lem:FixedBasis} and~\ref{lem:InterRigidVec}, these two subspaces cannot both be rigid invariant with respect to the same basis. Moreover, although both are Frobenius subalgebras, their intersection is not a Frobenius subalgebra, since it is degenerate.

\begin{example} \label{ex:Ben02} 
Consider the commutative \( \mathbb{C} \)-algebra
\(
U = \mathbb{C}[x, y, z]/(x^2, y^2, xz, yz, xy - z^2).
\)
This is a five-dimensional \( \mathbb{C} \)-algebra with basis \( \mathcal{B} = \{1, x, y, z, u\} \), where \( u = xy = z^2 \). It is a Frobenius algebra with respect to the nondegenerate associative bilinear form \( \kappa(a, b) = \tau_u(ab) \), where \( \tau_u(\cdot) \) denotes the coefficient of \( u \).

Let \( V \subseteq U \) be the unital subalgebra generated by \( x \) and \( y \), which has basis \( \{1, x, y, u\} \). Similarly, let \( W \subseteq U \) be the unital subalgebra generated by \( x \) and \( y + z \), which has basis \( \{1, x, y + z, u\} \). The intersection \( V \cap W \) has basis \( \{1, x, u\} \), with the relations \( x^2 = u^2 = xu = 0 \).
Although \( V \) and \( W \) are nondegenerate, their intersection \( V \cap W \) is degenerate. In fact, \( V \) and \( W \) are selfdual, but only \( V \) is \( \mathcal{B} \)-rigid invariant. The corresponding (conjugate-linear) automorphism \( M \) of \( U \), defined by
\[
M(1) = u, \quad M(x) = y, \quad M(y) = x, \quad M(z) = z, \quad M(u) = 1,
\]
satisfies \( MV = V \), while \( MW = \langle y, x + z \rangle \neq \langle x, y + z \rangle = W \).
\end{example}

%

The following example demonstrates that even if, for every basis \(\mathcal{B}\), the number of Frobenius \(\mathcal{B}\)-subalgebras (Notation \ref{not:B-sub}) is finite, the total number of Frobenius subalgebras can still be infinite.

\begin{example} \label{ex:NonFinite}
Let \( A = \mathbb{C}[x]/(x^4) \) be endowed with the Frobenius structure defined by the nondegenerate associative bilinear form
\(
\kappa(a,b) = \tau_{x^3}(ab),
\)
where \(\tau_{x^3}\) extracts the coefficient of \(x^3\).
The proper nontrivial Frobenius subalgebras of \(A\) are parametrized by \(\lambda \in \mathbb{C}\) and explicitly given by
\(
B_\lambda = \langle 1, \lambda x^2 + x^3 \rangle.
\)
With respect to the standard basis \(\mathcal{B}=\{1, x, x^2, x^3\}\), then $B_\lambda$ is \(\mathcal{B}\)-rigid invariant if and only if $\lambda=0$. For any other basis, this condition imposes a polynomial constraint on \(\lambda\) of degree at least one. Hence, only finitely many values of \(\lambda\) are possible.
%
\end{example}

%

Finally, an example of nondegenerate subspaces with nondegenerate intersection and degenerate sum.


\begin{example} \label{ex:NonDegSum}
Let $C = \mathbb{C}^4$ with basis $\{e_1, e_2, e_3, e_4\}$, and let $\kappa$ be the bilinear form defined by the matrix
\[
M = \begin{pmatrix}
1 & 0 & 0 & 0 \\
0 & 1 & 1 & 1 \\
0 & 1 & 1 & 0 \\
0 & 1 & 0 & 0
\end{pmatrix}
\]
Then $\kappa$ is nondegenerate since $\det M = -1 \neq 0$. Define the 2-dimensional subspaces $A = \langle e_1, e_2 \rangle$ and $B = \langle e_1, e_3 \rangle$.
\begin{itemize}
    \item $A$ and $B$ are nondegenerate since the restriction of $\kappa$ to both $A$ and $B$ has the identity matrix
    \(
    \begin{pmatrix}
    1 & 0 \\
    0 & 1
    \end{pmatrix}
    \),
    \item $A \cap B = \langle e_1 \rangle$ is nondegenerate since $\kappa(e_1, e_1) = 1$,
    \item $A + B = \langle e_1, e_2, e_3 \rangle$ is degenerate since the matrix below for the restriction of $\kappa$ to $A + B$ has determinant $0$.
    \[
    \begin{pmatrix}
    1 & 0 & 0 \\
    0 & 1 & 1 \\
    0 & 1 & 1
    \end{pmatrix}.
    \]
\end{itemize}
\noindent
Note that $A$, $B$, and $A \cap B$ are Frobenius algebras, whereas $A + B$ is not.
\end{example}


\subsection{Semisimple tensor categories} \label{sub:RigidSemisimple}

In a tensor category~$\mathcal{C}$, the notion of a basis~$\mathcal{B}$ for a vector space~$U$ naturally generalizes to a family of bases $\mathcal{B} = (\mathcal{B}_X)_{X \in \mathcal{O}(\mathcal{C})}$, where each~$\mathcal{B}_X$ is a basis of the multiplicity space~$\Hom_{\mathcal{C}}(X, C)$, for a fixed semisimple object~$C$ and each simple object~$X \in \mathcal{O}(\mathcal{C})$.

\begin{notation} \label{not:basis} 
We refer to such a family~$\mathcal{B}$ as a \emph{basis} of the semisimple object~$C$. In certain contexts, it may be more appropriate to consider the spaces $\Hom_{\mathcal{C}}(C, X)$ rather than $\Hom_{\mathcal{C}}(X, C)$. In such cases, we continue to write $\mathcal{B}$ when no confusion arises; otherwise, $\mathcal{B}'$.
\end{notation}


Let $C$ be an object in a $\mathbb{C}$-linear abelian rigid monoidal category $\mathcal{C}$, and let $(C^*, \ev_C, \coev_C)$ denote a (left) dual of $C$. Let $X \in \mathcal{O}(\mathcal{C})$ be a simple object. We define a bilinear form
$$
\kappa_{C,X} : \Hom_{\mathcal{C}}(X^*, C^*) \times \Hom_{\mathcal{C}}(X, C) \to \mathbb{C}
$$
by the relation
$$
\ev_C \circ (f \otimes g) = \kappa_{C,X}(f, g) \ev_X.
$$
Note that $\ev_C \circ (f \otimes g) \in \Hom_{\mathcal{C}}(X^* \otimes X, \one) = \mathbb{C}  \ev_X$, since
$$
\Hom_{\mathcal{C}}(X^* \otimes X, \one) \cong \End_{\mathcal C}(X) \cong \mathbb{C},
$$
by the natural adjunction isomorphism (see \cite[Proposition 2.10.8]{EGNO15}) and Schur's lemma, since $X$ is simple and $\mathbb{C}$ is algebraically closed.

\begin{lemma} \label{lem:BilToEv}
If $C$ is semisimple of finite length, then the collections $(\kappa_{C,X})_{X \in \mathcal{O}(\mathcal{C})}$ and $(\ev_X)_{X \in \mathcal{O}(\mathcal{C})}$ uniquely determine $\ev_C$.
\end{lemma}
\begin{proof}
Since $C$ is semisimple, there is a basis $(f_{X,i})$ of $\Hom_{\mathcal{C}}(X, C)$ and a basis $(g_{X,i})$ of $\Hom_{\mathcal{C}}(C, X)$ such that $$ g_{X,i} \circ f_{X,j} = \delta_{i,j} \id_X \text{ and }   \id_C = \sum_{X \in \mathcal{O}(\mathcal{C})} \sum_i f_{X,i} \circ g_{X,i}. $$
Thus, 
\begin{align*}
\ev_C =  \ev_C \circ (\id_{C^*} \otimes \id_C) &= \sum_{X,Y \in \mathcal{O}(\mathcal{C})} \sum_{i,j} \ev_C \circ (g_{X,i}^* \otimes f_{Y,j}) \circ (f_{X,i}^* \otimes g_{Y,j}) \\
&=    \sum_{X \in \mathcal{O}(\mathcal{C})} \sum_{i,j} \kappa_{C,X}(g_{X,i}^*, f_{X,j}) \ev_X \circ (f_{X,i}^* \otimes g_{X,j}). \qedhere
\end{align*}
\end{proof}
%
\begin{lemma} \label{lem:BilNonDeg}
If $C$ is semisimple of finite length, then each bilinear form $\kappa_{C,X}$ is nondegenerate.
\end{lemma}
\begin{proof}
Let $f \in \Hom_{\mathcal{C}}(X^*, C^*)$ be such that $\kappa_{C,X}(f, g) = 0$ for all $g \in \Hom_{\mathcal{C}}(X, C)$. 
We deduce that 
\begin{align*}
\ev_C \circ (f \otimes \id_C)
&= \ev_C \circ (f \otimes \sum_{X \in \mathcal{O}(\mathcal{C})} \sum_i f_{X,i} \circ g_{X,i} ) \\
&= \sum_{X \in \mathcal{O}(\mathcal{C})} \sum_i \kappa_{C,X}(f, f_{X,i}) \, \ev_X \circ (\id_{X^*} \otimes g_{X,i}) = 0.
\end{align*}
Finally, by applying the zigzag relation, $$f = \left((\ev_C \otimes \id_{C^*}) \circ (\id_{C^*} \otimes \coev_C)\right) \circ f =  \left( (\ev_C \circ (f \otimes \id_C)) \otimes \id_{C^*} \right) \circ (\id_{X^*} \otimes \coev_C) = 0. \qedhere $$
\end{proof}

The bilinear form $\kappa_{C,X}$ is completely determined by the conjugate-linear map $$M_{C,X}: \Hom_{\mathcal{C}}(X^*, C^*) \to \Hom_{\mathcal{C}}(X, C)$$ defined by $$\langle M_{C,X} f,g \rangle = \kappa_{C,X}(f,g),$$ where $\langle \cdot, \cdot \rangle$ is the Hermitian form, conjugate-linear in the \emph{first} argument, defined by $\langle g_i,g_j \rangle = \delta_{i,j}$, with respect to some basis $\mathcal{B}_X = (g_i)$ of $\Hom_{\mathcal{C}}(X, C)$. 

Similarly, consider the bilinear form $$\kappa'_{C,X}: \Hom_{\mathcal{C}}(C,X) \times \Hom_{\mathcal{C}}(C^*,X^*) \to \mathbb{C}$$ defined by $$(f \otimes g) \circ \coev_C = \kappa'_{C,X}(f,g)\coev_X,$$
and determined by the conjugate-linear map $$M'_{C,X}: \Hom_{\mathcal{C}}(C^*, X^*) \to  \Hom_{\mathcal{C}}(C, X),$$
defined by $$\langle f,M'_{C,X}g \rangle = \kappa'_{C,X}(f,g),$$ where $\langle \cdot, \cdot \rangle$ is the Hermitian form, conjugate-linear in the \emph{second} argument, defined by $\langle f_i,f_j \rangle = \delta_{i,j}$, with respect to some basis $\mathcal{B}'_X = (f_i)$ of $\Hom_{\mathcal{C}}(C, X)$. If $C$ is semisimple of finite length, then—similarly to Lemmas \ref{lem:BilToEv} and \ref{lem:BilNonDeg}—the families $(\kappa'_{C,X})_{X \in \mathcal{O}(\mathcal{C})}$ and $(\coev_X)_{X \in \mathcal{O}(\mathcal{C})}$ uniquely determine $\coev_C$, and each bilinear form $\kappa'_{C,X}$ is nondegenerate. As a consequence, the matrices $M_{C,X}$ and $M'_{C,X}$ are invertible. More precisely, consider the isomorphisms
\[
\phi: \Hom_{\mathcal{C}}(X, C) \to \Hom_{\mathcal{C}}(C^*, X^*) \quad \text{and} \quad \phi': \Hom_{\mathcal{C}}(C, X) \to \Hom_{\mathcal{C}}(X^*, C^*)
\]
induced by duality. Then we have the following:

\begin{lemma} \label{lem:ZigZag}
If $C$ is semisimple of finite length, then
\[
M_{C,X} \circ \phi' \circ M'_{C,X} \circ \phi = \id_{\Hom_{\mathcal{C}}(X, C)} \quad \text{and} \quad M'_{C,X} \circ \phi \circ M_{C,X} \circ \phi' = \id_{\Hom_{\mathcal{C}}(C, X)}.
\]
\end{lemma}

\begin{proof}
Expand $\ev_C$ and $\coev_C$ in the zigzag identity
\[
\id_{C^*} =  (\coev_C \otimes \id_{C^*}) \circ (\id_{C^*} \otimes \ev_C)
\]
using the bases introduced in the proof of Lemma~\ref{lem:BilToEv}. This yields the relation
\[
\sum_j \kappa_{C,X}(g_{X,i}^*, f_{X,j}) \, \kappa'_{C,X}(g_{X,j}, f^*_{X,l}) = \delta_{i,l},
\]
which is equivalent to the first identity. The second identity follows by a similar argument.
\end{proof}

\begin{lemma} \label{lem:SplitNonDeg}
Let \( A \) be a subobject of a semisimple object \( C \) of finite length. Its dual is ambient (as in Definition~\ref{def:InducedDual}) if and only if the restriction of the bilinear form \( \kappa'_{C,X} \) to $\Hom_{\mathcal{C}}(A, X) \times \Hom_{\mathcal{C}}(A^*, X^*)$ is \( \kappa'_{A,X} \), and is in particular nondegenerate.
\end{lemma}

\begin{proof}
Since \( C \) is semisimple, every subobject \( A \subseteq C \) splits, i.e., \( A \) is a direct summand of \( C \). Let \( i: A \to C \) and \( p: C \to A \) be morphisms such that \( p \circ i = \id_A \), so \( i \) is a split monomorphism and \( p \) a split epimorphism.

Assume that \( A \) has its dual ambient. By definition, $\coev_A = (p \otimes i^*)  \circ \coev_C.$
Since \( C \) is semisimple, we can choose bases \( (e_{X,a}) \) for \( \Hom_{\mathcal{C}}(X, C) \), \( (f_{X,b}) \) for \( \Hom_{\mathcal{C}}(A, X) \), \( (g_{X,c}) \) for \( \Hom_{\mathcal{C}}(X, A) \), \( (h_{X,d}) \) for \( \Hom_{\mathcal{C}}(C, X) \), such that
$
g_{X,c} = p \circ e_{X,c}$, $f_{X,b} = h_{X,b} \circ i,
$
and
\[
i = \sum_{X \in \mathcal{O}(\mathcal{C})} \sum_b e_{X,b} \circ f_{X,b}, \qquad 
p = \sum_{X \in \mathcal{O}(\mathcal{C})} \sum_c g_{X,c} \circ h_{X,c}.
\]

As in the proof of Lemma~\ref{lem:BilToEv}, and using the above identities, we compute:
\begin{align*}
\coev_A =  (p \otimes i^*) \circ  \coev_C 
&= \sum_{X \in \mathcal{O}(\mathcal{C})} \sum_{b,c} \kappa'_{C,X}(h_{X,c}, e_{X,b}^*)  \ev_X \circ (g_{X,c} \otimes f_{X,b}^*) \\
&= \sum_{X \in \mathcal{O}(\mathcal{C})} \sum_{b,c} \kappa'_{A,X}(f_{X,c}, g_{X,b}^*)  \ev_X \circ (g_{X,c} \otimes f_{X,b}^*).
\end{align*}

It follows that $
\kappa'_{A,X}(f_{X,c}, g_{X,b}^*) = \kappa'_{C,X}(h_{X,c}, e_{X,b}^*),
$
which shows that \( \kappa'_{A,X} \) is precisely the retriction of \( \kappa'_{C,X} \) to 
$
\Hom_{\mathcal{C}}(A, X) \times \Hom_{\mathcal{C}}(A^*, X^*).
$
Conversely, if $\kappa'_{A,X}$ is the restriction of $\kappa'_{C,X}$, then $\coev_A = (p \otimes i^*) \circ \coev_C$ by the formula above. Applying Lemma \ref{lem:ZigZag}, we can deduce that $\ev_A = \ev_C \circ (p^* \otimes i)$.
\end{proof}




\begin{definition} \label{def:RigidInvariance}
A (split) subobject \( (A, i, p) \) of a semisimple object~\( C \) of finite length in a \( \mathbb{C} \)-linear abelian rigid monoidal category~\( \mathcal{C} \) is said to be $\mathcal{B}'$-\emph{rigid invariant}, for some basis~$\mathcal{B}'$ of~$C$ (see Notation~\ref{not:basis}), if the following identity holds:
\[
M'_{C,X} \circ \left( \Hom_{\mathcal{C}}(A^*, X^*) \circ i^* \right)  = \Hom_{\mathcal{C}}(A, X) \circ p,
\]
where~$M'_{C,X}$ is defined as above with respect to the basis~$\mathcal{B}'_X$. We say that the subobject~$A$ is \emph{rigid invariant} if it is $\mathcal{B}'$-rigid invariant for \emph{every} basis~$\mathcal{B}'$ of~$C$.
\end{definition}

\begin{remark} \label{rk:WLOG}
Following Definition \ref{def:RigidInvariance}, we can assume without loss of generality that
\[
A = \bigoplus_{X \in \mathcal{O}(\mathcal{C})} \left( \Hom_{\mathcal{C}}(A, X) \circ p \right) \otimes X.
\]
Note that the triple \( (A^*, p^*, i^*) \) forms a split subobject of \( C^* \). If \( A \) is \( \mathcal{B} \)-rigid invariant, define \( N_{C,X} := (M'_{C,X})^{-1} \). Then we have
\[
A^* = \bigoplus_{X \in \mathcal{O}(\mathcal{C})} \left( \Hom_{\mathcal{C}}(A^*, X^*) \circ i^* \right) \otimes X^* = \bigoplus_{X \in \mathcal{O}(\mathcal{C})} \left( N_{C,X} \left( \Hom_{\mathcal{C}}(A, X) \circ p \right) \right) \otimes X^*.
\]
In other words, the transformation
\[
N_C: \bigoplus_{X \in \mathcal{O}(\mathcal{C})} V_X \otimes X \longmapsto \bigoplus_{X \in \mathcal{O}(\mathcal{C})} N_{C,X}(V_X) \otimes X^*
\]
explicitly realizes the ambient duality (Definition \ref{def:InducedDual}) on \( \mathcal{B} \)-rigid invariant subobjects of \( C \).
\end{remark}


\begin{lemma} \label{lem:FixedBasisCat}
Let \( \mathcal{C} \) be a \( \mathbb{C} \)-linear abelian rigid monoidal category. Let \( C \) in $\mathcal{C}$ be a semisimple object of finite length, and let \( (A,i,p)\) be a (split) subobject of $C$. Then \( A \) has its dual ambient (Definition~\ref{def:InducedDual}) if and only if it is \( \mathcal{B} \)-rigid invariant for some basis \( \mathcal{B} \) of \( C \).
\end{lemma}

\begin{proof}
Assume that \( A \) has its dual ambient. By Lemma~\ref{lem:SplitNonDeg}, the pairing \( \kappa'_{A,X} \) is nondegenerate. Then, as in Lemma~\ref{lem:FixedBasis}, there is a basis \( \mathcal{B} \) of \( C \) with respect to which
\[
\left\langle \left( \Hom_{\mathcal{C}}(A, X) \circ p \right)^{\perp_{\kappa'_{C,X}}},\; M'_{C,X} \left( \Hom_{\mathcal{C}}(A^*, X^*) \circ i^* \right) \right\rangle 
= \kappa'_{C,X} \left( \left( \Hom_{\mathcal{C}}(A, X) \circ p \right)^{\perp_{\kappa'_{C,X}}},\; \Hom_{\mathcal{C}}(A^*, X^*) \circ i^* \right) 
= \{0\}.
\]
Hence, we conclude that
\[
M'_{C,X} \left( \Hom_{\mathcal{C}}(A^*, X^*) \circ i^* \right) = \Hom_{\mathcal{C}}(A, X) \circ p, \quad \text{for all } X \in \mathcal{O}(\mathcal{C}),
\]
which means that \( A \) is \( \mathcal{B} \)-rigid invariant. Conversely, if \( A \) is \( \mathcal{B} \)-rigid invariant for some basis \( \mathcal{B} \) of \( C \), then the proof also proceeds as in Lemma~\ref{lem:FixedBasis} to show that \( A \) has its dual ambient.
\end{proof}


It is crucial to emphasize that in Lemma~\ref{lem:FixedBasisCat}, the phrase ``for some'' cannot be strengthened to ``for every'', as demonstrated in~$\VVec$ following Lemma~\ref{lem:FixedBasis} and in Example~\ref{ex:Ben02}. See also Remark~\ref{rk:UnitaryChange}.

In the semisimple case, the intersection and sum of subobjects can be defined in terms of the intersection and sum of their multiplicity spaces. In general see Remarks \ref{rk:intersection} and \ref{rk:union}. The primary motivation for introducing rigidity invariance is that it ensures the preservation of ambient selfduality (Definition~\ref{def:AmbientlySelfdual}) under intersection and sum. Recall the notion of coherent pair of subobjects from Definition \ref{def:CoherentPair}.


\begin{theorem} \label{thm:CoherentPairB}
Let \( A \) and \( B \) be Frobenius $\mathcal{B}$-subalgebra of \( C \) semisimple of finite length as an object in a $\mathbb{C}$-linear abelian rigid monoidal category. Then $(A,B)$ forms a coherent pair.
\end{theorem}
\begin{proof}
Since every Frobenius subalgebra is ambiently selfdual (Lemma~\ref{lem:FrobSubAmbiently}), the result follows from Lemma \ref{lem:RigidInvInterSum}.
\end{proof}

\begin{lemma} \label{lem:RigidInvInterSum}
Let $A$ and $B$ be $\mathcal{B}$-rigid invariant subobjects of $C$. Then both $A \cap B$ and $A + B$ are $\mathcal{B}$-rigid invariant. Moreover, if $A$, $B$, and $C$ are ambiently selfdual, so are $A \cap B$ and $A + B$, and so $(A,B)$ forms a coherent pair.
\end{lemma}

\begin{proof}
Since \( C \) is semisimple, the subobjects \( A \), \( B \), and \( A \cap B \) all are split. Let \( i_Y \colon Y \to C \) and \( p_Y \colon C \to Y \) denote the corresponding monomorphisms and epimorphisms for each \( Y \in \{A, B, A \cap B\} \) such that
\[
\Hom_{\mathcal{C}}(A \cap B, X) \circ p_{A \cap B} = \left( \Hom_{\mathcal{C}}(A, X) \circ p_A \right) \cap \left( \Hom_{\mathcal{C}}(B, X) \circ p_B \right),
\]
\[
\Hom_{\mathcal{C}}((A \cap B)^*, X^*) \circ i_{A \cap B}^* = \left( \Hom_{\mathcal{C}}(A^*, X^*) \circ i_A^* \right) \cap \left( \Hom_{\mathcal{C}}(B^*, X^*) \circ i_B^* \right).
\]

As in the proof of Lemma~\ref{lem:InterRigidVec}, we have
\[
M'_{C,X}\left( \left( \Hom_{\mathcal{C}}(A^*, X^*) \circ i_A^* \right) \cap \left( \Hom_{\mathcal{C}}(B^*, X^*) \circ i_B^* \right) \right) = \left( \Hom_{\mathcal{C}}(A, X) \circ p_A \right) \cap \left( \Hom_{\mathcal{C}}(B, X) \circ p_B \right),
\]
so it follows that
\[
M'_{C,X}\left( \Hom_{\mathcal{C}}((A \cap B)^*, X^*) \circ i_{A \cap B}^* \right) = \Hom_{\mathcal{C}}(A \cap B, X) \circ p_{A \cap B},
\]
which shows that \( A \cap B \) is $\mathcal{B}$-rigid invariant.

Now assume that \( A \), \( B \), and \( C \) are ambiently selfdual (Definition \ref{def:AmbientlySelfdual}). Then, for each \( Y \in \{A, B\} \), we can assume \( i_Y^* = p_Y \) and \( p_Y^* = i_Y \). By Remark~\ref{rk:WLOG}—where \( N_{C,X} := (M'_{C,X})^{-1} \)—together with some equalities written above,
\begin{align*}
(A \cap B)^* = N_C(A \cap B)
&= \bigoplus_{X \in \mathcal{O}(\mathcal{C})} N_{C,X}(\Hom_{\mathcal{C}}(A \cap B,X) \circ p_{A \cap B}) \otimes X^* \\
&= \bigoplus_{X \in \mathcal{O}(\mathcal{C})} (\Hom_{\mathcal{C}}((A \cap B)^*, X^*) \circ i_{A \cap B}^*) \otimes X^* \\
&= \bigoplus_{X \in \mathcal{O}(\mathcal{C})} \left( \Hom_{\mathcal{C}}(A^*, X^*) \circ i_A^* \right) \cap \left( \Hom_{\mathcal{C}}(B^*, X^*) \circ i_B^* \right) \otimes X^* \\
&= \bigoplus_{X \in \mathcal{O}(\mathcal{C})} \left( \Hom_{\mathcal{C}}(A, X^*) \circ p_A \right) \cap \left( \Hom_{\mathcal{C}}(B, X^*) \circ p_B \right) \otimes X^* \\
&= \bigoplus_{X \in \mathcal{O}(\mathcal{C})} \left( \Hom_{\mathcal{C}}(A, X) \circ p_A \right) \cap \left( \Hom_{\mathcal{C}}(B, X) \circ p_B \right) \otimes X \\
&= \bigoplus_{X \in \mathcal{O}(\mathcal{C})} \left( \Hom_{\mathcal{C}}(A \cap B, X) \circ p_{A \cap B} \right) \otimes X = A \cap B.
\end{align*}
For the sum, the proof is identical to that for the intersection, with $\cap$ simply replaced by $+$.
\end{proof}

It is important to stress that Lemma~\ref{lem:RigidInvInterSum} fails if $\mathcal{B}$-rigid invariance is replaced by the weaker requirement of having ambient dual, as illustrated in~$\VVec$ by Example~\ref{ex:Ben02}. Indeed, Lemma~\ref{lem:FixedBasisCat} shows that if $A$ and $B$ have their dual ambient, then $A$ is $\mathcal{B}$-rigid invariant and $B$ is $\mathcal{B}'$-rigid invariant for suitable bases $\mathcal{B}$ and $\mathcal{B}'$, which in general do not coincide (see also Remark~\ref{rk:UnitaryChange}). Similarly, without $\mathcal{B}$-rigidity invariance, one may ask whether the ambient selfduality of $A$, $B$, and $A \cap B$ implies that of $A+B$; however, Example~\ref{ex:NonDegSum} shows that this is not the case.





Alternatively, Dave Benson~\cite{BenMO1} exhibits examples of non-semisimple objects of finite length possessing two selfdual split subobjects whose intersection is not selfdual.

\subsection{Non-semisimple challenge} \label{sub:RigidNonSS}
The \emph{Krull--Schmidt theorem} states that every finite-length object~\( C \) in an abelian category admits a decomposition into a direct sum of indecomposable objects, uniquely determined up to isomorphism and permutation. Let~\( \mathcal{I}(C) \) denote a set of representatives for the isomorphism classes of indecomposable direct summands of~\( C \).

Although this is not required for the present paper, it is natural to consider how Lemmas~\ref{lem:BilToEv}--\ref{lem:RigidInvInterSum} might extend beyond the assumption that~\( C \) is semisimple. The main difficulty is that, in general, non-isomorphic indecomposable subobjects~\( X \) and~\( X' \) may intersect nontrivially or fail to be \emph{Hom-orthogonal}, i.e., \( \Hom_{\mathcal{C}}(X, X') \neq 0 \). Consequently, the spaces \( \Hom_{\mathcal{C}}(X, C) \) or \( \Hom_{\mathcal{C}}(C, X) \) may no longer serve as multiplicity spaces and might need to be replaced by appropriate subspaces. We should also assume that~\( A \) is a split subobject, replace the condition \( X \in \mathcal{O}(C) \) with \( X \in \mathcal{I}(C) \), and substitute the base field~\( \mathbb{C} \) with the endomorphism ring~\( \End_{\mathcal{C}}(X) \) in the entries of~\( M_{C,X} \) and~\( M'_{C,X} \), as well as in the image of~\( \kappa_{C,X} \) and~\( \kappa'_{C,X} \), defined by:
\[
\ev_C \circ (f \otimes g) = \ev_X \circ (\id_{X^*} \otimes \kappa_{C,X}(f, g)) \quad \text{and} \quad (f' \otimes g') \circ \coev_C = (\kappa'_{C,X}(f', g') \otimes \id_{X^*}) \circ \coev_X.
\]

One alternative approach is to define rigidity invariance through semisimplification as in Definition \ref{def:RigidInvariantNSS}.

%
%

\subsection{Unitary tensor categories}  \label{sub:RigidUnitary}

Let us conclude this section by presenting a simple and sufficient condition for a subobject of an object~$C$ in a unitary tensor category (see Remark~\ref{rk:UnitaryTensor}) to be rigid invariant, that is, $\mathcal{B}$-rigid invariant for \emph{every} basis~$\mathcal{B}$ of~$C$. We then show that this condition is always satisfied by unitary Frobenius subalgebras.

\begin{remark} \label{rk:UnitaryTensor}
In this paper, a \emph{unitary tensor category} refers to a tensor category $\mathcal{C}$ equipped with a dagger functor $()^\dag$, which maps each morphism $f \colon A \to B$ to a morphism $f^\dag \colon B \to A$, for all $A, B \in \mathcal{C}$. This functor, together with compatibility conditions, endows $\mathcal{C}$ with the structure of a $C^*$-category, as defined in \cite[\S 2.1]{JP17}. In that reference, the dagger is denoted by $()^*$, so the notation differs slightly, and such categories are referred to as $C^*$-tensor categories. It is also noted there that the tensor category associated with a finite-index subfactor (see Example \ref{ex:subf}) is an instance of a unitary tensor category.
\end{remark}

\begin{proposition} \label{prop:RigidInvUnitary}
Let \( (A,i) \) be a subobject of an object \( C \) in a unitary tensor category \( \mathcal{C} \). Assume that \( p=i^{\dagger} \) is a retraction of the monomorphism \( i \), so that
\(
i^{\dagger}\circ i=\id_A,
\)
and that
\(
(i^*)^{\dagger}=(i^{\dagger})^*.
\)
Then \( A \) is rigid invariant.
\end{proposition}
\begin{proof} 
Let $\mathcal{B}$ be any basis of~$C$. We proceed using the notation introduced earlier in this section:
\begin{lemma} \label{lem:UnitFormula}
Let \( f \in \Hom_{\mathcal{C}}(X,C) \) and \( h \in \Hom_{\mathcal{C}}(C,X) \). Then
\[
M_{C,X}((f^*)^{\dagger}) = f \quad \text{and} \quad M'_{C,X}((h^*)^{\dagger}) = h.
\]
\end{lemma}

\begin{proof}
By definition, for all \( f, g \in \Hom_{\mathcal{C}}(X, C) \), we have:
\[
\ev_C \circ ((f^*)^{\dagger} \otimes g) = \kappa_{C,X}((f^*)^{\dagger}, g)\, \ev_X.
\]
Applying \( \ev_X^{\dagger} \) on the right, we obtain:
\[
\ev_C \circ ((f^*)^{\dagger} \otimes g) \circ \ev_X^{\dagger} = \kappa_{C,X}((f^*)^{\dagger}, g)\, \ev_X \circ \ev_X^{\dagger} = \kappa_{C,X}((f^*)^{\dagger}, g)\, \FPdim(X).
\]
Therefore,
\[
\kappa_{C,X}((f^*)^{\dagger}, g) = \frac{1}{\FPdim(X)}\, \ev_C \circ ((f^*)^{\dagger} \otimes g) \circ \ev_X^{\dagger}.
\]
Now consider the Hermitian form, conjugate-linear in the first argument, on \( \Hom_{\mathcal{C}}(X, C) \) defined by
\[
\langle h_1, h_2 \rangle := \frac{1}{\FPdim(X)}\, \ev_C \circ (\id_{C^*} \otimes (h_2 \circ h_1^{\dagger})) \circ \ev_C^{\dagger}.
\]
Using the zig-zag identity and the definition of \( M_{C,X} \), we compute:
\[
\langle f, g \rangle = \frac{1}{\FPdim(X)}\, \ev_C \circ (\id_{C^*} \otimes (g \circ f^{\dagger})) \circ \ev_C^{\dagger} = \kappa_{C,X}((f^*)^{\dagger}, g) = \langle M_{C,X}((f^*)^{\dagger}), g \rangle.
\]
Since this equality holds for all \( g \), it follows by nondegeneracy of the Hermitian form that
\[
M_{C,X}((f^*)^{\dagger}) = f.
\]
Finally, using Lemma~\ref{lem:ZigZag} and above identity, we have:
\[
h = M'_{C,X}(M_{C,X}(h^*)^*) = M'_{C,X}((h^*)^{\dagger}). \qedhere
\] 
\end{proof}
By Lemma~\ref{lem:UnitFormula}, together with the assumptions \( (i^*)^{\dagger} = (i^{\dagger})^* \) and \( p=i^{\dagger} \), we obtain, for every \( g \in \Hom_{\mathcal{C}}(A,X) \),
\[
M_{C,X}'\left((g^*)^{\dagger} \circ i^*\right) = (((g^*)^{\dagger} \circ i^*)^{\dagger})^* = ((i^*)^{\dagger} \circ g^*)^* =((i^{\dagger})^* \circ g^*)^* = g \circ i^\dagger = g \circ p.
\]
The result follows directly from Definition~\ref{def:RigidInvariance} and $(\Hom_{\mathcal{C}}(A,X)^*)^{\dagger} = \Hom_{\mathcal{C}}(A^*,X^*)$.
\end{proof}
\begin{definition}[e.g., Definition~2.10 in~\cite{JP20}] \label{def:UnitFrob}
A Frobenius algebra \((C, m, e, \delta, \epsilon)\) in a unitary tensor category is called \emph{unitary} if it satisfies the following conditions:
\begin{itemize}
  \item it is separable (see Definition~\ref{def:connected});
  \item the composition \(\epsilon \circ m \circ \delta \circ e = \dim(C) \, \id_{\one}\);
  \item the structure maps are adjoint to each other: 
  \[
  m^{\dagger} = \delta \quad \text{(so } \delta^{\dagger} = m), \qquad e^{\dagger} = \epsilon \quad \text{(so } \epsilon^{\dagger} = e).
  \]
\end{itemize}
\end{definition}

The Frobenius algebra associated with a finite-index inclusion of \(C^*\)-algebras is unitary; see \cite[\S 3]{CHJP22}. Recall from Lemmas \ref{lem:multdualcomult} and \ref{lem:unitdualcounit} that  
\[
m^* = \delta, \quad \delta^* = m, \quad e^* = \epsilon, \quad \epsilon^* = e.
\]
\begin{proposition} \label{prop:UnitFrobSub}
In a unitary tensor category \(\mathcal{C}\), any unitary Frobenius subalgebra of a unitary Frobenius algebra is rigid invariant.
\end{proposition}
\begin{proof}
Let $C$ be a unitary Frobenius algebra, and let $B$ be a unitary Frobenius subalgebra. From Definition \ref{def:sub}, the unital algebra monomorphism \(i: B \to C\) satisfies  $i^* \circ i = \id_{B}  \text{ and } i^{**} = i$, which implies that \(i\) is also counital.

\begin{lemma} \label{lem:EvCoevDagger}
The following identities hold for all $Y \in \{B,C\}$, $$\ev_Y^* = \coev_Y = \ev_Y^\dagger \text{ and } \coev_Y^* = \ev_Y = \coev_Y^\dagger.$$
\end{lemma}
\begin{proof}
Immediate from Lemma \ref{lem:selfdual} and the identities above.
\end{proof}

\begin{lemma} \label{lem:UnitFrobSub}
The monomorphism \(i\) satisfies  $ i^{\dagger} \circ i = \id_{B} \text{ and } (i^*)^{\dagger} = (i^{\dagger})^*$.
\end{lemma}
\begin{proof}
A unitary Frobenius subalgebra can be realized as a finite-index intermediate subfactor, see Remark \ref{rk:subf}. In this setting, $i$ corresponds to the inclusion morphism and is therefore an isometry, meaning that the first identity holds. For the second identity, since $i^{**} = i$, then $^*i = i^*$. By Lemma \ref{lem:EvCoevDagger}, it then follows that $(i^*)^\dagger = (^*i)^\dagger = (i^\dagger)^*$.
\end{proof}
The proposition then follows from Lemma \ref{lem:UnitFrobSub} and Proposition \ref{prop:RigidInvUnitary}.
\end{proof}

Be aware that in the unitary setting, a morphism—specifically, the monomorphism $i$—must be adjointable (see the dagger structure in Remark~\ref{rk:UnitaryTensor}) to exist, which may not hold automatically (see Proposition~\ref{prop:AdjointEcomp}).


\section{Lattice structure} \label{sec:lattice}
In general, the Frobenius subalgerba poset (Definition \ref{def:FrobSubPoset}) fails to be a lattice because, as discussed in \S \ref{sec:RigidInvariant}, the intersection of two Frobenius subalgebras is not necessarily a Frobenius subalgebra. Since generalizing Watatani's theorem requires a lattice structure, such a generalization can only be performed on a sublattice of the Frobenius subalgebra poset. In specific contexts---such as intermediate subfactors or left coideal subalgebras---the poset inherently possesses a lattice structure. 

The main goal of this section is to characterize the sublattices in terms of ambient selfduality (Theorem~\ref{thm:AmbientSublattice}). In the appropriate setting, we deduce that every Frobenius $\mathcal{B}$-subalgebra poset is a sublattice (Theorem~\ref{thm:Bsublattice}).


\subsection{In abelian categories}

Following the approach in \cite[\S 1.5]{Fr64}, we begin by recalling the equivalence relation and the partial order for subobjects in a category. Two monomorphisms $i_1: A_1 \rightarrow C$ and $i_2: A_2 \rightarrow C$ are deemed \emph{equivalent} if there exist morphisms $i_{1,2}: A_1 \rightarrow A_2$ and $i_{2,1}: A_2 \rightarrow A_1$ such that $i_1 = i_2 \circ i_{1,2}$ and $i_2 = i_1 \circ i_{2,1}$.

A \emph{subobject} of $C$ is an equivalence class of monomorphisms into $C$. We define the subobject represented by $j_1: B_1 \rightarrow C$ to be \emph{contained} in the subobject represented by $j_2: B_2 \rightarrow C$ if there exists a morphism $j_{1,2}: B_1 \rightarrow B_2$ such that $j_1 = j_2 \circ j_{1,2}$. This containment relation defines a partial order on subobjects. When no confusion arises, we will simply write $B_1 \subseteq B_2$.

Following \cite[\S 2.1]{Fr64}, the \emph{intersection} of two subobjects of \(C\) is their greatest lower bound in the subobject poset of \(C\), which always exists in an abelian category \cite[Theorem 2.13]{Fr64}.

\begin{remark} \label{rk:intersection}
This notion of intersection is specific to the context of an abelian category. More generally, in any category with finite limits, an intersection exists and can be defined as the pullback of monomorphisms $i_A$ and $i_B$:
%
\[
 			\begin{tikzcd}
 			D \arrow{r}{j_A}  \arrow[swap]{d}{j_B} & A \arrow{d}{i_A} \\
 			B \arrow[swap]{r}{i_B} & C 
 			\end{tikzcd}
\]
A pullback is generically denoted as $A \times_C B$, but the pullback of monomorphisms can simply be denoted as $A \cap B$, with the monomorphism $i_A \circ j_A = i_B \circ j_B$ represented as $i_{A \cap B}$ (noting that a pullback of monomorphisms is also a monomorphism). The natural question is whether the Frobenius structure is preserved under intersection. However, as we observed in \S \ref{sec:RigidInvariant}, this is already not the case in $\VVec$.
\end{remark}

\begin{lemma} \label{lem:2pullback}
Let $(A,i_A)$ and $(B,i_B)$ be two subobject representatives of an object $C$ in an abelian category. Consider another object $C'$ with monomorphisms $i: C \to C'$ and $i_Y'= i \circ i_Y$ for all $Y \in \{ A, B \}$. Then, $(A,i_A')$ and $(B,i_B')$ are subobject representatives of $C'$, and the pullback of $(i_A,i_B)$ can be taken to be the same as that of $(i_A',i_B')$.

\[
 			\begin{tikzcd}
 		A \cap B \arrow{r}{j_A}  \arrow[swap]{d}{j_B} & A \arrow{d}{i_A} \arrow[ddr, bend left, "i_A '" ] & \\
 		B \arrow[swap]{r}{i_B} \arrow[swap,drr,bend right," i_B ' "] & C \arrow[dr, "i" ] & \\
 		& & C' 
 			\end{tikzcd}
 			\]
\end{lemma}

\begin{proof}
This follows directly from \cite{Fr64}, which states that the pullback of monomorphisms corresponds to the intersection defined as the greatest lower bound in the subobject poset of $C$ or $C'$. The relation $Y \subseteq C$ as subobjects of $C'$ holds as $i_Y' = i \circ i_Y$ for all $Y \in \{ A, B \}$.
\end{proof}





Dually, again following \cite{Fr64}, the \emph{union} \(A \cup B\) of two subobjects \(A\) and \(B\) of \(C\) is defined as their join (least upper bound) in the subobject poset of \(C\), which always exists in an abelian category. With these operations of intersection and union, the subobjects of \(C\) form a lattice.

\begin{remark} \label{rk:union}
This notion of union $A \cup B$ is specific to the context of an abelian category, where it is conventionally denoted by $A + B$. More generally, in any category with finite limits and finite colimits, a union exists and can be defined as the pushout of the monomorphisms $j_A$ and $j_B$ (themselves  pullbacks of monomorphisms $i_A$ and $i_B$):
\[
\begin{tikzcd}
A \cap B \arrow[r,"j_A"] \arrow[swap,d,"j_B"] & A \arrow[d,"k_A"] \arrow[ddr,bend left,"i_A"] & \\
B \arrow[r,"k_B"] \arrow[swap,drr,bend right,"i_B"] & A+B \arrow[swap,dr,dashed,"i_{A+B}"] & \\
& & C
\end{tikzcd}
\]
By the universal property of the pushout, there exists a unique monomorphism $i_{A+B} \colon A + B \to C$ such that $i_{A+B} \circ k_Y = i_Y$ for each $Y \in \{A, B\}$. 
\end{remark}

\begin{lemma} \label{lem:Inter+Sum}
Let $A$ and $B$ be subobjects of an object $C$ in an abelian category. If $A \cap B$ is a split subobject of both $A$ and $B$, then $A \oplus B \simeq (A \cap B) \oplus (A + B)$.
\end{lemma}

\begin{proof}
The characterization of the intersection and sum via a pullback-pushout square (Remarks \ref{rk:intersection} and \ref{rk:union}) is equivalent to the existence of the following short exact sequence:
\[
0 \to A \cap B \xrightarrow{\binom{j_A}{-j_B}} A \oplus B \xrightarrow{(k_A, k_B)} A + B \to 0.
\]
By assumption, $j_A$ and $j_B$ are split monomorphisms. It follows that $\binom{j_A}{-j_B}$ is also a split monomorphism. By the Splitting Lemma, this short exact sequence splits, which implies the isomorphism $A \oplus B \simeq (A \cap B) \oplus (A + B)$.
\end{proof}

\subsection{In abelian monoidal categories}

Roughly speaking, the following proposition states that the intersection of subalgebras is a subalgebra. The proof in a general abelian monoidal category is partly attributed to Brian Shin \cite{ShinMO}.

\begin{proposition} \label{prop:InterAlg}
Let $L \in \{A,B,C\}$, and let $(L,m_L,e_L)$ be unital algebras in an abelian monoidal category, along with unital algebra monomorphisms $i_L: L \to C$ for $L \in \{A,B\}$. Then there exists a unique unital algebra structure on the pullback $A \cap B$ such that the monomorphisms $j_L: A \cap B \to L$ are unital algebra morphisms.
\end{proposition}

\begin{proof}
For $L \in \{A,B\}$, consider the morphisms 
$$
f_L := i_L \circ m_L \circ (j_L \otimes j_L): (A \cap B) \otimes (A \cap B) \to C.
$$
Since $i_L$ is an algebra morphism, we have $i_L \circ m_L = m_C \circ (i_L \otimes i_L)$. Therefore, 
$$
f_L = m_C \circ (i_L \otimes i_L) \circ (j_L \otimes j_L) = m_C \circ ((i_L \circ j_L) \otimes (i_L \circ j_L)).
$$
However, $i_A \circ j_A = i_B \circ j_B$ by the definition of the pullback, so $f_A = f_B$. By universality, there exists a unique map 
$$
m_{A \cap B}:  (A \cap B) \otimes (A \cap B) \to A \cap B
$$
such that 
$$
j_L \circ m_{A \cap B} = m_L \circ (j_L \otimes j_L).
$$
Here is the corresponding diagram:
\[
 			\begin{tikzcd}
 			(A\cap B) \otimes (A\cap B) \arrow[rr, "j_A \otimes j_A"] \arrow[dr,dashed,"\exists! m_{A\cap B}"] \arrow[dd,"j_B \otimes j_B"]   &&  A \otimes A \arrow[d,"m_A"] \\
 			& A \cap B \arrow[r,"j_A"] \arrow[d,"j_B"] & A \arrow[d,"i_A"] \\
 			B \otimes B \arrow[swap,r,"m_B"] & B \arrow[swap,r,"i_B"] & C
 			\end{tikzcd}
\]

Similarly, for $L \in \{A,B\}$, since $i_L$ is a unital morphism, we have $i_A \circ e_A = e_C = i_B \circ e_B$. By universality, there exists a unique $e_{A \cap B}$ such that $j_L \circ e_{A \cap B} = e_L$. The corresponding diagram is as follows:
\[
 			\begin{tikzcd}
 			\one
 			\arrow[drr, bend left, "e_A"]
 			\arrow[swap,ddr, bend right, "e_B"]
 			\arrow[dr, dashed, "\exists! e_{A \cap B}"] & & \\
 			& A \cap B \arrow[r, "j_A"] \arrow[swap,d, "j_B"]
 			& A \arrow[d, "i_A"] \\
 			& B \arrow[swap,r, "i_B"]
 			& C
 			\end{tikzcd}
\]
The result follows from Lemma \ref{lem:TripleAlg}.
\end{proof}

\begin{corollary} \label{cor:InterCoalg}
In the rigid case, following Proposition~\ref{prop:InterAlg}, let $K \in \{A, B, A \cap B\}$. If $K$ is selfdual, then setting
$\delta_K := m_K^*$ and $\epsilon_K := e_K^*$ endows $K$ with a counital coalgebra structure. Moreover, for each $L \in \{A, B\}$, the map $j_L^*$ is a counital coalgebra epimorphism.
\end{corollary}

\begin{proof}
This follows directly from Lemmas~\ref{lem:Dualgebra} and~\ref{lem:DualgMor}, together with Proposition~\ref{prop:InterAlg}, using the fact that the dual of a monomorphism is an epimorphism.
\end{proof}

The following lemma and the proof of Lemma~\ref{lem:Ro} are partly due to Maxime Ramzi \cite{RamMO}. They are not used in this paper, but may be useful in future work.

\begin{lemma} \label{lem:Ram}
A square $ A_0 \to (A_1, A_2) \to A_3 $ with $ A_0 \to A_1 $ and $ A_2 \to A_3 $ being monomorphisms is a pullback square if and only if the induced morphism $ A_1 / A_0 \to A_3 / A_2 $ is a monomorphism.
\end{lemma}

\begin{proof}
The map $ A_1 \times_{A_3} A_2 \to A_1 $ is a monomorphism because it is pulled back from a monomorphism, and the inclusion $ A_0 \to A_1 $ factors through it. Therefore, the map $ A_1 / A_0 \to A_3 / A_2 $ factors through $ A_1 / A_0 \to A_1 / (A_1 \times_{A_3} A_2) $. For the total composite to be a monomorphism, this must also be a monomorphism, thus proving the lemma.
\end{proof}
The following lemma extends \cite[Chapter 14, Exercises 10-12]{Ro08}:

\begin{lemma} \label{lem:Ro}
Let $A, B, C, D, E, F$ be objects in an abelian rigid monoidal category.
\begin{enumerate}
\item If $A, B \subseteq C$, then $(A \otimes D) \cap (B \otimes D) \simeq (A \cap B) \otimes D,$ \label{lem:Ro1}
\item If $A \subseteq C$ and $B \subseteq D$, then $(A \otimes D) \cap (C \otimes B) \simeq (A \otimes B),$ \label{lem:Ro2}
\item If $A, B \subseteq C$ and $D, E \subseteq F$, then $(A \otimes D) \cap (B \otimes E) \simeq (A \cap B) \otimes (D \cap E).$ \label{lem:Ro3}
\end{enumerate}
\end{lemma}

\begin{proof}
By \cite[Exercise 1.6.4 and Proposition 2.10.8]{EGNO15}, the functor $-\otimes D$ is exact in an abelian rigid monoidal category. The intersection is defined as a pullback of monomorphisms. Since a pullback is a finite limit and an exact functor preserves finite limits, the result in (\ref{lem:Ro1}) follows.

For (\ref{lem:Ro2}), $X\otimes -$ preserves monomorphisms for all $X$, hence the square $A\otimes B\to (A \otimes D, C\otimes B) \to C\otimes D$ satisfies Lemma \ref{lem:Ram}. In particular, a tensor product of monomorphisms is a monomorphism (i.e., $A \otimes B \subseteq C \otimes D$).

For (\ref{lem:Ro3}), applying (\ref{lem:Ro2}) two times yields that 
$$
(A \otimes D) \cap (B \otimes E) \simeq (A \otimes F) \cap (C \otimes D) \cap (B \otimes F) \cap (C \otimes E).
$$
Next, applying (\ref{lem:Ro1}) with $-\otimes F$ and $C\otimes -$, along with the commutativity of intersections, gives us $((A\cap B) \otimes F)\cap (C\otimes (D \cap E))$. We then apply (\ref{lem:Ro2}) again to obtain $(A\cap B)\otimes (D\cap E)$.
\end{proof}

\begin{remark}
We deduce from the proof of Lemma \ref{lem:Ro} that in an abelian rigid monoidal category, the tensor product of two monomorphisms is a monomorphism.
\end{remark}

\subsection{Sublattices of the Frobenius subalgebra poset }
\begin{proposition}\label{prop:preInterWeak}
Let $A$ and $B$ be Frobenius subalgebras of $C$ in an abelian rigid monoidal category. Then their intersection $A \cap B$ is a unital subalgebra of $C$. Moreover, $A \cap B$ is a Frobenius subalgebra of $C$ if and only if it is ambiently selfdual.
\end{proposition}

\begin{proof}
By Proposition~\ref{prop:InterAlg}, the intersection $A \cap B$ is a unital subalgebra of $A$, and $j_A$ is a unital algebra monomorphism. Hence, $i_{A \cap B} = i_A \circ j_A$ is also a unital algebra monomorphism, being the composition of such morphisms. The result then follows from Proposition~\ref{prop:AmbientFrobSub}.
\end{proof}

Recall the notion of \emph{Frobenius $\mathcal{B}$-subalgebra} from Notation \ref{not:B-sub} and Definition~\ref{def:RigidInvariance}.

\begin{proposition} \label{prop:InterWeak}
Let \(A\) and \(B\) be Frobenius $\mathcal{B}$-subalgebras of \(C\), where \(C\) is semisimple of finite length as an object in a \(\mathbb{C}\)-linear abelian rigid monoidal category. Then \(A \cap B\) is a Frobenius $\mathcal{B}$-subalgebra of \(C\).
\end{proposition}

\begin{proof}
By Lemma~\ref{lem:RigidInvInterSum}, the intersection \(A \cap B\) is ambiently selfdual and $\mathcal{B}$-rigid invariant. The statement then follows directly from Proposition~\ref{prop:preInterWeak}.
\end{proof}

The relaxation of the semisimplicity assumption hinges on \S\ref{sub:RigidNonSS}. Let us recall the definition of a lattice:

\begin{definition} \label{def:lattice}
A \emph{meet-semilattice} $(L, \wedge)$ is a partially ordered set (poset) $L$ in which every pair of elements $a$ and $b$ has a unique infimum (or \emph{meet}) denoted by $a \wedge b$. A \emph{join-semilattice} $(L, \vee)$ is a poset $L$ in which every pair of elements $a$ and $b$ has a unique supremum (or \emph{join}) denoted by $a \vee b$. A \emph{lattice} $(L, \wedge, \vee)$ is a poset that is both a meet-semilattice and a join-semilattice.
\end{definition}

An element \(t\) of a poset \(L\) is called a \emph{top element} if \(t \ge a\) for all \(a \in L\). The \emph{height} of a poset is defined as the length of its longest strictly increasing chain $\cdots < a < b < \cdots$.

\begin{lemma} \label{lem:meet-semi}
A meet-semilattice with a top element and finite height is a lattice.
\end{lemma}
\begin{proof}
Let $a, b \in L$. The set $U$ of their common upper bounds is non-empty since it contains the top element. Because $L$ has finite height, every $x \in U$ is bounded below by some minimal element of $U$. If $c$ and $d$ are two such minimal elements, their meet $c \wedge d$ is also an upper bound for $a$ and $b$, meaning $c \wedge d \in U$. By minimality, $c = c \wedge d = d$. Therefore, $U$ possesses a unique minimal element, which must be the least upper bound $a \vee b$.
\end{proof}

\begin{theorem}\label{thm:AmbientSublattice}
Let $X$ be a Frobenius algebra of finite length in an abelian rigid monoidal category. A subposet of its Frobenius subalgebra poset can be extended to a sublattice with $[A] \wedge [B] = [A \cap B]$ if and only if the intersection of any two class representatives is ambiently selfdual.
\end{theorem}

\begin{proof}
One direction is immediate by Lemma \ref{lem:FrobSubAmbiently}. For the converse direction, by Proposition~\ref{prop:preInterWeak}, such a subposet can be extended to a meet-semilattice, with meet given by
\(
[A] \wedge [B] := [A \cap B],
\)
and with top element $[X]$. Since $X$ has finite length, the poset has finite height. The claim then follows from Lemma~\ref{lem:meet-semi}.
\end{proof}

The entire Frobenius subalgebra poset is not a lattice in general. For counterexamples in $\VVec$, see \cite{BenMO2, WillMO1} by Dave Benson and Will Sawin, with the first detailed in Example~\ref{ex:Ben02}. That is why the ambiently selfdual or rigid invariant assumptions matter. Recall that the notion of ambiently selfdual generalizes that of a nondegenerate subspace.

\begin{theorem}\label{thm:Bsublattice}
Let $X$ be a Frobenius algebra that is semisimple of finite length as an object in a $\mathbb{C}$-linear abelian rigid monoidal category. Then the subposet of its Frobenius subalgebra poset consisting of classes of Frobenius $\mathcal{B}$-subalgebras (the \emph{Frobenius $\mathcal{B}$-subalgebra poset}) is a lattice with $[A] \wedge [B] = [A \cap B]$.
\end{theorem}

\begin{proof}
This follows immediately from Lemma~\ref{lem:RigidInvInterSum} and Theorem~\ref{thm:AmbientSublattice}.
\end{proof}


\begin{corollary} \label{cor:lattice}
For every Frobenius algebra in a semisimple tensor category over $\mathbb{C}$, its Frobenius $\mathcal{B}$-subalgebra posets are lattices with $[A] \wedge [B] = [A \cap B]$.
\end{corollary}

As explained in \S\ref{sec:RigidInvariant}, the choice of a basis \( \mathcal{B} \) of \( X \) (see Notation~\ref{not:basis}) is essential: it determines a perspective that collapses the Frobenius subalgebra poset into a sublattice. Moreover, every Frobenius subalgebra appears in some such sublattice. To remove the dependence on a specific basis, we restrict to \emph{rigid invariant} Frobenius subalgebras, i.e., those that are $\mathcal{B}$-rigid invariant for every choice of basis \( \mathcal{B} \), as defined in Definition~\ref{def:RigidInvariance}.

\begin{corollary} \label{cor:lattice0}
For every Frobenius algebra $X$ that is semisimple of finite length in a \( \mathbb{C} \)-linear abelian rigid monoidal category, its rigid invariant Frobenius subalgebra poset is a lattice with $[A] \wedge [B] = [A \cap B]$.
\end{corollary}
\begin{proof}
Immediate from Theorem~\ref{thm:Bsublattice} together with the fact that an arbitrary intersection of sublattices is again a sublattice.
\end{proof}

\begin{corollary} \label{cor:lattice2}
For every Frobenius algebra in a semisimple tensor category over $\mathbb{C}$, its rigid invariant Frobenius subalgebra poset is a lattice with $[A] \wedge [B] = [A \cap B]$.
\end{corollary}

We arrive, quite naturally, at the unitary case:
\begin{corollary} \label{cor:latticeunitary}
For every unitary Frobenius algebra in a unitary tensor category, its unitary Frobenius subalgebra poset is a lattice with $[A] \wedge [B] = [A \cap B]$.
\end{corollary}

\begin{proof}
Immediate from Proposition~\ref{prop:UnitFrobSub} and Corollary~\ref{cor:lattice2}.
\end{proof}

Much like Corollary~\ref{cor:lattice2}, which fails in the absence of the rigid invariance assumption, Corollary~\ref{cor:latticeunitary} may also fail if the Frobenius subalgebra is not assumed to be unitary, given that Question~\ref{q:UnitFrob} remains open.


\section{Weak positivity} \label{sec:pseudo}

Let $ \mathbbm{k} $ be an algebraically closed field.

\begin{definition} \label{def:weakpositive} 
Let $ \mathcal{C} $ be a $ \mathbbm{k} $-linear monoidal category with a linear-simple unit object $ \one $. A nonzero object $ X \in \mathcal{C} $ is called \emph{weakly non-vanishing} if there exists an isomorphism $ \alpha \in \Hom_{\mathcal{C}}(X, X^{**}) $ such that for every nonzero idempotent morphism $ b \in \End_{\mathcal C}(X) $, the trace
$
\tr_\alpha(b) := \tr(\alpha \circ b),
$
as defined in Definition \ref{def:trace}, is nonzero. When the choice of $ \alpha $ is to be emphasized, we refer to $ X $ as $\alpha$-weakly non-vanishing.
If $ \mathbbm{k} = \mathbb{C} $ and $ \tr_\alpha(b) $ is always positive for all nonzero idempotents $ b $, then $ X $ is called \emph{weakly positive}, or more specifically, $\alpha$-weakly positive.
\end{definition}

\begin{lemma} \label{lem:trftrg}
Let $ X $ be an $ \alpha $-weakly positive object. For all nonzero idempotents $ f,g \in \End_{\mathcal{C}}(X) $ such that $ g \leq f $, it follows that $ \tr_{\alpha}(g) \leq \tr_{\alpha}(f) $. Moreover, if $ f \neq g $, then $ \tr_{\alpha}(g) < \tr_{\alpha}(f) $.
\end{lemma}

\begin{proof}
Recall that $ g \leq f $ if and only if $ g \circ f = g = f \circ g $. By the linearity of the trace, we have:
$$
\tr_{\alpha}(f) = \tr_{\alpha}(g) + \tr_{\alpha}(f-g).
$$
But $ f-g $ is an idempotent because:
$$
(f-g) \circ (f-g) = (f \circ f) - (f \circ g) - (g \circ f) + (g \circ g) = f - g - g + g = f - g.
$$
Thus, the result follows from the assumption that $ X $ is $ \alpha $-weakly positive.
\end{proof}

Recall the notion of (quasi-)pivotal structure $\phi$ in Definitions \ref{def:pivotal} and \ref{def:quasipivotal}, and of  trace $\tr_\phi$ below them. 
\begin{definition} \label{def:PositiveTensor}
A $\mathbb{C}$-linear rigid monoidal category $ \mathcal{C} $ with a linear-simple unit $ \one $ is called \emph{positive} if it has a pivotal structure $\phi$ such that $\tr_{\phi}(\id_X)$ is positive for all nonzero object $X$. If we want to highlight the choice of $ \phi $, we refer to $ \mathcal{C} $ as $ \phi $-positive.
\end{definition}

Recall the notion of a unitary tensor category as mentioned in Remark \ref{rk:UnitaryTensor}.

\begin{proposition} \label{prop:C*positive}
Every unitary tensor category is positive.
\end{proposition}

\begin{proof}
According to \cite[\S 2.3]{JP17}, any unitary tensor category $\mathcal{C}$ comes with a canonical spherical structure $\phi$ such that every nonzero object $X$ in $\mathcal{C}$, $\tr_{\phi}(\id_X) = \coev_X^\dag \circ \coev_X$.
By the definition of a $C^*$-category, for any morphism $f \in \Hom_{\mathcal{C}}(Y, Z)$, there exists a morphism $g \in \End_{\mathcal{C}}(Y)$ such that $f^\dag \circ f = g^\dag \circ g$, and $\End_{\mathcal{C}}(Y)$ carries a $C^*$-algebra structure (with involution $ ()^\dag $), so $g^\dag \circ g$ is a positive element.
Taking $Y = \one$, $Z = X \otimes X^*$, and $f = \coev_X$, we conclude that $\tr_{\phi}(\id_X)$ is positive.
\end{proof}

By the Frobenius-Perron theorem, in a positive finite tensor category, we have $ \tr_{\phi}(\id_X) = \FPdim(X) $.
\begin{proposition} \label{prop:pseudo}
A finite tensor category over $\mathbb{C}$ is positive if and only if it is a pseudo-unitary fusion category.
\end{proposition}

\begin{proof}
By \cite[Proposition 9.5.1]{EGNO15}, every pseudo-unitary fusion category is positive. Conversely, let $\mathcal{C}$ be a positive pivotal finite tensor category. Since finite tensor categories have enough projectives (\cite[Definition 1.8.6]{EGNO15}), there exists a nonzero projective object $P$ in $\mathcal{C}$. As observed in the proof of \cite[Theorem 6.6.1]{EGNO15}, if $\mathcal{C}$ is non-semisimple, then $\tr_{\phi}(\id_P)=0$. This contradicts positivity. Hence $\mathcal{C}$ is semisimple, and therefore a fusion category. Finally, every positive fusion category is intrinsically pseudo-unitary. This completes the proof.
\end{proof}

\begin{remark} \label{rk:PosVsUnitary}
The proof of Proposition \ref{prop:pseudo} shows that a positive \emph{finite} tensor category is necessarily semisimple. Without the finiteness assumption, however, semisimplicity can fail: for example, the Tannakian tensor category $\Rep_{\mathrm{fd}}(G)$ of finite-dimensional representations of a non-reductive subgroup $G \subset \mathrm{GL}(n,\mathbb{C})$ (e.g., the group of upper-triangular invertible matrices for $n \ge 2$) is positive but non-semisimple, and therefore non-unitary. More generally, if a Hopf algebra $H$ over the complex field is involutory (i.e., $S^2 = \id$), then both tensor categories $\Rep_{\mathrm{fd}}(H)$ and $\Corep_{\mathrm{fd}}(H)$ are positive (see \cite[Corollary 5.3.7, Remark 5.3.8]{EGNO15}), and we suspect the reverse implication for the latter. Thus, the converse of Proposition \ref{prop:C*positive} fails in general: there exist positive tensor categories that admit no unitary structure. In the finite case, however, the converse remains open. Indeed, by Proposition \ref{prop:pseudo}, positivity is equivalent to pseudo-unitarity for finite tensor categories. While every unitary fusion category is pseudo-unitary, it is not known whether every pseudo-unitary fusion category admits a compatible unitary structure (see \cite[\S 9.4]{EGNO15}).
\end{remark}



\begin{definition} \label{def:weaklyPositiveTensor}  Let $ \mathcal{C} $ be a $\mathbbm{k}$-linear rigid monoidal category with a linear-simple unit object $ \one $. We say that $ \mathcal{C} $ is \emph{weakly non-vanishing} if it is abelian and admits a quasi-pivotal structure $ \phi $ such that $ \tr_\phi(\id_X) $ is nonzero for every simple object $ X $. If we wish to emphasize the choice of $ \phi $, we refer to $ \mathcal{C} $ as \emph{$\phi$-weakly non-vanishing}. Without assuming $\mathcal{C}$ is abelian (since the condition is now on all nonzero objects rather than only simple ones) it is called \emph{weakly positive}, or more specifically \emph{$\phi$-weakly positive}, if $ \mathbbm{k} = \mathbb{C} $ and $ \tr_\phi(\id_X) $ is positive for every nonzero object $ X $.
\end{definition}

\begin{remark} \label{rk:PosWpos}
A $\phi$-positive category is $\phi$-weakly positive, since a pivotal structure is quasi-pivotal.
\end{remark}

\begin{proposition} \label{prop:WeaklyPositiveTensor}
Every nonzero object $ X $ in a Karoubian $\phi$-weakly positive category $ \mathcal{C} $ is $\phi_X$-weakly positive. Every simple object $ X $ in a $\phi$-weakly non-vanishing category $ \mathcal{C} $ is $\phi_X$-weakly non-vanishing.
\end{proposition}
\begin{proof}
For any nonzero idempotent $ b \in \End_{\mathcal C}(X) $, consider its decomposition $ b = i \circ p $, where $ p \circ i = \id_Y $ follows from the Karoubian structure. By Lemma \ref{lem:trfgpivotalquasi} and the assumption of $\phi$-weak positivity, we obtain:  
\[
\tr_{\phi_X}(b) = \tr_\phi(b) = \tr_\phi(i \circ p) = \tr_\phi(p \circ i) = \tr_\phi(\id_Y) > 0.
\]  
The proof for the weakly non-vanishing case follows analogously.
\end{proof}

Let $\mathcal{O}(\mathcal{C})$ be a set of isomorphism class representatives of the simple objects in an abelian monoidal category $\mathcal{C}$.

\begin{theorem} \label{thm:semisweakly}
A semisimple tensor category over $\mathbb{C}$ is weakly positive.
\end{theorem}

\begin{proof} 
We break the proof into the following lemmas:
\begin{lemma} \label{lem:NatFunctSimple}
Let $\mathcal{C}$ be a semisimple $\mathbbm{k}$-linear abelian category, and let $F, G: \mathcal{C} \to \mathcal{C}$ be $\mathbbm{k}$-linear (covariant) functors. Suppose that for each simple object $S$ in $\mathcal{O}(\mathcal{C})$, 
we are given a morphism $\phi_S: F(S) \to G(S).$ Then there exists a natural transformation $\phi: F \to G$ whose components at $S \in \mathcal{O}(\mathcal{C})$ are the given morphisms $\phi_S$.
\end{lemma}
\begin{proof}
Since $\mathcal{C}$ is semisimple, every object $V$ can, without loss of generality, be written as $V = \bigoplus_{S \in \mathcal{O}(\mathcal{C})} V_S \otimes S,$ where $V_S$ are multiplicity spaces $\Hom_{\mathcal{C}}(S,V)$. Because $F$ and $G$ are additive functors, they preserve direct sums, giving $F(V) = \bigoplus_{S} V_S \otimes F(S)$, $G(V) = \bigoplus_{S} V_S \otimes G(S)$. Define $\phi_V: F(V) \to G(V)$ as $\phi_V := \bigoplus_{S} \id_{V_S} \otimes \phi_S.$ To verify naturality, let $ f: V \to W $ be any morphism in $ \mathcal{C} $. Without loss of generality, we can express $V = \bigoplus_{S} V_S \otimes S$, $W = \bigoplus_{S} W_S \otimes S$, $f = \bigoplus_{S} f_S \otimes \id_S$. Here, each $ f_S: V_S \to W_S $ is a $ \mathbbm{k} $-linear map such that for any $ \alpha \in V_S = \Hom_{\mathcal{C}}(S,V) $, we have $ f_S(\alpha) = \alpha \circ f $. The functoriality and $\mathbbm{k}$-linearity of $F$ and $G$ imply $F(f) = \bigoplus_{S} f_S \otimes F(\id_S)$, $G(f) = \bigoplus_{S} f_S \otimes G(\id_S)$. To ensure naturality, we must verify that the following diagram commutes: $$\begin{tikzcd}
F(S) \ar[r, "F(\id_S)"] \ar[d, "\phi_{S}"'] & F(S) \ar[d, "\phi_{S}"] \\
G(S) \ar[r, "G(\id_S)"'] & G(S)
\end{tikzcd}$$ This holds since functoriality gives $F(\id_S) = \id_{F(S)}$ and $G(\id_S) = \id_{G(S)}$. Thus, $\phi$ is a natural transformation.
\end{proof}

\begin{lemma} \label{lem:semisquasi}
A semisimple multitensor category over $\mathbbm{k}$ has a quasi-pivotal structure.
\end{lemma}

\begin{proof}
By \cite[Proposition 4.8.1]{EGNO15}, for each $X$ in $\mathcal{O}(\mathcal{C})$, there exists an isomorphism $\phi_X: X \to X^{**}$. Applying Lemma \ref{lem:NatFunctSimple}, we obtain a natural isomorphism between $\id_{\mathcal{C}}$ and the double dual functor $()^{**}$.
\end{proof}

\begin{lemma} \label{lem:TensWNonzero}
A semisimple tensor category over $\mathbbm{k}$ is weakly non-vanishing.
\end{lemma}

\begin{proof}
By Lemma \ref{lem:semisquasi}, the category has a quasi-pivotal structure $\phi: \id_{\mathcal{C}} \to ()^{**}$. According to \cite[Proposition 4.8.4]{EGNO15}, $\tr_{\phi}(X)$ is nonzero for all $X$ in $\mathcal{O}(\mathcal{C})$.
\end{proof}

\begin{lemma} \label{lem:WNonZeroToWPositive}
A $\phi$-weakly non-vanishing semisimple tensor category over $\mathbb{C}$ is weakly positive.
\end{lemma}
\begin{proof}
By the polar decomposition, for each $X$ in $\mathcal{O}(\mathcal{C})$, we have 
$\tr_\phi(\id_X) = \rho_X e^{i \theta_X}$, 
where $\rho_X$ is positive. Define a modified isomorphism $\phi_X' := e^{-i \theta_X} \phi_X$ for each $X$. Applying Lemma \ref{lem:NatFunctSimple} to $(\phi_X')$ gives a new natural isomorphism $\phi'$, ensuring that  
$
\tr_{\phi'}(\id_X) = \rho_X > 0
$  
for all $X$ in $\mathcal{O}(\mathcal{C})$. The result follows.
\end{proof}

This completes the proof of the theorem.
\end{proof}

\begin{corollary} \label{cor:semisweaklyObject}
Every nonzero object in a semisimple tensor category over $\mathbb{C}$ is weakly positive.
\end{corollary}
\begin{proof}
Immediate by Proposition \ref{prop:WeaklyPositiveTensor} and Theorem \ref{thm:semisweakly}.
\end{proof}


Finally, we will see that in the semisimple case, the isomorphism that makes an object weakly positive (Definition \ref{def:weakpositive}) can be strongly restricted. 

\begin{lemma} \label{lem:idem}  
Let $M$ be a matrix in $M_n(\mathbb{C})$ such that for every nonzero idempotent $b$ in $M_n(\mathbb{C})$, the trace $\mathrm{Tr}(bM)$ is positive. Then $M$ is a positive scalar multiple of the identity.  
\end{lemma}  

\begin{proof}  
%
For any matrix $X \in M_n(\mathbb{C})$, the commutator $[X, b]$ decomposes as $[X, b] = Xb - bX = x - y$, where $x = (1-b)Xb$ and $y = bX(1-b)$. 
Observe that $bx = 0$ and $xb = x$, which implies $x^2 = 0$. Therefore, $b + zx$ is idempotent for all $z \in \mathbb{C}$. By assumption, $\mathrm{Tr}((b+zx)M) = \mathrm{Tr}(bM) + z\mathrm{Tr}(xM) > 0$, for all $z \in \mathbb{C}$. Thus, we must have $\mathrm{Tr}(xM) = 0$. 
Symmetrically, $yb = 0$ and $by = y$, making $b+zy$ idempotent, which similarly implies $\mathrm{Tr}(yM) = 0$. 
Thus, $\mathrm{Tr}([X, b]M) = \mathrm{Tr}(xM) - \mathrm{Tr}(yM) = 0$. Since $\mathrm{Tr}([X, b]M) = \mathrm{Tr}(X[b, M]) = 0$ for all $X$, the non-degeneracy of the trace pairing dictates that $[b, M] = 0$. 
Because idempotents linearly span $M_n(\mathbb{C})$, $M$ is in the commutant of $M_n(\mathbb{C})$, forcing $M = \lambda I$. Finally, $\mathrm{Tr}(b \lambda I) = \lambda \mathrm{Tr}(b) > 0$ yields $\lambda > 0$.
\end{proof}
\begin{proposition} \label{prop:WeakPosIso}
Let $ X $ be a semisimple $ \alpha $-weakly positive object in a $ \mathbb{C} $-linear abelian monoidal category $ \mathcal{C} $ with a linear-simple unit $ \mathbf{1} $. Assume that $X^{**} = X$. Then $\alpha$ can be chosen of the form $\alpha = \bigoplus_{S \in \mathcal{O}(X)} e^{-i \theta_S}\id_{M_S} \otimes \id_S$.
\end{proposition}

\begin{proof}
Since $ X $ is semisimple, we can express it as $X = \bigoplus_{S \in \mathcal{O}(X)} M_S \otimes S$,
where $ M_S $ is a multiplicity space and $\mathcal{O}(X)$ is the set of $S$ in $\mathcal{O}(\mathcal{C})$ with $ M_S \neq 0 $. Any idempotent $ b \in \End_{\mathcal C}(X)$ takes the form  $b = \bigoplus_{S} B_S \otimes \id_S$, where $ B_S $ is an idempotent matrix in $ \End_{\mathbb{C}}(M_S) $.
By assumption, if $ b \neq 0 $, then $ \tr_{\alpha}(b) $ is positive. By assumption, $ X^{**} = X $, so that $ S^{**} = S $ for all $ S \in \mathcal{O}(X)$.
The isomorphism $ \alpha $ must then be of the form  $\alpha = \bigoplus_{S} A_S \otimes \id_S$, where $ A_S $ is an invertible matrix in $ \End_{\mathbb{C}}(M_S) $. This gives $\tr_{\alpha}(b) = \tr(\alpha \circ b) = \sum_{S} \Tr(A_S B_S) \tr(\id_S)$, where $ \Tr $ denotes the usual matrix trace. Now, consider an idempotent $ b_S $ of the same form as $ b $, but with $ B_{S'} = \delta_{S,S'} B_S $. Then $\tr_{\alpha}(b_S) = \Tr(A_S B_S) \tr(\id_S)$.
Thus, for all nonzero idempotent $ B_S \in \End_{\mathbb{C}}(M_S) $, we find that  $\Tr(M B_S) > 0$, where $M = \tr(\id_S) A_S$. By Lemma \ref{lem:idem}, this implies that $ M $ is a positive scalar multiple of the identity. 
%
Let $\tr(\id_S) = \rho_S e^{i \theta_S}$ be the polar decomposition with $\rho_S$ positive, then there is a positive scalar $\eta_S$ such that $A_S =  \eta_S e^{-i \theta_S}\id_{M_S}$, thus $\alpha = \bigoplus_{S} \eta_S e^{-i \theta_S}\id_{M_S} \otimes \id_S$. Thus $
\tr_{\alpha}(b) = \sum_{S} \Tr(A_S B_S) \tr(\id_S) = \sum_{S} \Tr(\eta_S e^{-i \theta_S}\id_{M_S} B_S) \rho_S e^{i \theta_S} = \sum_{S} \eta_S\rho_S n_{b,S}$, where $n_{b,S}$ is the rank of $B_S$. Observe that the weak positivity is independent on the choice of the positive scalar $\eta_S$, thus we can take $\eta_S = 1$ for all $S \in \mathcal{O}(X)$. Thus, without loss of generality, we can have $\alpha = \bigoplus_{S} e^{-i \theta_S}\id_{M_S} \otimes \id_S$. \qedhere   
\end{proof}



\section{Semisimplification} \label{sec:semi} 
The notion of semisimplification for spherical tensor categories was first introduced in \cite{BW99}. This section follows the exposition in \cite{EO22}. 

\begin{definition} \label{def:negl} Let $\mathcal{C}$ be a rigid monoidal category equipped with a pivotal structure $\phi$. A morphism $f: X \to Y$ is said to be \emph{negligible} if, for every morphism $g: Y \to X$, the trace $\tr_{\phi}(f \circ g)$ is zero. \end{definition}

In a spherical tensor category $\mathcal{C}$, the collection $\mathcal{N}$ of negligible morphisms forms a tensor ideal. Consequently, the quotient category $\overline{\mathcal{C}} = \mathcal{C}/\mathcal{N}$ acquires the structure of a semisimple tensor category. In this new category, the morphism compositions and tensor products remain unchanged from those in $\mathcal{C}$, while the simple objects consist of the indecomposable objects of $\mathcal{C}$ that have nonzero dimension. Objects with zero dimension map to the zero object. This process is referred to as the \emph{semisimplification} of $\mathcal{C}$. 

\begin{remark} \label{rk:zeronil}
This result can be extended to $\mathbbm{k}$-linear Karoubian pivotal rigid monoidal categories, where all morphism spaces are finite-dimensional, and with a pivotal structure such that the trace of a nilpotent endomorphism is zero, and the left and right dimensions of indecomposable objects vanish simultaneously (see \cite[Theorem 2.6]{EO22}).
\end{remark}

\begin{lemma} \label{lem:phi1}
Let $\mathcal{C} $ be a $\mathbbm{k}$-linear rigid monoidal category with a pivotal structure $\phi $ and a linear-simple unit object $\one $. Then, $\phi_{\one} = \id_{\one} $.
\end{lemma}

\begin{proof}
Since $\one$ is linear-simple, there exists a scalar $c \in \mathbbm{k}$ such that $\phi_{\one} = c \, \id_{\one}$.  
Because $\phi$ is a monoidal natural isomorphism, $c$ is nonzero and
$
\phi_{\one} = \phi_{\one \otimes \one} = \phi_{\one} \otimes \phi_{\one}.
$
Substituting $\phi_{\one} = c \, \id_{\one}$, this gives
$
c \, \id_{\one} = (c \, \id_{\one}) \otimes (c \, \id_{\one}) = c^2 \, \id_{\one}.
$
Thus, $c^2 = c$. Hence $c=1$. Therefore, $\phi_{\one} = \id_{\one}$.
\end{proof}

\begin{lemma} \label{lem:specialnegl}
Let $\mathcal{C}$ be a $\mathbbm{k}$-linear rigid monoidal category with a pivotal structure $\phi$, where $\one$ is linear-simple. Suppose $(X, m, \delta, e, \epsilon)$ is a connected Frobenius algebra in $\mathcal{C}$. If $\tr(\phi_X)$ is nonzero, then none of the morphisms $m$, $\delta$, $e$, or $\epsilon$ are negligible.
\end{lemma}

\begin{proof}
By (co)unitality, we have
$
m \circ (e \otimes \id_X) = \id_X = (\epsilon \otimes \id_X) \circ \delta.
$
Taking traces, it follows that
\[
\tr_{\phi}(m \circ (e \otimes \id_X)) = \tr_{\phi}((\epsilon \otimes \id_X) \circ \delta) = \tr_{\phi}(\id_X) = \tr(\phi_X),
\]
which is nonzero by assumption. Moreover, by Lemma~\ref{lem:contr}, we have $\mu_X \, m \circ (\phi_X \otimes \id_X) \circ \delta = \tr(\phi_X) \, \id_X$, implying that $\mu_X$ is also nonzero. But $\mu_X := \epsilon \circ e = \tr(\epsilon \circ e) =  \tr_{\phi}(\epsilon \circ e)$,
since $\phi_{\one}=\id_{\one}$ (Lemma \ref{lem:phi1}). The result follows by Lemma~\ref{lem:trfgpivotalquasi}.
\end{proof}

\begin{proposition} \label{prop:posetsemi}
Let $\mathcal{C}$ be a $\mathbbm{k}$-linear rigid monoidal category with a pivotal structure $\phi$, a linear-simple unit object $\one$, and satisfying the conditions described in Remark~\ref{rk:zeronil} (e.g., a spherical tensor category). Let $X$ be a connected Frobenius algebra in $\mathcal{C}$ that is $\phi_X$-weakly non-vanishing (Definition \ref{def:weakpositive}). Then, $\overline{X}$ is also a Frobenius algebra in the semisimplification $\overline{\mathcal{C}}$. Furthermore, the Frobenius subalgebra poset (Definition \ref{def:FrobSubPoset}) of $X$ in $\mathcal{C}$ embeds into the Frobenius subalgebra poset of $\overline{X}$ in $\overline{\mathcal{C}}$.
\end{proposition}
\begin{proof}
The Frobenius algebra $X$ in $\mathcal{C}$ remains a Frobenius algebra $\overline{X}$ in $\overline{\mathcal{C}}$ by Lemma \ref{lem:specialnegl} and the preservation of morphism compositions and tensor products from $\mathcal{C}$, which ensures that the axioms are satisfied. For any indecomposable (so nonzero) component $Y$ of $X$, there exists a monomorphism $i: Y \to X$ and an epimorphism $p: X \to Y$ such that $i \circ p = b$ and $p \circ i = \id_Y$, indicating that $b$ is a nonzero idempotent (since $p = p \circ b$). Given that $X$ is $\phi_X$-weakly non-vanishing, we also have by Lemma \ref{lem:trfgpivotalquasi} that $\dim_{\phi}(Y) = \tr_{\phi}(\id_Y) = \tr_{\phi}(b) \neq 0$.
Therefore, $X$ has no indecomposable components of zero dimension, resulting in the poset embedding.
\end{proof}

The main purpose of Proposition~\ref{prop:posetsemi} is to extend the notion of a Frobenius $\mathcal{B}$-subalgebra beyond the semisimple setting---yielding Corollary~\ref{cor:allfinite1} via semisimplification---since a direct non-semisimple definition is problematic (\S\ref{sub:RigidNonSS}).

\begin{definition} \label{def:RigidInvariantNSS}
Under the assumptions of Proposition~\ref{prop:posetsemi}, a Frobenius $\mathcal{B}$-subalgebra $Y$ of $X$ in $\mathcal{C}$ is a Frobenius subalgebra such that $\overline{Y}$ is a Frobenius $\mathcal{B}$-subalgebra of $\overline{X}$ in $\overline{\mathcal{C}}$.
\end{definition}

\section{Formal angle} \label{sec:Angle}

We extend the notion of the angle between two intermediate subfactors introduced in~\cite{BDLR19} by defining a \emph{formal angle} for a coherent pair (Definition~\ref{def:CoherentPair}) of Frobenius subalgebras. This construction relies on the idempotent $b_{X'}$ from Definition~\ref{def:sub}, which generalizes the notion of a biprojection in subfactor planar algebra theory.

Let $\mathcal{C}$ be a $\mathbb{C}$-linear abelian rigid monoidal category in which the unit object~$\one$ is linear-simple. Let $X$ be a connected Frobenius algebra in $\mathcal{C}$.
Suppose there exists a quasi-pivotal structure $\phi$ (see Definition~\ref{def:quasipivotal}) such that $\alpha := \phi_X$ is an algebra isomorphism (Definition~\ref{def:FrobMor}) and $X$ is $\alpha$-weakly positive (Definition~\ref{def:weakpositive}). In particular, $\tr_{\alpha}(\id_X)$ is positive. Recall the partial order on idempotents in $\End_{\mathcal{C}}(X)$ introduced before Lemma \ref{lem:idemposet}. Let $(A,B)$ be a coherent pair of Frobenius subalgebras of $X$. Assume further that $\tr(b_A \circ b_B) \ne 0$ (see  Remark~\ref{rk:zerocase} for the case $\tr(b_A \circ b_B) = 0$). Applying Corollary~\ref{cor:landau}, we define the idempotent
\[
b_{AB} := \frac{\mu_X}{\tr(b_A \circ b_B)} \, b_A * b_B,
\]
see Remark~\ref{rk:AB}. This idempotent satisfies
\[
\tr(b_{AB}) = \frac{\tr(b_A) \, \tr(b_B)}{\tr(b_A \circ b_B)},
\quad \text{and} \quad
b_Y \le b_{AB} \quad \text{for all } Y \in \{A, B\}.
\]

\begin{remark} \label{rk:tralpha}
Since $ \tr_{\alpha}(\id_X)$ is nonzero, Corollary~\ref{cor:landau} implies that $ \tr(b) = \tr_{\alpha}(b) $ for all $ b \in \{b_A, b_B, b_{AB}\} $, and this scalar is positive because $ X $ is $ \alpha $-weakly positive and $ b $ is idempotent. Thus, $ \tr(b_A \circ b_B) = \frac{\tr(b_A) \tr(b_B)}{\tr(b_{AB})} $ is positive.
\end{remark}


Following the notation from \S \ref{sec:lattice}, we define \( b_{A \cap B} \) as \( i_{A \cap B} \circ i_{A \cap B}^* \), where \( i_{A \cap B} = i_Y \circ j_Y \) for \( Y \in \{A, B\} \).

\begin{lemma} \label{lem:Y>AnB}
For all $ Y \in \{A,B\} $, we have $ b_{A \cap B} \le b_Y $.  But $ \tr_{\alpha}(b_{A \cap B}) = \tr(b_{A \cap B}) $. Consequently, $ \tr(b_{A \cap B}) \le \tr(b_Y) $.
\end{lemma}

\begin{proof}
By definition, and the equality $ i_Y^* \circ i_Y = \id_Y $ (from Definition \ref{def:sub}):
\begin{align*}
b_Y \circ b_{A \cap B} &= i_Y \circ i_Y^* \circ (i_Y \circ j_Y) \circ (i_Y \circ j_Y)^* \\
&= i_Y \circ (i_Y^* \circ i_Y) \circ j_Y \circ (i_Y \circ j_Y)^* \\
&= (i_Y \circ j_Y) \circ (i_Y \circ j_Y)^* = b_{A \cap B}.
\end{align*}
Similarly, we find $ b_{A \cap B} \circ b_Y = b_{A \cap B} $. The second statement is immediate from Corollary \ref{cor:landau} since $A \cap B$ is also a Frobenius subalgebra of $C$ by Proposition \ref{prop:preInterWeak}. The last statement follows from Lemma \ref{lem:trftrg} and Remark \ref{rk:tralpha}.
\end{proof}


Let $b_{A+B}$ denote the idempotent $i_{A+B} \circ i_{A+B}^*$, where $i_{A+B}$ is defined as in Remark \ref{rk:union}. 

\begin{lemma} \label{lem:AB+}
The following equalities hold:
$$ \tr(b_{A+B}) = \tr(b_A) + \tr(b_B) - \tr(b_{A \cap B})  = \tr_{\alpha}(b_A) + \tr_{\alpha}(b_B) - \tr_{\alpha}(b_{A \cap B}) = \tr_{\alpha}(b_{A+B}) $$
\end{lemma}

\begin{proof}
Since $(A,B)$ is a coherent pair, then by Definition \ref{def:CoherentPair}, $A$, $B$, $A \cap B$ and $A+B$ are ambiently selfdual subobjects of $X$. By Definition \ref{def:AmbientlySelfdual} and Lemma \ref{lem:AmbientlyInter}, $A \cap B$ is a split subobject of both $A$ and $B$, so by Lemma~\ref{lem:Inter+Sum}, $A \oplus B \simeq (A \cap B) \oplus (A + B)$. Using Lemma \ref{lem:trfg} and the property $ i_Z^{**} = i_Z $ for $Z \in \{A, B, A \cap B, A+B\}$, we get $ \tr(b_Z) = \tr(i_Z \circ i_Z^*) = \tr(i_Z^* \circ i_Z) = \tr(\id_Z) $.  Therefore,
\begin{align*}
\tr(b_{A+B}) &= \tr(\id_{A+B}) = \tr(\id_{A \oplus B}) - \tr(\id_{A \cap B}) \\
&= \tr(\id_{A}) + \tr(\id_{B}) - \tr(\id_{A \cap B}) = \tr(b_A) + \tr(b_B) - \tr(b_{A \cap B}) \\
&= \tr_{\alpha}(b_A) + \tr_{\alpha}(b_B) - \tr_{\alpha}(b_{A \cap B}) = \tr_{\phi}(\id_{A}) + \tr_{\phi}(\id_{B}) - \tr_{\phi}(\id_{A \cap B}) \\
&=  \tr_{\phi}(\id_{A \oplus B}) - \tr_{\phi}(\id_{A \cap B}) =  \tr_{\phi}(\id_{A+B}) = \tr_{\alpha}(b_{A+B}),
\end{align*} 
Recall that $\alpha = \phi_X$ and $\phi$ is a quasi-pivotal structure (see Lemma \ref{lem:trfgpivotalquasi}).
\end{proof}
Throughout this section, it suffices to assume that $ X $ is $ \alpha $-weakly positive for some isomorphism $ \alpha $, except in the proof of Lemma~\ref{lem:AB+}, where we explicitly require that $ \alpha $ extends to a quasi-pivotal structure $ \phi $. Whether this extra assumption is genuinely necessary remains unclear. 

\begin{lemma} \label{lem:A+B>Y}
For all $ Y \in \{A, B\} $, we have $ b_Y \leq b_{A+B} $. But $ \tr_{\alpha}(b_{A+B}) = \tr(b_{A+B}) $. Thus $ \tr(b_Y) \leq \tr(b_{A+B}) $.
\end{lemma}

\begin{proof}
By definition, the equality $ i_{A+B}^* \circ i_{A+B} = \id_{A+B} $ and the notations in Remark \ref{rk:union},
\begin{align*}
b_{A+B} \circ b_Y &= (i_{A+B} \circ i_{A+B}^*) \circ (i_Y \circ i_Y^*) \\
&= i_{A+B} \circ i_{A+B}^* \circ (i_{A+B} \circ k_Y) \circ (i_{A+B} \circ k_Y)^* \\
&= i_{A+B} \circ (i_{A+B}^* \circ i_{A+B}) \circ k_Y \circ (i_{A+B} \circ k_Y)^* \\
&= (i_{A+B} \circ k_Y) \circ (i_{A+B} \circ k_Y)^* = b_Y.
\end{align*}
Similarly, we also find $ b_Y \circ b_{A+B} = b_Y $. The second statement is contained in Lemma \ref{lem:AB+}. The last statement follows from Lemma \ref{lem:trftrg} and Remark \ref{rk:tralpha}.
\end{proof}

%

\begin{lemma} \label{lem:joinidem}
Let $ P $ be a poset and let $ p, q_1, q_2 $ belong to $ P $ such that $ q_i \le p $ for $ i \in \{1,2\} $. If $ q_1 $ and $ q_2 $ have a unique supremum $ q_1 \vee q_2 $ (called the join), then $ q_1 \vee q_2 \le p $.
\end{lemma}

\begin{proof}
This statement is immediate.
\end{proof}

\begin{lemma} \label{lem:A+B<AB}
The inequality $ b_{A+B} \le b_{AB} $ holds. Consequently, $ \tr(b_{A+B}) \le \tr(b_{AB}) $.
\end{lemma}

\begin{proof}
This follows directly from the inequality $ b_Y \le b_{AB} $ and Lemma \ref{lem:joinidem}, since $ b_{A+B} = b_A \vee b_B $. The last statement follows from Lemma \ref{lem:trftrg}, Remark \ref{rk:tralpha} and Lemma \ref{lem:AB+}.
\end{proof}

By combining Lemmas \ref{lem:Y>AnB}, \ref{lem:AB+}, \ref{lem:A+B<AB}, and Corollary \ref{cor:landau}, for all $Y \in \{A,B\}$ we have:
$$
0 < \tr(b_Y) \le \tr(b_A) + \tr(b_B) - \tr(b_{A \cap B}) = \tr(b_{A+B}) \le \tr(b_{AB}) = \frac{\tr(b_A) \tr(b_B)}{\tr(b_A \circ b_B)},
$$
where positivity holds because $ b_Y $ is a nonzero idempotent and $X$ is $\alpha$-weakly positive. Therefore, we derive:
\begin{align*}
\tr(b_A \circ b_B) &\le \frac{\tr(b_A)\tr(b_B)}{\tr(b_A) + \tr(b_B) - \tr(b_{A \cap B})}, \\
\tr(b_A \circ b_B) - \tr(b_{A \cap B}) &\le \frac{\tr(b_A)\tr(b_B)}{\tr(b_A) + \tr(b_B) - \tr(b_{A \cap B})} - \tr(b_{A \cap B}), \\
&= \frac{\tr(b_A)\tr(b_B) - \tr(b_{A \cap B})(\tr(b_A) + \tr(b_B)) + \tr(b_{A \cap B})^2}{\tr(b_A) + \tr(b_B) - \tr(b_{A \cap B})},\\
&= \frac{(\tr(b_A)-\tr(b_{A \cap B})) (\tr(b_B) - \tr(b_{A \cap B}))}{\tr(b_A) + \tr(b_B) - \tr(b_{A \cap B})}.
\end{align*}
By Lemma \ref{lem:Y>AnB}, we have $b_{A \cap B} \le b_Y$ and
\[
(b_A - b_{A \cap B}) \circ (b_B - b_{A \cap B}) = b_A \circ b_B - b_{A \cap B}.
\]  
Moreover, if $ b_Y \neq b_{A \cap B} $, then $ \tr(b_Y - b_{A \cap B}) = \tr(b_Y) - \tr(b_{A \cap B}) = \tr_{\alpha}(b_Y) - \tr_{\alpha}(b_{A \cap B}) > 0 $ by Lemma \ref{lem:trftrg}, and 
\begin{align*}
\frac{\tr((b_A - b_{A \cap B}) \circ (b_B - b_{A \cap B}))}{\sqrt{\tr(b_A-b_{A \cap B})\tr(b_B-b_{A \cap B})}} &\le \frac{\sqrt{\tr(b_A-b_{A \cap B})\tr(b_B-b_{A \cap B})}}{\tr(b_A) + \tr(b_B) - \tr(b_{A \cap B})}, \\
&< \frac{\sqrt{\tr(b_A-b_{A \cap B})\tr(b_B-b_{A \cap B})}}{\tr(b_A-b_{A \cap B}) + \tr(b_B-b_{A \cap B})} \le \frac{1}{2}.
\end{align*}
The final inequality follows from the AM–GM inequality, which states that for two non-negative numbers $ x $ and $ y $, we have $ \frac{x + y}{2} \geq \sqrt{xy} $.

\begin{remark} \label{rk:zerocase}
We initially assumed $ \tr(b_A \circ b_B) $ to be nonzero; if it is zero, the inequality above holds trivially.\end{remark}

Let us define the \emph{formal angle} between $ b_A $ and $ b_B $ with respect to $ b_A \wedge b_B  = b_{A \cap B} $ as follows:
$$
\theta(b_A, b_B) :=
\begin{cases} 
\arccos\left(\tr\left( \frac{(b_A - b_{A \cap B})}{\sqrt{\tr(b_A-b_{A \cap B})}} \circ \frac{(b_B - b_{A \cap B})}{\sqrt{\tr(b_B-b_{A \cap B})}}\right) \right), & \text{if } b_A \neq b_B \\
0, & \text{if } b_A = b_B.
\end{cases}
$$

\begin{remark} \label{rk:formal}
The term \emph{angle} draws inspiration from the work in \cite{BDLR19}, while the descriptor \emph{formal} indicates that the trace $ \tr $ may not define an inner product space on $ \End_{\mathcal{C}}(X) $, as the tensor category is not assumed to be unitary, and thus it may not yield a genuine angle. After standard analytic continuation, $\cos(\arccos(z)) = z$ for all $z \in \mathbb{C}$.
\end{remark}

As shown above, if $b_A \neq b_B$, then $\theta(b_A, b_B) > \pi/3$, which is the key step in the proof of Watatani's theorem.%

\section{Watatani's finiteness theorem on tensor categories} \label{sec:Wat} 
This section is dedicated to the proof of Theorem \ref{thm:allfinite}—our most comprehensive generalization of Watatani’s finiteness theorem—along with its eight corollaries, naturally arriving at the unitary case.

\begin{definition} \label{def:minimal}
Let $ L $ be a poset and $ a \in L $. We define the subset $ m_a \subset L $ as the collection of elements $ b \in L $ such that if $ a < c \leq b $, then $ b = c $. In simpler terms, $ m_a $ is the set of minimal elements in $ L $ that are greater than $ a $.
\end{definition}

%

\begin{definition} \label{def:minimals}
Let $ L $ be a poset and $ a \in L $. Define the subset $ m^s_a \subset L $ as follows: $ m^0_a = \{a\} $ if $ s=0 $, and $ m^s_a = \bigcup_{b \in m^{s-1}_a} m_b $ if $ s \ge 1 $. In particular, $m^1_a = m_a$.
\end{definition}

\begin{remark} \label{rk:unionlat}
If $ L $ has finite height $ h $ and a bottom element $ 1 $, then $ L = \bigcup_{s=0}^h m^s_1 $.
\end{remark}

\begin{definition} \label{def:CoherentSublattice}
Let $X$ be a Frobenius algebra in an abelian rigid monoidal category. A sublattice $L$ of its Frobenius subalgebra poset (Definition~\ref{def:FrobSubPoset}) is called \emph{coherent} if it has a bottom element and, for all $[A],[B] \in L$, the pair $(A,B)$ is coherent (Definition~\ref{def:CoherentPair}) and $[A] \wedge [B] = [A \cap B]$.
\end{definition}

There are two main motivations for Definition \ref{def:CoherentSublattice}. First, it provides our largest setting in which Watatani's finiteness theorem can be generalized (see Theorem \ref{thm:allfinite}); second, it encompasses a large class of examples (see Proposition \ref{prop:Bcoherent}), properly covering the unitary case (see Proposition \ref{prop:UnitFrobSub}).

\begin{proposition} \label{prop:Bcoherent}
Let $X$ be a Frobenius algebra of finite length that is semisimple as an object in a $\mathbb{C}$-linear abelian rigid monoidal category, where the unit object $\mathbf{1}$ is simple.  
Then every Frobenius $\mathcal{B}$-subalgebra lattice (Theorem \ref{thm:Bsublattice}) forms a coherent sublattice of the Frobenius subalgebra poset of $X$.
\end{proposition}
\begin{proof}
By Lemma \ref{lem:b1nonzero}, the unit object $\mathbf{1}$ is a Frobenius subalgebra of $X$. Being evidently $\mathcal{B}$-rigid invariant, it serves as the bottom element of the lattice. The result then follows from Theorem \ref{thm:CoherentPairB}.
\end{proof}

\begin{theorem} \label{thm:allfinite}
Let $\mathcal{C}$ be a $\mathbb{C}$-linear abelian rigid monoidal category with a linear-simple unit object $\one$ and a quasi-pivotal structure $\phi$. Let $X$ be a finite-length connected Frobenius algebra in $\mathcal{C}$ that is $\phi_X$-weakly positive, where $\phi_X$ is an algebra isomorphism, and suppose that $\End_{\mathcal{C}}(X)$ is finite-dimensional. Then every coherent sublattice $L$ of the Frobenius subalgebra poset of $X$ is finite.
\end{theorem}
\begin{proof}
Following Proposition \ref{prop:biprojection}, we identify $L$ with $\{b_Y \ | \ [Y] \in L\}$. Let $b_0$ be the bottom element. The height of $L$ is finite, say $ h $, since $ X $ has finite length. By Remark \ref{rk:unionlat}, we have $ L =  \bigcup_{s=0}^h m^s_{b_0} $. The rest of the proof relies on the following lemma, where $\alpha:=\phi_X$.


\begin{lemma} \label{lem:min_finite}
Let $b$ be an element of $L$. The subset $m_b$ of $L$, as defined in Definition \ref{def:minimal}, is finite.
\end{lemma}

\begin{proof}
Note that $m_b = \{b_{A_i} \ | \ i \in I \}$, for some index set $I$, where $A_i$ is a Frobenius subalgebra of $X$. Define the vectors $$ v_i := \frac{b_{A_i} - b}{\sqrt{\tr(b_{A_i} - b)}} $$ for each $ i \in I $. Let $ V $ be the $\mathbb{R}$-subspace of $\End_{\mathcal{C}}(X)$ spanned by the vectors $(v_i)_{i \in I}$. 
Since $\End_{\mathcal{C}}(X)$ is finite-dimensional, there exists a finite subset $ J \subseteq I $ such that $(v_i)_{i \in J}$ forms a basis for $ V $.

We define a real inner product on $ V $ as follows:
$$
\left\langle \sum_{j \in J} \lambda_j v_j, \sum_{j \in J} \mu_j v_j \right\rangle := \sum_{j \in J} \lambda_j \mu_j.
$$
Let $ M $ be the matrix defined by $ ( \tr(v_i \circ v_j) )_{i,j \in J} $. For all  $ i, j \in I $, 
\begin{align*}
\langle v_i, M v_j \rangle &= \left\langle \sum_{k \in J} \lambda_k v_k, M \sum_{l \in J} \mu_l v_l \right\rangle \\ 
&= \sum_{k,l \in J} \lambda_k \mu_l \left\langle v_k, M v_l \right\rangle \\
&= \sum_{k,l \in J} \lambda_k \mu_l \tr(v_k \circ v_l)   \\
&= \tr((\sum_{k \in J} \lambda_k v_k) \circ (\sum_{l \in J} \mu_l v_l)) \\
&= \tr(v_i \circ v_j).
\end{align*}
If $ i \neq j $ then $ b_{A_i \cap A_j} = b $ by minimality, since $[A_i \cap A_j] = [A_i] \wedge [A_j]$. Therefore, based on the end of \S \ref{sec:Angle}, 
$$
\langle v_i, M v_j \rangle = \tr(v_i \circ v_j) = \cos(\theta(b_{A_i}, b_{A_j})) < \frac{1}{2}.
$$

Assume there exists $ i \in I $ and a sequence $ (i_n)_{n \in \mathbb{N}} $ such that $ i \neq i_n $ for all $ n $, and the angle $ \angle(v_i, v_{i_n}) $ tends to zero in the real inner product space $ V $.
As a result, there exists a sequence of positive real numbers $ (\alpha_n) $ such that 
$
\lim_{n \to \infty} (\alpha_n v_{i_n}) = v_i.
$
For all $ j \in I $, we have 
$$
\langle v_j, M v_j \rangle = \tr(v_j \circ v_j) = \tr\left(\left(\frac{b_{A_j} - b}{\sqrt{\tr(b_{A_j} - b)}}\right) \circ \left(\frac{b_{A_j} - b}{\sqrt{\tr(b_{A_j} - b)}}\right)\right) = \text{tr}\left(\frac{b_{A_j} - b}{\text{tr}(b_{A_j} - b)}\right) = 1.
$$
By continuity, we find that 
$$
\lim_{n \to \infty} \langle \alpha_n v_{i_n}, M \alpha_n v_{i_n} \rangle = \langle v_i, M v_i \rangle = 1.
$$
However, since 
$$
\langle \alpha_n v_{i_n}, M \alpha_n v_{i_n} \rangle = \alpha_n^2 \langle v_{i_n}, M v_{i_n} \rangle = \alpha_n^2,
$$
it follows that $\lim_{n \to \infty} \alpha_n = 1$.
Again, by continuity, we have 
$$
\lim_{n \to \infty} \langle v_{i}, M v_{i_n} \rangle = \lim_{n \to \infty} \alpha_n \langle v_{i}, M v_{i_n} \rangle = \lim_{n \to \infty} \langle v_{i}, M \alpha_n v_{i_n} \rangle = \langle v_{i}, M v_i \rangle = 1.
$$
This contradicts the inequality $ \langle v_{i}, M v_{i_n} \rangle < \frac{1}{2} $, which holds because we assumed that $ i \neq i_n $ for all $ n $. Thus, our initial assumption must be incorrect. This means that there exists a constant $ \beta > 0 $ such that for all $ i, j \in I $ with $ i \neq j $, the angle $ \angle(v_i, v_j) \ge \beta $. Since $ V $ is finite-dimensional, the index set $ I $ must also be finite.
\end{proof}

By Definition \ref{def:minimals}, Lemma \ref{lem:min_finite}, and induction, each set $ m^s_{b} $ is finite. The result follows by Remark \ref{rk:unionlat}. 
\end{proof}

\begin{remark} \label{rk:estimate}
An estimate of $\beta$, possibly involving $\Vert M \Vert$, should enable the derivation of an upper bound for the size of the finite lattice, analogous to \cite{BDLR19}.
\end{remark}

Let us provide a list of eight corollaries of Theorem \ref{thm:allfinite}: 

\begin{corollary} \label{cor:allfinite1.5}
Let $\mathcal{C}$ be a tensor category over $\mathbb{C}$ with a quasi-pivotal structure $\phi$. Let $X$ be a connected Frobenius algebra in $\mathcal{C}$ that is $\phi_X$-weakly positive with $\phi_X$ an algebra isomorphism. Then every coherent sublattice of the Frobenius subalgebra poset of $X$ is finite.
\end{corollary}
\begin{proof}
This follows directly from Theorem \ref{thm:allfinite}.
\end{proof}

\begin{corollary} \label{cor:allfinite1.6}
Let $X$ be a connected Frobenius algebra in a $\phi$-weakly positive tensor category such that $\phi_X$ an algebra isomorphism. Then every coherent sublattice of the Frobenius subalgebra poset of $X$ is finite.
\end{corollary}
\begin{proof}
This follows directly from Corollary \ref{cor:allfinite1.5}.
\end{proof}

\begin{corollary} \label{cor:allfinite1.75b}
Let $\mathcal{C}$ be a semisimple tensor category over $\mathbb{C}$, so weakly positive for some quasi-pivotal strucutre $\phi$. Let $X$ be a connected Frobenius algebra such that $\phi_X$ an algebra morphism. Then every coherent sublattice of the Frobenius subalgebra poset of $X$ is finite. In particular, every Frobenius $\mathcal{B}$-subalgebra lattice is finite.
\end{corollary}
\begin{proof}
Immediate by Corollary \ref{cor:allfinite1.6} and Theorem \ref{thm:semisweakly}. The last sentence applies Proposition \ref{prop:Bcoherent}.
\end{proof}
%

\begin{lemma} \label{lem:QuasiPivAlgMap}
Let $\mathcal{C}$ be a rigid monoidal category. Let $(X,m)$ be an algebra object in $\mathcal{C}$ such that $X^{**} = X$ and $m^{**} = m$. If there is a quasi-pivotal structure $\phi$ such that $\phi_{X \otimes X} = \phi_X \otimes \phi_X$, then $\phi_X$ is an algebra morphism.
\end{lemma}

\begin{proof}
By Definition \ref{def:quasipivotal}, a quasi-pivotal structure is a natural isomorphism from the identity functor $\id_{\mathcal{C}}$ to the double dual functor $(\_)^{**}$. By naturality, we have
\(
m \circ \phi_X = \phi_{X \otimes X} \circ m^{**}.
\)
By assumption, $\phi_{X \otimes X} = \phi_X \otimes \phi_X$. Substituting this into the equation above and using that $m^{**} = m$, we conclude that $\phi_X$ is an algebra morphism.
\end{proof}

Regarding Theorem \ref{thm:allfinite} and its corollaries, it follows from Lemmas \ref{lem:multdualcomult} and \ref{lem:QuasiPivAlgMap} that if $\phi_X$ satisfies $\phi_{X \otimes X} = \phi_X \otimes \phi_X$, then $\phi_X$ is necessarily an algebra isomorphism.


\begin{lemma} \label{lem:PivAlgMap}
Let $\mathcal{C}$ be a rigid monoidal category with a pivotal structure $\phi$. Let $(X,m)$ be an algebra object in $\mathcal{C}$ such that $X^{**} = X$ and $m^{**} = m$. Then $\phi_X$ is an algebra morphism.
\end{lemma}

\begin{proof}
This follows directly from Lemma \ref{lem:QuasiPivAlgMap}, since a pivotal structure, being monoidal, satisfies the identity $\phi_{X \otimes X} = \phi_X \otimes \phi_X$.
\end{proof}

\begin{corollary} \label{cor:allfinite1.75}
Let $\mathcal{C}$ be a tensor category over $\mathbb{C}$ with pivotal structure $\phi$. Let $X$ be a connected Frobenius algebra in $\mathcal{C}$ that is $\phi_X$-weakly positive. Then every coherent sublattice of the Frobenius subalgebra poset of $X$ is finite.
\end{corollary}
\begin{proof}
The result follows from Lemma \ref{lem:multdualcomult}, Lemma \ref{lem:PivAlgMap} and Corollary \ref{cor:allfinite1.5}.
\end{proof}

By applying semisimplification together with Definition~\ref{def:RigidInvariantNSS} for rigidity invariance without the assumption of semisimplicity, we obtain the following corollary. We note, however, that a more general statement can be derived by considering the conditions in Remark~\ref{rk:zeronil}.
\begin{corollary} \label{cor:allfinite1}
Let $\mathcal{C}$ be a tensor category over $\mathbb{C}$ with a spherical structure $\phi$. Let $X$ be a connected Frobenius algebra in $\mathcal{C}$ that is $\phi_X$-weakly positive. Then every Frobenius $\mathcal{B}$-subalgebra poset is finite.
\end{corollary}

\begin{proof}
The result follows from Proposition \ref{prop:posetsemi} and Corollary \ref{cor:allfinite1.75b}. Note that as an object in a tensor category, $ X $ has finite length and can be uniquely decomposed (up to isomorphism) into a direct sum of finitely many indecomposable components (Krull-Schmidt theorem). Due to weak positivity, all these indecomposables have nonzero dimensions. Consequently, $ \overline{X} $ in the semisimplification $ \overline{\mathcal{C}} $ also admits a unique decomposition (up to isomorphism) as a direct sum of finitely many simple objects corresponding to the aforementioned indecomposables.
\end{proof}

\begin{corollary} \label{cor:allfinite2b}
Let $X$ be a connected Frobenius algebra in a positive tensor category. Then every coherent sublattice of the Frobenius subalgebra poset of $X$ is finite.
\end{corollary}

\begin{proof}
The result follows from Definition \ref{def:PositiveTensor}, Remark \ref{rk:PosWpos}, Proposition \ref{prop:WeaklyPositiveTensor} and Corollary \ref{cor:allfinite1.75}.
\end{proof}

\begin{corollary} \label{cor:allfinite2}
Let $X$ be a connected Frobenius algebra in a positive semisimple tensor category. Then every Frobenius $\mathcal{B}$-subalgebra lattice is finite.
\end{corollary}

\begin{proof}
Immediate from Corollary \ref{cor:allfinite2b} and Proposition \ref{prop:Bcoherent}.
\end{proof}

The following corollary demonstrates that our framework recovers the classical version of Watatani's finiteness theorem in the unitary setting, where rigidity invariance is not required:

\begin{corollary} \label{cor:allfiniteUnitary}
Let \( X \) be a connected unitary Frobenius algebra in a unitary tensor category. Then its unitary Frobenius subalgebra lattice is finite.
\end{corollary}

\begin{proof}
The result follows from Proposition~\ref{prop:UnitFrobSub}, Proposition~\ref{prop:C*positive}, and Corollary~\ref{cor:allfinite2}.
\end{proof}

Unitary tensor categories are positive in the sense of Definition~\ref{def:PositiveTensor}, as shown in Proposition~\ref{prop:C*positive}. However, the converse may not hold; pseudo-unitary fusion categories also satisfy positivity, as established in Proposition~\ref{prop:pseudo}.

\begin{corollary} \label{cor:allfinite3}
Let $X$ be a connected Frobenius algebra in a pseudo-unitary fusion category over $\mathbb{C}$. Then every coherent sublattice of the Frobenius subalgebra poset of $X$ is finite. In particular, every Frobenius $\mathcal{B}$-subalgebra lattice is finite.
\end{corollary}

\begin{proof}
The result follows from Proposition \ref{prop:pseudo}, Corollaries \ref{cor:allfinite2b} and \ref{cor:allfinite2}.
\end{proof}



We conclude this section with several questions. First, regarding a general version of Watatani's theorem in tensor categories:

\begin{question} \label{qu:allfinitetensor}
Let $X$ be a connected Frobenius algebra in a tensor category over $\mathbb{C}$. 
Is the Frobenius subalgebra poset of $X$ necessarily a lattice? Furthermore, is it always finite?
\end{question}

Next, we restrict our attention to coherent sublattices in the semisimple setting:
\begin{question} \label{qu:allfinite}
Let $X$ be a connected Frobenius algebra in a semisimple tensor category over $\mathbb{C}$. 
Is every coherent sublattice of its Frobenius subalgebra poset finite?
\end{question}

\begin{question} \label{qu:allfinite4}
Let $\mathcal{C}$ and $\mathcal{D}$ be Grothendieck-equivalent semisimple tensor categories over $\mathbb{C}$. Suppose that for every connected Frobenius algebra in $\mathcal{C}$, every coherent sublattice of its Frobenius subalgebra poset is finite. Does the same property necessarily hold for $\mathcal{D}$?
\end{question}

According to Corollary \ref{cor:allfinite3}, if Question \ref{qu:allfinite4} has a positive answer, then any potential counterexample to Question \ref{qu:allfinite} would necessarily arise from a fusion category that is not Grothendieck-equivalent to a pseudo-unitary one. Such categories are rare in the literature; the first known example, the rank 6 modular fusion category $\mathcal{C}(\mathfrak{so}_5, 3/2)_{\mathrm{ad}}$, was described in \cite{Scho22}. It would be of interest to classify the connected Frobenius algebras within this category.

%
%
%
%



\section{Inclusion of $C^*$-algebras} \label{sec:IncC*}
Let $ A \subset B $ be a unital inclusion of $ \mathbbm{k} $-algebras. Define $ \mathcal{C} $ as the $ \mathbbm{k} $-linear abelian monoidal category of $ A $-bimodules, denoted $ \mathrm{Bimod}(A) $. More generally for rings, we refer the reader to \cite[Example 1.6]{Phh02}. It is convenient to use the (abuse of) notation $ A' \cap B $ to denote the commutant of $ A $ within $ B $, defined as $\{ b \in B \ | \ ab = ba \ \forall a \in A \}$, even though the term $ A' $ on its own may not make sense in general.

\begin{lemma} \label{lem:relative}
The $ \mathbbm{k} $-vector spaces $ \Hom_{\mathcal{C}}(_A A_A, {_A B_A}) $ and $ A' \cap B $ are isomorphic.
\end{lemma}

\begin{proof}
Consider the linear map $ T: A' \cap B \to \Hom_{\mathcal{C}}(_A A_A, {_A B_A}) $ defined by $ T(x) = \hat{x}: a \mapsto xa $. The map $ \hat{x} $ is a morphism in $ \mathcal{C} $ because $ \hat{x}(a_1 a a_2) = x a_1 a a_2 = a_1 x a a_2 = a_1 \hat{x}(a) a_2 $, where $ a_1 \in A$ commutes with $ x \in A' \cap B $. Next, define the linear map $ S: \Hom_{\mathcal{C}}(_A A_A, {_A B_A}) \to A' \cap B $ by $ S(f) = f(1) $. We will show that $ S $ and $ T $ are inverses of each other. First, we have $T(S(f)) = T(f(1)) = \widehat{f(1)}$. Since $ \widehat{f(1)}(a) = f(1)a = f(a) $ (because $ f $ is a bimodule morphism), we find $ \widehat{f(1)} = f $, implying $ T \circ S = \id $. Next, $T(S(x)) = T(\hat{x}) = \hat{x}(1) = x$, for $ x \in A' \cap B $, which shows that $ S \circ T = \id $. Thus, $ S $ and $ T $ are indeed inverses.
\end{proof}

\begin{definition} \label{def:irred}
A unital inclusion of algebras $ A \subseteq B $ is called \emph{irreducible} if $ A' \cap B $ is one-dimensional.
\end{definition}

\begin{lemma} \label{lem:IrredConn}  
A unital inclusion of algebras $ A \subseteq B $ is irreducible if and only if $ {_A B_A} $ is connected (and consequently, $ A $ and $ B $ are centerless, and $ \one $ is linear-simple).  
\end{lemma}

\begin{proof} 
Since $ {_A A_A} = \one $ serves as the unit in $ \mathcal{C} $, the result follows directly from Definitions \ref{def:connected} and \ref{def:irred}, together with Lemma \ref{lem:relative}.  
Furthermore, since $ A' \cap A $ and $ B' \cap B $ are both contained in $ A' \cap B $, the connectedness of $ {_A B_A} $ implies that $ A' \cap A $, $ B' \cap B $, and $ \Hom_{\mathcal{C}}(\one, \one) $ are all one-dimensional. This ensures that $ A $ and $ B $ are centerless, and that $ \one $ is linear-simple.
\end{proof}

Suppose $A \subseteq B$ is a unital inclusion of $C^*$-algebras equipped with a faithful conditional expectation $E: B \to A$. We view $B$ as a right $A$-$A$ correspondence (or Hilbert $C^*$-module), where the left and right actions are given by multiplication, and the right $A$-valued inner product is defined by $\inner{x}{y} = E(x^* y)$, for all $x, y \in B$. Let $A \subseteq C \subseteq B$ be an intermediate $C^*$-algebra. Then $C$ also inherits the structure of an $A$-$A$ correspondence with the right $A$-valued inner product given by $\inner{x}{y} = E(x^* y)$, for all $x, y \in C$.
Following \cite[Definition 2.4]{IW14} and \cite[\S 2.2.1]{GK24}:

\begin{definition} \label{def:Ecomp}
Let $A \subseteq B$ be a unital inclusion of $\mathrm{C}^*$-algebras equipped with a finite-index conditional expectation $E_A^B : B \to A$. An intermediate $\mathrm{C}^*$-subalgebra $A \subseteq C \subseteq B$ is said to be \emph{$E$-compatible} if there exists a conditional expectation $E_C^B : B \to C$ satisfying the compatibility condition $E_A^C \circ E_C^B = E_A^B$, where $E_A^C := E_A^B|_C : C \to A$ denotes the restriction of $E_A^B$ to $C$.
\end{definition}

\begin{proposition}  \label{prop:AdjointEcomp}
    Let $i: {_A}C_A \to {_A}B_A$ be the inclusion map (which is clearly an $A$-$A$ bimodule map). Then $i$ is adjointable if and only if the triple $(A \subseteq C \subseteq B)$ is $E$-compatible.
\end{proposition}

\begin{proof}
    Suppose the inclusion $(A \subseteq C \subseteq B)$ is $E$-compatible. By definition, there exists a conditional expectation $E^B_C: B \to C$ such that $E = E^C_A \circ E^B_C$. Consequently, for any $b \in B$ and $c \in C$, we have:
    \begin{equation*}
        \inner{i(c)}{b} = E^B_A (c^*b) = E^C_A (c^* E^B_C(b)) = \inner{c}{E^B_C(b)}.
    \end{equation*}
    By the faithfulness of the conditional expectations, it follows that $i$ is adjointable with $i^\dagger = E^B_C$.

    Conversely, suppose $i$ is adjointable. Then, for each $b \in B$ and $c \in C$:
    \begin{equation*}
        E^B_A (c^*b) = \inner{i(c)}{b} = \inner{c}{i^\dagger(b)} = E^C_A (c^* i^\dagger(b)).
    \end{equation*}
    Since $i$ is a unital inclusion, $i^\dagger$ is a projection of norm one from $B$ onto $C$. By Tomiyama's theorem \cite{T57}, $i^\dagger$ is a conditional expectation $E^B_C: B \to C$. Setting $c = 1$ in the identity above, we obtain $E^B_A = E^C_A \circ E^B_C$, which confirms that the inclusion is $E$-compatible.
\end{proof}

The following result was implicit in \cite{IW14} and was subsequently proven in detail in \cite[Corollary 4.2]{GK24}, allowing for a generalization beyond the irreducible case (\cite[Theorem 4.1]{GK24}). By applying our reformulation of Watatani's theorem on subfactors within the unitary tensor category setting (Corollary \ref{cor:allfiniteUnitary}) to $C^*$-algebras, we recover this deep result.

\begin{corollary} \label{cor:InoWat}
The $E$-compatible intermediate $C^*$-algebras of a finite-index unital irreducible inclusion of $C^*$-algebras form a finite lattice.
\end{corollary}
\begin{proof}
According to \cite{Wat90}, a finite index unital inclusion of $ C^* $-algebras is defined through a faithful conditional expectation $ E_A: B \to A $. As a result, based on \cite[Lemma 3.11]{CHJP22}, the object $ {_A B_A} $ is a unitary Frobenius algebra object within the unitary tensor category $ \mathrm{Bimod}(A) $, and it is connected by Lemma \ref{lem:IrredConn}. The result follows by Corollary \ref{cor:allfiniteUnitary}, since an intermediate $ C^* $-algebra corresponds to a unitary Frobenius subalgebra of $ {_A B_A} $ if and only if it is E-compatible by Proposition \ref{prop:AdjointEcomp}. 
\end{proof}

%

The following example of a unital irreducible finite-index inclusion of $C^*$-algebras features uncountably many intermediate $C^*$-algebras. This shows that $E$-compatibility is not always guaranteed for such an intermediate, proving that this assumption used in Corollary \ref{cor:InoWat} is necessary and cannot be omitted.

\begin{example} \label{ex:coburn}
Let $A_1$ be a centerless unital $C^*$-algebra equipped with an outer action $\alpha_1$ of the cyclic group $C_2$ (for instance, the Cuntz algebra $\mathcal{O}_2$). Let $A_2$ be a centerless unital $C^*$-algebra possessing uncountably many closed two-sided ideals $I_i$ (e.g., the Toeplitz algebra $\mathcal{T}$, whose ideal theory is determined by the quotient $\mathcal{T}/\mathcal{K} \cong C(S^1)$, see \cite[Theorem 2]{Co67}).
Define $A = A_1 \otimes A_2$ as the minimal tensor product, which remains centerless by \cite[Corollary 1]{HW73}. Let $\alpha = \alpha_1 \otimes \mathrm{id}$ and consider the crossed product $B = A \rtimes_\alpha C_2 = A \oplus Au$. The inclusion $A \subseteq B$ is irreducible (Lemma \ref{lem:irred}) with Watatani index 2.
For each ideal $I_i \subseteq A_2$, let $J_i = A_1 \otimes I_i$. Since each $J_i$ is $\alpha$-invariant, the subspace $D_i = A \oplus J_i u$ forms an intermediate $C^*$-algebra such that $A \subseteq D_i \subseteq B$. Because there are uncountably many distinct ideals $I_i$, we obtain uncountably many distinct intermediate $C^*$-algebras $D_i$.

\end{example}
Notably, there exists no conditional expectation from $B$ onto $D_i$ unless $D_i$ is $A$ or $B$ (Lemma~\ref{lem:CE}). Combining Corollary~\ref{cor:InoWat}, Remark~\ref{rk:Ecomp}, and a suitable topological argument should yield a more conceptual proof of Lemma~\ref{lem:CE}. Finally, if $B \subseteq B(H)$ with $B''$ a factor, then $D_i'' = B''$ unless $D_i = A$.
%
%

\begin{lemma} \label{lem:irred}
The inclusion $A \subseteq B$ from Example \ref{ex:coburn} is irreducible.
\end{lemma}

\begin{proof}
Let $x = a + bu \in A' \cap B$. For any $c \in A$, the relation $cx = xc$ implies $ca = ac$ and $cb = b\alpha(c)$. Since $A$ is centerless, $ca = ac$ for all $c \in A$ forces $a \in \mathbb{C}\id$. Taking the adjoint of the second equation yields $b^*c = \alpha(c)b^*$, and thus $b^*b\alpha(c) = b^*cb = \alpha(c)b^*b$. Because $\alpha$ is an automorphism, this shows $b^*b \in Z(A) = \mathbb{C}\id$. Similarly, $bb^* \in \mathbb{C}\id$. If $b \neq 0$, normalizing $b$ produces a unitary $v \in A$ satisfying $cv = v\alpha(c)$ for all $c \in A$. This implies $\alpha = \mathrm{Ad}_{v^*}$, contradicting the assumption that $\alpha$ is an outer automorphism. Hence $b = 0$, and we conclude $A' \cap B = \mathbb{C}\id$.
\end{proof}

%
%

\begin{lemma} \label{lem:CE}
There exists a conditional expectation from $B$ onto $D_i$ if and only if $D_i = A$ or $B$.
\end{lemma}
%
%
%

\begin{proof}
Recall that $B = A \oplus Au$ and $D_i = A \oplus J_i u$, where $J_i$ is a two-sided closed ideal in $A$. The case $D_i=A$ is clear. Assume $D_i \neq A$.
The existence of a conditional expectation from $B$ onto $D_i$ is equivalent to the existence of one from $A$ onto $J_i$. 
A two-sided closed ideal $J \subseteq A$ is the range of a conditional expectation if and only if there exists a two-sided closed ideal $I$ such that $A = I \oplus J$ (see, e.g., \cite[Remark 2]{PR17}). This decomposition holds if and only if there exists a central projection $p \in A$ such that $J = pA$, obtained by decomposing $\id \in A$ via $I \oplus J$. For $J$ nonzero, since $A$ is centerless, then $p=\id$, i.e. $J=A$. Consequently, $D_i=B$, which completes the proof.
\end{proof}

%
%

\begin{remark} \label{rk:Ecomp}
For a unital irreducible inclusion $A \subseteq B$ ($A' \cap B = \mathbb{C}\id$) of finite Watatani index, the space of $A$-bimodule maps is one-dimensional:
\( \mathrm{Hom}_{A,A}(B, A) \cong A' \cap B = \mathbb{C}\id. \)
Consequently, there exists a unique faithful conditional expectation $E^B_A$. It follows that any intermediate $C^*$-algebra $D$ admitting faithful conditional expectations $E^B_D$ and $E^D_A$ is necessarily $E$-compatible, as their composition must coincide with $E^B_A$ by uniqueness.
\end{remark}

Since $E$-compatibility always holds in the simple case \cite{IW14}, Corollary~\ref{cor:InoWat} represents a genuine strengthening of \cite[Corollary~3.9]{IW14} as it removes the requirement of simplicity.
As shown in Lemma \ref{lem:IrredConn}, the algebras $A$ and $B$ are centerless when the inclusion is irreducible. Consequently, it is essential to consider non-simple, centerless $C^*$-algebras (see Proposition \ref{prop:CenterlessNonsimple}). We begin by recalling some basic definitions and preliminary results.
\begin{definition} \label{def:simple}
A $ C^* $-algebra (or von Neumann algebra) is called \emph{simple} if it has no nonzero proper ideals that are closed under the operator norm topology (or the weak operator topology, respectively).
\end{definition}

\begin{proposition} \label{prop:factor}
A von Neumann algebra $ M $ is simple if and only if it is centerless (i.e., a factor).
\end{proposition}
\begin{proof}
The simplicity of a factor is specifically addressed in \cite[Proposition A.3.1]{JS97}. Conversely, if $ M $ is not a factor, it contains a nontrivial central projection $ p \in M \cap M' $. Thus, $ pMp $ forms a nontrivial closed ideal in $ M $, demonstrating that $ M $ is not simple.
\end{proof}

\begin{proposition} \label{prop:centerless}
Let $ H $ be a separable Hilbert space, and let $ A \subset B(H) $ be a unital $ C^* $-algebra. Let $ M := A'' $ be the von Neumann algebra generated by $ A $. If $ M $ is a factor then $ A $ is centerless.  
\end{proposition}
\begin{proof}
This follows immediately since $ (A \cap A')'' = \overline{A \cap A'}^{\text{wot}} \subseteq (\overline{A}^{\text{wot}} \cap \overline{A'}^{\text{wot}}) = M \cap M' $.
\end{proof}

Recall that the von Neumann algebra $ L(\Gamma) $ of a countable group $ \Gamma $ is a factor if and only if $ \Gamma $ is ICC (infinite conjugacy classes). 

\begin{proposition} \label{prop:CenterlessNonsimple}
Let $ \Gamma $ be a countable ICC group with a non-trivial amenable normal subgroup $N \subseteq \Gamma $. Then the reduced $ C^* $-algebra $ C_r^{\ast}(\Gamma) $ is centerless but non-simple.
\end{proposition}
\begin{proof}
Since $ \Gamma $ is a countable ICC group, $ L(\Gamma) $ is a factor. Therefore, by Proposition \ref{prop:centerless}, $ C_r^{\ast}(\Gamma) $ is centerless because $ C_r^{\ast}(\Gamma)'' = L(\Gamma) $. The non-simplicity follows from \cite[Proposition 3]{Ha07}.
\end{proof}

In particular, for all non-trivial ICC groups $ G $ and $ \Gamma $, with $ G $ being amenable (e.g., $ S_{\infty} $), $ C_r^{\ast}(G) $ and $ C_r^{\ast}(G \times \Gamma) $ are centerless, non-simple $ C^* $-algebras. In contrast, as noted in \cite{Ha07}, for every torsion-free non-elementary hyperbolic group $\Gamma$, the C*-algebra $C_r^*(\Gamma)$ is simple. This implies that $\Gamma$ has no non-trivial amenable normal subgroups (i.e., it has a trivial amenable radical) and is ICC.


\section{Hopf algebras} \label{sec:hopf}
By \cite{LS69, Par71, Skr07}, every finite-dimensional Hopf algebra, as well as each of its left coideal subalgebras, is a Frobenius algebra in $\VVec$. However, the application of Corollary~\ref{cor:allfinite3} requires a \emph{connected} Frobenius algebra object in a tensor category. Since an object of $\VVec$ is connected if and only if it is trivial, these classical results are not sufficient in the present setting. This motivates the refinement provided by Theorem~\ref{thm:HopfFrobRep}, formulated in the tensor category $\Corep(H)$ (see Definition~\ref{def:Corep}). As an application, assuming a positive answer to the natural Question~\ref{Q:TwoCoideal}, partially addressed by Theorem~\ref{thm:NonDegSum}, we recover \cite[Theorem~3.6]{EW14}.



\subsection{Basic definitions} \label{sub:HopfDef}
Let $\field$ be a field, and let $\mathcal{V}$ denote the tensor category of vector spaces over $\field$, with unit $\one=\field$, and braided by the flip $c_{U,V} : U \otimes V \longrightarrow V \otimes U$, $c_{U,V}(u \otimes v) = v \otimes u$. All definitions are purely categorical (with some elementwise reformulations provided), so that they remin valid for any braided tensor category \(\mathcal{V}\).

\begin{definition}
A \emph{bialgebra} in $\mathcal{V}$ is a quintuple $(H, m, e, \Delta, \varepsilon)$ consisting of:

\begin{itemize}
    \item an object $H$ in $\mathcal{C}$,
    \item multiplication morphism $m: H \otimes H \to H$,
    \item unit morphism $e: \one \to H$,
    \item comultiplication morphism $\Delta: H \to H \otimes H$,
    \item counit morphism $\varepsilon: H \to \one$,
\end{itemize}
that satisfies the following axioms:
\begin{enumerate}
    \item \emph{Associativity:} $m \circ (m \otimes \id) = m \circ (\id \otimes m)$,
    \item \emph{Unit axiom:} $m \circ (e \otimes \id) = \id = m \circ (\id \otimes e)$,
    \item \emph{Coassociativity:} $(\Delta \otimes \id) \circ \Delta = (\id \otimes \Delta) \circ \Delta$,
    \item \emph{Counit axiom:} $(\varepsilon \otimes \id) \circ \Delta  = \id = (\id \otimes \varepsilon) \circ \Delta$,
    \item \emph{Compatibility axioms:} $\Delta$ and $\varepsilon$ are unital algebra morphisms, 
\[    
    \Delta \circ m = (m \otimes m) \circ (\id_H \otimes c_{H,H} \otimes \id_H) \circ (\Delta \otimes \Delta), \ 
    \Delta \circ e = e \otimes e, \ 
    \varepsilon \circ m = \varepsilon \otimes \varepsilon, \ 
    \varepsilon \circ e = \id_\one,
\]   
which can be reformulated elementwise as follows, where $e(1) = 1_H$: for all $a,b \in H$,
\[
    \Delta(ab) = \Delta(a)\Delta(b), \ \Delta(1_H) = 1_H \otimes 1_H, \ \varepsilon(ab) = \varepsilon(a)\varepsilon(b), \ \varepsilon(1_H) = 1.
\]
Equivalently, this means that \(m\) and \(e\) are counital coalgebra morphisms.
\end{enumerate}
\end{definition}

\begin{definition}
An \emph{antipode} is a morphism $S: H \to H$ that satisfies the \emph{antipode axiom}:
\begin{equation*}
    m \circ (\id \otimes S) \circ \Delta = e \circ \varepsilon = m \circ (S \otimes \id) \circ \Delta,
\end{equation*}
which can be reformulated elementwise as follows, where $\Delta(a) = a_{(1)} \otimes a_{(2)}$ is the Sweedler notation: for all $a \in H$,
\begin{equation*}
     a_{(1)}S(a_{(2)}) = \varepsilon(a)1_H = S(a_{(1)})a_{(2)}.
\end{equation*}
\end{definition}

Recall that $S$ is an antihomomorphism of unital algebra and counital coalgebras (see e.g., \cite[Proposition 1.3.12]{Tim08}), 
\[
    S \circ m = m \circ (S \otimes S) \circ c_{H,H}, \ 
    S \circ e = e, \ 
    \Delta \circ S = c_{H,H} \circ (S \otimes S) \circ \Delta, \
    \epsilon \circ S = \epsilon,
\]
which can be reformulated elementwise as follows: for all $a,b,h$ in $H$,
$$S(ab) = S(b)S(a), \  S(1_H) = 1_H, \ (S(h))_{(1)} \otimes (S(h))_{(2)} = S(h_{(2)}) \otimes S(h_{(1)}), \ \varepsilon(S(h)) = \varepsilon(h).$$

\begin{definition} \label{def:hopf}
A \emph{Hopf algebra} in $\mathcal{V}$ is a bialgebra $(H, m, e, \Delta, \varepsilon)$ with an invertible antipode $S$.
\end{definition}
A Hopf algebra in $\mathcal{V}$ is usually referred to as a finite-dimensional Hopf algebra over the field $\field$.

\begin{definition} \label{def:Corep}
Let $H = (H, m, e, \Delta, \varepsilon, S)$ be a Hopf algebra in $\mathcal{V}$. The tensor category of left $H$-comodules, denoted $\mathcal{C} = \Corep(H)$, is defined as follows:

\begin{itemize}
    \item \emph{Objects:} The objects $V$ in $\mathcal{V}$ equipped with a morphism $\rho_V : V \to H \otimes V$ in $\mathcal{V}$ satisfying
    \[
        (\Delta \otimes \mathrm{id}_V) \circ \rho_V = (\mathrm{id}_H \otimes \rho_V) \circ \rho_V
        \ \text{ and } \ 
        (\varepsilon \otimes \mathrm{id}_V) \circ \rho_V = \mathrm{id}_V.
    \]
It is referred to as a \emph{coaction}, and elementwise it is written as $\rho_V(v) = v_{(-1)} \otimes v_{(0)}$ using Sweedler's notation.  
    \item \emph{Morphisms:} The morphisms $f : V \to W$ in $\mathcal{V}$ satisfying 
    \[
        \rho_W \circ f = (\mathrm{id}_H \otimes f) \circ \rho_V.
    \]
    \item \emph{Tensor product:} For $V, W $ in $\mathcal{C}$, their tensor product is the object $V \otimes W$ in $\mathcal{V}$ equipped with  the coaction
    \[
\rho_{V \otimes W} =  (m \otimes \id_V \otimes \id_W) \circ (\id_H \otimes c_{V,H} \otimes \id_W) \circ (\rho_V \otimes \rho_W).
\]
Elementwise, for all $v$ in $V$ and all $w$ in $W$,
    \[
        \rho_{V \otimes W}(v \otimes w) =  v_{(-1)} w_{(-1)} \otimes v_{(0)} \otimes w_{(0)}.
    \]
   \item \emph{Unit object:} The unit $\one$ in $\mathcal{V}$ with trivial coaction
    \(
        \rho_{\one} = e \otimes \id_{\one}
    \)
(elementwise,    
    \(
        \rho_{\field}(c) = 1_H \otimes c,
    \)
     for all $c$ in $\field$).    

\end{itemize}
\end{definition}

\begin{definition} \label{def:coideal}
Let $H$ be a Hopf algebra in $\mathcal{V}$. A \emph{left coideal subalgebra} is a unital algebra $K$ in $\mathcal{V}$ equipped with a unital algebra monomorphism $i: K \to H$ such that there exists a morphism $\rho_K: K \to H \otimes K$ satisfying
\(
    (\id_H \otimes i) \circ \rho_K = \Delta \circ i,
\)
which can be reformulated elementwise as a unital subalgebra $ K \subseteq H$ such that $\Delta(K) \subseteq H \otimes K$.
\end{definition}

As shown in Lemma \ref{lem:coidealobject}, a left coideal subalgebra of \(H\) in \(\mathcal{V}\) is precisely a unital subalgebra of \(H\) in \(\mathcal{C}\).

\begin{definition}\label{def:leftint}
A \emph{left integral} is a morphism \(\lambda \in \Hom_{\mathcal{C}}(H,\one)\). Equivalently, it is a morphism \(\lambda \in \Hom_{\mathcal{V}}(H,\one)\) satisfying
\((\id_H \otimes \lambda) \circ \Delta = \rho_{\one} \circ \lambda\)
(elementwise, $h_{(1)}\,\lambda(h_{(2)}) = \lambda(h)\,1_H$, for all $h$ in $H$).
\end{definition}

\subsection{Extra results on monoidal categories} \label{sub:HopfMonoidal}


\begin{lemma} \label{lem:AlgToFrob}
Let $(X,m,e)$ be a unital algebra object in a monoidal category $\mathcal{C}$, with a left dual $(X^*, \ev_X, \coev_X)$. If $X^* = X$ and $\ev_X = e^* \circ m,$ then $(X,m,e,m^*,e^*)$ is a Frobenius algebra object in $\mathcal{C}$. In addition, $m^{**}=m$,  $e^{**}=e$ and $\coev_X = \ev_X^* = m^* \circ e$.
\end{lemma}
\begin{proof}
Let us depict these morphisms and introduce $\delta$ as follows: 
\[ \raisebox{-5mm}{
	\begin{tikzpicture}[scale=0.8]
	\draw[blue,in=-90,out=-90,looseness=2] (-0.5,0.5) to (-1.5,0.5);
	 \draw[blue] (-1,-.1) to (-1,-.6);
	 \end{tikzpicture}} \coloneqq m , \hspace*{2mm}
	\raisebox{-6mm}{
		\begin{tikzpicture}[scale=0.8]
		\draw [blue] (-0.8,-.6) to (-.8,.6);
		\node at (-.8,.6) {${\color{blue}\bullet}$};
		 \end{tikzpicture}} \coloneqq e , \hspace*{2mm}
	 \raisebox{-7mm}{
	 	\begin{tikzpicture}[scale=0.8]
	 		\draw [blue] (-0.8,-.6) to (-.8,.6);
	 		\node at (-.8,-.6) {${\color{blue}\bullet}$};
	 	\end{tikzpicture}} \coloneqq e^*
 	=
 	\raisebox{-4mm}{
 	\begin{tikzpicture}[scale=0.8]
 	\draw[blue,in=-90,out=-90,looseness=2] (-0.5,0.5) to (-1.5,0.5);
 	\node at (-.5,.5) {${\color{blue}\bullet}$};
 	\end{tikzpicture}}	
 	=
 	\raisebox{-4mm}{
 		\begin{tikzpicture}[scale=0.8]
 			\draw[blue,in=-90,out=-90,looseness=2] (-0.5,0.5) to (-1.5,0.5);
 			\node at (-1.5,.5) {${\color{blue}\bullet}$};
 	\end{tikzpicture}},
 	\hspace*{2mm}
\raisebox{-4mm}{%
    \begin{tikzpicture}[scale=.8]
        \draw[blue,in=-90,out=-90,looseness=2] (0,0) to (1,0);
    \end{tikzpicture}%
} 
= \ev_X = e^* \circ m =
\raisebox{-6mm}{%
    \begin{tikzpicture}[scale=.8] 
        \draw[blue,in=-90,out=-90,looseness=2] (-0.5,0.5) to (-1.5,0.5);
        \draw[blue] (-1,-.1) to (-1,-.6);
        \node at (-1,-.6) {${\color{blue}\bullet}$};
    \end{tikzpicture}%
},
\hspace*{2mm}
\delta := 
\raisebox{-6mm}{%
    \begin{tikzpicture}[scale=0.6] 
        \draw[blue,in=90,out=90,looseness=2] (0,0) to (1,0);
        \draw[blue,in=-90,out=-90,looseness=2] (1,0) to (2,0);
        \draw[blue] (1.5,-.6) to (1.5,-1.2);
        \draw[blue] (0,0) to (0,-1.2);
        \draw[blue] (2,0) to (2,.9);
    \end{tikzpicture}%
} 
=
\raisebox{-6mm}{%
    \begin{tikzpicture}[scale=0.6]
        \draw[blue,in=-90,out=-90,looseness=2] (0,0) to (1,0);
        \draw[blue,in=90,out=90,looseness=2] (1,0) to (2,0);
        \draw[blue] (.5,-.6) to (.5,-1.2);
        \draw[blue] (0,0) to (0,.9);
        \draw[blue] (2,0) to (2,-1.2);
    \end{tikzpicture}%
}
\]

Two equalities above, involving $e^*$ and $\delta$, requires a proof, depicted below:
 \[  	\begin{tikzpicture}
  		\begin{scope}[xshift=0cm]
  			\draw[blue,in=-90,out=-90,looseness=2] (-0.5,0.5) to (-1.5,0.5);
  		\node at (-1.5,.5) {${\color{blue}\bullet}$};
  		\draw [color=DarkOrange, dashed, loop] (-1, -.2) circle (0.3cm);
  		\end{scope}
 
\begin{scope}[xshift=2cm]
 \draw[blue,in=-90,out=-90,looseness=2] (-0.5,0.5) to (-1.5,0.5);
\draw[blue] (-1,-.1) to (-1,-.6);
\node at (-1,-.6) {${\color{blue}\bullet}$};
\node at (-1.5,.5) {${\color{blue}\bullet}$};
\draw [color=DarkRed, dashed, loop] (-1, .3) circle (0.6cm);
\end{scope}

\begin{scope}[xshift=3cm]
		 			\draw [blue] (-0.8,-.6) to (-.8,.6);
			\node at (-.8,-.6) {${\color{blue}\bullet}$};
\end{scope}

\begin{scope}[xshift=5cm]
	\draw[blue,in=-90,out=-90,looseness=2] (-0.5,0.5) to (-1.5,0.5);
	\draw[blue] (-1,-.1) to (-1,-.6);
	\node at (-1,-.6) {${\color{blue}\bullet}$};
	\node at (-.5,.5) {${\color{blue}\bullet}$};
	\draw [color=DarkOrange, dashed, loop] (-1, -.3) circle (0.5cm);
\end{scope}

\begin{scope}[xshift=7cm]
	\draw[blue,in=-90,out=-90,looseness=2] (-0.5,0.5) to (-1.5,0.5);
	\node at (-.5,.5) {${\color{blue}\bullet}$};
\end{scope}

\node at (0,0) {$=$};
\node at (1.8,0) {$=$};
\node at (2.8,0) {$=$};
\node at (4.9,0) {$=$};

\node at (0,1) [align=center, above, DarkOrange,scale=.8] {$\ev_X = e^* \circ m$};
\node at (2.2,1) [align=center, above, DarkRed,scale=.9] {unitality};

 \draw[<-, DarkOrange] (0,.2) to (0,.9);
 \draw[<-, DarkRed] (1.8,.2) to (1.8,.9);
 \draw[<-, DarkRed] (2.8,.2) to (2.8,.9);
 \end{tikzpicture}
\]
from which we immediately see that $e^{**} = e$ by zigzag relation, and
\[ \begin{tikzpicture}[scale=.75]
\begin{scope}[xshift=0cm]
	\draw[blue,in=-90,out=-90,looseness=2] (0,0) to (1,0);
\draw[blue,in=90,out=90,looseness=2] (1,0) to (2,0);
\draw[blue] (.5,-.6) to (.5,-1.2);
\draw[blue] (0,0) to (0,1);
\draw[blue] (2,0) to (2,-1);
\draw [color=violet, dashed, loop] (.5,-.9) circle (0.2cm);
\end{scope}

\begin{scope}[xshift=4.5cm]
	\draw[blue,in=-90,out=-90,looseness=2] (0,0) to (1,0);
	\draw[blue,in=-90,out=-90,looseness=2] (.5,-.6) to (-.5,-.6);
	\draw[blue,in=90,out=90,looseness=2] (-.5,-.6) to (-1.5,-.6);
	\draw[blue,in=90,out=90,looseness=2] (1,0) to (2,0);
	\draw[blue] (-1.5,-.6) to (-1.5,-1);
	\draw[blue] (2,0) to (2,-.6);
	\draw [color=DarkOrange, dashed, loop] (0,-1.2) circle (0.4cm);
\end{scope}

\begin{scope}[xshift=8.8cm]
	\draw[blue,in=-90,out=-90,looseness=2] (0,0) to (1,0);
	\draw[blue,in=-90,out=-90,looseness=2] (.5,-.6) to (-.5,-.6);
	\draw[blue,in=90,out=90,looseness=2] (-.5,-.6) to (-1.5,-.6);
	\draw[blue,in=90,out=90,looseness=2] (1,0) to (2,0);
	\draw[blue] (-1.5,-.6) to (-1.5,-1);
	\draw[blue] (2,0) to (2,-.6);
	\draw[blue] (0,-1.2) to (0,-1.6);
	\node at (0,-1.6) {${\color{blue}\bullet}$};
	\draw [color=DarkGreen, dashed, loop] (.2,-.4) circle (.9cm);
\end{scope}

\begin{scope}[xshift=11.5cm]
	\draw[blue,in=90,out=90,looseness=2] (0,0) to (1,0);
	\draw[blue,in=-90,out=-90,looseness=2] (1,0) to (2,0);
	\draw[blue,in=-90,out=-90,looseness=2] (1.5,-.6) to (2.5,-.6);
	\draw[blue,in=90,out=90,looseness=2] (2.5,-.6) to (3.5,-.6);
	\draw[blue] (0,0) to (0,-1);
	\draw[blue] (2,-1.2) to (2,-1.6);
	\node at (2,-1.6) {${\color{blue}\bullet}$};
	\draw [color=DarkOrange, dashed, loop] (2,-1.3) circle (0.4cm);
\end{scope}

\begin{scope}[xshift=15.5cm]
	\draw[blue,in=90,out=90,looseness=2] (0,0) to (1,0);
	\draw[blue,in=-90,out=-90,looseness=2] (1,0) to (2,0);
	\draw[blue,in=-90,out=-90,looseness=2] (1.5,-.6) to (2.5,-.6);
	\draw[blue,in=90,out=90,looseness=2] (2.5,-.6) to (3.5,-.6);
	\draw[blue] (0,0) to (0,-1);
\end{scope}

\begin{scope}[xshift=19.5cm]
	\draw[blue,in=90,out=90,looseness=2] (0,0) to (1,0);
	\draw[blue,in=-90,out=-90,looseness=2] (1,0) to (2,0);
	\draw[blue] (1.5,-.6) to (1.5,-1.2);
	\draw[blue] (0,0) to (0,-1.2);
	\draw[blue] (2,0) to (2,1.2);
\end{scope}

\node at (2.5,-.2) {$=$};
\node at (6.9,-.2) {$=$};
\node at (11.2,-.2) {$=$};
\node at (15.25,-.2) {$=$};
\node at (19.2,-.2) {$=$};

\node at (2.5,.7) [align=center, above, DarkRed,scale=.8] {zig-zag};
\node at (6.9,.7) [align=center, above, DarkOrange,scale=.8] {$\ev_X = e^* \circ m$};
\node at (11.2,.7) [align=center, above, DarkGreen,scale=.8] {associativity};
\node at (19.2,.7) [align=center, above, DarkRed,scale=.8] {zig-zag};

\draw[->, DarkRed] (2.5, .7) to (2.5,0);
\draw[->, DarkOrange] (6.9, .7) to (6.9,0);
\draw[->, DarkGreen] (11.2, .7) to (11.2,0);
\draw[->, DarkRed] (19.2, .7) to (19.2,0);
\end{tikzpicture}
\]	
from which we deduce that $(X,\delta,e^*)$ is a counital coalgebra, as follows, where $\delta$ is depicted as in Definition \ref{def:coalgebra}:
\[ \begin{tikzpicture}[scale=.9]
	\begin{scope}[xshift=0cm]
		\draw[blue,in=90,out=90,looseness=2] (0,0) to (1,0);
		\node at (0,0) {$\textcolor{blue}{\bullet}$};
		\draw[blue] (.5,.6) to (.5,1.2);
		\draw [color=DarkOrange, dashed, loop] (.5,.6) circle (0.5cm);
	\end{scope}
	
	\begin{scope}[xshift=2cm,yshift=.7cm]
		\draw[blue,in=-90,out=-90,looseness=2] (0,0) to (1,0);
		\draw[blue,in=90,out=90,looseness=2] (1,0) to (2,0);
		\draw[blue] (.5,-.6) to (.5,-1.2);
		\draw[blue] (0,0) to (0,.6);
		\draw[blue] (2,0) to (2,-1);
		\node at (.5,-1.2) {${\color{blue}\bullet}$};
		\draw [color=DarkRed, dashed, loop] (.5,-.7) circle (0.6cm);
	\end{scope}
	
	\begin{scope}[xshift=5.4cm]
		\draw [blue] (0,0) to (0,1);
	\draw[blue,in=90,out=90,looseness=2] (0,1) to (.6,1);
	\draw[blue,in=-90,out=-90,looseness=2] (0,0) to (-.6,0);
	\draw[blue] (.6,1) to (.6,.1);
	\draw[blue] (-.6,0) to (-.6,.9);
	\end{scope}
	
	\begin{scope}[xshift=6.5cm]
		\draw[blue] (0,-.2) to (0,1.2);
	\end{scope}
	
	\begin{scope}[xshift=8cm]
		\draw [blue] (0,0) to (0,1);
		\draw[blue,in=90,out=90,looseness=2] (0,1) to (-.6,1);
		\draw[blue,in=-90,out=-90,looseness=2] (0,0) to (.6,0);
		\draw[blue] (-.6,1) to (-.6,.2);
		\draw[blue] (.6,0) to (.6,.8);
		\draw [color=DarkRed, dashed, loop] (.3,-.2) circle (0.4cm);
	\end{scope}
	
	\begin{scope}[xshift=9.5cm, yshift=.6cm]
		\draw[blue,in=90,out=90,looseness=2] (0,0) to (1,0);
		\draw[blue,in=-90,out=-90,looseness=2] (1,0) to (2,0);
		\draw[blue] (1.5,-.6) to (1.5,-1.2);
		\draw[blue] (0,0) to (0,-1.2);
		\draw[blue] (2,0) to (2,1);
		\node at (1.5,-1.2) {$\textcolor{blue}{\bullet}$};
		\draw [color=DarkOrange, dashed, loop] (1,.2) circle (1.2cm);
	\end{scope}
	
	\begin{scope}[xshift=12.5cm]
		\draw[blue,in=90,out=90,looseness=2] (0,0) to (1,0);
		\node at (1,0) {$\textcolor{blue}{\bullet}$};
		\draw[blue] (.5,.6) to (.5,1.2);
	\end{scope}
	
	\node at (1.5,.3) {$=$};
	\node at (4.25,.3) {$=$};
	\node at (6.25,.3) {$=$};
	\node at (6.9,.3) {$=$};
	\node at (8.9,.3) {$=$};
	\node at (12,.3) {$=$};
	
	\node at (1.4,1.2) [align=center, above, DarkOrange,scale=.8] {def. of $\delta$};
	\node at (4.25,1.35) [align=center, above, DarkRed,scale=.8] {$\ev_X = e^* \circ m$};
	
	\draw[->, DarkOrange] (1.5,1.2) to (1.5,.5);
	\draw[->, DarkRed] (4.25,1.2) to (4.25,.5);
\end{tikzpicture}\] 	
\[ \begin{tikzpicture}[scale=.8]
	\begin{scope}[xshift=0cm]
		\draw[blue,in=90,out=90,looseness=2] (0,0) to (1,0);
\draw[blue,in=90,out=90,looseness=2] (0.5,.6) to (-.5,.6);
\draw[blue] (-.5,.6) to (-.5,0);
\draw[blue] (0,1.2) to (0,1.6);
	\end{scope}
	
	\begin{scope}[xshift=2cm,yshift=1.1cm]
		\draw[blue,in=90,out=90,looseness=2] (0,0) to (1,0);
		\draw[blue,in=-90,out=-90,looseness=2] (1,0) to (2,0);
		\draw[blue,in=-90,out=-90,looseness=2] (1.5,-.6) to (2.5,-.6);
		\draw[blue,in=90,out=90,looseness=2] (2.5,-.6) to (3.5,-.6);
		\draw[blue] (0,0) to (0,-1);
		\draw[blue] (2,-1.2) to (2,-1.6);
		\draw[blue] (3.5,-.6) to (3.5,-1);
		\draw [color=DarkRed, dashed, loop] (1.9,-.4) circle (1cm);
	\end{scope}
	
	\begin{scope}[xshift=8cm,yshift=1.1cm]
		\draw[blue,in=-90,out=-90,looseness=2] (0,0) to (1,0);
		\draw[blue,in=-90,out=-90,looseness=2] (.5,-.6) to (-.5,-.6);
		\draw[blue,in=90,out=90,looseness=2] (-.5,-.6) to (-1.5,-.6);
		\draw[blue,in=90,out=90,looseness=2] (1,0) to (2,0);
		\draw[blue] (-1.5,-.6) to (-1.5,-1);
		\draw[blue] (2,0) to (2,-1);
		\draw[blue] (0,-1.2) to (0,-1.6);
	\end{scope}
	
	\begin{scope}[xshift=11cm]
		\draw[blue,in=90,out=90,looseness=2] (0,0) to (1,0);
		\draw[blue,in=90,out=90,looseness=2] (.5,.6) to (1.5,.6);
		\draw[blue] (1.5,.6) to (1.5,0);
		\draw[blue] (1,1.2) to (1,1.6);
	\end{scope}
	
	\node at (1.4,.4) {$=$};
	\node at (5.9,.4) {$=$};
	\node at (10.5,.4) {$=$};
	
	\node at (5.9,1.4) [align=center, above, DarkRed,scale=.8] {associativity};
	
	\draw[->, DarkRed] (5.9,1.4) to (5.9,.6);
	
\end{tikzpicture}\] 
Thus $(X,m,e,\delta,e^*)$ is weak Frobenius by Lemma \ref{lem:intermultdualcomult}, so Frobenius by Lemma \ref{lem:weakfull}. 

Finally, since $X^*=X$ and $e^{**}=e$, we have $m^* \circ e = (e^* \circ m)^* = \ev_X^{*} = \coev_{X^*} = \coev_{X}$. Thus, $\delta=m^*$ and $\delta^*=m$ by Lemma \ref{lem:multdualcomult}. The result follows.
\end{proof}
\begin{lemma} \label{lem:AnotherDual}
Let $X$ be an object in a monoidal category $\mathcal{C}$, with a left dual $(X^*, \ev_X, \coev_X)$, and let
$\Phi \in \Hom_{\mathcal{C}}(X,X^*)$
be an isomorphism. Let $\ev_X'$ and $\coev_X'$ denote the images of $\Phi$ and $\Phi^{-1}$ under the natural adjunction isomorphisms
$\Hom_{\mathcal{C}}(X,X^*)
\cong
\Hom_{\mathcal{C}}(X \otimes X,\one)$
and
$\Hom_{\mathcal{C}}(X^*,X)
\cong
\Hom_{\mathcal{C}}(\one,X \otimes X),$
respectively. Then $(X,\ev_X',\coev_X')$ is another left dual structure on $X$.
\end{lemma}
\begin{proof}
Let us show that pictorially as follows: 
\[\ev'_X = \raisebox{-6mm}{
	\begin{tikzpicture}
		\draw[blue,in=-90,out=-90,looseness=2] (0,0) to (1,0);
		\draw[blue] (0,0) to (0,.4);
		\draw[blue] (1,0) to (1,.4);
		\node[draw,thick,rounded corners,fill = white,scale=.8] at (0,0) {$\Phi$};
\end{tikzpicture}}, \hspace*{1mm}
\coev'_X = \raisebox{-4mm}{
	\begin{tikzpicture}
		\draw[blue,in=90,out=90,looseness=2] (0,0) to (1,0);
		\draw[blue] (0,0) to (0,-.4);
		\draw[blue] (1,0) to (1,-.4);
		\node[draw,thick,rounded corners,fill = white,scale=.7] at (1,0) {$\Phi^{-1}$};
\end{tikzpicture}} \, ,
\hspace*{2mm}
\raisebox{-8mm}{
\begin{tikzpicture}
	\draw [blue] (0,0) to (0,1);
	\draw[blue,in=90,out=90,looseness=2] (0,1) to (-.6,1);
	\draw[blue,in=-90,out=-90,looseness=2] (0,0) to (.6,0);
	\draw[blue] (-.6,1) to (-.6,.2);
	\draw[blue] (.6,0) to (.6,.8);
	\node[draw,thick,rounded corners,fill = white,scale=.7] at (0,.6) {$\Phi^{-1}$};
	\node[draw,thick,rounded corners,minimum width =20, fill = white,scale=.7] at (0,.1) {$\Phi$};
\end{tikzpicture}}
=
\raisebox{-8mm}{
	\begin{tikzpicture}
		\draw [blue] (0,0) to (0,1);
		\draw[blue,in=90,out=90,looseness=2] (0,1) to (-.6,1);
		\draw[blue,in=-90,out=-90,looseness=2] (0,0) to (.6,0);
		\draw[blue] (-.6,1) to (-.6,.2);
		\draw[blue] (.6,0) to (.6,.8);
	\end{tikzpicture}}
=
\raisebox{-4mm}{
\begin{tikzpicture}
	\draw[blue] (0,0) to (0,1.2);
\end{tikzpicture}}
=
\raisebox{-4mm}{
\begin{tikzpicture}
	\draw[blue] (0,-.4) to (0,1.2);
	\node[draw,thick,rounded corners,fill = white,scale=.7] at (0,.1) {$\Phi^{-1}$};
	\node[draw,thick,rounded corners,minimum width =20, fill = white,scale=.7] at (0,.6) {$\Phi$};
\end{tikzpicture}}
=
\raisebox{-6mm}{
	\begin{tikzpicture}
		\draw [blue] (0,0) to (0,1);
		\draw[blue,in=90,out=90,looseness=2] (0,1) to (.6,1);
		\draw[blue,in=-90,out=-90,looseness=2] (0,0) to (-.6,0);
		\draw[blue] (.6,1) to (.6,.1);
		\draw[blue] (-.6,0) to (-.6,.9);
		\node[draw,thick,rounded corners,fill = white,scale=.7] at (.6,.5) {$\Phi^{-1}$};
		\node[draw,thick,rounded corners,minimum width =20, fill = white,scale=.7] at (-.6,.1) {$\Phi$};
\end{tikzpicture}}. \qedhere
\]

\end{proof}

\begin{definition} \label{def:NonDeg}
Let $(X,m)$ be an algebra object in a monoidal category $\mathcal{C}$, with a left dual $X^*$. A morphism
$\lambda \in \Hom_{\mathcal{C}}(X,\one)$
is called \emph{nondegenerate} if the image of
$\lambda \circ m$
under the natural adjunction isomorphism
$\Hom_{\mathcal{C}}(X \otimes X,\one)
\cong
\Hom_{\mathcal{C}}(X,X^*)$
is an isomorphism
$\Phi.$
We depict these morphisms as follows:
\[ \lambda = \raisebox{-7mm}{
	\begin{tikzpicture}[scale=.9]
		\draw [blue] (-0.8,-.6) to (-.8,.6);
		\node at (-.8,-.6) {${\color{blue}\bullet}$};
\end{tikzpicture}},
 \hspace*{2mm}
 m =  \raisebox{-6mm}{
 	\begin{tikzpicture}[scale=.9]
 		\draw[blue,in=-90,out=-90,looseness=2] (-0.5,0.5) to (-1.5,0.5);
 		\draw[blue] (-1,-.1) to (-1,-.6);
 \end{tikzpicture}}, \hspace*{2mm}
\lambda \circ m = \raisebox{-6mm}{
	\begin{tikzpicture}[scale=.9]
		\draw[blue,in=-90,out=-90,looseness=2] (-0.5,0.5) to (-1.5,0.5);
		\draw[blue] (-1,-.1) to (-1,-.6);
		\node at (-1,-.6) {${\color{blue}\bullet}$};
\end{tikzpicture}},
\hspace*{2mm}
\Phi = \raisebox{-8mm}{
	\begin{tikzpicture}[scale=.8]
		\draw[blue,in=-90,out=-90,looseness=2] (0,0) to (1,0);
		\draw[blue,in=90,out=90,looseness=2] (1,0) to (2,0);
		\draw[blue] (.5,-.6) to (.5,-1.2);
		\draw[blue] (0,0) to (0,.6);
		\draw[blue] (2,0) to (2,-1);
		\node at (.5,-1.2) {${\color{blue}\bullet}$};
\end{tikzpicture}}
   \]
\end{definition}

\begin{proposition} \label{prop:NonDegToFrob}
Let $(X,m,e)$ be a unital algebra object in a monoidal category $\mathcal{C}$, with a left dual $X^*$. Let
$\Phi \in \Hom_{\mathcal{C}}(X,X^*)$
be the isomorphism associated with a nondegenerate morphism
$\lambda \in \Hom_{\mathcal{C}}(X,\one)$
as in Definition~\ref{def:NonDeg}. Let $(X,\ev_X',\coev_X')$ be the corresponding alternative left dual structure on $X$ constructed in Lemma~\ref{lem:AnotherDual}, and let $()^{\circ}$ be the associated duality functor. Then
$(X,m,e,m^{\circ},e^{\circ})$
is a Frobenius algebra object in $\mathcal{C}$. In addition, $e^{\circ} = \lambda$, $\ev_X' = e^{\circ} \circ m$ and $\coev_X' = (\ev_X')^{\circ} =  m^{\circ} \circ e$.
\end{proposition}
\begin{proof}
We first show that $\ev_X' = e^{\circ} \circ m$ as follows:
\[\ev'_X = \raisebox{-6mm}{
	\begin{tikzpicture}[scale=.9]
		\draw[blue,in=-90,out=-90,looseness=2] (0,0) to (1,0);
		\draw[blue] (0,0) to (0,.4);
		\draw[blue] (1,0) to (1,.4);
		\node[draw,thick,rounded corners,fill = white,scale=.8] at (0,0) {$\Phi$};
\end{tikzpicture}}
=
\raisebox{-11mm}{
	\begin{tikzpicture}[scale=.8]
		\draw[blue,in=-90,out=-90,looseness=2] (0,0) to (1,0);
		\draw[blue,in=90,out=90,looseness=2] (1,0) to (2,0);
		\draw[blue] (.5,-.6) to (.5,-1.2);
		\draw[blue] (0,0) to (0,.6);
		\draw[blue,in=-90,out=-90,looseness=2] (2,0) to (3,0);
		\draw[blue] (3,0) to (3,.6);
		\node at (.5,-1.2) {${\color{blue}\bullet}$};
\end{tikzpicture}}
=
\raisebox{-6mm}{
	\begin{tikzpicture}[scale=.9]
		\draw[blue,in=-90,out=-90,looseness=2] (-0.5,0.5) to (-1.5,0.5);
		\draw[blue] (-1,-.1) to (-1,-.6);
		\node at (-1,-.6) {${\color{blue}\bullet}$};
\end{tikzpicture}}
=
\lambda \circ m \  \ \text{and} \ \
\hspace*{2mm}
e^\circ = 
\raisebox{-6mm}{
	\begin{tikzpicture}[scale=.9]
		\draw[blue,in=-90,out=-90,looseness=2] (0,0) to (1,0);
		\draw[blue] (0,0) to (0,.4);
		\draw[blue] (1,0) to (1,.4);
		\node[draw,thick,rounded corners,fill = white,scale=.8] at (0,0) {$\Phi$};
		\node at (1,.4) {${\color{blue}\bullet}$};
\end{tikzpicture}}
=
\raisebox{-6mm}{
	\begin{tikzpicture}[scale=.9]
		\draw[blue,in=-90,out=-90,looseness=2] (-0.5,0.5) to (-1.5,0.5);
		\draw[blue] (-1,-.1) to (-1,-.6);
		\node at (-1,-.6) {${\color{blue}\bullet}$};
		\node at (-.5,.5) {${\color{blue}\bullet}$};
\end{tikzpicture}} 
=
\raisebox{-7mm}{
	\begin{tikzpicture}[scale=.9]
		\draw [blue] (-0.8,-.6) to (-.8,.6);
		\node at (-.8,-.6) {${\color{blue}\bullet}$};
\end{tikzpicture}}
=
\lambda.\]
Since $X^{\circ} = X$ by definition, the result follows by Lemma \ref{lem:AlgToFrob}.
\end{proof}

\subsection{Connected Frobenius algebras} \label{sub:HopfAppl}

This subsection provides a detailed proof of Theorem~\ref{thm:HopfFrobRep}. Although the result is presumably well known to experts, such a proof does not seem to be readily available in the literature, and we hope that this exposition will be useful to non-experts. Here $\mathcal{V}$ denotes the braided tensor category of finite-dimensional vector spaces over a field $\field$. A Hopf algebra in $\mathcal{V}$ is just a finite-dimensional Hopf algebra over $\field$. See Remark~\ref{rk:BraidedTensor} for possible extensions to an arbitrary braided tensor category \(\mathcal{V}\).





\begin{theorem}\label{thm:HopfFrobRep}
Let $H$ be a Hopf algebra in $\mathcal{V}$, and let $\mathcal{C}$ be the tensor category $\Corep(H)$. Then $H$ carries a natural structure of connected Frobenius algebra object in $\mathcal{C}$.
For a unital subalgebra $(K,i)$ of $H$, consider the following conditions:
\begin{enumerate}
    \item\label{itm:coideal} $K$ is a left coideal subalgebra of $H$;
    \item\label{itm:subalgobj} $K$ is a unital subalgebra object of $H$ in $\mathcal{C}$;
    \item\label{itm:frobsubalg} $K$ is a Frobenius subalgebra object of $H$ in $\mathcal{C}$;
\end{enumerate}
Then conditions \eqref{itm:coideal} and \eqref{itm:subalgobj} are equivalent. Moreover, if \eqref{itm:coideal} or \eqref{itm:subalgobj} hold, then \eqref{itm:frobsubalg} is equivalent to $\lambda \circ i$ be nondegenerate, with $\lambda\in \Hom_{\mathcal{C}}(H,\one)$ nonzero. 
In particular, if $H$ is semisimple and $\field = \mathbb{C}$, then condition \eqref{itm:coideal} implies this nondegeneracy. Consequently, in this case, conditions \eqref{itm:coideal}--\eqref{itm:frobsubalg} are all equivalent.
\end{theorem}
\begin{proof}
We first construct the Frobenius algebra $(H, m, e, \delta, \epsilon)$ step by step. Note that its counital coalgebra structure $(H, \delta, \epsilon)$ differs from $(H, \Delta, \varepsilon)$ arising from the Hopf algebra structure.

\begin{lemma}  \label{lem:AlgObj}
$(H, m, e)$ is a unital algebra object in $\mathcal{C}$.
\end{lemma}
\begin{proof}
We can endow $H$ with a $H$-comodule structure by setting $\rho_H = \Delta$, since $\Delta$ is coassociative and counital. The unit $e$ and the multipliciation $m$ are morphisms in $\mathcal{C}$ because
$$ \rho_H \circ e = \Delta \circ e = e \otimes e = (\id_H \otimes e) \circ (e \otimes \id_{\one}) = (\id_H \otimes e) \circ \rho_{\one}, \text{ and }$$
\begin{align*}
&\rho_H \circ m
= \Delta \circ m = (m \otimes m)\circ (\id_H \otimes c_{H,H} \otimes \id_H)\circ (\Delta \otimes \Delta) = \\
&(\id_H \otimes m)\circ (m \otimes \id_H \otimes \id_H)\circ (\id_H \otimes c_{H,H} \otimes \id_H)\circ (\Delta \otimes \Delta) = (\id_H \otimes m)\circ \rho_{H\otimes H}. \qedhere
\end{align*}
\end{proof}

\begin{lemma} \label{sublem:ResCoresMorGeneral}
Consider morphisms $f \in \Hom_{\mathcal{C}}(X,Y)$, $j \in \Hom_{\mathcal{C}}(Z,X)$, and a monomorphism $i \in \Hom_{\mathcal{C}}(T,Y)$. Let $g:Z \to T$ be a map such that $f \circ j = i \circ g$. Then $g$ is a morphism in $\mathcal{C}$.
\end{lemma}
\begin{proof}
By using the morphism identities for $f$, $j$ and $i$, together with the equality $f \circ j = i \circ g$, we obtain
\begin{align*}
&(\id_H \otimes i) \circ \rho_{T} \circ g = \rho_Y \circ i \circ g = \rho_Y \circ f \circ j = (\id_H \otimes f) \circ \rho_X \circ j 
= (\id_H \otimes f) \circ ( \id_H \otimes j) \circ \rho_{Z}  \\
&= (\id_H \otimes f \circ j) \circ \rho_{Z}
=(\id_H \otimes i \circ g) \circ \rho_{Z} 
=(\id_H \otimes i) \circ (\id_H \circ g) \circ \rho_{Z}.
\end{align*}
Thus, $\rho_{T} \circ g = (\id_H \otimes g) \circ \rho_{Z}$, i.e., $g$ is a morphism in $\mathcal{C}$, since $(\id_H \otimes i)$ is a monomorphism (left-cancellative).
\end{proof}

\begin{lemma} \label{sublem:ResCoresCoaction}
Let  $\rho_V: V \to H \otimes V$, $\rho_W: W \to H \otimes W$, and $i: W \to V$ be morphisms in $\mathcal{V}$, 
where $\rho_V$ is a coaction and $i$ is a monomorphism. 
If the condition 
\(
\rho_V \circ i = (\mathrm{id}_H \otimes i) \circ \rho_W
\)
holds, then $\rho_W$ is also a coaction.
\end{lemma}

\begin{proof}
Using the coaction property of $\rho_V$ together with the assumed condition, we have
\[
i \circ (\varepsilon \otimes \mathrm{id}_W) \circ \rho_W
= (\varepsilon \otimes \mathrm{id}_V) \circ (\mathrm{id}_H \otimes i) \circ \rho_W
= (\varepsilon \otimes \mathrm{id}_V) \circ \rho_V \circ i
= \mathrm{id}_V \circ i
= i.
\]
Since $i$ is left-cancellative, 
\(
(\varepsilon \otimes \mathrm{id}_W) \circ \rho_W = \mathrm{id}_W.
\) 
Similarly, we deduce that
\(
(\Delta \otimes \mathrm{id}_W) \circ \rho_W = (\mathrm{id}_H \otimes \rho_W) \circ \rho_W.
\) 
Hence, $\rho_W$ is a coaction.
\end{proof}

\begin{lemma} \label{lem:functor}
Let $F\colon \mathcal{A} \to \mathcal{B}$ be a left exact faithful functor between abelian categories. A morphism $f$ in $\mathcal{A}$ is a monomorphism if and only if $F(f)$ is a monomorphism in $\mathcal{B}$.
\end{lemma}

\begin{proof}
If $f \circ g = f \circ h$, then $F(f) \circ F(g) = F(f) \circ F(h)$. If $F(f)$ is a monomorphism, then $F(g) = F(h)$. Moreover, if $F$ is faithful, then $g = h$. Conversely, if $f$ is a monomorphism, then $\ker(f)=0$. If $F$ is left exact, it preserves kernels, hence $\ker(F(f)) = F(\ker(f)) = 0$. Therefore, $F(f)$ is a monomorphism.
\end{proof}

Recall from \cite[\S 5]{EGNO15} that the forgetful functor $\mathcal{C} \to \mathcal{V}$ is a tensor functor; in particular, it is faithful and exact.
\begin{lemma} \label{lem:coidealobject}
A unital subalgebra $K$ of $H$ in $\mathcal{V}$ is a left coideal subalgebra if and only if it is a unital subalgebra object of $H$ in $\mathcal{C}$.
\end{lemma}

\begin{proof}
Let $K$ be a left coideal subalgebra of $H$, namely (Definition \ref{def:coideal}), there is a unital algebra monomorphism $i$ in $\Hom_{\mathcal{V}}(K, H)$ such that there exists $\rho_K$ in $\Hom_{\mathcal{V}}(K, H \otimes K)$ satisfying $(\id_H \otimes i) \circ \rho_K = \Delta \circ i$.
This identity and Lemma \ref{sublem:ResCoresCoaction} imply that $\rho_K$ is a coaction, and hence $K$ becomes an object in $\mathcal{C}$. Furthermore, it means that $i$ is a morphism in $\mathcal{C}$, so a monomorphism in $\mathcal{C}$ (by Lemma \ref{lem:functor}), and consequently $i \otimes i$ is also a monomorphism in $\mathcal{C}$.
Since $i$ is a unital algebra monomorphism in $\mathcal{V}$ then $e_K = e \circ i$ and $m \circ (i \otimes i) = i \circ m_K$, which give them the structure of morphisms in $\mathcal{C}$, by Lemma \ref{sublem:ResCoresMorGeneral}, and $i$ is a unital algebra monomorphism in $\mathcal{C}$.


Conversely, if $K$ is a unital subalgebra object of $H$ in $\mathcal{C}$. Let $\rho_K: K \to H \otimes K$ be the coaction. Since $i$ is a morphism in $\mathcal{C}$, the identity $\Delta \circ i = (\id_H \otimes i) \circ \rho_K$ holds, which precisely means that $\rho_K$ satisfies the condition in Definition \ref{def:coideal}. Finally, since \(i\) is a unital algebra monomorphism in \(\mathcal{C}\), it is also one in \(\mathcal{V}\), see Lemma~\ref{lem:functor}.
\end{proof}

\begin{lemma} \label{lem:ConnectedFrob}
The Hopf algebra $H$ in $\mathcal{V}$ has the structure of a connected Frobenius algebra object in $\mathcal{C}$.
\end{lemma}
\begin{proof}
By the Larson-Sweedler theorem \cite[\S 2]{LS69}, the space
\(
\Hom_{\mathcal{C}}(H,\one)
\)
of left integrals (Definition~\ref{def:leftint}) is one-dimensional, say \(\mathbbm{k}\lambda\), and \(\lambda\) is nondegenerate in the sense of Definition~\ref{def:NonDeg}. Hence, by Proposition~\ref{prop:NonDegToFrob}, and using its notation, \((H,m,e,m^{\circ},e^{\circ})\) is a Frobenius algebra in \(\mathcal{C}\), with \(e^{\circ}=\lambda\).

Moreover, by the natural adjunction isomorphism,
\(
\Hom_{\mathcal{C}}(H,\one)
\simeq
\Hom_{\mathcal{C}}(\one,H^{\circ}),
\)
and \(H^{\circ}=H\) by the proof of Proposition~\ref{prop:NonDegToFrob}. Therefore,
\(
\dim \Hom_{\mathcal{C}}(\one,H)=1,
\)
that is, \(H\) is connected.
%
\end{proof}

\begin{remark}[For information only] \label{rk:casimir}
Let $c:=\coev_H'(1) \in H \otimes H $ be the \emph{Casimir element}. The fact that $\coev_H' \in \mathrm{Hom}_{\mathcal{C}}(\one,H \otimes H)$ reformulates as $\rho_{H \otimes H}(c) = 1_H \otimes c$ (i.e. $c$ is a \emph{coinvariant element}). By \S \ref{sub:HopfMonoidal}, the Frobenius comultiplication $\delta=m^{\circ} = (\id_H \otimes m) \circ (\coev_H' \otimes \id_H) = (m \otimes \id_H) \circ (\id_H \otimes \coev_H')$. Thus, with $c=\sum_i e_i \otimes f_i$, for all $h \in H$, $\delta(h) = (\id_H \otimes m)(h \otimes c) = (m \otimes \id_H)(c \otimes h) =  \sum_i he_i \otimes f_i = \sum_i e_i \otimes f_ih$,
and the Frobenius condition reformulates as follows: for all $x,y \in H$, $ \sum_i x(ye_i) \otimes f_i  = \sum_i (xy)e_i \otimes f_i  = \sum_i e_i \otimes f_i(xy) = \sum_i e_i \otimes (f_ix)y$.
\end{remark}

\begin{lemma} \label{lem:SSubb} Let $A$ be a semisimple unital subalgebra of $M_{n}(\mathbb{C})$. Then $A$ is conjugate to a $*$-subalgebra.
\end{lemma}

\begin{proof}
By the Artin-Wedderburn theorem, $A$ is spanned by a set of elements $\{e_{i,j}^{(k)}\}$ that satisfy the standard matrix unit relations $e_{i,j}^{(k)}e_{p,q}^{(l)}=\delta_{k,l}\delta_{j,p}e_{i,q}^{(k)}$ and $\sum_{k,i} e_{i,i}^{(k)} = I_{n}$. Define a new inner product on $\mathbb{C}^{n}$ by averaging over these units: $\langle x, y \rangle' = \sum_{k,i,j} \langle e_{i,j}^{(k)} x, e_{i,j}^{(k)} y \rangle$, where $\langle \cdot, \cdot \rangle$ is the former inner product. Direct computation shows that the adjoint of $e_{u,v}^{(w)}$ with respect to $\langle \cdot, \cdot \rangle'$ is $e_{v,u}^{(w)} \in A$; thus, $A$ is closed under this new adjoint operation $a \mapsto a^{\sharp}$.

Representing this inner product via a positive operator $P = V^{*}V$, we have $\langle x, y \rangle' = \langle Px, y \rangle$, which implies $a^{\sharp} = P^{-1} a^{*} P$. It follows that $V a^{\sharp} V^{-1} = (V a V^{-1})^{*}$. Therefore, the conjugate algebra $B = V A V^{-1}$ is closed under the former adjoint, making it a *-subalgebra of $M_{n}(\mathbb{C})$.
\end{proof}

\begin{lemma} \label{lem:tau|nondeg}
Let $\mathcal{M}$ be a finite-dimensional $C^*$-algebra, and let $\tau: \mathcal{M} \to \mathbb{C}$ be a faithful trace. If $A \subseteq \mathcal{M}$ is a semisimple subalgebra, then the restriction $\tau|_A$ is nondegenerate.
\end{lemma}

\begin{proof}
Since a finite-dimensional $C^*$-algebra is a direct sum of matrix algebras, Lemma \ref{lem:SSubb} implies that $A$ is conjugate to a $*$-subalgebra of $\mathcal{M}$. That is, there exists an invertible element $V \in \mathcal{M}$ such that for every $a \in A$, we have  
$(V a V^{-1})^* = V a^{\#} V^{-1}$ for some $a^{\#} \in A$. Using the tracial property of $\tau$ (i.e. $\tau(xy)=\tau(yx)$), we then obtain  
\[
\tau|_A(a a^{\#}) = \tau(a a^{\#}) = \tau(V^{-1} V a a^{\#}) = \tau(V a a^{\#} V^{-1}) = \tau(V a V^{-1} V a^{\#} V^{-1}) = \tau((V a V^{-1}) (V a V^{-1})^*).
\]  
Since $\tau$ is faithful on $\mathcal{M}$, it follows that $\tau|_A(a a^{\#}) = 0$ if and only if $a = 0$. Hence, $\tau|_A$ is nondegenerate.
\end{proof}

\begin{lemma} \label{lem:leftcoidealss}
If $H$ is semisimple over $\mathbb{C}$, then any left coideal subalgebra $K \subseteq H$ is semisimple.
\end{lemma}
\begin{proof}
This follows directly from \cite[Theorem 6.1(i)]{Skr07} together with \cite[Theorem 1.2(7)]{Ma92}.
\end{proof}

\begin{lemma} \label{lem:CoidealFrob}
If $H$ is semisimple over $\mathbb{C}$, then any left coideal subalgebra $K \subseteq H$ is a Frobenius subalgebra in $\mathcal{C}$.
\end{lemma}

\begin{proof}
By \cite[Proposition 2]{Rad94b}, the left integral $\lambda$ can be constructed using the trace function of the regular representation $r$ as
\(
\lambda(x) = \frac{1}{\dim H} \Tr(r(x) \circ S^2).
\)
Since $H$ is finite-dimensional and semisimple over $\mathbb{C}$, we have $S^2 = \id_H$ by \cite{LR88}, so that
\(
\lambda(x) = \frac{1}{\dim H} \Tr(r(x)).
\)
Hence, $\lambda$ defines a faithful trace on $H$, viewed as a finite-dimensional C$^*$-algebra, because $r(H)$ decomposes as a direct sum of matrix algebras.  
Applying Lemmas \ref{lem:tau|nondeg} and \ref{lem:leftcoidealss}, it follows that $\lambda|_K$ is nondegenerate. Then, by Lemma \ref{lem:coidealobject} and Proposition \ref{prop:NonDegToFrob}, using their notation, the tuple $(K, m_K, e_K, m_K^{\circ}, e_K^{\circ})$ forms a Frobenius algebra in $\mathcal{C}$.  
Finally, Lemma \ref{lem:AmbientRefor} and Proposition \ref{prop:AmbientFrobSub} reduce the verification that $K$ is a Frobenius subalgebra of $H$ (in the sense of Definition \ref{def:sub}) to the identity
\[
\ev_H' \circ (i \otimes i) 
= \lambda \circ m \circ (i \otimes i) 
= \lambda \circ i \circ m_K 
= \lambda|_K \circ m_K 
= \ev_K'. \qedhere
\]
\end{proof} 
This completes the proof of Theorem \ref{thm:HopfFrobRep}.
\end{proof}

\begin{remark}\label{rk:BraidedTensor}
Lemmas~\ref{lem:AlgObj} and~\ref{lem:coidealobject} remain valid for any braided tensor category \(\mathcal{V}\), since their proofs are purely categorical. For Lemma~\ref{lem:ConnectedFrob}, one needs a braided tensor category \(\mathcal{V}\) in which the Larson-Sweedler theorem remains valid. Is semisimplicity of \(\mathcal{V}\) sufficient? For Lemma~\ref{lem:CoidealFrob}, one further needs nondegeneracy to be inherited by unital subalgebras in \(\mathcal{C}\). Is it sufficient to assume that \(\mathcal{V}\) is unitary?
\end{remark}

\subsection{Applications}
This subsection applies Corollary~\ref{cor:allfinite3} to semisimple Hopf algebras, noting limitations in the nonsemisimple case. See Definition~\ref{def:CoherentSublattice} for the notion of a coherent sublattice of the Frobenius subalgebra poset.


\begin{corollary} \label{cor:allfinite1Hopf}
Let $H$ be a finite-dimensional semisimple Hopf algebra over $\mathbb{C}$. Then the Frobenius subalgebra poset of $H$ in $\Corep(H)$ is the left coideal subalgebra lattice, and every coherent sublattice is finite.
\end{corollary}    
\begin{proof}
By Theorem \ref{thm:HopfFrobRep}, the entire Frobenius subalgebra poset is the left coideal subalgebra poset, which is trivially a lattice. Recall from \cite[Chapter 5]{EGNO15} that $ \Corep(H) $ is an integral fusion category, and thus pseudo-unitary by Proposition \cite[Proposition 9.6.5]{EGNO15}. The result follows immediately from Corollary \ref{cor:allfinite3}. 
\end{proof}

We now ask whether the entire left coideal subalgebra lattice is coherent, which is equivalent to:

\begin{question}[\cite{PalMO26}]\label{Q:TwoCoideal}
Is the restriction $\lambda|_{L+K}$ necessarily nondegenerate for all left coideal subalgebras $L,K$?
\end{question}

By the proof of Lemma~\ref{lem:CoidealFrob}, the restrictions $\lambda|_L$ and $\lambda|_K$ are nondegenerate. Serge Skryabin mentioned privately that Question~\ref{Q:TwoCoideal} admits an affirmative answer in the special case where one of the two coideal subalgebras is a Hopf subalgebra, by means of standard techniques. However, he does not know how to approach the general case. For completeness, we include below a proof in this special case (Theorem \ref{thm:NonDegSum}). 

Recall that an algebra is called \emph{simple} if it admits no proper nontrivial two-sided ideals, whereas a coalgebra is called \emph{simple} if it admits no proper nontrivial subcoalgebras. A coalgebra \(C\) is called \emph{cosemisimple} if it is a direct sum of simple subcoalgebras, each formed by the sum of all simple subcomodules of \(C\) in a given isomorphism class. By \cite{LR88}, every finite-dimensional semisimple Hopf algebra $H$ over $\mathbb{C}$ is cosemisimple. Let \(V\) and \(W\) be left coideals of \(H\)---equivalently, subobject of $H$ in $\mathcal{C} = \Corep(H)$. The multiplication \(m\), the left integral \(\lambda\), and the inclusion maps \(i_V\) and \(i_W\) are morphisms in $\mathcal{C}$. Hence $\lambda \circ m \circ (i_V \otimes i_W)$ is in $\Hom_{\mathcal{C}}(V \otimes W , \one)\simeq \Hom_{\mathcal{C}}(V,W^*)$. This morphism is precisely the linearization of the restriction of the bilinear form to \(V \times W\). If \(V\) and \(W\) are simple, Schur's lemma implies that $\Hom_{\mathcal{C}}(V,W^*) = 0$ unless \(V \simeq W^*\). It follows that any two simple left coideals that are not dual up to isomorphism are orthogonal. By Lemma \ref{lem:CoidealFrob}, each left coideal subalgebra is isomorphic to its dual as left comodule.


\begin{theorem} \label{thm:NonDegSum}
Let $H$ be a finite-dimensional semisimple Hopf algebra over $\mathbb{C}$ with integral $\lambda$. If $K$ is a Hopf subalgebra and $L$ is a left coideal subalgebra, then $\lambda|_{L+K}$ is nondegenerate.
\end{theorem}
\begin{proof}
If $K$ is a Hopf subalgebra, then $H$ admits an orthogonal decomposition
\(
H = K \oplus^\perp C,
\)
where $C$ is a subcoalgebra orthogonal to $K$ with respect to the bilinear form associated with $\lambda$. Since $H$ is cosemisimple, every left coideal decomposes as a direct sum of simple left coideals, and non-dual simple left coideals are mutually orthogonal. It follows that
\(
L = (L \cap K) \oplus^\perp (L \cap C)
\)
and
\(
L+K = K \oplus^\perp (L \cap C).
\)
By the proof of Lemma~\ref{lem:CoidealFrob}, the restriction of $\lambda$ to each of the left coideal subalgebras $L$, $K$, and $L \cap K$ is nondegenerate. Hence the restriction of $\lambda$ to $L \cap C$ is also nondegenerate, and therefore so is its restriction to $L+K$, by the above orthogonal decompositions.
\end{proof}

\begin{corollary}\label{cor:HopfSubCoherent}
Let $H$ be a finite-dimensional semisimple Hopf algebra over $\mathbb{C}$. Then the Hopf subalgebra lattice forms a coherent sublattice of the Frobenius subalgebra lattice of $H$ in $\Corep(H)$.
\end{corollary}

The finiteness of the Hopf subalgebra lattice is immediate, since a finite-dimensional cosemisimple coalgebra has only finitely many subcoalgebras. By combining Corollaries~\ref{cor:allfinite1Hopf} and~\ref{cor:HopfSubCoherent}, a positive answer to Question~\ref{Q:TwoCoideal} would suffice to deduce the finiteness of the left coideal subalgebra lattice, recovering \cite[Theorem~3.6]{EW14}.


\begin{question}
In Theorem \ref{thm:HopfFrobRep}, is $K$ a Frobenius subalgebra object of $H$ in $\mathcal{C}$ even without the semisimple assumption on $H$?
\end{question}
  

Based on Proposition \ref{prop:pseudo} or Corollary \ref{cor:ZeroHopf}, it may be challenging to find a non-semisimple finite-dimensional Hopf algebra that satisfies the weakly positive assumption. However, if one exists, we can apply Corollary \ref{cor:allfinite1.75} as well.

We are still uncertain about how much the semisimple assumption can be relaxed, but we do know that it cannot be completely removed, even when limited to Hopf subalgebras.

\begin{example} \label{ex:nichols}
Nichols' Hopf algebras $ H_{2^n}$ for $ n \geq 1$, introduced in \cite{N78} and further explored in \cite[Example 5.5.8]{EGNO15}, serve as a family of counterexamples for $ n \geq 3$. The Hopf algebra $ H_{2^n}$ (for $ n \geq 1$) is $ 2^n$-dimensional and generated by elements $ g, x_1, \ldots, x_{n-1}$ subject to the relations:
$$
g^2 = 1, \quad x_i g = -g x_i, \quad x_i^2 = 0, \quad x_i x_j = -x_j x_i \quad (i \neq j),
$$
$$
\Delta(g) = g \otimes g, \quad \Delta(x_i) = g \otimes x_i + x_i \otimes 1, \quad \epsilon(g) = 1, \quad \epsilon(x_i) = 0, \quad S(g) = g, \quad S(x_i) = -g x_i.
$$
For $ n=1$, this corresponds to the group algebra of the cyclic group of order 2, and for $ n=2$, it represents Sweedler's 4-dimensional Hopf algebra. When $ n \geq 3$, it contains infinitely many Hopf subalgebras. Specifically, for any subspace $ E \subset \bigoplus_i \mathbb{C}x_i$, the subalgebra $ H_E = \langle E, g \rangle$ is a Hopf subalgebra, since for any $ x \in E$, we have $ \Delta(x) = g \otimes x + x \otimes 1$ and $ S(x) = -g x$. The dimension of $ H_E$ is $ 2^{\dim(E) + 1}$ and is completely determined by $ E$. Nichols' Hopf algebra $ H_{2^n}$ is unimodular if and only if $ n$ is odd. For $ n=2,3$, refer to \cite[Proposition 7 (e) and Proposition 10 (d)]{Rad94}, where they are denoted $ H_{(2,1,-1)}$ and $ U_{(2,1,-1)}$. For general cases, see the correction in \cite[Exercise 6.5.10(i)]{EtiCor}.
\end{example}   

Every Hopf subalgebra is a left coideal subalgebra, but not vice versa. For example, Sweedler's 4-dimensional Hopf algebra has finitely many Hopf subalgebras but infinitely many left coideal subalgebras \cite[Example 3.5]{EW14}.

\begin{example} \label{ex:double}
The quantum double $ D(H)$ of a finite-dimensional Hopf algebra $ H$ is a bicrossed product of $ H$ with $ H^{op*}$. Hence, $ H$ is a Hopf subalgebra of $ D(H)$. Moreover, $ D(H)$ is factorizable and therefore unimodular (\cite[Definition 6.5.7, Exercise 8.6.4(i), Propositions 7.14.6, 8.6.3, and 8.10.10]{EGNO15}). Consequently, if $ H$ has infinitely many Hopf subalgebras (as is the case with Nichols' Hopf algebra $ H_8$), then $ D(H)$ is a factorizable finite-dimensional Hopf algebra that also has infinitely many Hopf subalgebras.
\end{example}

\section{Other examples} \label{sec:other}
This section presents additional examples—connected canonical Frobenius algebras (\S\ref{sub:canon}), abstract spin chains (\S\ref{sub:ASC}), and vertex operator algebras (\S\ref{sub:voa})—that may fall within the scope of Theorem~\ref{thm:allfinite}. 

\subsection{Connected canonical Frobenius algebra} \label{sub:canon}

In \cite[Corollary 7.20.4]{EGNO15}, the canonical Frobenius algebra $\underline{\mathrm{Hom}}(\mathbf{1},\mathbf{1})$ in $\mathcal{C} \boxtimes \mathcal{C}^{\mathrm{op}}$ is examined for a unimodular \emph{multitensor} category $\mathcal{C}$. However,

\begin{proposition} \label{prop:canon}
The Frobenius algebra $\underline{\mathrm{Hom}}(\mathbf{1},\mathbf{1})$ is connected if and only if $\mathcal{C}$ is a tensor category.
\end{proposition}

\begin{proof}
Let $\mathcal{D}$ be a multitensor category and $\mathcal{M}$ a $\mathcal{D}$-module category with objects $M_1$ and $M_2$ in $\mathcal{M}$. According to \cite[(7.20)]{EGNO15}, the space $\underline{\mathrm{Hom}}(M_1, M_2)$ is defined via the natural isomorphism:
$$ \mathrm{Hom}_{\mathcal{M}}(X \otimes M_1, M_2) \simeq \mathrm{Hom}_{\mathcal{D}}(X, \underline{\mathrm{Hom}}(M_1, M_2)). $$

Assuming $\mathcal{M}$ is also monoidal with unit $\mathbf{1}_{\mathcal{M}}$, and letting $\mathbf{1}_{\mathcal{D}}$ be the unit of $\mathcal{D}$, we set $M_1 = M_2 = \mathbf{1}_{\mathcal{M}}$ and $X = \mathbf{1}_{\mathcal{D}}$. This yields:
$$ \mathrm{Hom}_{\mathcal{M}}(\mathbf{1}_{\mathcal{M}}, \mathbf{1}_{\mathcal{M}}) \simeq \mathrm{Hom}_{\mathcal{D}}(\mathbf{1}_{\mathcal{D}}, \underline{\mathrm{Hom}}(\mathbf{1}_{\mathcal{M}}, \mathbf{1}_{\mathcal{M}})). $$

Therefore, $\underline{\mathrm{Hom}}(\mathbf{1}_{\mathcal{M}}, \mathbf{1}_{\mathcal{M}})$ is connected if and only if $\mathbf{1}_{\mathcal{M}}$ is linear-simple, meaning $\mathrm{Hom}_{\mathcal{M}}(\mathbf{1}_{\mathcal{M}}, \mathbf{1}_{\mathcal{M}})$ is one-dimensional. For a multitensor category $\mathcal{M}$, this indicates that $\underline{\mathrm{Hom}}(\mathbf{1}_{\mathcal{M}}, \mathbf{1}_{\mathcal{M}})$ is connected if and only if $\mathcal{M}$ is a tensor category. The result follows by taking $\mathcal{D} = \mathcal{C} \boxtimes \mathcal{C}^{\mathrm{op}}$ and $\mathcal{M} = \mathcal{C}$.
\end{proof}

\begin{corollary} \label{cor:CanonHopf}
Let $H$ be a finite-dimensional unimodular Hopf algebra. Then $H^*$ serves as the \emph{connected} canonical Frobenius algebra in $\mathrm{Rep}(H \otimes H^{\mathrm{cop}})$, albeit with a different comultiplication and counit.
\end{corollary}

\begin{proof}
This characterization of the canonical Frobenius algebra is noted at the end of \cite[\S 7.20]{EGNO15}. The connectedness follows from Theorem \ref{prop:canon}.
\end{proof}

\begin{proposition}[\customcite{EGNO15}{Proposition 7.18.9}] \label{prop:zerotrace}
Let $\mathcal{C}$ be a unimodular non-semisimple finite tensor category. Let $f: \underline{\mathrm{Hom}}(\mathbf{1},\mathbf{1}) \to \underline{\mathrm{Hom}}(\mathbf{1},\mathbf{1})^{**}$ be a morphism in $\mathcal{C} \boxtimes \mathcal{C}^{\mathrm{op}}$. Then $\text{tr}(f) = 0$.
\end{proposition}
Referring to Lemma \ref{lem:selfdual}, we understand that a Frobenius algebra object is selfdual. Therefore, in Proposition \ref{prop:zerotrace}, the morphism $ f $ can be viewed as an isomorphism, which implies that the categorical dimension of $ \underline{\mathrm{Hom}}(\mathbf{1},\mathbf{1}) $ with respect to $ f $ must be zero.

\begin{corollary} \label{cor:ZeroTensor}
Let $\mathcal{C}$ be a unimodular non-semisimple finite tensor category. Assume that $\mathcal{C} \boxtimes \mathcal{C}^{\mathrm{op}}$ has a pivotal structure $\phi$. Then $\dim_{\phi}(\underline{\mathrm{Hom}}(\mathbf{1},\mathbf{1})) = 0$.
\end{corollary}
\begin{proof}
This follows directly from Proposition \ref{prop:zerotrace}.
\end{proof}

\begin{remark} \label{rk:failpositive}
In Corollary \ref{cor:ZeroTensor}, the canonical Frobenius algebra object $ X $ is not $ \phi_X $-weakly positive, which implies that Corollary \ref{cor:allfinite1} is not applicable.
\end{remark}

\begin{corollary} \label{cor:ZeroHopf}
Let $H$ be a finite-dimensional unimodular non-semisimple Hopf algebra. Assume that $\mathrm{Rep}(H \otimes H^{\mathrm{cop}})$ has a pivotal structure $\phi$. Then $\dim_{\phi}(H^*) = 0$.
\end{corollary}
\begin{proof}
This follows immediately from Corollaries \ref{cor:CanonHopf} and \ref{cor:ZeroTensor}.
\end{proof}

In Corollary \ref{cor:ZeroHopf}, we note that due to the fiber functor, $\FPdim$ is equal to $\dim_{\mathbbm{k}}$. However, in characteristic zero, $\dim_{\mathbbm{k}}(H^*)$ is non-zero. Thus, $\dim_{\phi}$ must differ from $\FPdim$ in this case.

\subsection{Abstract spin chains} \label{sub:ASC}
Quantum cellular automata (QCA) are quantum operations that reflect the core principles of unitarity and locality \cite{Watr95}. In \cite{JSW24}, the authors explore extensions of bounded-spread isomorphisms of symmetric local algebras to QCA defined on complete (or edge-restricted) local operator algebras. They introduce the notions of abstract spin chains and categorical inclusions, offering an algebraic framework to study these inclusions. In this section, we aim to demonstrate that the lattice of categorical inclusions (as discussed in \cite[\S 4]{JSW24}) corresponding to a connected, commutative Frobenius algebra (Q-system) is finite.  

As detailed in \cite[\S 3]{JSW24}, let $\mathcal{C}$ be an indecomposable unitary multi-fusion category, and let $X$ be an object in $\mathcal{C}$. For any finite interval $I \subset \mathbb{Z}$, define $A(\mathcal{C},X)_I$ as the finite-dimensional $*$-algebra $\End_{\mathcal{C}}(X^{\otimes |I|})$. The colimit $A(\mathcal{C},X)$ of $A(\mathcal{C},X)_I$, taken in the category of $*$-algebras, is called an \emph{abstract spin chain}.

Let $\mathcal{C}$ be an indecomposable multi-fusion category. Following \cite[Definition 3.2]{JSW24}, a \emph{quotient} of $\mathcal{C}$ is defined as an indecomposable multi-fusion category $\mathcal{D}$, equipped with a dominant\footnote{That is, \emph{surjective}, as defined in \cite[Definition 1.8.3]{EGNO15}.} tensor functor $F: \mathcal{C} \to \mathcal{D}$.

As noted in \cite{JSW24}, a dominant tensor functor $F: \mathcal{C} \to \mathcal{D}$ induces an inclusion of C$^*$-algebras:
$$ i_F: A(\mathcal{C},X) \hookrightarrow A(\mathcal{D},F(X)). $$

Given two quotients, $F: \mathcal{C} \to \mathcal{D}$ and $G: \mathcal{D} \to \mathcal{E}$, the composition $G \circ F: \mathcal{C} \to \mathcal{E}$ is also a quotient, and the inclusion satisfies $i_{G \circ F} = i_G \circ i_F$. These inclusions are known as \emph{categorical inclusions}.

According to \cite[\S 4]{JSW24}, citing a result from \cite{BN10}, for a quotient $F: \mathcal{C} \to \mathcal{D}$, there exists a connected commutative Frobenius algebra object $L$ in $\mathcal{Z}(\mathcal{C})$ such that $\mathcal{D}$ is equivalent to the category $\mathcal{C}_L$ of right $L$-modules. Moreover, $F$ is equivalent to the functor $F_L: \mathcal{C} \to \mathcal{C}_L$, where $x \mapsto x \otimes L$, and $L$ is identified with its image under the forgetful functor. Using this, \cite[\S 4]{JSW24} demonstrates an equivalence between the lattice of intermediate categorical inclusions between $A(\mathcal{C},X)$ and $A(\mathcal{C}_L,X)$, and the lattice of Frobenius subalgebras of $L$. Since the Drinfeld center of a multi-fusion category is a fusion category \cite[Exercise 7.13.7]{EGNO15}, we can apply Corollary \ref{cor:allfiniteUnitary} to get the following:

\begin{corollary} \label{cor:CatInc} The lattice of intermediate categorical inclusions between $A(\mathcal{C},X)$ and $A(\mathcal{C}_L,X)$ is finite. 
\end{corollary}
This provides a rigidity-type result for abstract spin chains.

\subsection{Vertex operator algebras} \label{sub:voa}

We posed the following question to Kenichi Shimizu in \cite{Shi24}:  

\begin{question}  
Under what conditions is a vertex operator algebra (VOA) $ V $ a connected Frobenius algebra object in $ \Rep(V) $?  
\end{question}  

Shimizu's response can be understood through the categorical perspective on VOA extensions, as explored in \cite{CKM24}. This work highlights that a VOA extension $ A $ of $ V $ naturally forms a commutative algebra in the representation category $ \Rep(V) $. This algebra is characterized by two key properties: it is connected (or haploid; see Definition \ref{def:connected}) and has a trivial twist, for the ribbon structure of $ \Rep(V) $.  

Although it is not fully established that every VOA extension gives rise to a Frobenius algebra, a related partial result is known. Specifically, if $ \mathcal{C} $ is a modular tensor category and $ A $ is a connected, commutative, and exact algebra in $ \mathcal{C} $, and if the category of local $ A $-modules, $ \mathcal{C}_A^{\text{loc}} $, forms a ribbon category with the same twist as $ \mathcal{C} $, then $ A $ is a symmetric Frobenius algebra \cite[Proposition 5.19]{SY24}. However, these assumptions may be too restrictive when applied to VOA extensions.  

%
%
%
%
%
%

\section{Quantum arithmetic} \label{sec:QA}
In this section, we explore further applications, particularly the generalization of Ore's theorem \cite{Pal20} and Euler's totient theorem \cite{Pal18} to tensor categories. Both generalizations depend on the existence of a finite lattice, which naturally aligns with Watatani's theorem on tensor categories. We will present these results without proofs. While they can currently be regarded as conjectures, we anticipate that their proofs could follow methods similar to those in \cite{Pal20} and \cite{Pal18}, but formalizing these proofs remains a valuable task for future work.

Additionally, we explore an extension of Robin's reformulation of the Riemann hypothesis \cite{Rob84}, which involves the sigma function, to tensor categories. We do not assert the relevance of this generalization but propose it as a potential direction for further investigation.

\begin{remark} \label{rk:spe}
We caution the reader that this section is largely speculative. Speculation involves imagining possibilities or exploring ideas without a solid foundation in evidence or structure. In contrast, a conjecture is a more formalized hypothesis, often based on observed patterns or logical reasoning, and is typically intended to be proven or disproven.
\end{remark}

\subsection{Ore's theorem} \label{sub:Ore}  

Øystein Ore demonstrated in \cite{Or38} in 1938 that a finite group is cyclic if and only if its subgroup lattice is distributive. He also extended this result in one direction as follows:

\begin{theorem}[\cite{Or38}, Theorem 7] \label{thm:Ore}
Let $[H,G]$ be a distributive interval of finite groups. Then there exists an element $g \in G$ such that $\langle Hg \rangle = G$.
\end{theorem}

The paper \cite{Pal20} generalizes Ore's Theorem \ref{thm:Ore} to planar algebras as follows, and applies it to establish a connection between combinatorics and representation theory.

\begin{theorem}[\cite{Pal20}] \label{thm:OrePlanar}
Let $\mathcal{P}$ be an irreducible subfactor planar algebra with a distributive biprojection lattice. Then there exists a minimal $2$-box projection that generates the identity biprojection.
\end{theorem} 

We propose the following generalization, assuming a positive answer to Question~\ref{qu:allfinitetensor}:
\begin{statement} \label{sta:Ore}
Let \( X \) be a connected Frobenius algebra object in a tensor category \( \mathcal{C} \), and assume that every coherent sublattice of its Frobenius subalgebra poset is distributive. Then there exists a minimal idempotent in \( \End_{\mathcal{C}}(X) \) that generates \( b_X = \id_X \).
\end{statement}
Since sublattices of distributive lattices are themselves distributive, it follows that if the Frobenius subalgebra poset of \( X \) is a distributive lattice, then every coherent sublattice is distributive. Motivated by this observation, we propose the following alternative to Statement~\ref{sta:Ore}: 
\begin{statement} \label{sta:Ore2}
Let \( X \) be a connected Frobenius algebra object in a tensor category \( \mathcal{C} \), and assume that its Frobenius subalgebra poset is a finite distributive lattice. Then there exists a minimal idempotent in \( \End_{\mathcal{C}}(X) \) that generates \( b_X = \id_X \).
\end{statement}


%
According to the ordering described before Lemma \ref{lem:idemposet}, an idempotent \( p \) is called \emph{minimal} if, for any nonzero idempotent \( q \) satisfying \( p \circ q = q = q \circ p \), it necessarily follows that \( p = q \). Moreover, by Proposition \ref{prop:biprojection}, a \emph{biprojection} is defined as the idempotent \( b_Y \) corresponding to a Frobenius subalgebra \( Y \subset X \). The biprojection \emph{generated} by an idempotent \( p \) is the minimal biprojection \( b \) such that \( b \ge p \).  

It would be valuable to establish a tensor-categorical analogue of Bisch’s theorem ~\cite[Theorem~3.2]{Bi94} (see also~\cite[Proposition~3.6]{BJ00}, \cite[Theorem~2 and Corollary~2.1]{Lan02}, and \cite[Theorem~4.12]{Liu16}), as this would yield a more transparent characterization of (generated) biprojections. The objective can be expressed through the following reformulation:

\begin{statement} \label{sta:liu}
A selfdual idempotent \( b \) corresponds to a Frobenius subalgebra if and only if \( b \ast b = \lambda b \), for some nonzero scalar \( \lambda \).
\end{statement}

%

\subsection{Dimension} \label{sub:Dim}

By \cite[Proposition 4.5.4]{EGNO15}, the Grothendieck ring of a tensor category is a transitive unital $\mathbb{Z}_+$-ring. Furthermore, by \cite[\S 3]{EGNO15}, any transitive unital $\mathbb{Z}_+$-ring of finite rank admits a unique unital ring homomorphism, the Frobenius-Perron dimension ($\FPdim$), which maps the basis elements to positive real numbers. However, without the finite-rank assumption, such a positive dimension function may not always exist. 

\begin{definition}  
A \emph{dimension} of a unital $\mathbb{Z}_+$-ring $R$ with basis $B$ is a unital ring homomorphism $\dim: R \to \mathbb{C}$, meaning a one-dimensional unital representation of $R$. A dimension is called \emph{positive} if $\dim(B) \subset \mathbb{R}_{>0}$.  
\end{definition}  

Given a unital $\mathbb{Z}_+$-ring $R$ with fusion data $(N_{i,j}^k)$, finding a dimension amounts to solving  
\[
\hspace*{-6cm} \text{(dimension equations)} \hspace*{2cm} d_i d_j = \sum_k N_{i,j}^k d_k, 
\]
with the normalization $d_1 = 1$ for the unit. Throughout this section, a tensor category is said to have a dimension $\dim$ if its Grothendieck ring does. Moreover, for a positive dimension, we will use the shorthand notation $|\cdot|$ instead of $\dim(\cdot)$, where $\cdot$ can refer to a basic element of the ring or an object of the category.

\begin{question}  
Let $R$ be a transitive unital $\mathbb{Z}_+$-ring with basis $(b_i)_{i \in I}$ and fusion data $(N_{i,j}^k)$ (or, if necessary, assume $R$ is the Grothendieck ring of a tensor category). Suppose that the fusion matrices $M_i = (N_{i,j}^k)_{k,j}$, acting on the Hilbert space $H = \ell^2(B)$, are bounded operators. Does there exist a positive dimension function $ |\cdot| $ on $ R $? If so, can it be chosen so that $ |b_i| = \| M_i \| $, where $ \| \cdot \| $ denotes the $ \ell^2 $-operator norm?
\end{question}  

The above question has an affirmative answer in the finite-rank case, as established by the Frobenius-Perron theorem. It should also hold if $ R $ is a limit of $\mathbb{Z}_+$-rings of finite rank.

A tensor category associated with a finite-index subfactor (refer to Example \ref{ex:subf}) qualifies as a finite tensor category if and only if the subfactor has a finite depth. In this case, the index is given by $\FPdim(X)$, where $ X $ is the corresponding Frobenius algebra object. In the infinite-depth case, a positive dimension still exist by Proposition \ref{prop:C*positive}.



\subsection{Euler's totient}  \label{sub:Euler} 

The traditional Euler's totient function, denoted as $\varphi(n)$, counts the number of positive integers up to $n$ that are relatively prime to $n$. Let $G$ be a finite group and $\mu$ the M\"obius function associated with its subgroup lattice $\mathcal{L}(G)$. In 1936, Philip Hall established in \cite{Hal36} that the Euler totient of a group $G$, as defined below, corresponds to the cardinality of the set $\{g \in G \ | \ \langle g \rangle = G\}$:
$$
\varphi(G) := \sum_{H \in \mathcal{L}(G)} \mu(H,G) |H|.
$$
Consequently, if $\varphi(G)$ is nonzero, then $G$ must be cyclic, and it follows that $\varphi(C_n) = \varphi(n)$. This result has been generalized to planar algebras as follows. 
\begin{theorem}[\cite{Pal18}] \label{thm:Euler}  
Let $\mathcal{P}$ be an irreducible subfactor planar algebra, and let $\mu$ denote the M\"obius function of its biprojection lattice $[e_1,\id]$. We introduce the Euler totient of the planar algebra $\mathcal{P}$ as follows:
$$
\varphi(\mathcal{P}) := \sum_{b \in [e_1,\id]} \mu(b,\id) |b:e_1|.
$$
If $\varphi(\mathcal{P})$ is nonzero, then there exists a minimal $2$-box projection that generates the identity biprojection.
\end{theorem}  

We propose the following generalization, which is subject to the same considerations outlined in \S \ref{sub:Ore}.

\begin{statement} \label{sta:Euler}
Let $ X $ be a connected Frobenius algebra object in a tensor category $ \mathcal{C} $ with dimension $ \dim $ (as described in \S \ref{sub:Dim}). Let $ \mu_{\mathcal{L}} $ represent the M\"obius function of a coherent sublattice $\mathcal{L}$ of the Frobenius subalgebra poset $[b_\one, b_X]$. We define the Euler totient of $ X $ with respect to $\mathcal{L}$ as follows:
$$
\varphi_{\mathcal{L}}(X) := \sum_{b_Y \in \mathcal{L}} \mu_{\mathcal{L}}(b_Y,b_X) \dim(Y).
$$
If $\varphi_{\mathcal{L}}(X)$ is nonzero for every coherent sublattice $\mathcal{L}$, then there is a minimal idempotent in $\End_{\mathcal C}(X)$ generating $b_X$. 
\end{statement}


\begin{question}
When $\varphi_{\mathcal{L}}(X)$ is nonzero for every coherent sublattice $\mathcal{L}$, can we also deduce that there exists a \emph{faithful} simple component $S$ of $X$, meaning that $X$ is a subobject of $S^{\otimes n}$ for sufficiently large $n$? 
\end{question}

The fact that $\varphi(n)$ is nonzero for any positive integer $n$ prompts the following speculation:
\begin{speculation} \label{spe:Euler}
If $[b_\one, b_X]$ is a distributive lattice, then $\varphi_{\mathcal{L}}(X)$ is nonzero for every coherent sublattice ${\mathcal{L}}$.
\end{speculation}  

We now present an application of Statement~\ref{sta:Euler}:
\begin{statement}
The minimal number of minimal idempotents required to generate \( b_X \) is bounded above by the minimal length \( \ell \) of a chain
\[
b_1 < b_2 < \cdots < b_{\ell+1} = b_X,
\]
such that \( \varphi_{\mathcal{L}}(b_i, b_{i+1}) \), defined analogously to~\cite[Definition~4.3]{Pal18}, is nonzero for every \( i \) and every \( \mathcal{L} \).
\end{statement}



\subsection{Depth} \label{sub:depth}
We aim to expand the concept of subfactor depth. Consider a Frobenius algebra object $ X $ in an abelian monoidal category $ \mathcal{C} $. We say that $ X $ has \emph{finite-depth} if the abelian monoidal subcategory it generates (which we can assume to be $ \mathcal{C} $ without loss of generality) contains a finite number of simple objects, up to isomorphism; otherwise, we call $ X $ \emph{infinite-depth}.
As stated in Remark \ref{rk:MugConverse}, we can express $ X $ as $ X = Y \circ Y^{\vee} $ for some object $ Y $ within a bimodule category, as detailed in Proposition \ref{prop:bimodYY}. In the finite-depth case, we define the \emph{depth} of $ X $ as the largest integer $ n = 2k $ or $ 2k+1 $ such that $ X^{\otimes k} $ or $ X^{\otimes k} \otimes Y $ has a new simple subquotient, up to isomorphism.

\begin{speculation} \label{spe:DepthWellDefined}
The concept of depth, as defined above, is independent of the choice of $ Y $. 
\end{speculation}



\begin{speculation} \label{spe:MaxDepth}
Let $ X_i $ be a Frobenius algebra in $\mathcal{C}_i$, with $i=1,2$. Then, $ X_1 \boxtimes X_2 $ is a Frobenius algebra in the Deligne tensor product $\mathcal{C}_1 \boxtimes \mathcal{C}_2$, and its depth is the maximum of the depths of $ X_1 $ and $ X_2 $.
\end{speculation}

According to \cite{Szy94}, a finite index irreducible subfactor is a Hopf $C^*$-algebra subfactor if and only if it has depth $2$. Drawing inspiration from this and Morita contexts of depth $2$ in \cite{Mug03}, we propose the following speculation:

\begin{speculation} \label{spe:fiber}
A finite tensor category admits a fiber functor if and only if it is generated by a connected Frobenius algebra of depth $2$.
\end{speculation}

Recall by \cite[Theorem 5.3.12]{EGNO15} that a finite tensor category admits a fiber functor if and only if it is equivalent to $\Rep(H)$ for some finite dimensional Hopf algebra $H$ (and see \cite[Theorem 5.4.1]{EGNO15} for the infinite case).

\subsection{Riemann hypothesis} \label{sub:RH} 
The sigma function $\sigma(n):=\sum_{d | n} d$ is defined as the sum of the positive divisors of $n$. Let $\gamma$ denote the Euler–Mascheroni constant. Then, we have \cite[Theorem 323]{HW08}:
$$
\limsup_{n \to \infty} \frac{\sigma(n)}{n \log \log n} = e^{\gamma}.
$$
In 1984, Guy Robin proved in \cite{Rob84} that the Riemann Hypothesis (RH) is true if and only if, for sufficiently large $n$,
$$
\sigma(n) < e^{\gamma} n \log \log n.
$$
Let $ X $ be a connected Frobenius algebra object in a tensor category $ \mathcal{C} $ with a positive dimension $ |\cdot| $ (see \S \ref{sub:Dim}). We define the sigma function of $ X $ as follows:  
$$ 
\sigma(X) := \sum_{d \in \mathcal{D}(X)} d, 
$$
where $ \mathcal{D}(X) $ is the divisor set of $ X $, defined as 
\[
\mathcal{D}(X) := \big\{ |Y| \;\big|\; Y \text{ is a Frobenius subalgebra of } X \big\}.
\] 
The set $ \mathcal{D}(X) $ is finite because the Grothendieck ring of a tensor category is a $ \mathbb{Z}_+ $-ring, as established more generally for multiring categories in \cite[\S 4.5]{EGNO15}, and $ |Y| $ depends only on the isomorphism class of $ Y $.   

We aim to extend the Riemann Hypothesis (RH) using Robin's reformulation within the class $ \mathfrak{C}_d $ of pairs $ (X, \mathcal{C}) $, where $ X $ is a connected Frobenius algebra object of finite depth $ d $ within the tensor category $ \mathcal{C} $ that it generates. By the definition of finite depth (see \S \ref{sub:depth}), $ \mathcal{C} $ is a finite tensor category and, as such, admits a positive dimension function $ |\cdot| = \FPdim $ (see \S \ref{sub:Dim}). This approach generalizes the quantum Riemann Hypothesis proposed in \cite{PalRHr}, which refines an earlier version from \cite{PalRH} that was subsequently disproven.  Let us speculate a (RH) of depth $ d $, denoted (RH$_d$).
\begin{speculation} \label{spe:RH}
For all $d \ge 2$, there is a constant $\gamma_d$ such that:
$$
\limsup_{(X,\mathcal{C}) \in \mathfrak{C}_d, \ |X| \to \infty} \frac{\sigma(X)}{|X| \log \log |X|}  = e^{\gamma_d},
$$
Furthermore, for all $(X,\mathcal{C}) \in \mathfrak{C}_d$ with $|X|$ large enough,
$$
\sigma(X) < e^{\gamma_d} |X| \log \log |X|.
$$
\end{speculation}

\begin{remark} \label{rk:RHinfty}
To speculate on an infinite-depth (RH), our approach would require restricting to $ X $ that generates a tensor category with a positive dimension $ |\cdot| $ (see \S \ref{sub:Dim}).
 \end{remark}

Let $ d \geq 2 $ be a positive integer, and let $ \mathcal{I}_d $ denote the set $ \big\{ |X| \;\big|\; (X,\mathcal{C}) \in \mathfrak{C}_d \big\} $. Using Speculation \ref{spe:MaxDepth}:
\begin{speculation} \label{spe:acc}
The (RH$_d$) as stated in Speculation \ref{spe:RH} implies that $\mathcal{I}_d$ is countable and has no accumulation points, and that $\gamma_d$ is strictly increasing in $d$.
\end{speculation}

\begin{question}
Can we further infer from (RH$_d$) that there is a minimum gap length in $\mathcal{I}_d$?
\end{question}

The following holds true for irreducible depth 2 subfactors, by Ore's theorem and the proof of Proposition \ref{prop:d2}.

\begin{speculation}  \label{spe:RHc}
Speculation \ref{spe:RH} can be simplified by focusing on the subclass $\mathfrak{D}_d$ which consists of $(X, \mathcal{C}) \in \mathfrak{C}_d$ where the Frobenius subalgebra poset of $X$ is a distributive lattice.
\end{speculation}



Given an intermediate subfactor $ N \subset P \subset M $, we obtain two subfactors: $ N \subset P $ and $ P \subset M $. However, the Frobenius subalgebra associated with the intermediate subfactor corresponds to $ N \subset P $ only. It is necessary to explore the categorical generalization of $ P \subset M $, and in the finite-depth case, the algebraic integer $ \frac{|M:N|}{|M:P|} = |P:N| $. 

\begin{speculation}  \label{spe:AlgInt}
Let $ Y $ be a Frobenius subalgebra of a finite-depth Frobenius algebra $ X $ in $ \mathcal{C} $. Then $ \frac{|X|}{|Y|} $ is an algebraic integer.
\end{speculation}

\begin{proposition} \label{prop:d2}
Assuming Speculations \ref{spe:fiber} and \ref{spe:AlgInt}, (RH$_2$) is equivalent to (RH).
\end{proposition}
\begin{proof}
In this case, $ \mathcal{D}(X) $ is a subset of $ \mathcal{D}(|X|) $, the set of divisors of the integer $|X|$. Consequently, $\sigma(X) \le \sigma(|X|)$. However, equality is attained when $ X $ is the connected Frobenius algebra associated with the irreducible cyclic group subfactor of index $ n = |X| $. This is due to the one-to-one correspondence among the divisors $ m $ of $ n $, the subgroups $ C_m $ of $ C_n $, and the Frobenius subalgebras.
\end{proof}




\section*{Acknowledgments}
We thank the anonymous referee for the insightful comments that improved the paper. We extend our gratitude to Keshab Chandra Bakshi, Dave Benson, Sebastian Burciu, Pavel Etingof, Jürgen Fuchs, Shamindra Ghosh, Dave Penneys, Maxime Ramzi, Ingo Runkel, Will Sawin, Kenichi Shimizu, Brian Shin, Serge Skryabin and Yasuo Watatani for insightful discussions. The names are listed in alphabetical order by first name. Mainak Ghosh is supported by the BJNSF (Grant No. 1S24063). Sebastien Palcoux is supported by the NSFC (Grant No. 12471031).

\end{document}